\documentclass[a4paper,12pt]{article}

\usepackage{amscd}
\usepackage{amssymb,amsfonts,amsmath,amsthm}
\usepackage{verbatim}

\usepackage{amsmath}
\usepackage{amsthm}
\usepackage{amssymb}

\usepackage{latexsym}
\usepackage{array}
\usepackage{enumerate}

\numberwithin{equation}{section}

\usepackage{amsmath,amssymb,mathrsfs,amsthm}
\usepackage[all]{xy}
\usepackage{graphicx}
\usepackage{ulem}
\usepackage{ascmac}
\usepackage{mathtools}

\newcommand{\Db}{{\bf D}^{\mathrm{b}}}
\newcommand{\Dbc}{{\bf D}_{\mathrm{c}}^{\mathrm{b}}}
\newcommand{\mer}{{\rm mero}}
\newcommand{\merc}{{\rm mero,c}}

\newcommand{\rHom}{\mathrm{RHom}}
\newcommand{\BDC}{{\mathbf{D}}^{\mathrm{b}}}

\newcommand{\Mod}{\mathrm{Mod}}

\newcommand{\CC}{\mathbb{C}}

\newcommand{\RR}{\mathbb{R}}
\newcommand{\QQ}{\mathbb{Q}}
\newcommand{\ZZ}{\mathbb{Z}}
\newcommand{\D}{\mathcal{D}}

\newcommand{\F}{\mathcal{F}}
\newcommand{\G}{\mathcal{G}}
\newcommand{\M}{\mathcal{M}}

\newcommand{\sho}{\mathcal{O}}
\newcommand{\R}{\mathcal{R}}

\newcommand{\SP}{\mathcal{P}}

\newcommand{\TT}{\mathcal{T}}

\newcommand{\CS}{\mathcal{S}}

\newcommand{\CH}{\mathcal{H}}
\newcommand{\CK}{\mathcal{K}}

\newcommand{\PP}{{\mathbb P}}

\newcommand{\LL}{{\mathbb L}}

\newcommand{\an}{{\rm an}}

\renewcommand{\dim}{{\rm dim}}

\newcommand{\e}{\varepsilon}

\newcommand{\supp}{{\rm supp}}

\newcommand{\grad}{{\rm grad}}

\newcommand{\distsq}{\delta}

\newcommand{\tl}[1]{\widetilde{#1}}
\newcommand{\uotimes}[1]{\otimes_{#1}}

\newcommand{\simto}{\overset{\sim}{\longrightarrow}}

\newcommand{\dsum}{\displaystyle \sum}

\newcommand{\op}{\mbox{\scriptsize op}}
\newcommand{\SD}{\mathcal{D}}
\newcommand{\SDop}{\mathcal{D}^{\mbox{\scriptsize op}}}
\newcommand{\SO}{\mathcal{O}}
\newcommand{\SA}{\mathcal{A}}
\newcommand{\SM}{\mathcal{M}}
\newcommand{\SN}{\mathcal{N}}
\newcommand{\SL}{\mathcal{L}}
\newcommand{\SK}{\mathcal{K}}
\newcommand{\SE}{\mathcal{E}}
\newcommand{\SF}{\mathcal{F}}
\newcommand{\SG}{\mathcal{G}}

\newcommand{\SU}{\mathcal{U}}
\newcommand{\SV}{\mathcal{V}}
\newcommand{\SW}{\mathcal{W}}

\newcommand{\Modcoh}{\mathrm{Mod}_{\mbox{\rm \scriptsize coh}}}
\newcommand{\Modhol}{\mathrm{Mod}_{\mbox{\rm \scriptsize hol}}}
\newcommand{\Modrh}{\mathrm{Mod}_{\mbox{\rm \scriptsize rh}}}

\newcommand{\BDCcoh}{{\mathbf{D}}^{\mathrm{b}}_{\mbox{\rm \scriptsize coh}}}

\newcommand{\BDChol}{{\mathbf{D}}^{\mathrm{b}}_{\mbox{\rm \scriptsize hol}}}
\newcommand{\BDCrh}{{\mathbf{D}}^{\mathrm{b}}_{\mbox{\rm \scriptsize rh}}}
\newcommand{\DD}{\mathbb{D}}
\newcommand{\Lotimes}[1]{\overset{L}{\otimes}_{#1}}
\newcommand{\Dotimes}{\overset{D}{\otimes}}

\newcommand{\Potimes}{\overset{+}{\otimes}}

\newcommand{\rhom}{{\rm R}{\mathcal{H}}om}
\newcommand{\rihom}{{\rm R}{\mathcal{I}}hom}
\newcommand{\Prihom}{{\rm R}{\mathcal{I}}hom^+}
\newcommand{\Prhom}{\rhom^+}
\newcommand{\I}{{\rm I}}
\newcommand{\che}[1]{\overset{\vee}{#1}}
\newcommand{\var}[1]{\overline{#1}}
\newcommand{\BEC}{{\mathbf{E}}^{\mathrm{b}}}
\newcommand{\Q}{\mathbf{Q}}
\newcommand{\EE}{\mathbb{E}}
 
\newcommand{\bs}{\backslash}

\newcommand{\inj}{``\varinjlim"}

\newcommand{\T}{{\rm T}}
\newcommand{\bfR}{\mathbf{R}}
\newcommand{\bfL}{\mathbf{L}}
\newcommand{\bfD}{\mathbf{D}}
\newcommand{\rmR}{{\rm R}}
\newcommand{\rmE}{{\rm E}}
\newcommand{\rmD}{{\rm D}}
\newcommand{\rmt}{{\rm t}}
\newcommand{\bfE}{\mathbf{E}}
\newcommand{\rmL}{{\rm L}}
\renewcommand{\Re}{\rm Re}

\newtheorem{theorem}{Theorem}[section]

\newtheorem{corollary}[theorem]{Corollary}
\newtheorem{lemma}[theorem]{Lemma}
\newtheorem{proposition}[theorem]{Proposition}

\theoremstyle{definition}
\newtheorem{definition}[theorem]{Definition}
\theoremstyle{remark}
\newtheorem{remark}[theorem]{\sc Remark}
\newtheorem{example}[theorem]{\sc Example}

\title{Fourier Transforms of Irregular Holonomic 
D-modules, \\ Singularities at Infinity of 
Meromorphic Functions \\ 
and Irregular Characteristic Cycles 
}

\author{Kiyoshi TAKEUCHI 
\footnote{Mathematical Institute, Tohoku University, 
Aramaki Aza-Aoba 6-3, Aobaku, Sendai, 980-8578, Japan. 
E-mail: takemicro@nifty.com} }

\date{}

\begin{document}
\maketitle

\begin{abstract}
Based on the recent developments in the irregular 
Riemann-Hilbert correspondence for holonomic D-modules 
and the Fourier-Sato transforms for enhanced 
ind-sheaves, we study the Fourier transforms of 
some irregular holonomic 
D-modules. For this purpose, the singularities 
of rational and meromorphic functions on complex 
affine varieties will be studied precisely, 
with the help of some new methods and tools such as meromorphic 
vanishing cycle functors. As a consequence, we show 
that the exponential factors and the irregularities 
of the  Fourier transform of a holonomic 
D-module are described geometrically by the stationary 
phase method, as in the classical case of 
dimension one. A new feature in the higher-dimensional 
case is that we have some extra  
rank jump of the Fourier transform 
produced by the singularities of the linear perturbations 
of the exponential factors at their points of 
indeterminacy. In the course of our study, 
not necessarily homogeneous Lagrangian cycles 
that we call irregular characteristic cycles 
will play a crucial role. 
\end{abstract}

\section{Introduction}\label{sec:1}

\indent The theory of Fourier transforms of 
D-modules is one of the most active areas in algebraic analysis. 
They interchange algebraic holonomic D-modules on the complex 
vector spaces $\CC^N$ with those on their duals. 
Until now, the case $N=1$ has been studied precisely by many 
mathematicians such as Bloch-Esnault \cite{BE04}, 
Malgrange \cite{Mal88}, Mochizuki 
\cite{Mochi10}, Sabbah \cite{Sab08} etc. 
On the other hand, after a groundbreaking 
development in the theory of irregular meromorphic 
connections by Kedlaya \cite{Ked10, Ked11} and Mochizuki 
\cite{Mochi11}, in \cite{DK16} D'Agnolo and Kashiwara 
established the Riemann-Hilbert correspondence for 
irregular holonomic D-modules 
(for another Riemann-Hilbert correspondence 
via filtered local 
systems, see also Sabbah \cite{Sab13}).  
For this purpose, 
they introduced enhanced ind-sheaves extending the 
classical notion of ind-sheaves introduced by 
Kashiwara-Schapira \cite{KS01}. 
Moreover in \cite{KS16-2}, 
Kashiwara and Schapira adapted this new notion to 
the Fourier-Sato transforms of Tamarkin \cite{Tama08} 
and developed a new theory of Fourier-Sato transforms 
for enhanced ind-sheaves which correspond to those 
for algebraic holonomic D-modules. Subsequently 
in \cite{DK17}, by using  
these results, D'Agnolo and Kashiwara 
studied Fourier transforms of 
holonomic D-modules on the affine line $\CC$ very precisely. 
In this case $N=1$, later 
Fourier transforms of regular and irregular holonomic 
D-modules were studied from various points of view by 
many authors such as D'Agnolo-Hien-Morando-Sabbah 
\cite{DHMS17}, Hohl \cite{hoh22}, 
Mochizuki \cite{Mochi18} and 
Barco-Hien-Hohl-Sevenheck \cite{BHHS22} etc. 
However, in contrast to these achievements in $N=1$, 
we know only very little in the 
higher-dimensional case $N \geq 2$. 
The aim of this paper is to clarify 
this situation by extending our previous 
results for Fourier transforms of 
regular holonomic D-modules in 
\cite{IT20a} and \cite{IT20b} to more general 
holonomic D-modules. 
For this purpose, we will study the singularities 
of rational and meromorphic functions on complex 
affine varieties precisely, 
by using some new methods and tools such as meromorphic 
vanishing cycle functors. 

\medskip 
\indent First, let us briefly recall the definition of 
Fourier transforms of algebraic D-modules. 
Let $X=\CC_z^N$ be a complex vector space 
and $Y=\CC_w^N$ its dual. 
We regard them as algebraic varieties and 
use the notations $\SD_X$ and $\SD_Y$ for 
the sheaves of the 
rings of ``algebraic'' differential operators on them. 
Denote by $\Modcoh(\SD_X)$ (resp. $\Modhol(\SD_X)$) 
the category of coherent (resp. holonomic) $\SD_X$-modules. 
Let $W_N := \CC[z, \partial_z]\simeq\Gamma(X; \SD_X)$ and 
$W^\ast_N := \CC[w, \partial_w]\simeq\Gamma(Y; \SD_Y)$ 
be the Weyl algebras over $X$ and $Y$, respectively. 
Then there exists a ring isomorphism 
\begin{equation}
W_N\simto W^\ast_N\hspace{30pt}
(z_i\mapsto-\partial_{w_i},\ \partial_{z_i}\mapsto w_i),  
\end{equation}
by which we can endow a left $W_N$-module $M$ with 
a structure of a left $W_N^\ast$-module.
We call it the Fourier transform of $M$ 
and denote it by $M^\wedge$. 
For a ring $R$ we denote by $\Mod_f(R)$ 
the category of finitely generated $R$-modules. 
Recall that for the affine algebraic varieties $X$ 
and $Y$ we have the equivalences of categories 
\begin{align} 
\Modcoh(\SD_X)
&\simeq
\Mod_f(\Gamma(X; \SD_X)) = \Mod_f(W_N),\\
\Modcoh(\SD_Y)
&\simeq
\Mod_f(\Gamma(Y; \SD_Y)) = \Mod_f(W^\ast_N)
\end{align}
obtained by taking global sections 
(see e.g. \cite[Propositions 1.4.4 and 1.4.13]{HTT08}).
Thus, for a coherent $\SD_X$-module $\SM\in\Modcoh(\SD_X)$
we can define its Fourier 
transform $\SM^\wedge\in\Modcoh(\SD_Y)$. 
It follows that we obtain an equivalence of categories
\begin{equation}
( \cdot )^\wedge : \Modhol(\SD_X)\simto \Modhol(\SD_Y)
\end{equation}
between the subcategories of holonomic D-modules 
(see e.g. \cite[Proposition 3.2.7]{HTT08} for the details). 
Although the definition of Fourier transforms of 
holonomic D-modules is so simple, in general it is 
hard to describe their properties. 
First of all, the Fourier transform $\SM^\wedge$ of a regular 
holonomic $\SD_X$-module $\SM$ is 
not necessarily regular. For the regularity of 
$\SM^\wedge$ we need some very strong condition on $\SM$. 
Let $X^{\an} = \CC^N$ be the underlying complex 
manifold of $X= \CC^N$ that we sometimes denote by $X$ 
for simplicity. 
Recall that an algebraic constructible sheaf 
$\SF \in\Dbc (X^{\an}):= \Dbc (\CC_{X^{\an}})$ on $X^{\an}= \CC^N$ 
is called monodromic if 
its cohomology sheaves are locally constant 
on each $\CC^*$-orbit in $X^{\an}= \CC^N$ (see 
Verdier \cite{Ver83}). 
Then the following beautiful theorem 
is due to Brylinski \cite{Bry86}. 

\begin{theorem}[{\rm (Brylinski \cite{Bry86})}]\label{th-Bry}
Let $\SM$ be an algebraic regular holonomic 
D-module on $X= \CC^N$. Assume that its 
solution complex $Sol_{X}(\SM)$ is 
monodromic. Then its 
Fourier transform $\SM^\wedge$ is regular 
and $Sol_{Y}(\SM^\wedge)$ is monodromic. 
\end{theorem}

In \cite{IT20a} and \cite{IT20b}, removing the 
monodromicity assumption in this theorem, 
the author and Ito studied the Fourier transforms of general 
regular holonomic D-modules $\SM$ on $X= \CC^N$ and 
described their smooth loci, exponential factors, irregularities and 
characteristic cycles etc. in terms of the geometry 
of $\SM$. Moreover, as was clarified in 
\cite{IT20b}, if $\SM$ is regular holonomic then 
$Sol_{Y}(\SM^\wedge)$ is monodromic. 
For the other important contributions in the 
regular case, see also Daia \cite{Dai00}. 
In this paper, removing also 
the regularity assumption 
in \cite{Dai00}, \cite{IT20a} and \cite{IT20b}, 
we study the Fourier transforms of more general 
holonomic D-modules $\SM$. 
Namely, we aim at finding a way to get a unified 
generalization of the results in \cite{DK17} 
and \cite{IT20a}. 

\medskip 
\indent Now let us explain our mains results 
more precisely. Mostly 
in this paper, we consider the following special but basic 
holonomic D-modules. For a rational function 
$f= \frac{P}{Q}: X \setminus Q^{-1}(0) 
\longrightarrow \CC$ ($P, Q \in 
\Gamma(X; \SO_X) \simeq \CC [z_1,z_2, \ldots, z_N]$, 
$Q \not= 0$) on $X= \CC^N_z$ we set 
$U:= X \setminus Q^{-1}(0)$ and 
define an 
exponential $\SD_X$-module $\SE^f_{U|X} 
\in\Modhol (\SD_X)$ as in the analytic case 
(see Subsection \ref{sec:7}). In what 
follows, we always assume that $P$ and $Q$ are 
coprime. Then $I(f):= P^{-1}(0) \cap Q^{-1}(0)
 \subset X= \CC^N$ 
is noting but the set of the points of 
indeterminacy of the rational function 
$f= \frac{P}{Q}$. 

\begin{definition}\label{defi-etdint}
We say that a holonomic $\SD_X$-module 
$\SM \in\Modhol(\SD_X)$ is an exponentially twisted holonomic 
D-module if there exist a regular holonomic 
$\SD_X$-module $\SN \in\Modrh(\SD_X)$ and a rational function 
$f= \frac{P}{Q}: U=X \setminus Q^{-1}(0) 
\longrightarrow \CC$ ($P, Q \in 
\Gamma(X; \SO_X), 
Q \not= 0$) on $X= \CC^N_z$ such that we have an isomorphism 
\begin{equation}
\SM \simeq 
\SN \Dotimes \SE^f_{U|X}, 
\end{equation}
where 
for the hypersurface $D:=Q^{-1}(0) \subset X$ we have 
$\bigl( \SN \Dotimes \SE^f_{U|X} \bigr)(*D) 
\simeq \SN \Dotimes \SE^f_{U|X}$ and hence the right 
hand side is concentrated in degree $0$. 
\end{definition}

Exponentially twisted holonomic 
D-modules can be considered as 
natural prototypes or building blocks 
for general holonomic D-modules, in view of 
the recent progress in \cite{Ked10, Ked11}, 
\cite{Mochi11} and \cite{DK16}. 
Since for the moment there is no efficient 
way to describe the enhanced solution complexes 
of general holonomic D-modules in higher 
dimensions $N \geq 2$, exponentially twisted holonomic 
D-modules are the most typical 
holonomic D-modules to which the 
general theory of \cite{KS16-2} is 
applicable. From now on, we fix an 
exponentially twisted holonomic $\SD_X$-module 
$\SM \in\Modhol(\SD_X)$ such that 
$\SM \simeq 
\SN \Dotimes \SE^f_{U|X}$ 
for a regular holonomic 
$\SD_X$-module $\SN \in\Modrh(\SD_X)$ and a rational function 
$f= \frac{P}{Q}: U= X \setminus Q^{-1}(0) 
\longrightarrow \CC$ on $X$ and explain 
our results on its 
Fourier transform $\SM^\wedge$. 
Let us call $f$ the exponential factor of $\SM$. We set 
\begin{equation}
K:= Sol_X ( \SN ) \in \Dbc(X^{\an}). 
\end{equation}
Here we use the convention that 
its shift $K[N] \in \Dbc(X^{\an})$ 
is a perverse sheaf on $X^{\an}$ and 
denote the support of $K \in \Dbc(X^{\an})$ 
by $Z \subset X$. Let $\pi : X^{\an} \times \RR 
\longrightarrow X^{\an}$ be the projection. 
Then for the enhanced sheaf 
\begin{equation}
F:= \pi^{-1}K \otimes {\rm E}_{U^{\an}|X^{\an}
}^{{\rm Re}f} = 
\pi^{-1}K \otimes 
\Big( \CC_{\{ (z,t) \in X \times \RR \ | \ z \in U, 
t + {\rm Re }f(z) \geq 0 \}} \Big)
\in \BEC_+(\CC_{X^{\an}})
\end{equation}
on the underlying real analytic manifold $X_{\RR}
= X^{\an}$ of $X$, we have an isomorphism   
\begin{equation}\label{eq-sit} 
Sol_{X}^\rmE( \SM ) \simeq 
\CC^\rmE_{X^{\an}} \Potimes F
\simeq 
\underset{a\to+\infty}{\inj} \ 
\pi^{-1}K \otimes 
\Big( 
\CC_{\{ (z,t) \in X \times \RR \ | \ z \in U, 
t + {\rm Re }f(z) \geq a \}} \Big)
\end{equation}
of enhanced ind-sheaves on $X_{\RR}$ 
(see Section \ref{uni-sec:2}). We define 
its enhanced micro-support ${\rm SS}^{{\rm E}}(F) \subset 
(T^*X_{\RR}) \times \RR$ and the reduced one 
${\rm SS}_{{\rm irr}}(F) 
\subset T^*X_{\RR}$ as in D'Agnolo-Kashiwara \cite{DK17} 
and Tamarkin \cite{Tama08}. 
In this paper, we call ${\rm SS}_{{\rm irr}}(F)$ 
the irregular micro-support of $F$. 
Note that for the (not necessarily homogeneous) 
complex Lagrangian submanifold 
\begin{equation}
\Lambda^f := 
\{ (z, df(z)) \ | \ z \in U=X \setminus Q^{-1}(0) \} 
\subset T^*U 
\end{equation}
of $T^*U$ via the natural identification $(T^*U)_{\RR} 
\simeq T^*U_{\RR}$ we have 
\begin{equation}\label{lgsvforf}
{\rm SS}_{{\rm irr}}(F) \cap T^*U_{\RR}
= \bigl( {\rm SS}(K) \cap T^*U_{\RR} \bigr) + \Lambda^f. 
\end{equation}
Moreover we can show that ${\rm SS}_{{\rm irr}}(F)$ 
is contained in a complex isotropic analytic 
subset of $(T^*X)_{\RR} \simeq T^*X_{\RR}$ 
(see Lemma \ref{iso-analy}). 
However, for the moment it is not clear for us if 
${\rm SS}_{{\rm irr}}(F)$ itself 
is a complex Lagrangian analytic subset or not. 
This prevents us from applying the theory of 
Fourier-Sato transforms for enhanced ind-sheaves 
developed by Kashiwara-Schapira \cite{KS16-2} 
especially when $I(f)=P^{-1}(0) \cap Q^{-1}(0) 
\not= \emptyset$. 
To overcome this difficulty, we define a (not necessarily 
homogeneous) complex Lagrangian analytic 
subset ${\rm SS}^{\CC}_{{\rm irr}}( \SF )$ 
of $(T^*X)_{\RR} \simeq T^*X_{\RR}$ 
in a different way 
as follows and use it instead of ${\rm SS}_{{\rm irr}}(F)$. 
First, as a complex analogue of the 
$\RR$-constructible sheaf 
$F= \pi^{-1}K 
\otimes {\rm E}_{U^{\an}|X^{\an}}^{{\rm Re}f}$ 
on $X_{\RR} \times \RR$, by
the (not necessarily) closed embedding 
\begin{equation}
i_{-f}: U= X \setminus Q^{-1}(0) \hookrightarrow 
X \times \CC, \qquad (z \longmapsto (z,-f(z)))
\end{equation}
associated to the rational function 
$-f:U \longrightarrow \CC$ we set 
\begin{equation}
\SF := (i_{-f})_! (K|_U) \quad 
\in \Db (X \times \CC ).
\end{equation}
Then we can easily show that 
it is a constructible sheaf 
on $X \times \CC$ and hence its micro-support 
${\rm SS} ( \SF )$ 
is a homogeneous complex Lagrangian 
analytic subset of $T^*(X \times \CC )$. 
Then as in the definitions of ${\rm SS}^{{\rm E}}(F)
\subset (T^*X_{\RR}) \times \RR$ 
and ${\rm SS}_{{\rm irr}}(F) \subset T^*X_{\RR}$, 
by forgetting the homogeneity of ${\rm SS} ( \SF )$ 
we define the following subsets: 
\begin{equation}
{\rm SS}^{{\rm E}, \CC}( \SF )
\subset T^*X \times \CC, \qquad 
{\rm SS}_{{\rm irr}}^{\CC}( \SF ) \subset T^*X. 
\end{equation}
See Section \ref{sec:19} for the details. 
It is clear that ${\rm SS}^{{\rm E}, \CC}( \SF )$ 
is a complex analytic subset of $(T^*X) \times \CC$. 
Let us call it the enhanced micro-support of 
$\SF \in \Dbc (X \times \CC )$. We can also show that 
${\rm SS}_{{\rm irr}}^{\CC}( \SF )$ is a 
(not necessarily 
homogeneous) complex Lagrangian analytic 
subset of $(T^*X)_{\RR} \simeq T^*X_{\RR}$ and call it 
the irregular micro-support of $\SF$. 
As a byproduct of the proof, we obtain also a 
(not necessarily 
homogeneous) Lagrangian cycle in 
$T^*X \simeq X \times Y$ supported by 
${\rm SS}_{{\rm irr}}^{\CC}( \SF ) \subset T^*X$ 
(see the proof of Lemma \ref{iso-analy-new}). 
As we can show that it depends only on $\SM$, 
we call it the irregular characteristic 
cycle of $\SM$ and denote it by ${\rm CC}_{{\rm irr}}( \SM )$. 
For a point $w \in Y= \CC^N$ we define a 
rational function $f^{w}$ on $X= \CC^N$ by 
\begin{equation}
f^{w} : U=X \setminus Q^{-1}(0) \longrightarrow \CC \qquad 
(z \longmapsto   \langle  z, w \rangle -f(z)),  
\end{equation}
where for $z=(z_1, \ldots, z_N) \in X$ and 
$w=(w_1, \ldots, w_N) \in Y$ we set 
\begin{equation}
\langle z, w \rangle := \sum_{i=1}^N z_i w_i \qquad \in \CC. 
\end{equation}
Let
\begin{equation}
X\overset{p}{\longleftarrow}X\times 
Y\overset{q}{\longrightarrow}Y
\end{equation}
be the projections. 
Then by using our meromorphic vanishing cycle functors 
$\phi^{\merc}_{f^{w}-c}( \cdot ):  \Dbc(X^{\an})
\longrightarrow 
\Dbc(X^{\an})$ ($c \in \CC$) with compact support 
(see Section \ref{sec:s2} for their definition 
and basic properties), 
we define the following subset of 
$(T^*X)_{\RR} \simeq T^*X_{\RR}$. 

\begin{definition}\label{defi-13int}
For the enhanced sheaf $F= \pi^{-1}K 
\otimes {\rm E}_{U^{\an}|X^{\an}}^{{\rm Re}f} \in 
\BEC_+(\CC_{X^{\an}})$ we define a subset 
${\rm SS}_{{\rm eva}}(F)
\subset T^*X_{\RR}$ of $T^*X_{\RR}
\simeq (T^*X)_{\RR}$ by 
\begin{equation}
{\rm SS}_{{\rm eva}}(F):= \{ (z,w) \in (T^*X)_{\RR} 
\ | \ \phi^{\merc}_{f^{w}-c}(K)_{z} \not= 0 
\ \text{for some} \ c \in \CC \}. 
\end{equation}
\end{definition}

Note that we can calculate the stalks of 
the meromorphic vanishing cycles 
and hence ${\rm SS}_{{\rm eva}}(F)$ 
by using resolutions of singularities of 
$f$ and the results in 
\cite[Section 5]{MT11} (see 
the final part of Section \ref{sec:s2}). 
We then will show that 
${\rm SS}_{{\rm eva}}(F)$ is contained in 
$\Lambda:= {\rm SS}^{\CC}_{{\rm irr}}( \SF )$ 
and there exists a non-empty Zariski open 
subset $\Omega \subset Y$ of $Y= \CC^N$ 
such that the restriction of the projection 
$q: T^*X \simeq X \times Y \simeq T^*Y \longrightarrow Y$ 
to $\Lambda= {\rm SS}^{\CC}_{{\rm irr}}( \SF ) \subset T^*X$ 
is an unramified finite covering over 
$\Omega \subset Y$ and satisfies the following properties. 
For the precise construction of 
$\Omega$ see Lemma \ref{lem-fin-cover}.  
Let $V \subset \Omega^{\an}$ be 
a contractible open subset 
of $\Omega^{\an}$ so that for the decomposition 
\begin{equation}
q^{-1}(V)\cap \Lambda
= \Lambda_{V,1} 
\sqcup \Lambda_{V,2} \sqcup \cdots \cdots \sqcup \Lambda_{V,d}
\end{equation}
of $q^{-1}(V)\cap \Lambda$ into its connected components 
$\Lambda_{V,i}$ ($1 \leq i \leq d$) we have 
$\Lambda_{V,i} \simto V$ for any $1 \leq i \leq d$. 
Then by our choice of $\Omega \subset Y= \CC^N$, 
$Z_i:=p ( \Lambda_{V,i} ) \subset X= \CC^N_z$ 
($1 \leq i \leq d$) are complex submanifolds of 
$X= \CC^N$ and we can renumber them 
so that for some $1 \leq r \leq d$ we have 
$Z_i \subset U=X \setminus Q^{-1}(0)$ (resp. 
$Z_i \subset I(f)=P^{-1}(0) \cap Q^{-1}(0)$) if $1 \leq i \leq r$ 
(resp. if $r+1 \leq i \leq d$). Let 
$g_i:V \longrightarrow \CC$ ($1 \leq i \leq d$) 
be holomorphic functions on 
$V \subset \Omega \subset Y= \CC^N$ such that 
\begin{equation}
\chi ( \Lambda_{V,i} )
= \Lambda^{g_i}:= \{ (w, dg_i(w)) \ | \ w \in V \} 
\subset T^*Y, 
\end{equation}
where we used the symplectic transformation 
of \cite{DK17}: 
\begin{equation}
\chi : T^*X \simto T^*Y \qquad ((z,w) \longmapsto (w,-z)). 
\end{equation}
Note that such $g_i:V \longrightarrow \CC$ are 
uniquely defined modulo constant functions 
on $V$. For $1 \leq i \leq d$ and $w \in V$ let 
$\zeta^{(i)} (w) \in X= \CC^N_z$ be the 
unique point of $Z_i \subset X$ such that 
$(\zeta^{(i)}(w), w) \in  \Lambda_{V,i} 
\subset \Lambda$ so that we have 
\begin{equation}
\Lambda_{V,i} = 
\{ (\zeta^{(i)}(w), w) \ | \ w \in V \} 
\subset T^*Y \simeq T^*X. 
\end{equation}
Then by \eqref{lgsvforf} 
for any $1 \leq i \leq r$ and $w \in V$ the point 
\begin{equation}
( \zeta^{(i)}(w), df^w( \zeta^{(i)}(w)))= 
( \zeta^{(i)}(w), w- df( \zeta^{(i)}(w))) \in 
(T^*X)_{\RR} 
\end{equation}
is contained in the smooth part of ${\rm SS}(K|_U)
\subset (T^*U)_{\RR}$. 
We denote by $m(i) \geq 1$ the multiplicity of 
the regular holonomic 
$\SD_X$-module $\SN \in\Modrh(\SD_X)$ (or 
of the perverse sheaf $K[N]$) there. 
In fact, for any such $i$ we have an equality 
\begin{equation}
g_i(w) \equiv 
f(\zeta^{(i)}(w))- \langle \zeta^{(i)}(w), w \rangle 
=- f^w(\zeta^{(i)}(w)) \qquad (w \in V) 
\end{equation}
modulo constant functions on $V$ and can define 
$g_i: V \longrightarrow \CC$ by 
\begin{equation}
g_i(w) := - f^w(\zeta^{(i)}(w)) \qquad (w \in V).  
\end{equation}

Also for $r+1 \leq i \leq d$ and $w \in V$, we can show that 
there exist only finitely many $c \in \CC$ 
such that 
\begin{equation}
\phi^{\merc}_{f^w -c}( K)_{\zeta^{(i)}(w)} \not= 0
\end{equation}
and for any such $c \in \CC$ we have a concentration 
\begin{equation}
H^j \phi^{\merc}_{f^w -c}( K)_{\zeta^{(i)}(w)} \simeq 0 
\qquad (j \not= N-1). 
\end{equation}
Moreover, by our choice of $\Omega \subset Y= \CC^N$ 
the dimension of the only non-trivial 
cohomology group 
$H^{N-1} \phi^{\merc}_{f^w -c}( K)_{\zeta^{(i)}(w)} \not= 0$ 
is constant with respect to $w \in V \subset \Omega$. 
For $r+1 \leq i \leq d$ we thus can set 
\begin{equation}
m(i) := 
\dsum_{c \in \CC} {\rm dim } 
H^{N-1} \phi^{\merc}_{f^w -c}( K)_{\zeta^{(i)}(w)} 
\geq 1. 
\end{equation}
Namely for $r+1 \leq i \leq d$ 
the multiplicity $m(i) \geq 1$ is defined 
by the singularities of the rational functions 
$f^w$ ($w \in V$) at their points of 
indeterminacy $\zeta^{(i)}(w) \in I(f^w)= I(f)= 
P^{-1} \cap Q^{-1}(0)$. In fact, for any such $i$ 
we can also show that there exist (distinct) constants 
$a_1,a_2, \ldots, a_{n_i} \in \CC$ such that 
we have 
\begin{equation}
\{ c \in \CC \ | \ 
 \phi^{\merc}_{f^w -c}( K)_{\zeta^{(i)}(w)} \not= 0 \} 
= 
\{ a_1-g_i(w), a_2-g_i(w), \ldots, a_{n_i}-g_i(w) \} 
\end{equation}
for any $w \in V$. 

Let $i_Y : Y=\CC^N\xhookrightarrow{\ \ \ }\var{Y}=\PP^N$ 
be the projective compactification of $Y$. 
We extend the Fourier transform $\SM^{\wedge} \in\Modhol(\SD_Y)$ 
to the holonomic D-module 
$\tl{\SM^{\wedge}} := i_{Y \ast} ( \SM^{\wedge} ) 
\simeq \bfD i_{Y\ast} ( \SM^{\wedge} )$ 
on $\var{Y}$. 
Let $\var{Y}^{\an}$ be the underlying complex manifold 
of $\var{Y}$ and 
define the analytification 
$\tl{\SM^{\wedge}}^{\an}\in\Modhol(\SD_{\var{Y}^{\an}})$ 
of $\tl{\SM^{\wedge}}$ by 
$\tl{\SM^{\wedge}}^{\an} := \SO_{\var{Y}^{\an}}
\otimes_{\SO_{\var{Y}}}\tl{\SM^{\wedge}}$. 
Then we have the following 
formula for the enhanced solution complex 
\begin{equation}
Sol_{\var Y}^\rmE(\tl{\SM^{\wedge}}) := 
Sol_{\var{Y}^{\an}}^\rmE( \tl{\SM^{\wedge}}^{\an} )
\in \BEC(\I\CC_{\var{Y}^{\an}})
\end{equation}
of $\tl{\SM^{\wedge}}^{\an}$. 

\begin{theorem}\label{th-A} 
In the situation as above, we have an isomorphism
\begin{align*}
\pi^{-1}\CC_V \otimes
\Big(Sol_{\var Y}^\rmE( \tl{\SM^\wedge} )\Big)
& \simeq 
\bigoplus_{i=1}^d 
\bigl( \EE_{V^{\an}| \var{Y}^{\an}}^{{\rm Re}g_i}
\bigr)^{\oplus m(i)}
\\ 
& \simeq 
\bigoplus_{i=1}^d
\Big(\underset{a\to+\infty}{\inj}\ 
\CC_{\{ (w,t) \in \var{Y}^{\an} \times \RR \ | \ 
w \in V, \ 
t + {\rm Re }g_i(w) \geq a \}} \Big)^{\oplus m(i)}
\end{align*}
of enhanced ind-sheaves on $\var{Y}^{\an}$. 
In particular, the restriction 
$\SM^\wedge |_{\Omega}$ of the Fourier 
transform $\SM^\wedge$ to $\Omega \subset Y= \CC^N$ is an 
algebraic integrable connection of rank 
$\sum_{i=1}^d m(i)$. 
\end{theorem}

By Theorem \ref{th-A} we can describe the 
exponential factors and the irregularities 
of the Fourier transform $\SM^{\wedge}$ of $\SM$ 
along various submanifolds of $Y= \CC^N$. 
See Section \ref{sec:9} for the details. 
By patching the 
(not necessarily homogeneous) Lagrangian cycles  
\begin{equation}
\dsum_{i=1}^d  m(i) \cdot 
\bigl[ \Lambda_{V,i} \bigr]
\end{equation}
in the open subsets $q^{-1}(V)= X \times V \subset 
X \times Y \simeq T^*X$ for various 
$V \subset \Omega$ we obtain a Lagrangian cycle 
globally defined in 
$q^{-1}( \Omega )= X \times \Omega \subset 
T^*X$. By our definitions of the multiplicities $m(i) \geq 1$, 
it turns out that it coincides with the restriction of 
the irregular characteristic 
cycle ${\rm CC}_{{\rm irr}}( \SM )$ of $\SM$ 
to $q^{-1}( \Omega )= X \times \Omega \subset T^*X$. 
Then the last assertion of Theorem \ref{th-A} means that 
the generic rank of the holonomic D-module 
$\SM^\wedge$ is equal to the covering degree of 
${\rm CC}_{{\rm irr}}( \SM )$ over 
$\Omega \subset Y= \CC^N$. We thus 
find that if $I(f)= P^{-1} \cap Q^{-1}(0) 
\not= \emptyset$ there may exist some extra  
rank jump of $\SM^\wedge$ produced by 
the singularities of the rational functions 
$f^w$ ($w \in \Omega$) i.e. the linear perturbations 
of the exponential factor $f$ of 
$\SM$. This shows that the structures of 
the Fourier transforms of irregular 
holonomic D-modules are much more involved 
than those of the regular ones studied in 
\cite{Bry86}, \cite{Dai00}, \cite{IT20a} 
and \cite{IT20b} etc. 

\medskip 
\indent Our proof of Theorem \ref{th-A} is 
similar to that of \cite[Theorem 4.4]{IT20a} and 
relies on a Morse theory for the Morse functions 
${\rm Re }f^w: U=X \setminus Q^{-1}(0) 
\longrightarrow \RR$ ($w \in \Omega$) and 
$K|_U \in \Dbc (U^{\an})$ as well as 
the theory of Fourier-Sato transforms 
for enhanced ind-sheaves of 
Kashiwara-Schapira \cite{KS16-2}. 
However, it requires also 
more careful analysis on the singularities 
of the rational functions 
$f^w: U=X \setminus Q^{-1}(0) 
\longrightarrow \CC$ ($w \in \Omega$) 
at infinity and at their points of 
indeterminacy. To treat their singularities at 
infinity, we extend the classical notion of 
tameness at infinity introduced by 
Broughton \cite{Bro88} for polynomial 
functions on $\CC^N$ to that for 
rational and meromorphic functions on 
smooth subvarieties of $\CC^N$ and 
obtain a transversality theorem (see 
Proposition \ref{prop-7}) similar to 
the ones proved by N{\'e}methi and Zaharia 
\cite{NZ90}, \cite{NZ92} 
(see also \cite[Sections 2 and 3]{NST21} for 
related results). 
Let $\CS$ be an algebraic Whitney stratification of $Z \cap U 
\subset U=X \setminus Q^{-1}(0)$ adapted to 
$K|_U \in \Dbc ( U)$ such that 
\begin{equation}
{\rm SS }(K|_U) \subset \bigcup_{S \in \CS}T^*_SX 
\end{equation}
Then in Lemma \ref{lem-31} we will show that 
for any $w \in \Omega$ and stratum 
$S \in \CS$ in $\CS$ such that $T^*_SX 
\subset {\rm SS }(K|_U)$ the restriction 
$f^w|_{S}:S \longrightarrow \CC$ 
of $f^w$ to $S \subset U$ is tame 
at infinity in our (generalized) sense. 
This implies that the morphisms 
$f^w: U \longrightarrow \CC$ ($w \in \Omega$) 
and $K|_U \in \Dbc (U^{\an})$ 
satisfy one of the conditions of Kashiwara and Schapira 
in their non-proper direct image 
theorem \cite[Theorem 4.4.1]{KS85} 
for the family of open subsets 
$U_r \subset U$ ($r>0$) of $U$ defined by 
\begin{equation}
U_r:= \{ z \in U \ | \ 
 ||z|| <r \} \subset U \qquad (r>0)
\end{equation}
(see Proposition \ref{prop-micros}). We 
then show that for the given rational function 
$f: U \longrightarrow \CC$ we can apply 
\cite[Theorem 4.4.1]{KS85} to its 
generic linear perturbations 
$\tilde{f} = -f^w: U \longrightarrow \CC$ 
($w \in \Omega$) and $K|_U \in \Dbc (U^{\an})$ 
to obtain the following byproduct
of our study of Fourier transforms 
(see Proposition \ref{prop-micros-new} and Lemma \ref{lem-31-new}). 

\begin{proposition}\label{th-B} 
Let $f : U \longrightarrow \CC$ and 
$K \in \Dbc (X^{\an})$ be as above. 
Then for generic linear perturbations 
$\tilde{f} : U \longrightarrow \CC$ of $f$ we have: 
For any $\tau_0 \in \CC$ 
there exist $R \gg 0$ and 
$0< \varepsilon \ll 1$ such that for the 
inclusion maps $i_r: U_r \hookrightarrow U$ 
($r>R$) we have isomorphisms 
\begin{equation}
 \rmR \tilde{f}_! (i_r)_! (i_r)^{-1} (K|_U) 
\simto  \rmR \tilde{f}_! (K|_U) 
\end{equation}
on the open subset $\{ \tau \in \CC \ | \ 
 | \tau - \tau_0| < \varepsilon \} \subset \CC$ of $\CC$. 
Moreover for such $\tilde{f}$ we have 
\begin{equation}
{\rm SS }( \rmR \tilde{f}_! (K|_U) ) \subset 
\varpi \rho^{-1} {\rm SS }(K|_U), 
\end{equation}
where we used the natural morphisms 
\begin{equation}
T^* \CC 
\overset{\varpi}{\longleftarrow}
U \times_{\CC} T^* \CC 
\overset{\rho}{\longrightarrow}
T^* U
\end{equation}
associated to 
$\tilde{f} : U \longrightarrow \CC$. 
\end{proposition}

This in particular implies that for any 
algebraic constructible sheaf 
$K \in \Dbc (X^{\an})$ the geometric condition 
of Kashiwara and Schapira 
in \cite[Theorem 4.4.1]{KS85} is satisfied 
by generic polynomial functions on $X= \CC^N$. 
To treat the singularities of $f^w: U 
\longrightarrow \CC$ ($w \in \Omega$) 
at their points of indeterminacy, 
we develop a general theory of meromorphic 
vanishing cycle functors (with compact 
support) by modifying the similar ones 
studied previously by Raibaut \cite{Rai13} 
and Nguyen-Takeuchi \cite{NT22}. 
See Section \ref{sec:s2} for the details. 

\medskip 
\indent It is clear that we can readily 
extend Theorem \ref{th-A} to all 
holonomic D-modules on $X= \CC^N$ having  
only exponentially twisted composition factors. 
 Recall that the Fourier 
transform is an exact functor. 
Moreover, if for a holonomic D-module 
$\SN$ on $X= \CC^N$ there exists an 
enhanced sheaf $G$ globally defined 
on $X_{\RR}$ and having 
a local expression similar to that of 
the above one $F$ such that 
\begin{equation}\label{eq-sitit} 
Sol_{X}^\rmE( \SN ) \simeq 
\CC^\rmE_{X^{\an}} \Potimes G, 
\end{equation}
then we can obtain a formula for 
the enhanced ind-sheaf 
\begin{equation}
\pi^{-1}\CC_V \otimes
\Big(Sol_{\var Y}^\rmE( \tl{\SN^\wedge} )\Big)
\quad \in \BEC(\I\CC_{\var{Y}^{\an}})
\end{equation}
($V$ is an open subset of $Y= \CC^N$) on $\var{Y}^{\an}$ 
along the same line 
as in the proof of Theorem \ref{th-A}. 
Recall that in \cite{DK16} D'Agnolo and 
Kashiwara proved that 
for a complex manifold $X$ and 
an exponential D-module $\mathscr{E}_{U | X}^f$ 
on it associated to a meromorphic function 
$f\in\SO_X(\ast D)$ along 
a closed hypersurface 
$D \subset X$ and $U=X \setminus D$ 
we have an isomorphism 
\begin{equation}
Sol_X^\rmE\big( \mathscr{E}_{U | X}^f \big) 
\simeq \EE_{U | X}^{\Re f} 
\simeq 
\underset{a\to+\infty}{\inj} \ 
\Big( 
\CC_{\{ (z,t) \in X \times \RR \ | \ z \in U, 
t + {\rm Re }f(z) \geq a \}} \Big)
\end{equation}
of enhanced ind-sheaves on $X^{\an}=X_{\RR}$. 
This formula for the enhanced solution complex 
$Sol_X^\rmE\big( \mathscr{E}_{U | X}^f \big)$ played 
a central role in the proof of the main results 
in \cite{DK16}. In Section \ref{sec:11} 
we shall try to extend it to arbitrary meromorphic 
connections. Our argument is a 
higher-dimensional analogue of those 
of Kashiwara-Schapira \cite[Section 7]{KS03} 
and Morando \cite[Section 2.1]{Morando}. We believe 
that this approach would lead us to a 
full generalization of Theorem \ref{th-A} to arbitrary 
holonomic D-modules on $X= \CC^N$.

\bigskip
\noindent{\bf Acknowledgement:} 
The author thanks Professors Tat Thang Nguyen, 
Adam Parusinski and J{\"o}rg Sch{\"u}rmann for 
several discussions with them  
during the preparation of this paper.

\section{Meromorphic Nearby and Vanishing Cycle 
\\ Functors}\label{sec:s2}

In this section, we recall the definitions of the 
meromorphic nearby cycle functors introduced in 
Nguyen-Takeuchi \cite{NT22} and 
Raibaut \cite{Rai13} and prove their basic properties. 
In this paper we essentially 
follow the terminology of 
\cite{Dim04}, \cite{HTT08} and \cite{KS90}. 
For a topological space $X$ denote by $\Db(X)$ the 
derived category whose objects are 
bounded complexes of sheaves 
of $\CC_X$-modules on $X$. If $X$ is a complex manifold, 
we denote by $\Dbc(X)$ the full 
subcategory of $\Db(X)$ consisting of 
constructible objects and adopt the 
convention that $\CC_X[ {\rm dim} X] \in \Dbc(X)$ 
is a perverse sheaf on $X$.

Let $X$ be a complex manifold and $P(z), Q(z)$ 
holomorphic functions on it. Assume that $Q(z)$ 
is not identically zero on each connected component 
of $X$. Then we define a meromorphic function $f(z)$ on 
$X$ by 
\begin{equation}
f(z)= \frac{P(z)}{Q(z)}  \qquad (z \in X \setminus Q^{-1}(0)). 
\end{equation}
Let us set $I(f)=P^{-1}(0) \cap Q^{-1}(0) \subset X$. 
If $P$ and $Q$ are coprime in the local ring 
$\sho_{X,z}$ at a point $z \in X$, then 
$I(f)$ is nothing but the set of the indeterminacy 
points of $f$ on a neighborhood of $z$. Note that 
the set $I(f)$ depends on the pair $(P(z), Q(z))$ of 
holomorphic functions representing $f(z)$. 
For example, if we take a holomorphic function 
$R(z)$ on $X$ (which is not identically zero on 
each connected component of $X$) and set 
\begin{equation}
g(z)= \frac{P(z)R(z)}{Q(z)R(z)}  \qquad (z \in X 
\setminus (Q^{-1}(0) \cup R^{-1}(0)) ), 
\end{equation}
then the set $I(g)=I(f) \cup R^{-1}(0)$ might be 
bigger than $I(f)$. In this way, we 
distinguish $f(z)= \frac{P(z)}{Q(z)}$ from 
$g(z)= \frac{P(z)R(z)}{Q(z)R(z)}$ even if 
their values coincide over an open dense 
subset of $X$. 
This is the convention due to 
Gusein-Zade, Luengo and Melle-Hern\'andez 
\cite{GLM98} etc. 
Now we recall the following 
fundamental theorem due to \cite{GLM98}
and \cite{GLM01}. In what follows, we assume that 
$X$ is connected and $P$ is not identically zero on it. 

\begin{theorem}\label{the-fib} 
(Gusein-Zade, Luengo and Melle-Hern\'andez 
\cite{GLM98}, \cite{GLM01}) 
For any point $z_0 \in P^{-1}(0)$ there exists 
$\e_0> 0$ such that for any $0< \e < \e_0$ 
and the open ball $B_{\e}(z_0) \subset X$ of 
radius $\e >0$ with center at $z_0$ 
(in a local chart of $X$) the restriction 
\begin{equation}
B_{\e}(z_0) \setminus Q^{-1}(0) 
\longrightarrow \CC
\end{equation}
of $f: X \setminus Q^{-1}(0) 
\longrightarrow \CC$ is a locally trivial 
fibration over a sufficiently small 
punctured disk in $\CC$ with center at 
the origin $0 \in \CC$ 
\end{theorem}

We call the fiber in this theorem the Milnor fiber of 
the meromorphic function $f(z)= \frac{P(z)}{Q(z)}$ 
at $z_0 \in P^{-1}(0)$ and denote it by $F_{z_0}$. 
For the meromorphic function $f(z)= \frac{P(z)}{Q(z)}$ let 
\begin{equation}
i_f: X \setminus Q^{-1}(0) \hookrightarrow 
X \times \CC_{\tau}
\end{equation}
be the (not necessarily) 
closed embedding defined by $z \longmapsto (z,f(z))$. 
Let $\tau : X \times \CC \rightarrow \CC$ be the 
second projection. Then 
for $\F \in \Db(X)$ we set 
\begin{equation}
\psi_f^{\mer}( \F ):= \psi_{\tau}( \rmR i_{f*}
( \F |_{X \setminus Q^{-1}(0)}) ) 
\in \Db(X)
\end{equation}
(see Nguyen-Takeuchi \cite{NT22}). 
We call $\psi_f^{\mer}( \F )$ the meromorphic 
nearby cycle sheaf of $\F$ along $f$. Moreover we set 
\begin{equation}
\psi_f^{\merc}( \F ):= \psi_{\tau}( i_{f!}
( \F |_{X \setminus Q^{-1}(0)}) ) 
\in \Db(X) 
\end{equation}
(see Raibaut \cite{Rai13}) and call it the meromorphic 
nearby cycle sheaf with compact support 
of $\F$ along $f$. Then the proof of 
the following lemma is similar to those 
of \cite[Lemma 2.1 and Remark 2.3]{NT22}. 

\begin{lemma}\label{lem-1} 
\begin{enumerate}
\item[\rm{(i)}] The support of $\psi_f^{\merc}( \F )$ is 
contained in $P^{-1}(0)$. 
\item[\rm{(ii)}] There exists an isomorphism 
\begin{equation}
\psi_f^{\merc}( \F ) \simto 
\psi_f^{\merc}( \F_{X \setminus (P^{-1}(0) \cup Q^{-1}(0))} ). 
\end{equation}
\item[\rm{(iii)}] Assume that the meromorphic 
function $f(z)= \frac{P(z)}{Q(z)}$ 
is holomorphic on a neighborhood of a point 
$z_0 \in X$ i.e. there exists a holomorphic function 
$g(z)$ defined on a neighborhood of 
$z_0 \in X$ such that $P(z)=Q(z) \cdot g(z)$ on it. 
Then 
we have an isomorphism 
\begin{equation}
\psi_f^{\merc}( \F )_{z_0} \simeq 
\psi_g ( \F_{X \setminus Q^{-1}(0)} )_{z_0}
\end{equation}
for the classical (holomorphic) nearby cycle 
functor $\psi_g( \cdot )$. 
\item[\rm{(iv)}] 
For any point $z_0 \in P^{-1}(0)$ and 
$j \in \ZZ$ we have an isomorphism 
\begin{equation}
H^j \psi_f^{\merc}( \F )_{z_0} \simeq 
H^j_c ( \overline{F_{z_0}} \setminus Q^{-1}(0) ; \F ), 
\end{equation}
where $H^j_c ( \cdot )$ stands for the hypercohomology 
group with compact support. 
\end{enumerate}
\end{lemma}

\begin{example}\label{exe-1} 
Consider the case where $X$ is the complex plane $\CC^2$ 
endowed with the standard coordinate $z=(z_1,z_2)=(x,y)$ and 
let $P,Q \in \CC [x,y]$ be polynomials on $X= \CC^2$ 
coprime each other such that $P(0)=Q(0)=0$. Assume that 
the complex curve $P^{-1}(0) \subset X$ (resp. $Q^{-1}(0) \subset X$) 
has an isolated singular point (resp. is smooth) 
at the origin $0=(0,0) \in X= \CC^2$. Then for $c \in \CC$ 
the fiber $f^{-1}(c) \subset X \setminus Q^{-1}(0)$ of the 
rational function $f= \frac{P}{Q}: X \setminus Q^{-1}(0) 
\longrightarrow \CC$ is explicitly described as follows: 
\begin{equation}
f^{-1}(c)= \{ (x,y) \in X \setminus Q^{-1}(0) \ | \ 
P(x,y)-c \cdot Q(x,y)=0 \} \quad \subset X \setminus Q^{-1}(0). 
\end{equation}
On the other hand, by \cite[Corollary 2.8]{Mil68} the set 
$\Sigma_f \subset  \CC$ of the critical values of 
$f= \frac{P}{Q}: X \setminus Q^{-1}(0) 
\longrightarrow \CC$ is finite and hence generic 
fibers of $f$ are smooth. This implies that for 
$c \in \CC^*= \CC \setminus \{ 0 \}$ such that 
$0<|c| \ll 1$ the complex curve 
\begin{equation}
\overline{f^{-1}(c)}= \{ (x,y) \in X = \CC^2 \ | \ 
P(x,y)-c \cdot Q(x,y)=0 \} \quad \subset X  
\end{equation}
is smooth outside the origin $0=(0,0) \in X= \CC^2$. 
Moreover by the conditions 
\begin{equation}
\Bigl( \frac{\partial P}{\partial x} (0,0), 
\frac{\partial P}{\partial y} (0,0) \Bigr) 
= (0,0), \quad 
\Bigl( \frac{\partial Q}{\partial x} (0,0), 
\frac{\partial Q}{\partial y} (0,0) \Bigr) 
\not= (0,0)
\end{equation}
we see that for such $c \in \CC^*$ the complex curve 
$\overline{f^{-1}(c)} \subset X$ is smooth also at 
the origin. For a complex number $c \in \CC^*$ and 
$\varepsilon >0$ such that $0<|c| \ll \varepsilon \ll 1$ 
let us explain the topology of the Milnor fiber 
\begin{equation}
F_0 = (B_{\varepsilon}(0) \setminus Q^{-1}(0)) \cap f^{-1}(c)
 = (B_{\varepsilon}(0) \cap \overline{f^{-1}(c)} ) \setminus 
\{ 0 \}
\end{equation}
of $f= \frac{P}{Q}: X \setminus Q^{-1}(0) 
\longrightarrow \CC$ at the origin $0=(0,0) \in I(f)= 
P^{-1}(0) \cap Q^{-1}(0)$. We set $L:=P^{-1}(0) \cap 
\partial B_{\varepsilon}(0) \subset 
S_{\varepsilon}(0):= \partial B_{\varepsilon}(0) \simeq S^3$ 
and call it the link of the complex curve 
$P^{-1}(0) \subset X= \CC^2$ at the origin. The small 
sphere $S_{\varepsilon}(0)$ being compact, 
the boundary $\partial \overline{F_{0}}$ of the closure 
$\overline{F_{0}}$ of $F_0$ is homeomorphic to $L$. 
We know also that the link $L$ is 
a disjoint union of some circles $S^1$ embedded in 
$S_{\varepsilon}(0) \simeq S^3$. We denote the number of 
circles $S^1$ in $L$ by $m$ and shall explain how to 
calculate it from now. For this purpose, let 
$F(P)_0$ be the Milnor fiber of $P^{-1}(0) \subset X$ 
at the origin $0=(0,0) \in X$ and consider the 
monodromy automorphism 
\begin{equation}
\Phi (P)_1: H^1( F(P)_0 ; \CC) \simto  H^1( F(P)_0 ; \CC)
\end{equation}
of its first cohomology group $H^1( F(P)_0 ; \CC)$. 
Denote by $N_1$ the number of the Jordan blocks for 
the eigenvalue $1$ in $\Phi (P)_1$. Then by the results in 
\cite[Section 8]{Mil68} (see also the proof of 
\cite[Lemma 4.3]{NT05}) we have $m=N_1+1$. Since by 
the classical theory of the mixed Hodge structure  
of $H^1( F(P)_0 ; \CC)$ we know in this case of 
${\rm dim} X=2$ that the maximal possible size of such 
Jordan blocks is one, we can calculate the number 
$N_1$ by the theory of monodromy zeta functions 
(see \cite[Section 2]{Take23} for a review of this 
subject). We thus have seen that the interior of 
$\overline{F_{0}}$ is a smooth Riemann surface with 
the boundary $\partial \overline{F_{0}}$ which is 
isomorphic to the disjoint union 
$S^1 \sqcup \cdots \sqcup S^1$ of $N_1+1$ copies of 
$S^1$. Let $M$ be a compact oriented surface obtained by 
attaching  $N_1+1$ ($2$-dimensional) disks to 
$\overline{F_{0}}$ along its boundary 
$\partial \overline{F_{0}}$ and denote its genus 
by $g(M)$. By taking a triangulation of $\overline{F_{0}}$ 
we see that 
$H_c^2( \overline{F_{0}} ; \CC_X) \simeq 
H^2( \overline{F_{0}} ; \CC_X) \simeq 0$. Then by 
the short exact sequence 
\begin{equation}
0 \longrightarrow \CC_{M \setminus \overline{F_{0}}} \longrightarrow  
\CC_M  \longrightarrow \CC_{\overline{F_{0}}} 
 \longrightarrow 0 
\end{equation}
we obtain isomorphisms 
\begin{equation}
H^j( \overline{F_{0}} ; \CC_X)
\simeq 
\begin{cases}
\CC & (j=0) \\ 
\\
\CC^{2 g(M)+N_1} & (j=1) \\
\\
0 & (\mbox{otherwise}). 
\end{cases}
\end{equation}
Moreover, by $\overline{F_{0}} \setminus Q^{-1}(0)= 
\overline{F_{0}} \setminus \{ 0 \}$ we can show that 
there exist isomorphisms 
\begin{equation}
H^j \psi_f^{\merc}( \CC_X )_{0}
 \simeq 
H_c^j( \overline{F_{0}} \setminus Q^{-1}(0); \CC_X)
 \simeq 
\begin{cases}
\CC^{2 g(M)+N_1} & (j=1) \\ 
\\
0 & (\mbox{otherwise}). 
\end{cases}
\end{equation}
Compare this result with the perversity of 
$\phi_f^{\merc}( \CC_X [1] )$ which will be proved 
soon below in this section. For the above $c \in \CC^*$ and 
$\varepsilon^{\prime} >0$ such that 
$0< \varepsilon^{\prime} \ll |c| \ll \varepsilon$ we set 
\begin{equation}
G_0 := (B_{\varepsilon^{\prime}}(0) \setminus Q^{-1}(0)) \cap f^{-1}(c)
= (B_{\varepsilon^{\prime}}(0) \cap \overline{f^{-1}(c)} ) \setminus 
\{ 0 \}. 
\end{equation}
Then the closure $\overline{G_{0}}$ of $G_0$ 
is homeomorphic to the ($2$-dimensional) disk and 
$\overline{G_{0}} \setminus Q^{-1}(0)= 
\overline{G_{0}} \setminus \{ 0 \}$. We thus obtain 
isomorphisms 
\begin{equation}
H_c^j( \overline{G_{0}} \setminus Q^{-1}(0); \CC_X) 
\simeq 0 \qquad (j \in \ZZ ). 
\end{equation}
The difference of the topology of $G_0$ from that of $F_0$ 
comes from the fact that there exists $r>0$ such that 
$\varepsilon^{\prime} < r < \varepsilon$ and 
the sphere $S_r(0)= \partial B_r(0)$ is tangent to 
the smooth curve $\overline{f^{-1}(c)}$. 
This explains the reason why one should distinguish $F_0$ 
from $G_0$ carefully in the study of the Milnor fiber $F_0$ of 
$f= \frac{P}{Q}: X \setminus Q^{-1}(0) 
\longrightarrow \CC$. 
\end{example}

The following result is an analogue 
of \cite[Theorem 2.2]{NT22} for the meromorphic 
nearby cycle sheaf with compact support 
of $\F$ along $f$. 

\begin{theorem}\label{the-1} 
\begin{enumerate}
\item[\rm{(i)}] If $\F \in \Db(X)$ is constructible, then 
$\psi_f^{\merc}( \F ) \in \Db(X)$ is also constructible. 
\item[\rm{(ii)}] If $\F \in \Db(X)$ is perverse, then 
$\psi_f^{\merc}( \F )[-1] \in \Db(X)$ is also perverse. 
\end{enumerate}
\end{theorem}

\begin{proof} Assume that $\F \in \Db(X)$ is 
constructible. Define a hypersurface $W$ of 
$X \times \CC_{\tau}$ by 
\begin{equation}
W= \{ (z, \tau ) \in X \times \CC \ | \ 
P(z)- \tau Q(z)=0 \} 
\end{equation}
and let $\kappa : W \rightarrow X$ be the 
restriction of the first projection 
$X \times \CC_{\tau} \rightarrow X$ to it. 
Then $\kappa$ induces an isomorphism 
\begin{equation}
\kappa^{-1} (X \setminus Q^{-1}(0)) 
\simto X \setminus Q^{-1}(0)
\end{equation}
and $\kappa^{-1} (X \setminus Q^{-1}(0))$ 
is nothing but the graph 
\begin{equation}
\{ (z, f(z)) \in (X \setminus Q^{-1}(0)) 
\times \CC \ | \ z \in X \setminus Q^{-1}(0) \} 
\end{equation}
of $f: X \setminus Q^{-1}(0) \rightarrow \CC$. 
In this way, we identify $X \setminus Q^{-1}(0)$ 
and the open subset $\kappa^{-1} (X \setminus Q^{-1}(0))$ 
of $W$. Let 
\begin{equation}
j_f: (X \setminus Q^{-1}(0)) \times \CC_{\tau} 
\hookrightarrow X \times \CC_{\tau} 
\end{equation}
and $i_W: W \hookrightarrow X \times \CC_{\tau}$ 
be the inclusion maps. Then as in the proof 
of \cite[Theorem 2.2]{NT22} 
we obtain an isomorphism 
\begin{equation}
\G := i_{f!}( \F |_{X \setminus Q^{-1}(0)}) \simeq 
j_{f_!}j_f^{-1}(i_{W*} \kappa^{-1} \F ). 
\end{equation}
From this we see that $\G = 
i_{f!}( \F |_{X \setminus Q^{-1}(0)})$ is constructible. 
Now the assertion (i) is clear. 
Assume that $\F \in \Db(X)$ is perverse. 
Then, although $\CH := i_{W*} \kappa^{-1} \F
\in \Db(X \times \CC )$ is not 
necessarily perverse (up to some shift) 
in general, its 
restriction 
\begin{equation}
j_f^{-1} \CH \simeq \CH |_{(X \setminus Q^{-1}(0)) 
\times \CC_{\tau}}
\end{equation}
to $(X \setminus Q^{-1}(0)) \times \CC_t$ is preverse 
(up to some shift). Moreover, its Verdier dual 
$\mathbb{D}_{X \times \CC} ( \CH ) \in \Db(X \times \CC )$ 
satisfies the same property. It follows 
from the proof of \cite[Theorem 2.2 (ii)]{NT22} that 
\begin{equation}
 \rmR j_{f_*}j_f^{-1} \mathbb{D}_{X \times \CC} ( \CH ) 
\in \Db(X \times \CC )
\end{equation}
is perverse (up to some shift). By the isomorphism 
\begin{equation}
\mathbb{D}_{X \times \CC} ( \G ) 
= \mathbb{D}_{X \times \CC} ( j_{f_!}j_f^{-1}( \CH ) ) 
\simeq \rmR j_{f_*}j_f^{-1}
\mathbb{D}_{X \times \CC} (  \CH  ), 
\end{equation}
we find that $\G = j_{f_!}j_f^{-1}( \CH ) \simeq 
i_{f!}( \F |_{X \setminus Q^{-1}(0)}) \in \Db(X \times \CC )$ 
is also  perverse (up to some shift). 
Then the assertion of (ii) immediately follows from 
the t-exactness of the nearby cycle functor 
\begin{equation}
\psi_{\tau}( \cdot ): \Db(X \times \CC ) \longrightarrow \Db( X ). 
\end{equation}
\end{proof}

By this theorem we obtain a functor 
\begin{equation}
\psi_f^{\merc }( \cdot ) : 
\Dbc (X) \longrightarrow \Dbc (X). 
\end{equation}
Let $i_0: X \hookrightarrow X \times \CC_{\tau}$ 
($x \longmapsto (x,0)$) be the inclusion map 
and for $\F \in \Db(X)$ set 
\begin{equation}
\phi_f^{\merc}( \F ):= \phi_{\tau}( i_{f!}
( \F |_{X \setminus Q^{-1}(0)}) ) 
\in \Db(X). 
\end{equation}
We call it the meromorphic 
vanishing cycle sheaf with compact support 
of $\F$ along $f$. Then there exists a 
distinguished triangle 
\begin{equation}\label{dist}
i_0^{-1} i_{f!}( \F |_{X \setminus Q^{-1}(0)}) 
\longrightarrow \psi_f^{\merc}( \F )
\longrightarrow \phi_f^{\merc}( \F ) 
\overset{+1}{\longrightarrow}
\end{equation}
in $\Db(X)$. Moreover, since for the inclusion map 
$j: P^{-1}(0) \setminus  Q^{-1}(0) 
\hookrightarrow X$ we have a Cartesian diagram 
\begin{equation}
\begin{CD}
P^{-1}(0) \setminus  Q^{-1}(0)   @>{j}>> X 
\\
@VVV   @VV{i_0}V
\\
X \setminus  Q^{-1}(0)   @>{i_f}>>  
X \times \CC_t
\end{CD}
\end{equation}
of inclusion maps, we obtain an isomorphism 
\begin{equation}\label{dist-isom}
i_0^{-1} i_{f!}( \F |_{X \setminus Q^{-1}(0)}) 
\simeq 
j_! ( \F |_{P^{-1}(0) \setminus  Q^{-1}(0) } ) 
\simeq \F_{P^{-1}(0) \setminus  Q^{-1}(0)}. 
\end{equation}
We thus see that also the functor 
\begin{equation}
\phi_f^{\merc }( \cdot ) : \Db (X) \longrightarrow \Db (X)
\end{equation}
preserves the constructibility and 
the perversity (up to some shift). 
Moreover by \eqref{dist} and \eqref{dist-isom},  
for any $\F \in \Db(X)$ we have isomorphisms 
\begin{equation}
\phi_f^{\merc}( \F )_{z_0} \simeq 
\psi_f^{\merc}( \F )_{z_0} \qquad 
(z_0 \in Q^{-1}(0)). 
\end{equation}
The proof of the following lemma is similar to those 
of \cite[Lemma 2.1 and Remark 2.3]{NT22}. 

\begin{lemma}\label{lem-2} 
\begin{enumerate}
\item[\rm{(i)}] The support of $\phi_f^{\merc}( \F )$ is 
contained in $P^{-1}(0)$. 
\item[\rm{(ii)}] There exists an isomorphism 
\begin{equation}
\phi_f^{\merc}( \F ) \simto 
\phi_f^{\merc}( \F_{X \setminus Q^{-1}(0)} ). 
\end{equation}
\item[\rm{(iii)}] Assume that the meromorphic 
function $f(z)= \frac{P(z)}{Q(z)}$ 
is holomorphic on a neighborhood of a point 
$z_0 \in X$ i.e. there exists a holomorphic function 
$g(z)$ defined on a neighborhood of 
$z_0 \in X$ such that $P(z)=Q(z) \cdot g(z)$ on it. 
Then 
we have an isomorphism 
\begin{equation}
\phi_f^{\merc}( \F )_{z_0} \simeq 
\phi_g ( \F_{X \setminus Q^{-1}(0)} )_{z_0}
\end{equation}
for the classical (holomorphic) vanishing cycle 
functor $\phi_g( \cdot )$. 
\end{enumerate}
\end{lemma}

The following result is a cohomological generalization of 
Theorem \ref{the-fib}. 

\begin{proposition}\label{PVC} 
Let $\F \in \Dbc(X)$ be a constructible sheaf on $X$. 
Then for any $z_0 \in I(f)=P^{-1}(0) \cap Q^{-1}(0)$ 
and $c \in \CC$ there exist small $0< \e \ll 1$ 
such that the cohomology sheaves $H^j \rmR  f_! \bigl( 
\F_{\overline{B_{\e}(z_0)}}|_{X  \setminus Q^{-1}(0)} \bigr)$ 
$(j \in \ZZ )$ are local systems over a sufficiently small punctured 
disk centered at $c \in \CC$. Moreover we have 
\begin{equation}\label{VAc} 
\rmR  f_! \bigl( 
\F_{\overline{B_{\e}(z_0)}}|_{X  \setminus Q^{-1}(0)} \bigr)_c 
\simeq 
\rmR  (f-c)_! \bigl( 
\F_{\overline{B_{\e}(z_0)}}|_{X  \setminus Q^{-1}(0)} \bigr)_0 
\simeq 0
\end{equation}
and there exist isomorphisms 
\begin{align}
& \phi_{\tau -c} \Bigl( 
\rmR  f_! \bigl( 
\F_{\overline{B_{\e}(z_0)}}|_{X  \setminus Q^{-1}(0)} \bigr) \Bigr) 
 \simeq 
\psi_{\tau -c} \Bigl( 
\rmR  f_! \bigl( 
\F_{\overline{B_{\e}(z_0)}}|_{X  \setminus Q^{-1}(0)} \bigr) \Bigr) 
\\
& \simeq \psi_{f-c}^{\merc} ( \F )_{z_0} 
\simeq \phi_{f-c}^{\merc} ( \F )_{z_0}.  
\end{align}
\end{proposition}

\begin{proof}
Note that for $P_c:=P-c \cdot Q$ we have 
$f-c= \frac{P_c}{Q}$. Then we obtain the first 
assertion by (the proof of)  Lemma \ref{lem-1} (iv) 
(see also the proof of \cite[Theorem 2.6]{Take23}). 
We can also show the vanishing \eqref{VAc} 
by the cone theorem proved by \cite[Theorem 2.10]{Mil68} and 
\cite[Lemma 3.2]{BV72}. Indeed, let 
$X= \sqcup_{\alpha \in A} X_{\alpha}$ be a 
stratification of $X$ such that $H^j( \F |_{X_{\alpha}})$ 
is locally constant for any $j \in \ZZ$ and $\alpha \in A$. 
Then the cone theorem implies that there exist small $0< \e \ll 1$ 
such that for any $\alpha \in A$ the subsets 
\begin{equation}
\overline{X_{\alpha}} \cap P_c^{-1}(0) \cap \overline{B_{\e}(z_0)}, 
\qquad 
\overline{X_{\alpha}} \cap P_c^{-1}(0) \cap Q^{-1}(0) \cap \overline{B_{\e}(z_0)}
\end{equation}
of $\overline{B_{\e}(z_0)}$ are contractible and hence 
\begin{equation}
\rmR  (f-c)_! \bigl( 
\F_{\overline{X_{\alpha}} \cap \overline{B_{\e}(z_0)}}|_{X  \setminus 
Q^{-1}(0)} \bigr)_0 
\simeq 0. 
\end{equation}
By decomposing the support of $\F$ with respect to the 
stratification $X= \sqcup_{\alpha \in A} X_{\alpha}$ of $X$ 
we thus can prove the vanishing \eqref{VAc}. 
Together with Lemma \ref{lem-1} (iv), we obtain also 
the remaining assertions. 
\end{proof}

\begin{proposition}\label{FNTVC} 
For a constructible sheaf $\F \in \Dbc(X)$ on $X$ and a 
point $z_0 \in I(f)=P^{-1}(0) \cap Q^{-1}(0)$, 
there exists a finite subset $\Sigma \subset \CC$ 
of $\CC$ such that we have 
\begin{equation}
\phi_{f-c}^{\merc} ( \F )_{z_0} \simeq 0. 
\end{equation}
for any $c \in \CC \setminus \Sigma$. 
\end{proposition}

\begin{proof}
Note that for the point $z_0 \in I(f)$ 
(under the natural identification of $X \times \{ 0 \}$ with 
$X \times \{ c \}$) we have an 
isomorphism 
\begin{equation}\label{EE-1} 
\phi_{f-c}^{\merc} ( \F )_{z_0} \simeq 
\phi_{\tau -c} 
\Bigl( i_{f!}( \F |_{X \setminus Q^{-1}(0)}) \Bigr)_{(z_0,c)}. 
\end{equation}
Let $\PP = \PP^1 = \CC \sqcup \{ \infty \}$ be the 
projective compactification of $\CC$ and 
$i: X \times \CC \hookrightarrow X \times \PP$ the 
inclusion map. Then by the proof of Theorem \ref{the-1} (i) 
we see also that $\widetilde{\G} := 
i_! \Bigl( i_{f!}( \F |_{X \setminus Q^{-1}(0)}) \Bigr) \in 
\Db(X \times \PP )$ is constructible. Let $\CS$ be a Whitney stratification 
of $X \times \PP$ adapted to $\widetilde{\G}$ such that 
\begin{equation}\label{EE-2} 
{\rm SS} ( \widetilde{\G} ) \subset 
\bigcup_{S \in \CS} T^*_S(X \times \PP ). 
\end{equation}
Then by the theorem in \cite[page 43]{GM88}, after 
refining $\CS$ if necessary, we may assume that 
for $\CS$ and a Whitney stratification $\CS_0$ of 
$X$ the projection $X \times \PP \longrightarrow X$ 
is a stratified fiber bundle as in its assertion. 
We may assume also that the one point set 
$\{ z_0 \} \subset X$ is a stratum of $\CS_0$. 
Then there exists a finite subset 
$\{ c_1,c_2, \ldots, c_k \} \subset \PP$ of $\PP$ 
such that the strata in $\CS$ projecting to 
the one $\{ z_0 \} \subset X$ in $\CS_0$ are 
$\{ (z_0, c_i) \} \subset X \times \PP$ 
($1 \leq i \leq k$) and $\{ z_0 \} \times 
( \PP \setminus  \{ c_1,c_2, \ldots, c_k \}) \subset X \times \PP$. 
Let us set $\Sigma := \CC \cap 
\{ c_1,c_2, \ldots, c_k \} \subset \CC$. Then by 
\eqref{EE-1}, \eqref{EE-2} and \cite[Proposition 8.6.3]{KS90} 
we obtain the desired vanishing  
\begin{equation}
\phi_{f-c}^{\merc} ( \F )_{z_0} \simeq 0. 
\end{equation}
for any $c \in \CC \setminus \Sigma$. 
\end{proof}

The following result is an analogue   
for $\psi_f^{\merc}( \cdot )$ and 
$\phi_f^{\merc}( \cdot )$ of the 
classical one for $\psi_f( \cdot )$ 
and $\phi_f( \cdot )$ 
(see e.g. \cite[Proposition 4.2.11]{Dim04},  
\cite[Exercise VIII.15]{KS90} 
and \cite[Proposition 2.4]{NT22}). 

\begin{proposition}\label{PDI} 
Let $\nu  : Y \longrightarrow X$ be a proper surjective 
morphism of complex manifolds and 
$f \circ \nu $ the meromorphic function on $Y$ defined by 
\begin{equation}
f \circ \nu  = \frac{P \circ \nu }{Q \circ \nu }. 
\end{equation}
Then we have: 
\begin{enumerate}
\item[\rm{(i)}] 
For $\G \in \Db(Y)$ 
there exists isomorphisms 
\begin{equation}
\psi_{f}^{\merc} (  \rmR \nu_* \G ) \simeq 
 \rmR \nu_* \psi_{f \circ \nu}^{\merc} ( \G ), \qquad 
\phi_{f}^{\merc} (  \rmR \nu_* \G ) \simeq 
 \rmR \nu_* \phi_{f \circ \nu}^{\merc} ( \G ). 
\end{equation}
\item[\rm{(ii)}] 
If moreover $\nu$ induces an isomorphism 
\begin{equation}
Y \setminus \nu^{-1}(P^{-1}(0) \cup Q^{-1}(0)) 
\simto X \setminus (P^{-1}(0) \cup Q^{-1}(0)), 
\end{equation}
then for $\F \in \Db(X)$ 
there exists an isomorphism 
\begin{equation}
\psi_{f}^{\merc} ( \F ) \simeq 
 \rmR \nu_* \psi_{f \circ \nu}^{\merc} ( \nu^{-1} \F ). 
\end{equation}
\item[\rm{(iii)}] 
If moreover $\nu$ induces an isomorphism 
\begin{equation}
Y \setminus \nu^{-1} Q^{-1}(0)
\simto X \setminus Q^{-1}(0), 
\end{equation}
then for $\F \in \Db(X)$ 
there exists an isomorphism 
\begin{equation}
\phi_{f}^{\merc} ( \F ) \simeq 
 \rmR \nu_* \phi_{f \circ \nu}^{\merc} ( \nu^{-1} \F ). 
\end{equation}
\end{enumerate}
\end{proposition}

\begin{lemma}\label{lem-a1} 
Assume that the hypersurfaces $P^{-1}(0), Q^{-1}(0) 
\subset X$ are smooth and 
intersect transversally in a neighborhood $U 
\subset X$ of a point $z_0 \in I(f)=P^{-1}(0) \cap Q^{-1}(0)$. 
Then we have 
\begin{equation}
\phi_f^{\merc}( \CC_X )_{z_0} \simeq 
\psi_f^{\merc}( \CC_X )_{z_0} \simeq 0. 
\end{equation}
\end{lemma}

\begin{proof}
Let $\nu  : Y \longrightarrow U$ be the blow-up of 
$U$ along the complex submanifold 
$I(f) \cap U \subset U$ of codimension two. 
Then $\nu$ induces an isomorphism 
\begin{equation}
Y \setminus \nu^{-1} (Q|_U)^{-1}(0)
\simto U \setminus (Q|_U)^{-1}(0)
\end{equation}
and we can apply Proposition \ref{PDI} (iii) 
to the meromorphic function 
$f|_U:U \longrightarrow \CC$ and 
$\nu  : Y \longrightarrow U$ to obtain an isomorphism 
\begin{equation}
\phi_{f}^{\merc} ( \CC_X )_{z_0}= 
\phi_{f|_U}^{\merc} ( \CC_U )_{z_0}
\simeq 
\rmR \Gamma ( \nu^{-1}(z_0); 
\phi_{(f|_U) \circ \nu}^{\merc} ( \CC_Y )). 
\end{equation}
By our construction of the functor $\phi_{(f|_U) 
\circ \nu}^{\merc}( \cdot )$, the support of 
$\phi_{(f|_U) \circ \nu}^{\merc} ( \CC_Y )$ is 
contained in the proper transform of 
$P^{-1}(0) \cap U$ in $Y$.  
Moreover by Lemma \ref{lem-2} (iii), 
we can easily show that the stalk of 
$\phi_{(f|_U) \circ \nu}^{\merc} ( \CC_Y )$ 
at each point of $\nu^{-1}(z_0)$ is 
isomorphic to zero 
(see also e.g. \cite[Lemma 2.9]{Take23}). 
\end{proof}

\begin{corollary}\label{cor-a1} 
For a point $z_0 \in I(f)=P^{-1}(0) \cap Q^{-1}(0)$ 
assume that there exists its neighborhood $U 
\subset X$ such that the hypersurfaces $P^{-1}(0), Q^{-1}(0) 
\subset X$ are smooth and 
intersect transversally in $U \setminus \{ z_0 \}$. 
Then we have the concentration 
\begin{equation}
H^j \phi_{f}^{\merc} ( \CC_X )_{z_0} \simeq 
H^j \psi_{f}^{\merc} ( \CC_X )_{z_0} \simeq 0 
\qquad (j \not= {\rm dim} X-1). 
\end{equation}
\end{corollary}

\begin{proof}
By our assumption and Lemma \ref{lem-a1}, the support of 
\begin{equation}
\phi_{f}^{\merc} ( \CC_X )|_U \simeq 
\phi_{f|_U}^{\merc} ( \CC_U )
\end{equation} 
is contained in the one point set $\{ z_0 \} 
\subset U$. Then the assertion 
immediately follows from the 
perversity of 
$\phi_{f}^{\merc} ( \CC_X [{\rm dim} X] )[-1]$. 
\end{proof}

By the proof of Corollary \ref{cor-a1}, we obtain also the 
following result. 

\begin{proposition}\label{prop-a1} 
Let $\CK \in \Dbc(X)$ be a perverse sheaf on $X$ and 
for a point $z_0 \in I(f)=P^{-1}(0) \cap Q^{-1}(0)$ 
assume that there exists its neighborhood $U 
\subset X$ such that the support of 
\begin{equation}
\phi_{f}^{\merc} ( \CK )|_U \simeq 
\phi_{f|_U}^{\merc} ( \CK |_U )
\end{equation} 
is contained in the one point set $\{ z_0 \} 
\subset U$. 
Then we have the concentration 
\begin{equation}
H^j \phi_{f}^{\merc} ( \CK )_{z_0} \simeq 
H^j \psi_{f}^{\merc} ( \CK )_{z_0} \simeq 0 
\qquad (j \not= -1). 
\end{equation}
\end{proposition}

In the situation of Proposition \ref{prop-a1}, 
we can calculate the dimension of the only 
non-trivial cohomology group 
\begin{equation}
H^{-1} \phi_{f}^{\merc} ( \CK )_{z_0} \simeq 
H^{-1} \psi_{f}^{\merc} ( \CK )_{z_0} 
\end{equation}
as follows. 
For a constructible sheaf $\F \in \Dbc (X)$ we set 
\begin{equation}
\chi ( \F, z_0):= \dsum_{j \in \ZZ} (-1)^j 
{\rm dim } H^j \phi_{f}^{\merc}( \F)_{z_0}
\qquad \in \ZZ 
\end{equation}
so that for any distinguished triangle 
$\F^{\prime} \longrightarrow \F 
\longrightarrow \F^{\prime \prime}
\overset{+1}{\longrightarrow}$ in $\Dbc (X)$ 
we have $\chi ( \F, z_0)= \chi ( \F^{\prime}, z_0)+ 
\chi ( \F^{\prime \prime}, z_0)$. 
Let us set $D:=P^{-1}(0) \cup Q^{-1}(0) \subset X$. 
Then by Lemma \ref{lem-1} (ii) we may assume that 
$\CK_{X \setminus D} \simeq \CK$. By  
decomposing the the support of 
the perverse sheaf $\CK_{X \setminus D} \simeq \CK 
\in \Dbc (X)$ 
with respect to a stratification 
of $X$ adapted to $\CK$ and the divisor $D \subset X$ 
in $X$, it suffices to 
calculate $\chi ( \F, z_0) \in \ZZ$ for some 
$\F \in \Dbc (X)$ such that for a 
local system $L$ on a stratum $S \subset X$ 
contained in 
$X \setminus D$ and the inclusion 
map $j:S \hookrightarrow X$ we have 
$\F=j_! L$. On the other hand, as in 
the proof of \cite[Theorem 3.6]{MT13}, we can 
construct a proper morphism 
$\nu_1 : \widetilde{X}_1 \longrightarrow X$ 
of a complex manifold $\widetilde{X}_1$ 
which induces an isomorphism over the open 
subset $X \setminus D \subset X$ 
such that the divisor $D^{\prime}:= 
\nu_1^{-1}(D)= (P \circ \nu_1 )^{-1}(0) \cup 
(Q \circ \nu_1 )^{-1}(0)$ in $\widetilde{X}_1$ is 
normal crossing and the rational function 
$f \circ \nu_1 := (P \circ \nu_1 )/(Q \circ \nu_1 )$ 
on $\widetilde{X}_1$ has no point of indeterminacy 
on the whole $\widetilde{X}_1$. Namely 
the pole and zero divisors of 
the rational function 
$f \circ \nu_1$ are disjoint in $D^{\prime}$. 
Let $g: \widetilde{X}_1 \longrightarrow \PP$ 
be the holomorphic map defined by $f \circ \nu_1$. 
In this situation, there exists a (not necessarily closed) 
embedding $i_S: S \hookrightarrow \widetilde{X}_1$ 
such that $\nu_1 \circ i_S=j$. Then we have 
$\nu_1^{-1} \F = \nu_1^{-1}  j_! L \simeq 
i_{S !}L$ and hence by Proposition \ref{PDI} (ii) 
and Lemma \ref{lem-1} (iii) we obtain an isomorphism 
\begin{equation}
\psi_{f}^{\merc} ( \F)_{z_0} \simeq 
\rmR \Gamma ( \nu_1^{-1}(z_0) \cap g^{-1}(0) 
; \psi_g( i_{S !}L )). 
\end{equation}
We set $S_1:=i_S(S) \subset \widetilde{X}_1$. 
Denote by $\overline{S_1}$ its closure in $\widetilde{X}_1$ 
and let $i_{\overline{S_1}}: \overline{S_1}
 \hookrightarrow \widetilde{X}_1$ 
be the inclusion map. Then there exists a proper morphism 
$\nu_2 : T \longrightarrow \overline{S_1}$ 
of a complex manifold $T$ 
which induces an isomorphism over the open 
subset $S_1$ of $\overline{S_1}$ 
such that $E:= \nu_2^{-1}( \overline{S_1} \setminus S_1 )
\subset T$ is a normal crossing divisor in $T$. 
Set $T^{\circ} := 
T \setminus E \simeq S_1 \simeq S$ and let $\iota : T^{\circ} \hookrightarrow T$ 
be the inclusion map. 
Then for the proper morphism $\nu := \nu_1 \circ i_{\overline{S_1}} 
\circ \nu_2: T  \longrightarrow X$ 
and $\widetilde{g} :=  g \circ i_{\overline{S_1}} 
\circ \nu_2 = (g|_{\overline{S_1}}) \circ \nu_2 
: T  \longrightarrow \PP$ we 
obtain an isomorphism 
\begin{equation}
\psi_{f}^{\merc}( \F)_{z_0} \simeq 
\rmR \Gamma ( \nu^{-1}(z_0) \cap  \widetilde{g}^{-1}(0); 
\psi_{\widetilde{g}}
( \iota_! L )).
\end{equation}
Since the divisor $E \subset T$ is 
normal crossing, one can calculate 
\begin{equation}
\dsum_{j \in \ZZ} (-1)^j 
{\rm dim } H^j ( \nu^{-1}(z_0 ) \cap \widetilde{g}^{-1}( 0 ); 
\psi_{\widetilde{g}}
( \iota_! L )) 
\quad \in \ZZ 
\end{equation}
by the results in \cite[Section 5]{MT11} 
(see also e.g. \cite[Lemma 2.18]{Take23}).

\begin{example}\label{exe-2} 
As in Example \ref{exe-1} consider the case where 
$X$ is the complex plane $\CC^2$ 
endowed with the standard coordinate $z=(z_1,z_2)=(x,y)$ and 
for two non-zero complex numbers $a,b \in \CC^*= \CC 
\setminus \{ 0 \}$ set $P(x,y)=ax^2+bxy+y^3$, $Q(x,y)=x$ and 
\begin{equation}
f(x,y)= \frac{P(x,y)}{Q(x,y)} = 
\frac{ax^2+bxy+y^3}{x} \qquad 
((x,y) \in X \setminus Q^{-1}(0)). 
\end{equation}
Then the complex curve $P^{-1}(0) \subset X$ (resp. $Q^{-1}(0) \subset X$) 
has an isolated singular point (resp. is smooth) 
at the origin $0=(0,0) \in X= \CC^2$. Namely the rational function 
$f= \frac{P}{Q}: X \setminus Q^{-1}(0) 
\longrightarrow \CC$ satisfies the conditions in Example \ref{exe-1}. 
Since moreover in this case $P^{-1}(0) \subset X$ is 
Newton non-degenerate at the origin (see \cite[Definition 3.5]{Take23} 
for the definition), by \cite[Theorem 8.9 (i)]{Take23} 
the number $N_1$ of the Jordan blocks for 
the eigenvalue $1$ in its first Milnor monodromy is 
equal to $1$. As we explained in Example \ref{exe-1}, 
we can show it also by Varchenko's formula for 
monodromy zeta functions in \cite{Var76}. 
Hence in this case, by the arguments in 
Example \ref{exe-1} the boundary $\partial \overline{F_{0}}$ 
of the closure $\overline{F_{0}}$ of 
the Milnor fiber $F_0$ of $f$ at the origin $0=(0,0) \in I(f)$ 
is homeomorphic to the disjoint union 
$S^1 \sqcup S^1$ of two copies of $S^1$ and 
there exist isomorphisms 
\begin{equation}
H^j \psi_f^{\merc}( \CC_X )_{0}
\simeq 
H^j \phi_f^{\merc}( \CC_X )_{0} 
\simeq 
\begin{cases}
\CC^{2 g(M)+1} & (j=1) \\ 
\\
0 & (\mbox{otherwise}),  
\end{cases}
\end{equation}
where $g(M)$ is the genus of the 
compact oriented surface $M$ obtained by 
attaching  two ($2$-dimensional) disks to 
$\overline{F_{0}}$ along its boundary 
$\partial \overline{F_{0}}$. From now, we shall 
show that $g(M)$ is equal to $0$. For this purpose, we set 
$D_1:=P^{-1}(0)$, $D_2:=Q^{-1}(0)$ and 
$D:=D_1 \cup D_2$ and construct 
a proper morphism $\nu :Y \longrightarrow X$ 
of a $2$-dimensional complex manifold $Y$ 
which induces an isomorphism over the open 
subset $X \setminus D \subset X$ 
such that the divisor $\nu^{-1}(D)$ in $Y$ is 
normal crossing as follows. First, let $\nu_1 :X_1 \longrightarrow X$ 
be the blow-up of $X= \CC^2$ along the origin 
$\{ 0 \} \subset X= \CC^2$ and $E_1:= \nu_1^{-1}(0) 
( \simeq \PP^1) \subset X_1$ its exceptional divisor. 
We denote the proper transforms of $D_1$ and $D_2$ 
in $X_1$ by $\widetilde{D_1} \subset X_1$ and 
$\widetilde{D_2} \subset X_1$ respectively. 
Then we can easily see that the set 
$E_1 \cap \widetilde{D_1}$ consists of 
two points and one of them is the unique point 
in the one point set $E_1 \cap \widetilde{D_2}$ 
and the other one is in $E_1 \setminus \widetilde{D_2}$. 
Let $\nu_2 :X_2 \longrightarrow X_1$ 
be the blow-up of $X_1$ along the point 
$E_1 \cap \widetilde{D_2} \subset X_1$ and 
$E_2 ( \simeq \PP^1) \subset X_2$ its exceptional divisor. 
By abuse of notations, we denote the proper transform 
of $\widetilde{D_2} \subset X_1$ 
in $X_2$ also by $\widetilde{D_2} \subset X_2$. 
Then the set $E_2 \cap \widetilde{D_2}$ consists of 
one point and the meromorphic function 
$f \circ \nu_1 \circ \nu_2$ on $X_2$ has a zero 
(resp. pole) of order $1$ along the divisor 
$E_2 \subset X_2$ (resp. $\widetilde{D_2} \subset X_2$). 
This implies that $f \circ \nu_1 \circ \nu_2$ still has 
indeterminacy at the point $E_2 \cap \widetilde{D_2}$. 
In order to eliminate it, let us consider the 
blow-up $\nu_3 :X_3 \longrightarrow X_2$ 
of $X_2$ along the point $E_2 \cap \widetilde{D_2} 
\subset X_2$. Set $Y:=X_3$ and $\nu:= 
\nu_1 \circ \nu_2 \circ \nu_3: Y \longrightarrow X$ and 
let $E_3 ( \simeq \PP^1 ) \subset Y$ be the exceptional divisor of 
the last bolw-up $\nu_3 :Y=X_3 \longrightarrow X_2$. 
Again by abuse of notations, we denote the proper 
transforms of $D_1$, $D_2$, $E_1$ and $E_2$ in $Y=X_3$ by 
$\widetilde{D_1}$, $\widetilde{D_2}$, $\widetilde{E_1}$ and 
$\widetilde{E_2}$ respectively. 
Then the meromorphic function $f \circ \nu$ on $Y$ 
has a zero (resp. pole) of order $1$ along the divisors 
$\widetilde{D_1}, \widetilde{E_1}, \widetilde{E_2} 
\subset Y$ (resp. along the divisor $\widetilde{D_2} 
\subset Y$). Moreover $f \circ \nu$ has neither zero 
nor pole along the divisor $E_3 \subset Y$ and we have 
$( \widetilde{D_1} \cup \widetilde{E_1} \cup \widetilde{E_2}) 
\cap \widetilde{D_2} = \emptyset$. This implies that 
the meromorphic function $f \circ \nu$ has no point 
of indeterminacy on the whole $Y$. Moreover 
the divisor 
\begin{equation}
D^{\prime}:= \nu^{-1}(D) = 
( \widetilde{D_1} \cup \widetilde{E_1} \cup \widetilde{E_2}) 
\cup E_3 \cup \widetilde{D_2} 
\qquad \subset Y 
\end{equation}
in $Y$ is normal crossing. Note that among the smooth divisors 
$\widetilde{D_1}, \widetilde{E_1} \simeq \PP^1, \widetilde{E_2} \simeq \PP^1, 
E_3 \simeq \PP^1$ in $Y$ we have only the following four 
intersection points: 
\begin{equation}
A:= \widetilde{D_1} \cap \widetilde{E_1}, \ 
B:= \widetilde{D_1} \cap \widetilde{E_2}, \ 
A^{\prime}:= \widetilde{E_1} \cap \widetilde{E_2}, \ 
B^{\prime}:= \widetilde{E_2} \cap E_3
\end{equation}
and we set 
\begin{equation}\label{newdiv}
\widetilde{E_1^{\circ}} := \widetilde{E_1} \setminus \{ A, A^{\prime} \}, 
\qquad 
\widetilde{E_2^{\circ}} := \widetilde{E_2} \setminus \{ B, 
A^{\prime}, B^{\prime} \}. 
\end{equation}
Let $g: Y \longrightarrow \PP^1$ be the holomorphic map defined by 
the meromorphic function $f \circ \nu$. Then we have 
\begin{equation}
g^{-1}(0)= \widetilde{D_1} \cup \widetilde{E_1} \cup \widetilde{E_2}, 
\qquad 
\nu^{-1}(0)= \widetilde{E_1} \cup \widetilde{E_2} \cup E_3
\end{equation}
and hence by Lemma \ref{lem-1} (i), (ii), (iii) 
and Proposition \ref{PDI} (i) we obtain isomorphisms  
\begin{align}
\psi_f^{\merc}( \CC_X )_{0} 
& \simeq \psi_f^{\merc}( \CC_{X \setminus D} )_{0} 
\simeq \psi_f^{\merc}( \rmR \nu_* \CC_{Y \setminus D^{\prime}} )_{0} 
\\
& \simeq \rmR \Gamma ( \nu^{-1}(0); 
\psi_{f \circ \nu}^{\merc}( \CC_{Y \setminus D^{\prime}} ) )
\\
& \simeq \rmR \Gamma ( \nu^{-1}(0); 
\psi_{g}( \CC_{Y \setminus D^{\prime}} ) )
\\
& \simeq \rmR \Gamma ( \nu^{-1}(0) \cap g^{-1}(0) ; 
\psi_{g}( \CC_{Y \setminus D^{\prime}} ) )
\\
& \simeq \rmR \Gamma ( \widetilde{E_1} \cup \widetilde{E_2} 
 ; \psi_{g}( \CC_{Y \setminus D^{\prime}} ) ). 
\end{align}
Since we know that the complex $\psi_f^{\merc}( \CC_X )_{0}$ is 
concentrated in degree $1$, we obtain also 
\begin{equation}
{\rm dim} H^1 \psi_f^{\merc}( \CC_X )_{0} = - 
\dsum_{j \in \ZZ} (-1)^j {\rm dim} H^j
( \widetilde{E_1} \cup \widetilde{E_2} 
 ; \psi_{g}( \CC_{Y \setminus D^{\prime}} ) ). 
\end{equation}
Now let $\chi ( \psi_{g}( \CC_{Y \setminus D^{\prime}} ) )$ 
be the $\ZZ$-valued constructible function on 
$g^{-1}(0)$ obtained by taking the local Euler-Poincar{\'e} 
index of $\psi_{g}( \CC_{Y \setminus D^{\prime}} )$ at 
each point of $g^{-1}(0)$. Then we have an equality 
\begin{equation}
\dsum_{j \in \ZZ} (-1)^j {\rm dim} H^j
( \widetilde{E_1} \cup \widetilde{E_2} 
 ; \psi_{g}( \CC_{Y \setminus D^{\prime}} ) )
= \int_{\widetilde{E_1} \cup \widetilde{E_2}} 
\chi ( \psi_{g}( \CC_{Y \setminus D^{\prime}} ) ), 
\end{equation}
where $\int_{\widetilde{E_1} \cup \widetilde{E_2}}( \cdot )$ stands for 
the topological (Euler) integral over the 
analytic subset $\widetilde{E_1} \cup \widetilde{E_2} \subset 
g^{-1}(0)$. 
Moreover by the results in \cite[Section 5]{MT11} 
we can easily show that 
\begin{equation}
\chi ( \psi_{g}( \CC_{Y \setminus D^{\prime}} ) )
|_{\widetilde{E_1} \cup \widetilde{E_2}} = 
{\bf 1}_{\widetilde{E_1^{\circ}}} + {\bf 1}_{\widetilde{E_2^{\circ}}}. 
\end{equation}
Together with \eqref{newdiv} we thus 
obtain the desired equality 
\begin{equation}
{\rm dim} H^1 \psi_f^{\merc}( \CC_X )_{0} = 
- \chi ( \widetilde{E_1^{\circ}} ) - \chi ( \widetilde{E_2^{\circ}} )
=1. 
\end{equation}
\end{example}

\section{Singularities at Infinity of 
Meromorphic Functions}\label{sec:s3}

In this section, we study the singularities at 
infinity of rational and 
meromorphic functions on complex 
affine subvarieties of $X= \CC^N_z$. 
Although we treat only rational functions here, 
as is clear from the proofs our results 
hold true also for (possibly multi-valued) 
meromorphic functions with moderate growth 
at infinity. We leave their precise formulations 
to the readers. First, let us consider 
a complex submanifold $S \subset X= \CC^N_z$ of 
$X= \CC^N_z$. Denote by $X_{\RR}$ (resp. $S_{\RR}$) 
the underlying real analytic manifold of $X$ 
(resp. $S$). Then for any point $z \in S_{\RR}$ 
the tangent space $T_zS_{\RR}$ is naturally 
identified with a linear subspace of 
$T_zX_{\RR} \simeq \CC^N$. In fact, 
we can easily see that it has also 
a structure of $\CC$-linear subspace of 
dimension $\dim S$ i.e. $T_zS_{\RR} \simeq \CC^{\dim S}$.
For $u=(u_1, \ldots, u_N), v=(v_1, \ldots, v_N) \in \CC^N$ we set 
\begin{equation}
\langle u, v \rangle := \sum_{i=1}^N u_iv_i \qquad \in \CC. 
\end{equation}
Let 
\begin{equation}
( \cdot, \cdot ): \CC^N \times \CC^N \longrightarrow \CC 
\qquad ((u,v) \longmapsto \langle u, \overline{v} \rangle )
\end{equation}
be the standard Hermitian inner product of $\CC^N$ and 
for $z \in \CC^N$ define its norm $|| z || \geq 0$ by 
$|| z ||= \sqrt{(z,z)}$. For a point $z \in \CC^N$ 
set $V:=T_zX_{\RR} \simeq \CC^N$ and denote by 
$V^*$ its dual space ${\rm Hom}_{\RR}(V, \RR ) \simeq 
T^*_zX_{\RR}$. Then by the perfect pairing 
${\rm Re} ( \cdot, \cdot ): \CC^N \times \CC^N 
\longrightarrow \RR$ we obtain an isomorphism 
\begin{equation}
\phi: V=T_zX_{\RR} \simto V^* =T^*_zX_{\RR}
\qquad (v \longmapsto \{ u \mapsto {\rm Re} (u,v) \} ). 
\end{equation}
Define a real-valued function $\distsq : X_{\RR} 
\longrightarrow \RR$ by 
\begin{equation}
\distsq (z):= \frac{1}{2} || z ||^2=  \frac{1}{2} (z,z). 
\end{equation}
Then for any point $z \in  X_{\RR}$, the 
cotangent vector $d \distsq (z) \in T^*_zX_{\RR}$ 
corresponds to the point $z \in T_zX_{\RR} \simeq \CC^N$ 
itself via the isomorphism 
$\phi: V=T_zX_{\RR} \simto V^* =T^*_zX_{\RR}$. 
Fix a point $z \in S_{\RR} \subset X_{\RR}$ and 
set $W:=T_zS_{\RR} \simeq \CC^{\dim S} \subset 
V:=T_zX_{\RR} \simeq \CC^N$. Denote by $V^*$ 
(resp. $W^*$) the dual space 
${\rm Hom}_{\RR}(V, \RR ) \simeq 
T^*_zX_{\RR}$ (resp. ${\rm Hom}_{\RR}(W, \RR ) \simeq 
T^*_zS_{\RR}$).  Then by the perfect pairing 
${\rm Re} ( \cdot, \cdot ): W \times W 
\longrightarrow \RR$ we obtain also an isomorphism 
$\psi: W=T_zS_{\RR} \simto W^* =T^*_zS_{\RR}$. 

\begin{lemma}\label{lem-3} 
Let $\Psi_z: V^*= T^*_zX_{\RR} \longrightarrow 
W^*= T^*_zS_{\RR}$ be the ($\RR$-linear) morphism 
induced by the inclusion map $S_{\RR} \hookrightarrow 
X_{\RR}$. Then there exists a $\CC$-linear morphism 
$\Phi_z: V= T_zX_{\RR} \longrightarrow 
W= T_zS_{\RR}$ such that the diagram 
\begin{equation}\label{comm-dia}
\begin{CD}
V   @>{\phi}>{\sim}> V^*
\\
@V{\Phi_z}VV   @VV{\Psi_z}V
\\
W   @>{\psi}>{\sim}>  W^* 
\end{CD}
\end{equation}
communtes. 
\end{lemma}

\begin{proof}
Assume first that $S$ is a complex hypersurface in 
$X= \CC^N$. For a holomorphic function 
$h$ defined on a neighborhood of the point 
$z \in S$ in $X= \CC^N$ such that $S=h^{-1}(0)$ 
we set 
\begin{equation}
\grad h(z): = 
\Bigl( \frac{\partial h}{\partial z_1}, 
\frac{\partial h}{\partial z_2}, 
\ldots \ldots, 
\frac{\partial h}{\partial z_N} 
\Bigr) \in \CC^N. 
\end{equation}
Then for any smooth curve $p(t):(- \varepsilon, \varepsilon ) 
\longrightarrow S_{\RR}$ 
($ \varepsilon >0$) 
such that $p(0)=z$ we have 
\begin{equation}
\frac{d}{dt}h(p(t))|_{t=0}= 
\Bigl\langle \frac{dp}{dt}(0), \grad h(z) \Bigr\rangle =0. 
\end{equation}
This implies that the $\CC$-linear subspace 
$W=T_zS_{\RR} \subset V=T_zX_{\RR}= \CC^N$ 
is explicitly given by 
\begin{equation}
W= \{ v \in V \ | \ 
\langle v, \grad h(z) \rangle = 
(v, \overline{\grad h(z)} )=0 \}, 
\end{equation}
where we set 
\begin{equation}
\overline{\grad h(z)}: = 
\Bigl( \overline{\frac{\partial h}{\partial z_1}}, 
\overline{\frac{\partial h}{\partial z_2}}, 
\ldots \ldots, 
\overline{\frac{\partial h}{\partial z_N}}
\Bigr) \in \CC^N. 
\end{equation}
Hence we can define a $\CC$-linear morphism 
$\Phi_z: V \longrightarrow W$ by 
the formula 
\begin{equation}\label{restequa}
\Phi_z(v)= v- 
\frac{(v, \overline{\grad h(z)})}{(
\overline{\grad h(z)}, \overline{\grad h(z)})}
\cdot \overline{\grad h(z)} \quad \in W.  
\end{equation}
Then it is straightforward to check that 
the diagram \eqref{comm-dia} commutes. 
If $\dim S< \dim X-1=N-1$, we repeat this 
argument. 
\end{proof}

In what follows, we assume that 
$S \subset X= \CC^N$ is a smooth Zariski 
locally closed (complex) affine subvariety 
i.e. a smooth quasi-affine subvariety of $X= \CC^N$ 

\begin{lemma}\label{lem-4} 
Let $p(t):(0, \varepsilon ) \longrightarrow S 
\subset X= \CC^N$ be an analytic curve on $S$ 
such that 
\begin{equation}
\lim_{t \to +0} || p(t) || = + \infty 
\end{equation}
and its expansion at infinity is of the form 
\begin{equation}
p(t)=a t^{\alpha} + ( \text{higher order terms} )
\qquad (a \in \CC^N \setminus \{ 0 \}, 
\alpha <0). 
\end{equation}
Define an analytic curve 
$q(t):(0, \varepsilon ) \longrightarrow 
X= \CC^N$ by 
\begin{equation}
q(t)= \Phi_{p(t)}(p(t)) \qquad (0<t< \varepsilon ). 
\end{equation}
Recall that the element $d \distsq (p(t)) \in T^*_{p(t)}X_{\RR}$ 
corresponds to $p(t) \in T_{p(t)}X_{\RR}= \CC^N$ 
via the isomorphism 
$T_{p(t)}X_{\RR} \simto T^*_{p(t)}X_{\RR}$ 
and the point $\Phi_{p(t)}(p(t)) \in T_{p(t)}S_{\RR}$ 
is considered as an element of $T_{p(t)}X_{\RR}= \CC^N$. 
Then the expansion at infinity of $q(t)$ 
has the same top term as that of $p(t)$: 
\begin{equation}
q(t)=a t^{\alpha} + ( \text{higher order terms} )
\qquad (a \in \CC^N \setminus \{ 0 \}, 
\alpha <0). 
\end{equation}
\end{lemma}

\begin{proof}
As in the proof of Lemma \ref{lem-3}, 
assume first that $S$ is a complex hypersurface in 
$X= \CC^N$. Consider a holomorphic function 
$h$ defined on a neighborhood of a point 
$z \in S$ in $X= \CC^N$ such that $S=h^{-1}(0)$. 
First of all, note that we have 
\begin{equation}
p(t)= \frac{t}{\alpha} \cdot \frac{dp}{dt}(t)+ 
( \text{higher order terms} )
\end{equation}
and set 
\begin{equation}
r(t):=( \text{higher order terms} )=
p(t)- \frac{t}{\alpha} \cdot \frac{dp}{dt}(t). 
\end{equation}
Since $\frac{dp}{dt}(t)$ is a tangent vector of 
the manifold $S_{\RR}$ at the point $p(t) \in 
S_{\RR}$, by the proof of Lemma \ref{lem-3} 
we have 
\begin{equation}
\Bigl( 
\frac{dp}{dt}(t),  \overline{\grad h(p(t))} 
\Bigr) = 
\Bigl\langle 
\frac{dp}{dt}(t), \grad h(p(t))  
\Bigr\rangle =0. 
\end{equation}
Then by \eqref{restequa} we obtain 
\begin{equation}
q(t)= \Phi_{p(t)}(p(t))= p(t)- 
\frac{(r(t), \overline{\grad h(p(t))})}{(
\overline{\grad h(p(t))}, \overline{\grad h(p(t))})}
\cdot \overline{\grad h(p(t))} 
\end{equation}
from which the asssertion immediately follows. 
If $\dim S< \dim X-1=N-1$, we repeat this 
argument. 
\end{proof}

From now on, let us consider also a rational function 
$f= \frac{P}{Q}: X \setminus Q^{-1}(0) 
\longrightarrow \CC$ ($P, Q \in 
\Gamma(X; \SO_X) \simeq \CC [z_1,z_2, \ldots, z_N]$, 
$Q \not= 0$) on the smooth algebraic variety 
$X= \CC^N_z$ and set 
$U:= X \setminus Q^{-1}(0)$. Let $S \subset U$ 
be a smooth quasi-affine subvariety of $U$ and 
denote the restriction $f|_S: S \longrightarrow \CC$ 
of $f$ to it by $g$. Then 
for the cotangent vector $d {\rm Re} g(z) \in T^*_zS_{\RR}$ 
at a point $z \in S_{\RR}$ 
we define its norm $||  d {\rm Re} g(z) || \geq 0$ 
to be that of the element $\grad {\rm Re} g(z) \in T_zS_{\RR}
\subset T_zX_{\RR} \simeq \CC^N$ which corresponds to 
it by the isomorphism 
$\psi : T_zS_{\RR} \simto T^*_zS_{\RR}$. For 
polynomial functions $f: \CC^N \longrightarrow \CC$ 
on $\CC^N$ the following condition was first introduced by 
Broughton \cite{Bro88}. 

\begin{definition}\label{def-7} 
Let $S \subset U$ and $g:S \longrightarrow \CC$ be 
as above. Then we say that $g$ is tame 
at infinity if there exist $R \gg 0$ and 
$0< \varepsilon \ll 1$ such that 
\begin{equation}
\{ z \in S \ | \ ||  d {\rm Re} g(z) || < \varepsilon, \ 
|| z || >R \} = \emptyset. 
\end{equation}
\end{definition}
For polynomial functions $f: \CC^N \longrightarrow \CC$ 
on $\CC^N$ the following result was essentially obtained in 
N{\'e}methi-Zaharia \cite{NZ90} and \cite{NZ92}. 
Here we modify their arguments with the help of 
Lemma \ref{lem-4}. See also the proof of 
Nguyen-Pham-Pham \cite[Theorem 3.1]{NPP19} for a 
similar result on complete intersection subvarieties 
$S \subset X= \CC^N$ and in a different formulation 
in terms of Rabier's norm defined in \cite{R97}. 

\begin{proposition}\label{prop-7} 
Let $S \subset U$ and $g:S \longrightarrow \CC$ be 
as above and assume that $g$ is tame 
at infinity. We define a subset $M(g)$ of $S$ by 
\begin{equation}
M(g)=  \{ z \in S \ | \   
d( \distsq |_S)(z) \in \RR d {\rm Re} g(z) + \RR d {\rm Im} g(z)  \}. 
\end{equation} 
Then for any $\tau_0 \in \CC$  there exist $R \gg 0$ and 
$0< \varepsilon \ll 1$ such that 
\begin{equation}
M(g) \cap \{ z \in S \ | \ |g(z)- \tau_0| < \varepsilon, \ 
|| z || >R \} = \emptyset. 
\end{equation}
\end{proposition}

\begin{proof}
Fix a point $z \in S$ and recall that 
we set $g=f|_S$. We denote by 
$\grad {\rm Re} f(z) \in T_zX_{\RR} \simeq \CC^N$ 
the element which corresponds to 
$d {\rm Re} f(z) \in T^*_zX_{\RR}$
by the isomorphism 
$\phi : T_zX_{\RR} \simto T^*_z X_{\RR}$. Then 
by the Cauchy-Riemann equations we have 
\begin{align}\label{C-R-eqn}
\grad {\rm Re} f(z) 
& =\Bigl( 
\frac{\partial {\rm Re} f}{\partial x_1}
+ i \frac{\partial {\rm Re} f}{\partial y_1}, 
\ldots \ldots, 
\frac{\partial {\rm Re} f}{\partial x_N}
+ i \frac{\partial {\rm Re} f}{\partial y_N}
\Bigr)
\\
& = 
\overline{\grad f(z)}= 
\Bigl( \overline{\frac{\partial f}{\partial z_1}}, 
\ldots \ldots, 
\overline{\frac{\partial f}{\partial z_N}}
\Bigr), 
\end{align}
where we set $z_i=x_i+i y_i$ ($1 \leq i \leq N$). 
Moreover we have 
\begin{equation}
\grad {\rm Im} f(z) 
=\Bigl( 
\frac{\partial {\rm Im} f}{\partial x_1}
+ i \frac{\partial {\rm Im} f}{\partial y_1}, 
\ldots \ldots, 
\frac{\partial {\rm Im} f}{\partial x_N}
+ i \frac{\partial {\rm Im} f}{\partial y_N}
\Bigr)
 = i \cdot \grad {\rm Re} f(z). 
\end{equation}
By Lemma \ref{lem-3}
this implies that for a point $z \in S_{\RR}$ 
such that $||z|| \gg 0$ the condition 
\begin{equation}
d( \distsq |_S)(z) \in \RR d {\rm Re} g(z) + \RR d {\rm Im} g(z) 
\end{equation}
is equivalent to the one that there exists $\lambda \in \CC^*
:= \CC \setminus \{ 0 \}$ such that 
\begin{equation}
\grad ( \distsq |_S)(z) = \Phi_z ( \lambda \cdot \grad {\rm Re} f(z)). 
\end{equation}
Here we used the fact that if $||z|| \gg 0$ we have 
$d( \distsq |_S)(z) \not= 0$ (see e.g. \cite[Corollary 2.8]{Mil68}). 
Since we have $\grad ( \distsq |_S)(z)= \Phi_z 
( \grad \distsq (z))= \Phi_z(z)$ and $\Phi_z$ is $\CC$-linear, 
it is also equivalent to the one that there exists $\lambda \in \CC^*
:= \CC \setminus \{ 0 \}$ such that 
\begin{equation}
\lambda \cdot \Phi_z(z) = 
\Phi_z( \grad {\rm Re} f(z) )= \grad {\rm Re} g(z). 
\end{equation}
Now let us prove the assertion by showing a contradiction. 
Assume to the contrary that there exists a sequence 
$z_n \in M(g) \subset S \subset X= \CC^N$ ($n=1,2, \ldots$) such that 
\begin{equation}
\lim_{n \to + \infty} |g(z_n)- \tau_0| =0,  
\qquad 
\lim_{n \to + \infty} || z_n || = + \infty. 
\end{equation}
Then by the curve selection lemma \cite[Lemma 4]{NZ92} 
of N{\'e}methi and Zaharia, there exists 
an analytic curve  $p(t):(0, \varepsilon ) \longrightarrow M(g)
\subset S \subset X= \CC^N$ 
such that 
\begin{equation}
\lim_{t \to +0} |g(p(t))- \tau_0| =0,  
\qquad 
\lim_{t \to +0} || p(t) || = + \infty. 
\end{equation}
Assume that its expansion at infinity is of the form 
\begin{equation}
p(t)=a t^{\alpha} + ( \text{higher order terms} )
\qquad (a \in \CC^N \setminus \{ 0 \}, 
\alpha <0). 
\end{equation}
and define $q(t):(0, \varepsilon ) \longrightarrow X= \CC^N$ by 
\begin{equation}
q(t)= \Phi_{p(t)}(p(t)) \qquad (0<t< \varepsilon ). 
\end{equation}
Then by Lemma \ref{lem-4} 
the expansion at infinity of $q(t)$ 
has the same top term as that of $p(t)$: 
\begin{equation}
q(t)=a t^{\alpha} + ( \text{higher order terms} )
\qquad (a \in \CC^N \setminus \{ 0 \}, 
\alpha <0). 
\end{equation}
Moreover, by the condition 
\begin{equation}
\lim_{t \to +0} g(p(t)) = \tau_0,  
\end{equation}
the expansion at infinity 
of $g(p(t))$ is of the form 
\begin{equation}\label{expainf} 
g(p(t))=b t^{\beta} + ( \text{higher order terms} )
\qquad (b \in \CC^* , \beta \geq 0). 
\end{equation}
Since $g$ is tame at infinity, the expansion at infinity 
of $\grad {\rm Re} g(p(t)): (0, \varepsilon ) \longrightarrow \CC^N$ 
is of the form 
\begin{equation}
\grad {\rm Re} g(p(t))=c t^{\gamma} + ( \text{higher order terms} )
\qquad (c \in \CC^N \setminus \{ 0 \}, \gamma \leq 0). 
\end{equation}
Then by the condition $p(t) \in M(g) \subset S$ 
($0 < t < \varepsilon$) and the above argument, 
there exists an analytic curve 
$\lambda (t):  (0, \varepsilon ) \longrightarrow \CC^*$ 
such that 
\begin{equation}
\lambda (t) \cdot q(t)= \lambda (t) \cdot \Phi_{p(t)}(p(t))
= \grad {\rm Re} g(p(t)) 
\qquad (0 < t < \varepsilon ). 
\end{equation}
If the expansion at infinity 
of $\lambda (t)$ is of the form 
\begin{equation}
\lambda (t)= \lambda_0 t^{\gamma - \alpha} + ( \text{higher order terms} )
\qquad ( \lambda_0 \in \CC^*), 
\end{equation}
we have $c= \lambda_0 a \in \CC^N \setminus \{ 0 \}$ and 
\begin{equation}
\grad {\rm Re} g(p(t))= \lambda_0 a t^{\gamma} + ( \text{higher order terms} )
\qquad ( \gamma \leq 0). 
\end{equation}
On the other hand, for the Hermitian inner product 
$( \cdot, \cdot )$ we have 
\begin{align*}
\frac{d}{dt}g(p(t))
&  =  \frac{d}{dt}f(p(t)) = 
\Bigl( \frac{dp}{dt}(t), \overline{\grad f(p(t))} \Bigr) 
\\
& =
\Bigl( \frac{dp}{dt}(t), \grad {\rm Re} f(p(t)) \Bigr) 
\\
& = 
\Bigl( \frac{dp}{dt}(t), \Phi_{p(t)} ( \grad {\rm Re} f(p(t)) ) \Bigr) 
\\
& = 
\Bigl( \frac{dp}{dt}(t), \grad {\rm Re} g(p(t)) \Bigr) 
\\
& = 
\Bigl( \frac{dp}{dt}(t), \lambda (t) \cdot q(t) \Bigr) 
\\
& = 
( \alpha a, \lambda_0 a) \cdot 
t^{\alpha - 1 + \gamma} + ( \text{higher order terms} ),  
\end{align*}
where we used \eqref{restequa} in the fourth equality 
(see also the proof of Lemma \ref{lem-3}). 
Since $( \alpha a, \lambda_0 a)= \alpha 
\overline{\lambda_0} || a ||^2 \not= 0$, 
$\alpha - 1 + \gamma <-1$ and the left hand side 
$\frac{d}{dt}g(p(t))$ is of order $> -1$ in $t$ 
by \eqref{expainf},  
we obtain the desired contradiction. This 
completes the proof. 
\end{proof}

For $r>0$ let $B_r(0)= \{ z \in X= \CC^N \ | \ 
|| z || <r \} \subset X= \CC^N$ be the open ball 
of radius $r$ centered at the origin $0 \in 
X= \CC^N$. Then Propositon \ref{prop-7} 
means that if $g:S \longrightarrow \CC$ is 
tame at infinity the boundary $\partial B_r(0) 
\cap S$ of $B_r(0) \cap S$ for $r \gg 0$ 
intersects the fibers $g^{-1}( \tau ) \subset S$ 
($| \tau - \tau_0| \ll 1$) of $g$ transversally. 
Now let $K \in \Dbc (X^{\an})$ be 
an algebraic constructible sheaf 
on $X^{\an}$. Namely we assume that $K$ is 
adapted to an algebraic stratification of $X^{\an}
= \CC^N$. Then the micro-support 
${\rm SS }(K) \subset T^*X^{\an}$ of $K$ is 
a homogeneous Lagrangian subvariety of 
$(T^*X)^{\an} \simeq T^*X^{\an}$ and 
there exists an algebraic Whitney 
stratification $\CS$ of $X= \CC^N_z$ 
such that 
\begin{equation}
{\rm SS }(K) \subset \bigsqcup_{S \in \CS} T^*_SX. 
\end{equation}
Recall that by the Whitney condition the right 
hand side is a closed subset in $(T^*X)^{\an} \simeq T^*X^{\an}$. 
For $r>0$ we define an open subset 
$U_r \subset U$ of $U=X \setminus Q^{-1}(0)$ by 
\begin{equation}
U_r:= U \cap B_r(0)= 
\{ z \in U \ | \ ||z|| <r \} \subset U \subset X= \CC^N. 
\end{equation}
Then in order to apply the direct image theorem 
\cite[Theorem 4.4.1]{KS85} for non-proper maps of 
Kashiwara and Schapira to 
$K|_U \in \Dbc (U^{\an})$ and $f: U=X \setminus Q^{-1}(0) 
\longrightarrow \CC$, we shall prove the 
following result. 

\begin{proposition}\label{prop-micros} 
Let $\CS$ be the Whitney 
stratification of $X= \CC^N_z$ as above and 
assume that for any stratum $S \in \CS$ in it 
such that $S \cap U \not= \emptyset$ and 
$T^*_SX \subset {\rm SS }(K)$ 
the restriction $f|_{S \cap U}: S \cap U 
\longrightarrow \CC$ of $f$ to $S \cap U$ 
is tame at infinity. 
Then for any $\tau_0 \in \CC$  there exist $R \gg 0$ and 
$0< \varepsilon \ll 1$ such that we have 
\begin{equation}
N^*(U_r) \cap \Bigl\{ 
\overline{{\rm SS }(K|_U) + \rho 
(U \times_{\CC} T^* \CC) }
\Bigr\} \subset T^*_UU 
\end{equation}
(for the definition of $N^*(U_r) \subset T^*U$ 
see \cite[Definition 5.3.6]{KS90}) over 
the open subset $\{ z \in U \ | \ 
|f(z)- \tau_0| < \varepsilon \} 
\subset U$ of $U$ for any $r \geq R$, 
where $\rho : U \times_{\CC} T^* \CC 
\longrightarrow T^*U$ is the natural morphism 
associated to $f: U \longrightarrow \CC$. 
Moreover the same is true even after replacing $N^*(U_r)$ 
by its antipodal set $N^*(U_r)^{{\rm a}}$. 
\end{proposition}

\begin{proof}
Since the subset $\overline{{\rm SS }(K|_U) + \rho 
(U \times_{\CC} T^* \CC) } \subset T^*U$ of $T^*U$ 
is stable by the antipodal map of $T^*U$, 
it suffices to consider only $N^*(U_r)$. 
We prove the assertion by showing a contradiction. 
Assume to the contrary that there exists a sequence 
$z_n \in U \subset X= \CC^N$ ($n=1,2, \ldots$) such that 
\begin{equation}
\lim_{n \to + \infty} |f(z_n)- \tau_0| =0,  
\qquad 
\lim_{n \to + \infty} || z_n || = + \infty 
\end{equation}
and for the non-zero inner conormal vector 
$- d \distsq (z_n) \in \CC^N$ of the real hypersurface 
$\partial U_{||z_n||} \subset U$ at 
the point $z_n \in U$ we have 
\begin{equation}
(z_n, w_n):=(z_n, - d \distsq (z_n)) \in 
\overline{{\rm SS }(K|_U) + \rho 
(U \times_{\CC} T^* \CC) } \subset T^*U
\end{equation}
for any $n=1,2, \ldots$. By taking a 
subsequence, we may assume that there 
exists a stratum $S \in \CS$ such that 
$S \cap U \not= \emptyset$ and 
$z_n \in S \cap U$ for any $n=1,2, \ldots$. 
First, let us consider the case where 
$T^*_SX \subset {\rm SS }(K)$. Then for 
each $n=1,2, \ldots$ there exist a 
stratum $S^{\prime} \in \CS$ such that 
$S \subset \overline{S^{\prime}}$ and 
$T^*_{S^{\prime}}X \subset {\rm SS }(K)$ 
and sequences $(z_{nm}, w_{nm}) \in T^*_{S^{\prime} \cap U}X$, 
$\lambda_{nm} \in \CC$ ($m=1,2, \ldots$) such that 
\begin{equation}\label{esseneqa} 
\lim_{m \to + \infty} z_{nm} =z_n, 
\qquad 
\lim_{m \to + \infty} 
\bigl( w_{nm} + \lambda_{nm} \cdot df( z_{nm} ) 
\bigr) =w_n. 
\end{equation}
Since $f|_{S \cap U}: S \cap U 
\longrightarrow \CC$ is tame at infinity 
by our assumption, there exists 
$0< \varepsilon_0 \ll 1$ such that 
for large enough $n \gg 0$ we have 
the condition $|| d(f|_{S \cap U})(z_n)|| 
\geq \varepsilon_0 >0$. On the other hand, 
by the Whitney condition of $\CS$ 
a subsequence of 
$w_{nm}/||w_{nm}|| \in \CC^N$  ($m=1,2, \ldots$, 
$||w_{nm}|| \not= 0$) 
converges to a point in $(T^*_SX)_{z_n}$. 
As $df(z) \in \CC^N$ is holomorphic with respect to 
$z$ and bounded on a neighborhood of the point $z_n \in 
S \cap U$ in $U$, these conditions 
in together imply that all the sequences 
$w_{nm} \in \CC^N$, $\lambda_{nm} \cdot df( z_{nm} ) \in \CC^N$, 
$\lambda_{nm} \in \CC$ ($m=1,2, \ldots$) 
are bounded. Indeed, if the sequence $w_{nm} \in \CC^N$ ($m=1,2, \ldots$) 
is not bounded, by \eqref{esseneqa} a subsequence 
of $\lambda_{nm} \cdot df( z_{nm} ) \in \CC^N$ ($m=1,2, \ldots$) 
goes to at infinity of a direction in $(T^*_SX)_{z_n}$. 
This contradicts with the condition $|| d(f|_{S \cap U})(z_n)|| 
\geq \varepsilon_0 >0$. 
Hence, by taking their subsequences, 
we may assume that all of them converge. 
In particular, by the Whitney condition of $\CS$ we have 
\begin{equation}
\lim_{m \to + \infty} 
w_{nm} \in (T^*_SX)_{z_n}. 
\end{equation}
Hence for any $n=1,2, \ldots$ 
there exists $\lambda_n \in \CC$ such that 
\begin{equation}
d( \distsq |_{S \cap U})(z_n)= \lambda_n \cdot 
d(f|_{S \cap U})(z_n) 
\end{equation}
and we get a contradiction by Proposition 
\ref{prop-7}. Next, let us consider the case where 
$T^*_SX$ is not contained in ${\rm SS }(K)$. 
Also in this case, for 
each $n=1,2, \ldots$ there exist a 
stratum $S^{\prime} \in \CS$ such that 
$S \subset \overline{S^{\prime}}$ and 
$T^*_{S^{\prime}}X \subset {\rm SS }(K)$ 
and sequences $(z_{nm}, w_{nm}) \in T^*_{S^{\prime} \cap U}X$, 
$\lambda_{nm} \in \CC$ ($m=1,2, \ldots$) such that 
\begin{equation}
\lim_{m \to + \infty} z_{nm} =z_n, 
\qquad 
\lim_{m \to + \infty} 
\bigl( w_{nm} + \lambda_{nm} \cdot df( z_{nm} ) 
\bigr) =w_n. 
\end{equation}
By taking their subsequences and using 
the Whitney condition of $\CS$, we may assume that 
the tangent planes $T_{z_{nm}}S^{\prime} \subset \CC^N$ of 
$S^{\prime}$ at $z_{nm} \in S^{\prime}$ 
converge to a linear subspace 
$\TT_n \subset T_{z_n}X \simeq \CC^N$ 
such that $T_{z_n}S \subset \TT_n$ as 
$m$ tends to $+ \infty$. Let 
\begin{equation}
\Phi_n: T^*_{z_n}X \simeq \CC^N \longrightarrow 
\TT^*_n:= {\rm Hom}_{\CC}( \TT_n, \CC )
\end{equation}
be the surjective linear map associated to 
the inclusion map $\TT_n \hookrightarrow T_{z_n}X \simeq \CC^N$. 
Then by our assumption for the stratum $S^{\prime}$, 
there exists $0< \varepsilon_0 \ll 1$ such that 
for large enough $n \gg 0$ we have 
the condition $|| \Phi_n \bigl( df(z_n) \bigr) || 
\geq \varepsilon_0 >0$. 
Similarly to the previous case, we can thus 
show that all the sequences 
$w_{nm} \in \CC^N$, $\lambda_{nm} \cdot df( z_{nm} ) \in \CC^N$, 
$\lambda_{nm} \in \CC$ ($m=1,2, \ldots$) 
are bounded. By taking subsequences, 
we can assume also that all of them converge. 
Hence for any $n=1,2, \ldots$ 
there exists $\lambda_n \in \CC$ such that 
\begin{equation}
 \Phi_n \bigl( d \distsq (z_n) \bigr) = \lambda_n \cdot 
 \Phi_n \bigl( df(z_n) \bigr). 
\end{equation}
By taking a subsequence of $z_n \in S \cap U$ 
($n=1,2, \ldots$), we may assume that 
there exists a stratum $S^{\prime} \in \CS$ 
 such that $S \subset \overline{S^{\prime}}$,  
$T^*_{S^{\prime}}X \subset {\rm SS }(K)$,  
and for any 
$n=1,2, \ldots$ the linear subspace 
$\TT_n \subset T_{z_n}X \simeq \CC^N$ at $z_n \in S \cap U$ 
is a limit of some tangent spaces 
of $S^{\prime}$. 
We fix such $S^{\prime}$ once and for all. 
Set $l:= {\rm dim }S^{\prime}= {\rm dim } \TT_n$ 
($n=1,2, \ldots$) and let ${\rm Gr}$ be the 
complex Grassmann manifold consisting of 
$l$-dimensional linear subspaces of $\CC^N$. 
Let $A$ be a subset of $(S \cap U) \times {\rm Gr}$ 
consisting of pairs $(z, \TT ) \in (S \cap U) \times {\rm Gr}$ 
such that there exists 
a sequence $z_m \in S^{\prime} \cap U$ ($m=1,2, \ldots$) 
in the stratum $S^{\prime}$ such that 
\begin{equation}
\lim_{m \to + \infty} z_m =z, 
\qquad 
\lim_{m \to + \infty} T_{z_m} S^{\prime} 
= \TT
\end{equation}
and for the surjective linear map  
\begin{equation}
\Phi_{(z, \TT )}: T^*_{z}X \simeq \CC^N \longrightarrow 
\TT^*:= {\rm Hom}_{\CC}( \TT , \CC )
\end{equation}
associated to the inclusion map $\TT \hookrightarrow T_{z}X 
\simeq \CC^N$ we have 
\begin{equation}
\Phi_{(z, \TT )} \bigl( d \distsq (z) \bigr) = \lambda \cdot 
\Phi_{(z, \TT )} \bigl( df(z) \bigr)
\end{equation}
for some $\lambda \in \CC$. Then 
we can easily show that $A \subset (S \cap U) \times {\rm Gr} 
\subset X \times {\rm Gr}$ is a semi-analytic 
subset of $X \times {\rm Gr}$. 
In the above argument, we obtained a sequence 
$(z_n, \TT_n) \in A \subset (S \cap U) \times {\rm Gr}$ 
in $A$ such that 
\begin{equation}
\lim_{n \to + \infty} |f(z_n)- \tau_0| =0,  
\qquad 
\lim_{n \to + \infty} || z_n || = + \infty. 
\end{equation}
Since the complex Grassmann manifold ${\rm Gr}$ 
is covered by finitely many open subsets 
isomorphic to $\CC^{l(N-l)}$, we can apply 
the curve selection lemma \cite[Lemma 4]{NZ92} 
of N{\'e}methi and Zaharia to find 
an analytic curve 
\begin{equation}
q(t)=(p(t), \TT (t)) :(0, \varepsilon ) \longrightarrow 
A \subset (S \cap U) \times {\rm Gr}
\end{equation}
in $A$ satisfying the conditions 
\begin{equation}
\lim_{t \to +0} |f(p(t))- \tau_0| =0,  
\qquad 
\lim_{t \to +0} || p(t) || = + \infty. 
\end{equation}
Then we obtain a contradiction as in 
the proof of Proposition \ref{prop-7}. 
More precisely, by taking a family of conormal 
vectors of the planes $\TT (t) \subset \CC^N$ 
which depend analytically on 
$t \in (0, \varepsilon )$, we obtain a result 
similar to Lemma \ref{lem-4} and 
can apply the proof of Proposition \ref{prop-7} 
to our situation. 
This completes the proof. 
\end{proof}

By the proof of Proposition \ref{prop-micros}, 
we obtain also the followng simple consequence. 

\begin{lemma}
In the situation of Proposition \ref{prop-micros}, 
for any $\tau_0 \in \CC$  there exist $R \gg 0$ and 
$0< \varepsilon \ll 1$ such that we have 
\begin{equation}
df(z) \notin {\rm SS }(K|_U)
\end{equation}
for any $z \in U$ such that $|| z || \geq R$ 
and $| f(z)- \tau_0|< \varepsilon$. 
\end{lemma} 

First, let us consider the 
special case where $I(f)=P^{-1}(0) \cap Q^{-1}(0)
= \emptyset$. Note that this condition is 
satisfied for example if $f= \frac{P}{Q}$ is a polynomial 
i.e. $Q=1$ or $P=1$. 

\begin{theorem}\label{np-vc}
In the situation of Proposition \ref{prop-micros}, 
assume also that $I(f)=P^{-1}(0) \cap Q^{-1}(0)
= \emptyset$. 
Then we have the following results. 
\begin{enumerate}
\item[\rm{(i)}] 
For any $\tau_0 \in \CC$ there exist $R \gg 0$ and 
$0< \varepsilon \ll 1$ such that for the 
inclusion maps $i_r: B_r(0) \hookrightarrow X= \CC^N$ 
($r>R$) we have isomorphisms 
\begin{equation}
 \rmR f_! (i_r)_! ( i_r^{-1} K ) 
\simeq  \rmR f_!  K_{B_r(0)}
\simto  \rmR f_! K
\end{equation}
on the open subset $\{ \tau \in \CC \ | \ 
 | \tau - \tau_0| < \varepsilon \} \subset \CC$ of $\CC$. 
\item[\rm{(ii)}] 
For any $\tau_0 \in \CC$, if we assume also that 
\begin{equation}
\{ z \in f^{-1}( \tau_0) \ | \ 
df(z) \in {\rm SS }(K) \} \subset 
f^{-1}( \tau_0) \cap {\rm supp }(K) 
\end{equation}
is a finite subset of $f^{-1}( \tau_0) \cap {\rm supp }(K)$ 
and let $p_1,p_2, \ldots, p_k$ be the points in it, 
then for the vanishing cycle 
$\phi_{\tau - \tau_0}( \rmR f_! K ) \in 
\BDC (\{ \tau_0 \} )$ of $\rmR f_! K$ 
along the function $\tau - \tau_0: \CC \longrightarrow \CC$ 
we have an isomorphism 
\begin{equation}
\phi_{\tau - \tau_0}( \rmR f_! K) \simeq 
\bigoplus_{i=1}^k 
\phi_{f- \tau_0}( K)_{p_i}. 
\end{equation}
If we assume moreover that 
$K[N] \in \Dbc (X^{\an})$ is a perverse sheaf, 
then we have also a concentration 
\begin{equation}
H^j \phi_{\tau - \tau_0}( \rmR f_! K ) \simeq 0 
\qquad (j \not= N-1). 
\end{equation}
\end{enumerate}
\end{theorem}

\begin{proof}
By combining \cite[Theorem 4.4.1]{KS85} with 
Proposition \ref{prop-micros} 
we obtain the isomorphism in (i). 
Let us prove (ii). For $R \gg 0$ such that 
$p_1,p_2, \ldots, p_k \in B_R(0)$ and 
$0< \varepsilon \ll 1$ in (i) 
let us fix $r>R$. 
Then the morphism $f$ being proper on 
the support of 
\begin{equation}
(i_r)_! ( i_r^{-1} K ) 
\simeq K_{B_r(0)}, 
\end{equation}
by (i) we obtain an isomorphism 
\begin{equation}
\phi_{\tau - \tau_0}( \rmR f_!  K )  \simeq 
\phi_{\tau - \tau_0} (  \rmR  f_* K_{B_r(0)} ) 
 \simeq 
 \rmR  \Gamma ( f^{-1}( \tau_0); 
\phi_{f- \tau_0} ( K_{B_r(0)} ) ). 
\end{equation}
Moreover by our assumption, the 
complex hypersurface $f^{-1}( \tau_0) \subset 
X= \CC^N$ is smooth on a neighborhood of 
\begin{equation}
\bigl( f^{-1}( \tau_0) \setminus 
\{ p_1,p_2, \ldots, p_k \} 
\bigr) \cap {\rm supp }(K)
\end{equation}
in $X= \CC^N$. Then by 
\cite[Propositon 8.6.3]{KS90} we see that 
the support of $\phi_{f- \tau_0} ( K_{B_r(0)} ) 
\in \BDC ( f^{-1}( \tau_0) )$ is contained in the set 
\begin{equation}
\{ p_1,p_2, \ldots, p_k \} \sqcup 
\bigl( f^{-1}( \tau_0) \cap \partial B_r(0) \bigr). 
\end{equation}
Now let us take $r^{\prime} >R$ such that 
$r> r^{\prime} >R$. Then by the proof of 
\cite[Theorem 4.4.1]{KS85} there exists an 
isomorphism 
\begin{equation}
\rmR f_! K_{B_{r^{\prime}}(0)}
\simto  \rmR f_! K_{B_r(0)}
\end{equation}
on the open subset $\{ \tau \in \CC \ | \ 
 | \tau - \tau_0| < \varepsilon \} \subset \CC$ of $\CC$. 
In other words, we have a vanishing 
\begin{equation}
\rmR f_! ( K_{B_r(0) \setminus B_{r^{\prime}}(0)} ) 
\simeq 0
\end{equation}
there. This implies that we have 
\begin{equation}
 \rmR  \Gamma ( f^{-1}( \tau_0); 
\phi_{f- \tau_0} ( K_{B_r(0) \setminus B_{r^{\prime}}(0)} ) )
\simeq 0. 
\end{equation}
Since the support of $\phi_{f- \tau_0} ( 
K_{B_r(0) \setminus B_{r^{\prime}}(0)} ) 
\in \BDC ( f^{-1}( \tau_0) )$ is contained in the set 
\begin{equation}
\bigl( f^{-1}( \tau_0) \cap \partial B_r(0) \bigr)
 \sqcup 
\bigl( f^{-1}( \tau_0) \cap \partial B_{r^{\prime}}(0) \bigr)
\end{equation}
and hence the object $\rmR \Gamma ( f^{-1}( \tau_0) \cap \partial B_r(0); 
\phi_{f- \tau_0} ( K_{B_r(0)} ) )$ is a direct 
summand of the one $ \rmR  \Gamma ( f^{-1}( \tau_0); 
\phi_{f- \tau_0}( K_{B_r(0) \setminus B_{r^{\prime}}(0)} ) ) 
\simeq 0$, we obtain also a vanishing 
\begin{equation}
\rmR  \Gamma ( f^{-1}( \tau_0) \cap \partial B_r(0) ; 
\phi_{f- \tau_0} ( K_{B_r(0)} ) )
\simeq 0. 
\end{equation}
From this, the first assertion of (ii) immediately follows. 
To obtain the last one, it suffices to use  
the t-exactness of the functor 
$\phi_{f- \tau_0}( \cdot )[-1]$. 
\end{proof}

Similarly, by Proposition \ref{prop-micros} and 
\cite[Theorem 4.4.1]{KS85} we obtain 
the following result. 

\begin{theorem}\label{np-vc-new}
In the situation of Proposition \ref{prop-micros}, 
assume also that $I(f)=P^{-1}(0) \cap Q^{-1}(0)
= \emptyset$. 
Then we have the following results. 
\begin{enumerate}
\item[\rm{(i)}] 
For any $\tau_0 \in \CC$ there exist $R \gg 0$ and 
$0< \varepsilon \ll 1$ such that for the 
inclusion maps $i_r: B_r(0) \hookrightarrow X= \CC^N$ 
($r>R$) we have isomorphisms 
\begin{equation}
 \rmR f_* \rmR  (i_r)_* ( i_r^{-1} K ) 
\simeq  \rmR f_*  K_{B_r(0)}
\simto  \rmR f_* K
\end{equation}
on the open subset $\{ \tau \in \CC \ | \ 
 | \tau - \tau_0| < \varepsilon \} \subset \CC$ of $\CC$. 
\item[\rm{(ii)}] 
For any $\tau_0 \in \CC$, if we assume also that 
\begin{equation}
\{ z \in f^{-1}( \tau_0) \ | \ 
df(z) \in {\rm SS }(K) \} \subset 
f^{-1}( \tau_0) \cap {\rm supp }(K) 
\end{equation}
is a finite subset of $f^{-1}( \tau_0) \cap {\rm supp }(K)$ 
and let $p_1,p_2, \ldots, p_k$ be the points in it, 
then for the vanishing cycle 
$\phi_{\tau - \tau_0}( \rmR f_* K ) \in 
\BDC (\{ \tau_0 \} )$ of $\rmR f_* K$ 
along the function $\tau - \tau_0: \CC \longrightarrow \CC$ 
we have an isomorphism 
\begin{equation}
\phi_{\tau - \tau_0}( \rmR f_* K) \simeq 
\bigoplus_{i=1}^k 
\phi_{f- \tau_0}( K)_{p_i}. 
\end{equation}
If we assume moreover that 
$K[N] \in \Dbc (X^{\an})$ is a perverse sheaf, 
then we have also a concentration 
\begin{equation}
H^j \phi_{\tau - \tau_0}( \rmR f_* K ) \simeq 0 
\qquad (j \not= N-1). 
\end{equation}
\end{enumerate}
\end{theorem}

Next, let us consider the problem in the general case i.e. 
in the absence of the condition $I(f)=P^{-1}(0) \cap Q^{-1}(0)
= \emptyset$. For this purpose, we regard 
the projection $\tau : X \times \CC 
\longrightarrow \CC$ as a holomorphic function on 
$X \times \CC$ and denote it by $h$. 

\begin{definition}\label{def-7new} 
Let $S \subset X \times \CC$ be a smooth quasi-affine 
subvariety of $X \times \CC$. 
Then we say that the restriction $h|_S$ of 
the function $h$ to $S$ is relatively tame 
at infinity for the projection $X \times \CC 
\longrightarrow X$ if there exist $R \gg 0$ and 
$0< \varepsilon \ll 1$ such that 
\begin{equation}
\{ (z, \tau ) \in S \ | \ || 
 d {\rm Re} (h|_S)(z, \tau ) || < \varepsilon, \ 
|| z || >R \} = \emptyset. 
\end{equation}
\end{definition}

For the rational function $f= \frac{P}{Q}: U= 
X \setminus Q^{-1}(0) \longrightarrow \CC$ and 
the algebraic constructible sheaf 
$K \in \Dbc (X^{\an})$ by using the (not necessarily) 
closed embedding $i_f:U \longrightarrow X \times \CC$
($z \longmapsto (z,f(z))$) we set 
\begin{equation}
L:= i_{f !} (K|_U) \in \Dbc (X^{\an} \times \CC ). 
\end{equation}
Then there exists an algebraic Whitney 
stratification $\CS$ of $X \times \CC$ such that 
\begin{equation}
{\rm SS }(L) \subset \bigsqcup_{S \in \CS} 
T^*_S(X \times \CC ). 
\end{equation}
Moreover we may assume also that for $\CS$ and a Whitney 
stratification $\CS_0$ of $X = \CC^N$ 
the projection $X \times \CC 
\longrightarrow X$ is a stratified fiber 
bundle as in the assertion of the theorem in 
\cite[page 43]{GM88}. Indeed, after extending 
$L$ to a constructible sheaf on $X \times \PP$ 
we can apply this theorem to the 
proper morphism $X \times \PP 
\longrightarrow X$ ($(z, \tau ) \longmapsto z$). 
Let us fix such Whitney stratifications 
$\CS$ and $\CS_0$. Then we shall say that a stratum 
$S \in \CS$ in $\CS$ is horizontal if 
for its projection $S_0 \in \CS_0$ to $X= \CC^N$ 
we have ${\rm dim }S_0 = {\rm dim }S$ i.e. 
the surjective 
submersion $S \longrightarrow S_0$ induced 
by the projection 
$X \times \CC \longrightarrow X$ 
is a finite covering. Obviously, if 
$S \in \CS$ is not horizontal then 
for its projection $S_0 \in \CS_0$ to $X= \CC^N$ 
we have ${\rm dim }S_0 = {\rm dim }S-1$ and hence 
the restriction $h|_S$ of 
the function $h$ to $S$ is relatively tame 
at infinity for the projection $X \times \CC 
\longrightarrow X$. Namely, our relative 
tameness at infinity of $h|_S$ is a 
constraint only for horizontal strata $S \in \CS$. 
Moreover it is easy to see that if for 
a horizontal stratum $S \in \CS$ in $\CS$ 
there exists a holomorphic function $g$ 
on its projection $S_0 \in \CS_0$ to $X$ 
such that 
\begin{equation}
S= \{ (z, g(z)) \ | \ z \in S_0 \} \subset 
X \times \CC
\end{equation}
then the relative 
tameness at infinity of $h|_S$ is equivalent 
to the tameness at infinity of $g:S_0 
\longrightarrow \CC$ in Definition \ref{def-7}. 
The proof of the following proposition is very similar 
to that of Proposition \ref{prop-micros} and 
we omit it. 

\begin{proposition}\label{prop-micros-new} 
Let $\CS$ be the Whitney 
stratification of $X \times \CC$ as above and 
assume that for any stratum $S \in \CS$ in it 
such that 
$T^*_S(X \times \CC ) \subset {\rm SS }(L)$ 
the restriction $h|_{S}: S 
\longrightarrow \CC$ of the function $h$ to 
$S \subset X \times \CC$ 
is relatively tame at infinity for 
the projection $X \times \CC 
\longrightarrow X$. 
Then for any $\tau_0 \in \CC$  there exist $R \gg 0$ and 
$0< \varepsilon \ll 1$ such that we have 
\begin{equation}
N^*( B_r(0) \times \CC ) \cap \Bigl\{ 
\overline{{\rm SS }(L) + \rho 
\bigl( (X \times \CC ) \times_{\CC} T^* \CC \bigr) }
\Bigr\} \subset T^*_{(X \times \CC )} (X \times \CC ) 
\end{equation}
over 
the open subset $\{ (z, \tau ) \in X \times \CC \ | \ 
| \tau - \tau_0| < \varepsilon \} 
\subset X \times \CC$ of $X \times \CC$ for any $r \geq R$, 
where $\rho : ( X \times \CC ) \times_{\CC} T^* \CC 
\hookrightarrow T^* ( X \times \CC )$ is the closed embedding 
associated to the projection $h: X \times \CC \longrightarrow \CC$. 
Moreover the same is true even after replacing 
$N^*( B_r(0) \times \CC )$ 
by its antipodal set $N^*( B_r(0) \times \CC )^{{\rm a}}$. 
\end{proposition}

\begin{theorem}\label{prop-8-new} 
In the situation of Proposition \ref{prop-micros-new}
we have the following results. 
\begin{enumerate}
\item[\rm{(i)}] 
For any $\tau_0 \in \CC$  
there exist $R \gg 0$ and 
$0< \varepsilon \ll 1$ such that for the 
inclusion maps $i_r: U_r=U \cap B_r(0) \hookrightarrow U$ 
($r>R$) we have isomorphisms 
\begin{equation}
 \rmR f_! (i_r)_! (i_r)^{-1} (K|_U) 
\simeq  \rmR f_!  (K|_U)_{U_r}
\simto  \rmR f_! (K|_U)
\end{equation}
on the open subset $\{ \tau \in \CC \ | \ 
 | \tau - \tau_0| < \varepsilon \} \subset \CC$ of $\CC$. 
\item[\rm{(ii)}] 
For $\tau_0 \in \CC$ we assume also that 
\begin{equation}
\{ z \in f^{-1}( \tau_0) \ | \ 
df(z) \in {\rm SS }(K) \} \subset 
f^{-1}( \tau_0) \cap {\rm supp }(K|_U) 
\end{equation}
and 
\begin{equation}
\{ z \in I(f) \ | \ 
\phi_{f- \tau_0}^{\merc} (K)_z \not= 0 \} \subset 
I(f) \cap {\rm supp }(K) 
\end{equation}
are finite subsets of $Z= {\rm supp }(K)$ 
and denote them by $\{ p_1,p_2, \ldots, p_k \}$ 
and $\{ q_1,q_2, \ldots, q_l \}$ respectively. 
Then for the vanishing cycle 
$\phi_{\tau - \tau_0}( \rmR f_! (K|_U) ) \in 
\BDC (\{ \tau_0 \} )$ of $\rmR f_!  (K|_U)$ 
along the function $\tau - \tau_0: \CC \longrightarrow \CC$ 
we have an isomorphism 
\begin{equation}\label{decomp} 
\phi_{\tau - \tau_0}( \rmR f_! (K|_U) ) \simeq 
\Bigl\{ \bigoplus_{i=1}^k 
\phi_{f- \tau_0}( K)_{p_i} \Bigr\} 
\oplus 
\Bigl\{ \bigoplus_{i=1}^l 
\phi_{f- \tau_0}^{\merc}( K)_{q_i} \Bigr\}. 
\end{equation}
If we assume moreover that 
$K[N] \in \Dbc (X^{\an})$ is a perverse sheaf, 
then we have also a concentration 
\begin{equation}
H^j \phi_{\tau - \tau_0}( \rmR f_! (K|_U)) \simeq 0 
\qquad (j \not= N-1). 
\end{equation}
\end{enumerate}
\end{theorem}

\begin{proof}
By  \cite[Theorem 4.4.1]{KS85} and  
Proposition \ref{prop-micros-new} 
there exist $R \gg 0$ and 
$0< \varepsilon \ll 1$ such that for the 
inclusion maps $j_r: B_r(0) \times \CC \hookrightarrow 
X \times \CC$ 
($r>R$) we have isomorphisms 
\begin{equation}
 \rmR h_! (j_r)_! (j_r)^{-1} L
\simeq  \rmR h_! L_{B_r(0) \times \CC}
\simto  \rmR h_! L
\end{equation}
on the open subset $\{ \tau \in \CC \ | \ 
 | \tau - \tau_0| < \varepsilon \} \subset \CC$ of $\CC$. 
Moreover we have $\rmR h_! L \simeq  \rmR f_! (K|_U)$ 
and for any $r>0$ there exist an isomorphism 
\begin{equation}
\rmR h_! L_{B_r(0) \times \CC} \simeq 
\rmR h_! \Bigl( 
i_{f!} (K|_U) \otimes \CC_{B_r(0) \times \CC}
\Bigr) \simeq 
\rmR f_! (K|_U)_{U_r}. 
\end{equation}
We thus obtain the isomorphism in (i). 
Let us prove (ii). Let $R \gg 0$ and 
$0< \varepsilon \ll 1$ be as in (i). Here 
we take $R>0$ so that also the condition 
$p_1, p_2, \ldots, p_k, q_1, q_2, \ldots, q_l \in 
B_R(0)$ is satisfied. Then by the 
proof of (i), for any $r>R$ we obtain isomorphisms  
\begin{align*}
& \phi_{\tau - \tau_0}( \rmR f_! (K|_U) )  \simeq 
\phi_{\tau - \tau_0} ( \rmR h_! L_{B_r(0) \times \CC} )
\\
& \simeq \rmR \Gamma (X \times \{ \tau_0 \} ; 
\phi_{h - \tau_0} ( L_{B_r(0) \times \CC}) ),  
\end{align*}
where in the second isomorphism we used the fact 
that the morphism $h: X \times \CC \longrightarrow \CC$ 
is proper on the support of $ L_{B_r(0) \times \CC}$. 
On the other hand, there exist also isomorphisms 
\begin{equation}
\phi_{h - \tau_0} ( L) \simeq 
\phi_{h - \tau_0} ( i_{f!} (K|_U) )  \simeq 
\phi_{f - \tau_0}^{\merc}(K) 
\end{equation}
under the natural identification $h^{-1}( \tau_0) =
X \times \{ \tau_0 \} \simeq X$. 
Similarly, we have isomorphisms 
\begin{equation}
\phi_{h - \tau_0} ( L_{B_r(0) \times \CC}) \simeq 
\phi_{h - \tau_0} ( i_{f!} (K_{B_r(0)}|_U) )  \simeq 
\phi_{f - \tau_0}^{\merc}(K_{B_r(0)}). 
\end{equation}
Then together with our assumptions, this implies 
that for any $r>R$ the support of the vanishing cycle 
$\phi_{h - \tau_0} ( L_{B_r(0) \times \CC})$ is 
contained in the set 
\begin{equation}
\{ p_1,p_2, \ldots, p_k \} \sqcup 
\{ q_1,q_2, \ldots, q_l \} \sqcup 
\partial B_r(0). 
\end{equation}
Now let us take $r, r^{\prime} >R$ such that 
$r> r^{\prime}$. Then by the proof of 
\cite[Theorem 4.4.1]{KS85} there exists an 
isomorphism 
\begin{equation}
\rmR h_! L_{B_{r^{\prime}}(0) \times \CC}
\simto  \rmR h_! L_{B_r(0) \times \CC}
\end{equation}
on the open subset $\{ \tau \in \CC \ | \ 
 | \tau - \tau_0| < \varepsilon \} \subset \CC$ of $\CC$. 
In other words, we have a vanishing 
\begin{equation}
\rmR h_! L_{(B_r(0) \setminus B_{r^{\prime}}(0)) \times \CC}  
\simeq 0
\end{equation}
there. This implies that we have 
\begin{equation}
 \rmR  \Gamma ( X \times \{ \tau_0 \} ; 
\phi_{h- \tau_0} ( 
L_{(B_r(0) \setminus B_{r^{\prime}}(0)) \times \CC} ) )
\simeq 0. 
\end{equation}
Since the support of 
$\phi_{h- \tau_0} ( 
L_{(B_r(0) \setminus B_{r^{\prime}}(0)) \times \CC} )$ 
is contained in the set 
\begin{equation}
\partial B_r(0) 
 \sqcup 
\partial B_{r^{\prime}}(0) 
\end{equation}
and hence the object $\rmR \Gamma ( \partial B_r(0); 
\phi_{h- \tau_0} (  L_{B_r(0) \times \CC} ) )$ is a direct 
summand of the one $\rmR \Gamma ( X \times \{ \tau_0 \} ; 
\phi_{h- \tau_0} ( 
L_{(B_r(0) \setminus B_{r^{\prime}}(0)) \times \CC} ) )
\simeq 0$, we obtain also a vanishing 
\begin{equation}
\rmR \Gamma ( \partial B_r(0); 
\phi_{h- \tau_0} (  L_{B_r(0) \times \CC} ) )
\simeq 0. 
\end{equation}
From this, the first assertion of (ii) immediately follows. 
The second one follows from the fact that 
the functors $\phi_{f- \tau_0}( \cdot )[-1]$ and 
$\phi_{f- \tau_0}^{\merc}( \cdot )[-1]$ preserve 
the perversity. This completes the proof. 
\end{proof}

Finally, to end this section, we shall introduce 
some geometric consequences of Proposition \ref{prop-7} 
which generalize the main results of Broughton \cite{Bro88}. 
Let $f \in \CC [z_1,z_2, \ldots, z_N]$ be a polynomial 
on $X= \CC^N$ and $S \subset X$ a smooth subvariety of 
$X= \CC^N$. Let $g:S \longrightarrow \CC$ be the 
restriction of $f:X= \CC^N \longrightarrow \CC$ to $S \subset X$. 
Then it is well-known that there exists a finite 
subset $B \subset \CC$ of $\CC$ such that the 
restriction $g^{-1}( \CC \setminus B) \longrightarrow \CC \setminus B$ 
of $g$ is a ${\rm C}^{\infty}$-locally trivial fibration. 
We denote by $B_g$ the smallest finite subset of 
$\CC$ satisfying this property and call it the 
bifurcation set of $g$. By this definition, it is 
clear that the set $\Sigma_g:= g( {\rm Sing} g) 
\subset \CC$ of the critical values of $g$ is contained in 
$B_g$. Note that as was observed in \cite{Bro88} 
there are a lot of polynomial maps $g:S \longrightarrow \CC$ 
such that $B_g \not= \Sigma_g$. Nevertheless, 
by Proposition \ref{prop-7} we can easily show 
the following result as in the proof of 
\cite[Theorem 1]{NZ90}. 

\begin{theorem}\label{nexsai} 
(cf. Nguyen-Pham-Pham \cite[Theorem 3.1]{NPP19}) 
Assume that $g:S \longrightarrow \CC$ is tame at infinity. 
Then we have $B_g = \Sigma_g$. 
\end{theorem}
By Theorem \ref{np-vc-new}, we obtain also the 
analogues of the results in Sabbah \cite[Section 8]{Sab06}. 
Here, instead of Sabbah's cohomological tameness, 
we assume our topological one in Definition \ref{def-7}. 
In particular, we obtain the following results. 

\begin{lemma}\label{suplem-1} 
(cf. Sabbah \cite[Lemma 8.5]{Sab06}) 
Assume that $n:= {\rm dim} S-1 \geq 1$ and 
$g:S \longrightarrow \CC$ is tame at infinity. 
Then we have 
\begin{equation}
H^j \rmR g_* \CC_S
\simeq 0 \qquad (j \notin [0,n]). 
\end{equation}
Moreover for any $1 \leq j \leq n-1$ the direct image 
sheaf $H^j \rmR g_* \CC_S$ is a constant sheaf on 
$\CC$ of rank ${\rm dim} H^j(S; \CC )$. 
\end{lemma} 
As is the proof of  \cite[Lemma 8.5]{Sab06}, 
in the situation of Lemma \ref{suplem-1} we 
see that for any point $c \in \CC$ 
we have isomorphisms 
\begin{equation}
H^j ( \rmR g_* \CC_S)_c
\simeq H^j(g^{-1}(c); \CC )
 \qquad (j \in \ZZ ). 
\end{equation}
We thus obtain the following very simple 
consequence of Lemma \ref{suplem-1}. 

\begin{corollary}\label{suplcoro-1} 
In the situation of Lemma \ref{suplem-1}, 
for the finite subset $B_g= \Sigma_g \subset \CC$ 
of $\CC$ and a base point $c \in \CC \setminus B_g$ 
the monodromy representations 
\begin{equation}
\rho_{g,c,j}: 
\pi_1( \CC \setminus B_g; c) \longrightarrow 
{\rm Aut} \Bigl( H^j(g^{-1}(c); \CC ) \Bigr) 
 \qquad (0 \leq j \leq n-1)
\end{equation}
are trivial. 
\end{corollary} 
By Corollaly \ref{suplcoro-1} the monodromy representation 
\begin{equation}
\rho_{g,c,n}: 
\pi_1( \CC \setminus B_g; c) \longrightarrow 
{\rm Aut} \Bigl( H^n(g^{-1}(c); \CC ) \Bigr) 
\end{equation}
in the top degree $n= {\rm dim} g^{-1}(c)$ 
is the only non-trivial one and we can calculate 
its eigenvalues by the theory of monodromy zeta 
functions (see e.g. \cite[Section 2]{Take23}). 
For a polynomial map $h: S \longrightarrow \CC$ on $S$ 
we denote by $\mu (h)$ the sum of the Milnor numbers of 
$h$. If the set ${\rm Sing} h \subset S$ of the 
critical points of $h$ is not finite, we set 
$\mu (h):= + \infty$. 

\begin{lemma}\label{suplem-2} 
(cf. Sabbah \cite[Remark just after Lemma 8.5]{Sab06}) 
In the situation of Lemma \ref{suplem-1}, the 
generic rank of the constructible sheaf 
$H^n \rmR g_* \CC_S$ on $\CC$ is equal to 
the number 
\begin{equation}
\mu (g) + {\dim} H^n(S; \CC) - {\dim} H^{n+1}(S; \CC). 
\end{equation}
\end{lemma} 
Moreover, by the proof of \cite[Lemma 8.5]{Sab06} 
and Lemma \ref{suplem-2} we obtain the following result. 

\begin{corollary}\label{suplcorol-2} 
In the situation of Lemma \ref{suplem-2}, 
for a point $c_0 \in \CC$ 
we denote by $\mu_{c_0}$ the sum of the Milnor 
numbers of the hypersurface $g^{-1}(c_0) \subset S$ 
and set $\mu (g)^{\prime}:= 
\mu (g) + {\dim} H^n(S; \CC) - {\dim} H^{n+1}(S; \CC)$. 
Then for any point $c_0 \in \CC$ we have 
\begin{equation}
{\rm dim} H^n(g^{-1}(c_0); \CC )= 
\begin{cases}
\mu (g)^{\prime}  & (c_0 \notin B_g= \Sigma_g )\\
\\
\mu (g)^{\prime}- \mu_{c_0} & (c_0 \in B_g= \Sigma_g ).
\end{cases}
\end{equation}
\end{corollary} 
Now let $Y= \CC^N_w$ be the dual vector space of $X= \CC^N_z$ 
and consider the Lagrangian subvariety 
\begin{equation}
\Lambda^g:= \{ (z,dg(z)) \ | \ z \in S \} \qquad \subset 
T^* S
\end{equation}
of $T^*S$ and the natural morphisms 
\begin{equation}
T^* X 
\overset{\varpi}{\longleftarrow}
S \times_{X} T^* X
\overset{\rho}{\longrightarrow}
T^* S
\end{equation}
associated to the inclusion map 
$S \hookrightarrow X$. Then $\varpi \rho^{-1} \Lambda^g 
\subset T^*X$ is a Lagrangian subvariety of $T^*X$. 
For a point $w=(w_1,w_2, \ldots, w_N) 
\in Y= \CC^N$ we define a linear perturbation 
$g^{(w)}: S \longrightarrow \CC$ of $g: S \longrightarrow \CC$ by 
\begin{equation}
g^{(w)}(z):=g(z)- \dsum_{j=1}^N w_jz_j \qquad (z \in S \subset X). 
\end{equation}

\begin{lemma}\label{suplem-3} 
The polynomial map 
$g:S \longrightarrow \CC$ is tame at infinity 
if and only if the restriction $\varpi \rho^{-1} \Lambda^g 
\longrightarrow Y$ of the projection $T^*X \simeq 
X \times Y \longrightarrow Y$ to $\varpi \rho^{-1} \Lambda^g
\subset T^*X$ 
is proper over a sufficiently small open ball 
$B_{\varepsilon}(0) \subset Y= \CC^N$ 
$(0< \varepsilon \ll 1)$ centered at the origin 
$0 \in Y= \CC^N$. 
\end{lemma}

\begin{proof}
First note that the fiber of the morphism 
$\varpi \rho^{-1} \Lambda^g 
\longrightarrow Y$ at a point $w \in Y= \CC^N$ is 
naturally identified with the set 
\begin{equation}
{\rm Sing} g^{(w)}= \{ z \in S \ | \ dg^{(w)}(z)=0 
 \} \qquad \subset S. 
\end{equation}
Then the last condition in the lemma is equivalent to 
the one that there exist $R \gg 0$ and $0< \varepsilon \ll 1$ 
such that 
\begin{equation}
 \{ z \in S \ | \ dg^{(w)}(z)=0, || z || >R 
 \} = \emptyset
\end{equation}
for any $w \in B_{\varepsilon}(0) \subset Y= \CC^N$. 
By the Cauchy-Riemann equation, for any $z \in S$ and 
$w \in Y= \CC^N$ we have also an equivalence 
\begin{equation}
dg^{(w)}(z)=0  \ \Longleftrightarrow \ 
{\rm grad} {\rm Re} g^{(w)}(z)=0. 
\end{equation}
Moreover, for the surjective $\CC$-linear map 
$\Phi_z: T_zX_{\RR} \simeq \CC^N \longrightarrow 
T_zS_{\RR} \simeq \CC^{{\rm dim} S}$ in Lemma \ref{lem-3}, 
by \eqref{C-R-eqn} we obtain an equality 
\begin{equation}
{\rm grad} {\rm Re} g^{(w)}(z)=
{\rm grad} {\rm Re} g(z)- 
\Phi_z \left(
\begin{array}{c}
\overline{w_1}    
\\
\vdots 
\\
\overline{w_N}      
\end{array}
\right). 
\end{equation} 
Since by the proof of  Lemma \ref{lem-3} the $\CC$-linear map 
$\Phi_z$ is the orthogonal projection to 
$T_zS_{\RR} \simeq \CC^{{\rm dim} S}$ with respect to 
the Hermitian metric of $T_zX_{\RR} \simeq \CC^N$, 
we see that the open subset 
\begin{equation}
\Bigl\{ 
- \Phi_z \left(
\begin{array}{c}
\overline{w_1}    
\\
\vdots 
\\
\overline{w_N}      
\end{array}
\right)  \ | \ 
w \in B_{\varepsilon}(0)  \Bigr\} \quad \subset T_zS_{\RR}
\end{equation} 
is the $\varepsilon$-ball in 
$T_zS_{\RR} \simeq \CC^{{\rm dim} S}$ 
centered at the origin. From this 
we immediately obtain the assertion. 
\end{proof}
By the proof of Lemma \ref{suplem-3} and the 
algebraicity of $\varpi \rho^{-1} \Lambda^g
\subset T^*X$ we can easily see that generic 
linear perturbations 
$g^{(w)}$ ($w \in Y= \CC^N$) 
of $g$ are tame at infinity. More precisely, 
we have the following result. 

\begin{lemma}\label{suplemma-4} 
Let $\Omega \subset Y = \CC^N$ be the maximal Zariski open 
subset of $Y$ such that the base change of 
the morphism 
$\varpi \rho^{-1} \Lambda^g 
\longrightarrow Y$ by the 
inclusion map $\Omega \hookrightarrow Y$ is a 
(possibly ramified) finite covering. 
Then for $w \in Y = \CC^N$ the linear perturbation  
$g^{(w)}$ of $g$ is tame at infinity if and only if 
$w \in \Omega$. 
\end{lemma}
As in \cite[Proposition 3.1]{Bro88} we obtain also the   
following characterization of the tameness 
at infinity of $g: S \longrightarrow \CC$.

\begin{proposition}\label{propt} 
The polynomial map $g: S \longrightarrow \CC$ is 
tame at infinity if and only if $\mu (g) < + \infty$ and 
$\mu (g)= \mu (g^{(w)})$ for all sufficiently small 
$w \in Y= \CC^N$. 
\end{proposition}
From now on, we shall introduce our results on the 
bouquet decompositions of the fibers 
$h^{-1}(c) \subset S$ $(c \in \CC )$ of 
some tame polynomial maps $h:S \longrightarrow \CC$.  
For this purpose, first we recall the following 
fundamental theorem due to Broughton \cite[Theorem 1.2]{Bro88}. 
For a polynomial map $h:S \longrightarrow \CC$ having 
only isolated singular points and a point $c_0 \in \CC$ 
we denote by $\mu_{c_0}$ the sum of the Milnor 
numbers of the hypersurface $h^{-1}(c_0) \subset S$ and 
set $\mu (h,c_0):= \mu (h)- \mu_{c_0}$. 

\begin{theorem}\label{bdecom-br} 
(Broughton \cite[Theorem 1.2]{Bro88}) 
Assume that a polynomial map 
$h:X= \CC^N \longrightarrow \CC$ is tame at infinity. 
Then for any point $c_0 \in \CC$ its fiber 
$h^{-1}(c_0) \subset X= \CC^N$ of $h$ 
has the homotopy type of a bouquet 
$S^{N-1} \vee \cdots \vee S^{N-1}$ of 
some $(N-1)$-dimensional spheres $S^{N-1}$. 
Moreover the number of the spheres $S^{N-1}$ in 
the bouquet decomposition is equal to $\mu (h,c_0)$. 
\end{theorem}
Recall that this beautiful result was the starting 
point of the intensive activities in the study of 
the monodromies at infinity of polynomial maps 
$h:X= \CC^N \longrightarrow \CC$ (see 
e.g. \cite[Sections 4 and 7]{Take23} for 
the details). Motivated by it, 
we introduce the following class of 
subvarieties of $X= \CC^N$. 

\begin{definition}
We say that a smooth complete intersection subvariety 
$Z \subset X= \CC^N$ is a bouquet variety if it 
has the homotopy type of a bouquet of 
some spheres of dimension ${\rm dim} Z$. 
\end{definition}
Note that if a polynomial map 
$h:X= \CC^N \longrightarrow \CC$ is tame at infinity 
then by Theorem \ref{bdecom-br} generic fibers of 
$h$ are bouquet varieties. 

\begin{theorem}\label{bdecom-tk} 
Assume that $n:= {\rm dim} S-1 \geq 1$, the smooth subvariety 
$S \subset X= \CC^N$ is a bouquet variety and the polynomial map 
$g:S \longrightarrow \CC$ is tame at infinity. 
Denote by $\mu_S$ the number of the spheres $S^{n+1}$ in 
the bouquet decomposition of $S$. 
Then for any $c_0 \in \CC$ its fiber $g^{-1}(c_0) \subset S$ of $g$ 
has the homotopy type of a bouquet 
$S^{n} \vee \cdots \vee S^{n}$ of 
some $n$-dimensional spheres $S^{n}$. 
Moreover the number of the spheres $S^{n}$ in 
the bouquet decomposition is equal to $\mu (g, c_0)- \mu_S$. 
\end{theorem}

\begin{proof}
With Proposition \ref{prop-7} and its consequences 
above at hands, the proof is very similar to 
that of \cite[Theorem 1.2]{Bro88} and for it we 
use a Morse theory for the Morse function 
$\varphi (z):=| g(z)-c_0 |^2$ $(z \in S_{\RR})$ on $S_{\RR}$. 
First of all, by Proposition \ref{prop-7} for a 
sufficiently small $0< \varepsilon \ll 1$ the 
level set $\{ z \in S_{\RR} \ | \ \varphi (z) < \varepsilon \} 
\subset S_{\RR}$ of $\varphi$ is homotopic to 
the fiber $g^{-1}(c_0)$. Moreover by the Cauchy-Riemann equation 
we can easily see that a point $z_0 \in S \setminus g^{-1}(c_0)$ 
is a critical point of the Morse function 
\begin{equation}
\varphi (z)=| g(z)-c_0 |^2 = \exp \Bigl( 
2 {\rm Re} \log (g(z)-c_0) \Bigr)
\end{equation}
if and only if it is a critical point of $g$. 
As in the proof of  \cite[Theorem 1.2]{Bro88}, 
by slightly perturbing $g:S \longrightarrow \CC$ 
we may assume that any critical point 
$z_0 \in S \setminus g^{-1}(c_0)$ of $g$ is 
non-degenerate i.e. of complex Morse type. Then we can 
easily show that any critical point 
$z_0 \in S \setminus g^{-1}(c_0)$ of 
$\varphi$ is non-degenerate and has the Morse index 
${\rm dim} S=n+1$. This would then imply that for a 
sufficiently large $t \gg 0$ the 
level set $\{ z \in S_{\RR} \ | \ \varphi (z) < t \} 
\subset S_{\RR}$ of $\varphi$ has the homotopy 
type of the CW complex obtained by attaching 
some $(n+1)$-dimensional cells to $g^{-1}(c_0)$ and 
the number of such cells is equal to $\mu (g, c_0)= 
\mu (g)- \mu_{c_0}$. 
However, to justify such a Morse-theoretical argument 
on the ``non-compact" manifold $S_{\RR}$ we have to 
show that each integral curve of the gradient vector 
field ${\rm grad} \varphi / ||{\rm grad} \varphi ||^2$ 
does not go to infinity in a finite time 
(see the last part of the proof of \cite[Theorem 1.2]{Bro88}
and that of Milnor \cite[Theorem 2.10]{Mil68}). 
For this purpose, it suffices to show that 
for some $R \gg 0$ the function $||{\rm grad} \varphi ||^{-1}$ 
is bounded on the set $\{ z \in S_{\RR} \ | \ 
 ||z|| \geq R, \varphi (z) \geq \varepsilon \}$. 
First, we define  a function $\widetilde{\varphi}: 
X_{\RR} \longrightarrow \RR$ on $X_{\RR}$ by 
\begin{equation}
\widetilde{\varphi}(z):= |f(z)-c_0|^2 \qquad 
(z \in X_{\RR}). 
\end{equation}
Then our Morse function $\varphi: S_{\RR} \longrightarrow \RR$
is the restriction of $\widetilde{\varphi}$ to 
$S_{\RR} \subset X_{\RR}$. Moreover, for any point 
$z \in X_{\RR}$ we can easily show an equality 
\begin{equation}
{\rm grad} \widetilde{\varphi}(z)= 2(f(z)-c_0) \cdot 
\overline{{\rm grad}f(z)}
\end{equation}
in $T_zX_{\RR} \simeq \CC^N$. Recall that by \eqref{C-R-eqn} 
we also have 
\begin{equation} 
\overline{\grad f(z)} = \grad {\rm Re} f(z). 
\end{equation}
Then by Lemma \ref{lem-3}, for any point 
$z \in S_{\RR}$ there exists a surjective 
$\CC$-linear map  
$\Phi_z: T_zX_{\RR} \longrightarrow T_zS_{\RR}$ 
such that  
\begin{align*}
|| {\rm grad} \varphi ||^2
&  = 
\Bigl\langle \grad \varphi (z), \overline{\grad \varphi (z)} \Bigr\rangle  
\\
&  = 
\Bigl( \grad \varphi (z), \grad \varphi (z) \Bigr) 
\\
&  = 
\Bigl( \Phi_z ( \grad \widetilde{\varphi} (z)), 
\Phi_z ( \grad \widetilde{\varphi} (z)) \Bigr) 
\\
&  = 
4 |g(z)-c_0|^2 \cdot || \Phi_z( \grad {\rm Re} f(z)) ||^2 
\\
&  = 
4 |g(z)-c_0|^2 \cdot || \grad {\rm Re} g(z) ||^2. 
\end{align*} 
Since $g$ is tame at inifinity and hence the function 
$|| \grad {\rm Re} g ||= || d{\rm Re} g || \geq 0$ on 
$S_{\RR}$ is bounded away from $0$ at infinity, now 
the desired boundedness of $||{\rm grad} \varphi ||^{-1}$ 
immediately follows.  
Then as in the proof of \cite[Theorem 1.2]{Bro88}, by 
the homotopy exact sequence associated to 
the pair $(S, g^{-1}(c_0))$ of topological spaces we obtain 
\begin{equation}
\pi_j(g^{-1}(c_0)) \simeq 0 \qquad 
(j< n= {\rm dim} g^{-1}(c_0)). 
\end{equation}
Note also that by the 
Andreotti-Frankel theorem in \cite{AF59} and Hamm's 
one in \cite{H83} the complete intersection subvariety 
$g^{-1}(c_0) \subset X= \CC^N$ of $X= \CC^N$ has the 
homotopy type of a CW complex of 
dimension $\leq n= {\rm dim} g^{-1}(c_0)$. 
Then we can apply the argument in the proof of Milnor 
\cite[Theorem 6.5]{Mil68} to 
$g^{-1}(c_0)$ to show that its (top-dimensional) 
homology group $H_{n}(g^{-1}(c_0))$ is a free $\ZZ$-module 
and the first assertion immediately 
follows from Whitehead's theorem as in 
the proof of \cite[Theorem 1.2]{Bro88}. We 
can also show the second assertion by Corollary \ref{suplcorol-2}. 
This completes the proof. 
\end{proof}
Note that in \cite[Theorem 4.6 and Corollary 4.7]{Tib98} 
Tibar also obtained a result similar to the first 
part of Theorem \ref{bdecom-tk} under a different condition 
at infinity in an appropriately chosen compactification of $S$. 
For $1 \leq k \leq N$ we define a decreasing sequence 
\begin{equation}
S_0:=X= \CC^N \supset S_1 \supset S_2 \supset \cdots \cdots 
\supset S_k 
\end{equation}
of smooth complete intersection subvarieties 
of $X= \CC^N$ such that ${\rm dim} S_i=N-i$ 
$(0 \leq i \leq k)$ as follows. First we take a 
tame polynomial $g_0:S_0=X= \CC^N \longrightarrow \CC$
and a point $c_0 \in \CC \setminus B_{g_0}$ and set 
$S_1:=g_0^{-1}(c_0) \subset S_0$. Next we 
repeat this construction and define 
$S_i$ for $i \geq 2$ recursively. Namely, for 
each $1 \leq i \leq k-1$ after defining $S_i$ we 
take a tame polynomial $g_i:S_i \longrightarrow \CC$
and a point $c_i \in \CC \setminus B_{g_i}$ and set 
$S_{i+1}:=g_i^{-1}(c_i) \subset S_i$. Then by 
Theorem \ref{bdecom-tk} we see that 
$S_1,S_2, \ldots, S_k$ are bouquet varieties. 
If for each $0 \leq i \leq k$ we take a 
polynomial map $f_i:X= \CC^N \longrightarrow \CC$ 
such that $f_i|_{S_i}=g_i$ and 
set $h_i:=f_i-c_i: X= \CC^N \longrightarrow \CC$, 
then we obtain 
a smooth complete intersection subvariety 
\begin{equation}
S_k= \{ z \in X= \CC^N \ | \ 
h_1(z)=h_2(z)= \cdots = h_k(z)=0 \} \qquad \subset X= \CC^N 
\end{equation}
of $X= \CC^N$ which has the homotopy type of a bouquet 
$S^{N-k} \vee \cdots \vee S^{N-k}$ of 
some $(N-k)$-dimensional spheres $S^{N-k}$. 
Moreover, by Lemmas \ref{suplem-3} and \ref{suplemma-4} 
we can explicitly choose such polynomials 
$h_1,h_2, \ldots, h_k$. 
This explains the reason why 
the results on the monodromies at infinity of 
the polynomial maps on the complete intersection 
subvarieties of $X= \CC^N$ studied in 
Esterov-Takeuchi \cite[Section 6]{ET12} and 
Matsui-Takeuchi \cite[Section 5]{MT11b} are very similar 
to the ones in the local case of 
Hamm \cite{H71}, Esterov-Takeuchi \cite[Sections 4 and 5]{ET12}, 
Matsui-Takeuchi \cite[Theorem 3.12]{MT11} 
and Oka \cite{O90}. Indeed, by Lemma \ref{suplem-3} we can 
easily show that in the Newton non-degenerate 
setting of \cite[Section 6]{ET12} and \cite[Section 5]{MT11b} 
the generic fibers admit a bouquet decomposition 
(under some weak assumptions).

\section{An Overview on Ind-sheaves and the Irregular 
\\ Riemann-Hilbert Correspondence}\label{uni-sec:2}
In this section, we briefly recall some basic notions 
and results which will be used in this paper. 
We assume here that the reader is familiar with the 
theory of sheaves and functors in the framework of 
derived categories. For them we follow the terminologies 
in \cite{KS90} etc. For a topological space $M$ 
denote by $\BDC(\CC_M)$ the derived category 
consisting of bounded 
complexes of sheaves of $\CC$-vector spaces on it.

\subsection{Ind-sheaves}\label{sec:3}
We recall some basic notions 
and results on ind-sheaves. References are made to 
Kashiwara-Schapira \cite{KS01} and \cite{KS06}. 
Let $M$ be a good topological space (which is locally compact, 
Hausdorff, countable at infinity and has finite soft dimension). 
We denote by $\Mod(\CC_M)$ the abelian category of sheaves 
of $\CC$-vector spaces on it and by $\I\CC_M$ 
that of ind-sheaves. Then there exists a 
natural exact embedding $\iota_M : \Mod(\CC_M)\to\I\CC_M$ 
of categories. We sometimes omit it.
It has an exact left adjoint $\alpha_M$, 
that has in turn an exact fully faithful
left adjoint functor $\beta_M$: 
\begin{equation}
\xymatrix@C=60pt{\Mod(\CC_M)  \ar@<1.0ex>[r]^-{\iota_{M}} 
 \ar@<-1.0ex>[r]_- {\beta_{M}} & \I\CC_M 
\ar@<0.0ex>[l]|-{\alpha_{M}}}.
\end{equation}

The category $\I\CC_M$ does not have enough injectives. 
Nevertheless, we can construct the derived category $\BDC(\I\CC_M)$ 
for ind-sheaves and the Grothendieck six operations among them. 
We denote by $\otimes$ and $\rihom$ the operations 
of tensor products and internal homs respectively. 
If $f : M\to N$ be a continuous map, we denote 
by $f^{-1}, \rmR f_\ast, f^!$ and $\rmR f_{!!}$ the 
operations of inverse images,
direct images, proper inverse images and proper direct images 
respectively. 
We set also $\rhom := \alpha_M\circ\rihom$. 
We thus obtain the functors
\begin{align*}
\iota_M &: \BDC(\CC_M)\to \BDC(\I\CC_M),\\
\alpha_M &: \BDC(\I\CC_M)\to \BDC(\CC_M),\\
\beta_M &: \BDC(\CC_M)\to \BDC(\I\CC_M),\\
\otimes &: \BDC(\I\CC_M)\times\BDC(\I\CC_M)\to\BDC(\I\CC_M), \\
\rihom &: \BDC(\I\CC_M)^{\op}\times\BDC(\I\CC_M)\to\BDC(\I\CC_M), \\
\rhom &: \BDC(\I\CC_M)^{\op}\times\BDC(\I\CC_M)\to\BDC(\CC_M), \\
\rmR f_\ast &: \BDC(\I\CC_M)\to\BDC(\I\CC_N),\\
f^{-1} &: \BDC(\I\CC_N)\to\BDC(\I\CC_M),\\
\rmR f_{!!} &: \BDC(\I\CC_M)\to\BDC(\I\CC_N),\\
f^! &: \BDC(\I\CC_N)\to\BDC(\I\CC_M).
\end{align*}
Note that $(f^{-1}, \rmR f_\ast)$ and 
$(\rmR f_{!!}, f^!)$ are pairs of adjoint functors.
We may summarize the commutativity of the various functors 
we have introduced in the table below. 
Here, $``\circ"$ means that the functors commute,
and $``\times"$ they do not.
\begin{table}[h]
\begin{equation*}
   \begin{tabular}{l||c|c|c|c|c|c|c}
    {} & $\otimes$ & $f^{-1}$ & $\rmR f_\ast$ & $f^!$ & 
$\rmR f_{!!}$ & $\underset{}{\varinjlim}$ &  $\varprojlim$ \\ \hline \hline 
    $\underset{}{\iota}$ & $\circ$ & $\circ$ & $\circ$ & 
$\circ$ & $\times$ & $\times$ & $\circ$  \\ \hline
    $\underset{}{\alpha}$ & $\circ$ & $\circ$ & $\circ$ & 
$\times$ & $\circ$ & $\circ$ & $\circ$ \\ \hline
    $\underset{}{\beta}$ & $\circ$ & $\circ$ & $\times$  
& $\times$ & $\times$ & $\circ$& $\times$\\ \hline 
     $\underset{}{\varinjlim}$ & $\circ$ & $\circ$ & 
$\times$ & $\circ$ & $\circ$ &\multicolumn{2}{|c}{} \\\cline{1-6}
     $\underset{}{\varprojlim}$ & $\times$ & $\times$ & 
$\circ$ & $\times$ & $\times$ &\multicolumn{2}{|c}{} \\\cline{1-6}
   \end{tabular}
   \end{equation*}
\end{table}

\subsection{Ind-sheaves on Bordered Spaces}\label{sec:4} 
For the results in this subsection, we refer to 
D'Agnolo-Kashiwara \cite{DK16}. 
A bordered space is a pair $M_{\infty} = (M, \che{M})$ of
a good topological space $\che{M}$ and 
an open subset $M\subset\che{M}$.
A morphism $f : (M, \che{M})\to (N, \che{N})$ of bordered spaces
is a continuous map $f : M\to N$ such that the first projection
$\che{M}\times\che{N}\to\che{M}$ is proper on
the closure $\var{\Gamma}_f$ of the graph $\Gamma_f$ of $f$ 
in $\che{M}\times\che{N}$.
If also the second projection $\var{\Gamma}_f\to\che{N}$ is proper, 
we say that $f$ is semi-proper. 
The category of good topological spaces embeds into that
of bordered spaces by the identification $M = (M, M)$. 
We define the triangulated category of ind-sheaves on 
$M_{\infty} = (M, \che{M})$ by 
\begin{equation}
\BDC(\I\CC_{M_\infty}) := 
\BDC(\I\CC_{\che{M}})/\BDC(\I\CC_{\che{M}\backslash M}).
\end{equation}
The quotient functor
\begin{equation}
\mathbf{q} : \BDC(\I\CC_{\che{M}})\to\BDC(\I\CC_{M_\infty})
\end{equation}
has a left adjoint $\mathbf{l}$ and a right 
adjoint $\mathbf{r}$, both fully faithful, defined by 
\begin{equation}
\mathbf{l}(\mathbf{q} F) := \CC_M\otimes F,\hspace{25pt} 
\mathbf{r}(\mathbf{q} F) := \rihom(\CC_M, F). 
\end{equation}
For a morphism $f : M_\infty\to N_\infty$ 
of bordered spaces, we have 
the Grothendieck's operations 
\begin{align*} 
\otimes &: \BDC(\I\CC_{M_\infty})\times
\BDC(\I\CC_{M_\infty})\to\BDC(\I\CC_{M_\infty}), \\
\rihom &: \BDC(\I\CC_{M_\infty})^{\op}\times
\BDC(\I\CC_{M_\infty})\to\BDC(\I\CC_{M_\infty}), \\
\rmR f_\ast &: \BDC(\I\CC_{M_\infty})\to\BDC(\I\CC_{N_\infty}),\\
f^{-1} &: \BDC(\I\CC_{N_\infty})\to\BDC(\I\CC_{M_\infty}),\\
\rmR f_{!!} &: \BDC(\I\CC_{M_\infty})\to\BDC(\I\CC_{N_\infty}),\\
f^! &: \BDC(\I\CC_{N_\infty})\to\BDC(\I\CC_{M_\infty}) \\
\end{align*}
(see \cite[Definitions 3.3.1 and 3.3.4]{DK16} for the details). 
Moreover, there exists a natural embedding of categories 
\begin{equation}
\xymatrix@C=30pt@M=10pt{\BDC(\CC_M)\ar@{^{(}->}[r] & 
\BDC(\I\CC_{M_\infty}).}
\end{equation}

\subsection{Enhanced Sheaves}\label{sec:5}
For the results in this subsection, see 
Kashiwara-Schapira \cite{KS16-2} and 
D'Agnolo-Kashiwara \cite{DK17}. 
Let $M$ be a good topological space. 
We consider the maps 
\begin{equation}
M\times\RR^2\xrightarrow{p_1, p_2, \mu}M
\times\RR\overset{\pi}{\longrightarrow}M
\end{equation}
where $p_1, p_2$ are the first and the second projections 
and we set $\pi (x,t):=x$ and 
$\mu(x, t_1, t_2) := (x, t_1+t_2)$. 
Then the convolution functors for 
sheaves on $M \times \RR$ are defined by
\begin{align*}
F_1\Potimes F_2 &:= \rmR \mu_!(p_1^{-1}F_1\otimes p_2^{-1}F_2),\\
\Prhom(F_1, F_2) &:= \rmR p_{1\ast}\rhom(p_2^{-1}F_1, \mu^!F_2).
\end{align*}
We define the triangulated category of enhanced sheaves on $M$ by 
\begin{equation}
\BEC(\CC_M) := \BDC(\CC_{M\times\RR})/\pi^{-1}\BDC(\CC_M).
\end{equation}
Then the quotient functor
\begin{equation}
\Q : \BDC(\CC_{M \times \RR} )  \to\BEC(\CC_M)
\end{equation}
has fully faithful left and right adjoints 
$\bfL^\rmE, \bfR^\rmE$ defined by 
\begin{equation}
\bfL^\rmE(\Q F) := (\CC_{\{t\geq0\}}\oplus\CC_{\{t\leq 0\}})
\Potimes F ,\hspace{20pt} \bfR^\rmE(\Q G) :
=\Prhom(\CC_{\{t\geq0\}}\oplus\CC_{\{t\leq 0\}}, G), 
\end{equation}
where $\{t\geq0\}$ stands for $\{(x, t)\in 
M\times\RR\ |\ t\geq0\}$ and $\{t\leq0\}$ is 
defined similarly. The convolution functors 
are defined also for enhanced sheaves. We denote them 
by the same symbols $\Potimes$, $\Prhom$. 
For a continuous map $f : M \to N $, we 
can define naturally the operations 
$\bfE f^{-1}$, $\bfE f_\ast$, $\bfE f^!$, $\bfE f_{!}$ 
for enhanced sheaves. 
We have also a 
natural embedding $\e : \BDC( \CC_M) \to \BEC( \CC_M)$ defined by 
\begin{equation}
\e(F) := \Q(\CC_{\{t\geq0\}}\otimes\pi^{-1}F). 
\end{equation}
For a continuous function $\varphi : U\to \RR$ 
defined on an open subset $U \subset M$ of $M$  we define 
the exponential enhanced sheaf by 
\begin{equation}
{\rm E}_{U|M}^\varphi := 
\Q(\CC_{\{t+\varphi\geq0\}} ), 
\end{equation}
where $\{t+\varphi\geq0\}$ stands for 
$\{(x, t)\in M\times{\RR}\ |\ x\in U, t+\varphi(x)\geq0\}$.

\subsection{Enhanced Ind-sheaves}\label{sec:6}
We recall some basic notions 
and results on enhanced ind-sheaves. References are made to 
D'Agnolo-Kashiwara \cite{DK16} and 
Kashiwara-Schapira \cite{KS16}. 
Let $M$ be a good topological space.
Set $\RR_\infty := (\RR, \var{\RR})$ for 
$\var{\RR} := \RR\sqcup\{-\infty, +\infty\}$,
and let $t\in\RR$ be the affine coordinate. 
We consider the maps 
\begin{equation}
M\times\RR_\infty^2\xrightarrow{p_1, p_2, \mu}M
\times\RR_\infty\overset{\pi}{\longrightarrow}M
\end{equation}
where $p_1, p_2$ and $\pi$ are morphisms
of bordered spaces induced by the projections.
And $\mu$ is a morphism of bordered spaces induced by the map 
$M\times\RR^2\ni(x, t_1, t_2)\mapsto(x, t_1+t_2)\in M\times\RR$.
Then the convolution functors for 
ind-sheaves on $M \times \RR_\infty$ are defined by
\begin{align*}
F_1\Potimes F_2 &:= \rmR\mu_{!!}(p_1^{-1}F_1\otimes p_2^{-1}F_2),\\
\Prihom(F_1, F_2) &:= \rmR p_{1\ast}\rihom(p_2^{-1}F_1, \mu^!F_2).
\end{align*}
Now we define the triangulated category 
of enhanced ind-sheaves on $M$ by 
\begin{equation}
\BEC(\I\CC_M) := \BDC(\I\CC_{M 
\times\RR_\infty})/\pi^{-1}\BDC(\I\CC_M).
\end{equation}
Note that we have a natural embedding of categories
\begin{equation}
\BEC(\CC_M) \xhookrightarrow{\ \ \ }\BEC(\I\CC_M).
\end{equation}
The quotient functor
\begin{equation}
\Q : \BDC(\I\CC_{M\times\RR_\infty})\to\BEC(\I\CC_M)
\end{equation}
has fully faithful left and right adjoints 
$\bfL^\rmE,\bfR^\rmE$ defined by 
\begin{equation}
\bfL^\rmE(\Q K) := (\CC_{\{t\geq0\}}\oplus\CC_{\{t\leq 0\}})
\Potimes K ,\hspace{20pt} \bfR^\rmE(\Q K) :
=\Prihom(\CC_{\{t\geq0\}}\oplus\CC_{\{t\leq 0\}}, K), 
\end{equation}
where $\{t\geq0\}$ stands for 
$\{(x, t)\in M\times\var{\RR}\ |\ t\in\RR, t\geq0\}$ 
and $\{t\leq0\}$ is defined similarly.

The convolution functors 
are defined also for enhanced ind-sheaves. We denote them 
by the same symbols $\Potimes$, $\Prihom$. 
For a continuous map $f : M \to N $, we 
can define also the operations 
$\bfE f^{-1}$, $\bfE f_\ast$, $\bfE f^!$, $\bfE f_{!!}$ 
for enhanced ind-sheaves. For example, 
by the natural morphism $\tl{f}: M \times \RR_{\infty} \to 
N \times \RR_{\infty}$ of bordered spaces associated to 
$f$ we set $\bfE f_\ast ( \Q K)= \Q(\rmR \tl{f}_{\ast}(K))$. 
The other operations are defined similarly. 
We thus obtain the six operations $\Potimes$, $\Prihom$,
$\bfE f^{-1}$, $\bfE f_\ast$, $\bfE f^!$, $\bfE f_{!!}$ 
for enhanced ind-sheaves .
Moreover we denote by $\rmD_M^\rmE$ 
the Verdier duality functor for enhanced ind-sheaves.
We have outer hom functors
\begin{align*}
\rihom^\rmE(K_1, K_2) &:= \rmR\pi_\ast\rihom(\bfL^\rmE K_1, 
\bfL^\rmE K_2)
\simeq \rmR\pi_\ast\rihom(\bfL^\rmE K_1, \bfR^\rmE K_2),\\
\rhom^\rmE(K_1, K_2) &:= \alpha_M\rihom^\rmE(K_1, K_2),\\
\rHom^\rmE(K_1, K_2) &:=\rmR\Gamma(M; \rhom^\rmE(K_1, K_2)),
\end{align*}
with values in $\BDC(\I\CC_M), 
\BDC(\CC_M)$ and $\BDC(\CC)$, respectively. 
Moreover for $F\in\BDC(\I\CC_M)$ and $K\in\BEC(\I\CC_M)$ 
the objects 
\begin{align*}
\pi^{-1}F\otimes K &:=\Q(\pi^{-1}F\otimes \bfL^\rmE K),\\
\rihom(\pi^{-1}F, K) &:=\Q\big(\rihom(\pi^{-1}F, \bfR^\rmE K)\big). 
\end{align*}
in $\BEC(\I\CC_M)$ are well-defined. 
Set $\CC_M^\rmE := \Q 
\Bigl(``\underset{a\to +\infty}{\varinjlim}"\ \CC_{\{t\geq a\}}
\Bigr)\in\BEC(\I\CC_M)$. 
Then we have 
natural embeddings $\e, e : \BDC(\I\CC_M) \to 
\BEC(\I\CC_M)$ defined by 
\begin{align*}
\e(F) & := \Q(\CC_{\{t\geq0\}}\otimes\pi^{-1}F) \\ 
e(F) &:=  \CC_M^\rmE\otimes\pi^{-1}F
\simeq \CC_M^\rmE\Potimes\e(F). 
\end{align*}

For a continuous function $\varphi : U\to \RR$ 
defined on an open subset $U \subset M$ of $M$  we define 
the exponential enhanced ind-sheaf by 
\begin{equation}
\EE_{U|M}^\varphi := 
\CC_M^\rmE\Potimes {\rm E}_{U|M}^\varphi
=
\CC_M^\rmE\Potimes
\Q\CC_{\{t+\varphi\geq0\}}
=
 \Q 
\Bigl(``\underset{a\to +\infty}{\varinjlim}"\ 
\CC_{\{t+\varphi\geq a\}}
\Bigr)
\end{equation}
where $\{t+\varphi\geq0\}$ stands for 
$\{(x, t)\in M\times\var{\RR}\ |\ t\in\RR, x\in U, 
t+\varphi(x)\geq0\}$.

\subsection{D-modules}\label{sec:7}
In this subsection we recall some basic notions 
and results on D-modules. 
References are made to 
\cite{HTT08}, \cite[\S 7]{KS01}, 
\cite[\S 8, 9]{DK16},
\cite[\S 3, 4, 7]{KS16} and 
\cite[\S 4, 5, 6, 7, 8]{Kas16}. 
For a complex manifold $X$ we denote by 
$d_X$ its complex dimension. 
Denote by $\SO_X, \Omega_X$ and $\SD_X$ 
the sheaves of holomorphic functions, 
holomorphic differential forms of top degree 
and holomorphic differential operators, respectively. 
Let $\BDC(\SD_X)$ be the bounded derived category 
of left $\SD_X$-modules and $\BDC(\SDop_X)$ 
be that of right $\SD_X$-modules. 
Moreover we denote by $\BDCcoh(\SD_X)$, $\BDC_{\rm good}(\SD_X)$,
$\BDChol(\SD_X)$ and $\BDCrh(\SD_X)$ the full triangulated subcategories
of $\BDC(\SD_X)$ consisting of objects with coherent, good, 
holonomic and regular holonomic cohomologies, 
respectively.
For a morphism $f : X\to Y$ of complex manifolds, 
denote by $\Dotimes, \rhom_{\SD_X}, \bfD f_\ast, \bfD f^\ast$ 
the standard operations for D-modules. 
We define also the duality functor $\DD_X : \BDCcoh(\SD_X)^{\op} 
\simto \BDCcoh(\SD_X)$ 
by 
\begin{equation}
\DD_X(\SM):=\rhom_{\SD_X}(\SM, \SD_X)
\uotimes{\SO_X}\Omega^{\otimes-1}_X[d_X].
\end{equation}
Note that there exists 
an equivalence of categories 
$( \cdot )^{\rm r} : \Mod(\SD_X)\simto\Mod(\SDop_X)$ given by 
\begin{equation}
\SM^{\rm r}:=\Omega_X\uotimes{\SO_X}\SM.
\end{equation}
The classical de Rham and solution functors are defined by  
\begin{align*}
DR_X &:  \BDCcoh (\SD_X)\to\BDC(\CC_X),
\hspace{40pt}\SM \longmapsto \Omega_X\Lotimes{\SD_X}\SM, \\
Sol_X &: \BDCcoh (\SD_X)^{\op}\to\BDC(\CC_X),
\hspace{30pt}\SM \longmapsto \rhom_{\SD_X}(\SM, \SO_X).
\end{align*}
Then for $\SM\in\BDCcoh(\SD_X)$ 
we have an isomorphism $Sol_X(\SM)[d_X]\simeq DR_X(\DD_X\SM)$. 
For a closed hypersurface $D\subset X$ 
in $X$ we denote by $\SO_X(\ast D)$ 
the sheaf of meromorphic functions on $X$ with poles in $D$. 
Then for $\SM\in\BDC(\SD_X)$ we set 
\begin{equation}
\SM(\ast D) := \SM\Dotimes\SO_X(\ast D).
\end{equation}
For $f\in\SO_X(\ast D)$ and $U := X\bs D$, set 
\begin{align*}
\SD_Xe^f &:= \SD_X/\{P \in \SD_X  \ |\ Pe^f|_U = 0\}, \\
\SE_{U|X}^f &:= \SD_Xe^f(\ast D).
\end{align*}
Note that $\SE_{U|X}^f$ is holonomic and there exists an isomorphism 
\begin{equation}
\DD_X(\SE_{U|X}^f)(\ast D) \simeq \SE_{U|X}^{-f}.
\end{equation}
Namely $\SE_{U|X}^f$ is a meromorphic connection associated 
to $d+df$. 

One defines the ind-sheaf $\SO_X^\rmt$ of tempered 
holomorphic functions 
as the Dolbeault complex with coefficients in the ind-sheaf 
of tempered distributions. 
More precisely, denoting by $\overline{X}$ the complex conjugate manifold 
to $X$ and by $X_{\RR}$ the underlying real analytic manifold of $X$,
we set
\begin{equation}
\SO_X^\rmt := \rihom_{\SD_{\overline{X}}}(
\SO_{\overline{X}}, \mathcal{D}b_{X_\RR}^\rmt), 
\end{equation}
where $\mathcal{D}b_{X_\RR}^\rmt$ is the ind-sheaf of 
tempered distributions on $X_\RR$
(for the definition see \cite[Definition 7.2.5]{KS01}). 
Moreover, we set 
\begin{equation}
\Omega_X^\rmt := \beta_X\Omega_X\otimes_{\beta_X\SO_X}\SO_X^\rmt.
\end{equation}
Then the tempered de Rham and solution functors 
are defined by 
\begin{align*}
DR_X^\rmt &: \BDCcoh (\SD_X)\to\BDC(\I\CC_X),
\hspace{40pt} 
\SM \longmapsto \Omega_X^\rmt\Lotimes{\SD_X}\SM, 
\\
Sol_X^\rmt &: \BDCcoh (\SD_X)^{\op}\to\BDC(\I\CC_X), 
\hspace{40pt}
\SM \longmapsto \rihom_{\SD_X}(\SM, \SO_X^\rmt). 
\end{align*}
Note that we have isomorphisms
\begin{align*}
Sol_X(\SM) &\simeq \alpha_XSol_X^\rmt(\SM), \\
DR_X(\SM) &\simeq \alpha_XDR_X^\rmt(\SM), \\
Sol_X^\rmt(\SM)[d_X] &\simeq DR_X^\rmt(\DD_X\SM).
\end{align*}

Let $i : X\times\RR_\infty\to X\times\PP$ be
the natural morphism of bordered spaces and
$\tau\in\CC\subset\PP$ the affine coordinate 
such that $\tau|_\RR$ is that of $\RR$. 
We then define objects $\SO_X^\rmE\in\BEC(\I\SD_X)$ and 
$\Omega_X^\rmE\in\BEC(\I\SD_X^{\op})$ by 
\begin{align*}
\SO_X^\rmE &:= \rihom_{\SD_{\overline{X}}}(
\SO_{\overline{X}}, \mathcal{D}b_{X_\RR}^{\T})\\
&\simeq i^!\bigl((\SE_{\CC|\PP}^{-\tau})^{\rm r} 
\Lotimes{\SD_\PP}\SO_{X\times\PP}^\rmt\bigr)[1]
\simeq i^!\rihom_{\SD_\PP}(\SE_{\CC|\PP}^{\tau}, 
\SO_{X\times\PP}^\rmt)[2],\\
\Omega_X^\rmE &:= \Omega_X\Lotimes{\SO_X}\SO_X^\rmE\simeq 
i^!(\Omega_{X\times\PP}^\rmt\Lotimes{\SD_\PP}
\SE_{\CC|\PP}^{-\tau})[1], 
\end{align*}
where $\mathcal{D}b_{X_\RR}^{\T}$ 
stand for the enhanced ind-sheaf 
of tempered distributions on $X_\RR$ 
(for the definition see \cite[Definition 8.1.1]{DK16}). 
We call $\SO_X^\rmE$ the enhanced ind-sheaf 
of tempered holomorphic functions. 
Note that there exists an isomorphism
\begin{equation}
i_0^!\bfR^\rmE\SO_X^\rmE\simeq\SO_X^\rmt, 
\end{equation}
where $i_0 : X\to X\times\RR_\infty$ is the 
inclusion map of bordered spaces induced by $x\mapsto (x, 0)$. 
The enhanced de Rham and solution functors 
are defined by 
\begin{align*}
DR_X^\rmE &: \BDCcoh (\SD_X)\to\BEC(\I\CC_X),
\hspace{40pt} 
\SM \longmapsto \Omega_X^\rmE\Lotimes{\SD_X}\SM,
\\
Sol_X^\rmE &: \BDCcoh (\SD_X)^{\op}\to\BEC(\I\CC_X), 
\hspace{40pt} 
\SM \longmapsto \rihom_{\SD_X}(\SM, \SO_X^\rmE). 
\end{align*}
Then for $\SM\in\BDCcoh(\SD_X)$ 
we have isomorphism $Sol_X^\rmE(\SM)[d_X]\simeq DR_X^\rmE(\DD_X\SM)$ 
and 
$Sol^{\rmt}_X(\M)\simeq
i_0^!{\bfR^\rmE}Sol_X^{\rmE}(\M).$ 
We recall the following results of \cite{DK16}.

\begin{theorem}
\begin{enumerate}\label{thm-4}
\item[\rm{(i)}] For $\SM\in\BDC_{\rm hol}(\SD_X)$ there
is an isomorphism in $\BEC(\I\CC_X)$
\begin{equation}
\rmD_X^\rmE\bigl(DR_X^\rmE(\SM)\bigr) \simeq Sol_X^\rmE(\SM)[d_X].
\end{equation}

\item[\rm{(ii)}] Let $f : X\to Y$ be a morphism of complex manifolds.
Then for $\SN\in\BDC_{\rm hol}(\SD_Y)$ 
there is an isomorphism in $\BEC(\I\CC_X)$
\begin{equation}
Sol_X^\rmE({\bfD} f^\ast\SN) \simeq \bfE f^{-1}Sol_Y^\rmE(\SN).
\end{equation}

\item[\rm{(iii)}] Let $f : X\to Y$ be a morphism of complex manifolds
and $\SM\in\BDC_{\rm good}(\SD_X)\cap\BDC_{\rm hol}(\SD_X)$.
If $\supp(\SM)$ is proper over Y then there 
is an isomorphism in $\BEC(\I\CC_Y)$
\begin{equation}
Sol_Y^\rmE({\bfD} f_\ast\SM)[d_Y] \simeq 
\bfE f_\ast Sol_X^\rmE(\SM )[d_X].
\end{equation}

\item[\rm{(iv)}] For $\SM_1, \SM_2\in\BDC_{\rm hol}(\SD_X)$, 
there exists an isomorphism in $\BEC(\I\CC_X)$
\begin{equation}
Sol_X^\rmE(\SM_1\Dotimes\SM_2)\simeq Sol_X^\rmE(\SM_1)
\Potimes Sol_X^\rmE(\SM_2).
\end{equation}

\item[\rm{(v)}] If $\SM\in\BDC_{\rm hol}(\SD_X)$ and $D\subset X$
is a closed hypersurface, then there 
are isomorphisms in $\BEC(\I\CC_X)$
\begin{align*}
Sol_X^\rmE(\SM(\ast D)) &\simeq \pi^{-1}
\CC_{X\bs D}\otimes Sol_X^\rmE(\SM),\\
DR_X^\rmE(\SM(\ast D)) &\simeq \rihom(
\pi^{-1}\CC_{X\bs D}, DR_X^\rmE(\SM)).
\end{align*}

\item[\rm{(vi)}] Let $D$ be a closed hypersurface in $X$ and 
$f\in\SO_X(\ast D)$ a meromorphic function along $D$. 
Set $U:=X \setminus D \subset X$. 
Then there exists an isomorphism in $\BEC(\I\CC_X)$
\begin{equation}
Sol_X^\rmE\big(\mathscr{E}_{U | X}^f \big) 
\simeq \EE_{U | X}^{\Re f}.
\end{equation}

\item[\rm{(vii)}] For $\SL \in \BDCrh(\SD_X)$ and 
$\SM \in\BDC_{\rm hol}(\SD_X)$, 
there exists an isomorphism in $\BEC(\I\CC_X)$
\begin{equation}
DR_X^\rmE(\SL \Dotimes \SM )\simeq \rihom ( 
\pi^{-1} 
Sol_X^\rmE(\SL ), DR_X^\rmE(\SM) ).
\end{equation}
\end{enumerate}
\end{theorem}

We also have the following corollary of Theorem \ref{thm-4}. 

\begin{corollary}\label{cor-41}
For $\SL \in \BDCrh(\SD_X)$ and 
$\SM \in\BDC_{\rm hol}(\SD_X)$, 
there exists an isomorphism in $\BEC(\I\CC_X)$
\begin{equation}
Sol_X^\rmE(\SL \Dotimes \SM )\simeq 
\pi^{-1} 
Sol_X (\SL ) \otimes Sol_X^\rmE(\SM).
\end{equation}
\end{corollary}

\begin{proof}
Let 
\begin{equation}
\rmD_X^\rmE ( \cdot ): 
\BEC(\I\CC_X)^{\rm op} \longrightarrow 
\BEC(\I\CC_X), \qquad F \longmapsto 
\Prihom (F, \omega^\rmE_X) 
\end{equation}
be the Verdier dual functor defined in 
\cite[Section 4.8]{DK16}, where 
we set $\omega^\rmE_X = 
\CC_X^\rmE \otimes \pi^{-1} \omega_X$. 
Moreover we set $L= Sol_X (\SL )$ 
and $K= Sol_X^\rmE(\SM)$. 
Then we have isomorphisms 
\begin{align*}
 \rmD_X^\rmE ( \pi^{-1} L \otimes K) 
\simeq & \rihom ( \pi^{-1} L, \rmD_X^\rmE (K)) 
\\
\simeq & 
\rihom ( \pi^{-1} L, DR_X^\rmE(\SM) [d_X]) 
\\
\simeq & 
DR_X^\rmE(\SL \Dotimes \SM ) [ d_X ], 
\end{align*}
where the first isomorphism follows from 
the definition of the bifunctor 
$\Prihom ( \cdot, \cdot )$ and 
in the second (resp. third) isomorphism 
we used Theorem \ref{thm-4} (i) 
(resp. Theorem \ref{thm-4} (vii)). By taking 
the Verdier dual $\rmD_X^\rmE ( \cdot )$ of 
the both sides, we obtain the desired 
isomorphism. 
\end{proof}

We recall also the following theorem of \cite{DK16}. 

\begin{theorem}[{\cite[Theorem 9.5.3 (Irregular 
Riemann-Hilbert Correspondence)]{DK16}}]\label{cor-4}
There exists an isomorphism functorial 
with respect to $\SM\in\BDC_{\rm hol}(\SD_X) :$
\begin{equation}
\SM\simto \rhom^\rmE(Sol_X^\rmE(\SM), \SO_X^\rmE)
\end{equation}
in $\BDC_{\rm hol}(\SD_X)$.
\end{theorem}

For the proof of the main results in \cite{DK16}, 
a key role was played by the following proposition. 
Here we give a very short new proof 
deduced directly from \cite[Propositon 7.3 and Remark 7.4]{KS03} 
(see also \cite[Proposition 3.14]{IT20a}). 

\begin{proposition}[{\cite[Proposition 6.2.2]{DK16}}]\label{pro-dk}
Let $D$ be a closed hypersurface in $X$ and 
$f\in\SO_X(\ast D)$ a meromorphic function along $D$. 
Set $U:=X \setminus D \subset X$. 
Then there exists an isomorphism in $\BDC(\I\CC_X)$ 
\begin{equation}
DR_X^\rmt \big(\mathscr{E}_{U | X}^{-f} \big) 
\simeq 
\rihom \Big( \CC_U , 
\underset{a\to+\infty}{\inj}\ 
\CC_{ \{ \Re f<a \} } \Big)[d_X].
\end{equation}
\end{proposition}

\begin{proof}
Let $z$ be the standard holomorphic coordinate of $\CC$. 
As in \cite[Lemma 6.2.5]{DK16}, first we consider the 
exponetial D-module 
$\mathscr{E}_{\CC | \PP}^{z}$ on the one dimensional 
projective space $\PP = \PP^1$ associated to 
the meromorphic function $f=z \in 
\SO_{\PP}(\ast \{ \infty \} )$. 
Then for the holomorphic coordinate $\zeta := \frac{1}{z}$ 
on a neighborhood of the point 
$\infty \in \PP$ there exist isomorphisms 
\begin{equation}
\mathscr{E}_{\CC | \PP}^{z} \simeq \SD_{\PP} \exp \Bigl( \frac{1}{\zeta} \Bigr) 
\simeq \SD_{\PP}/ \SD_{\PP}( \zeta^2 \partial_{\zeta} +1). 
\end{equation}
This implies that on a neighborhood of $\infty \in \PP$ 
we have isomorphisms 
\begin{align}
\DD_{\PP} \big( \mathscr{E}_{\CC | \PP}^{z} \big) 
& \simeq 
\SD_{\PP}/ \SD_{\PP}( - \partial_{\zeta} \zeta^2+1) 
\simeq \SD_{\PP} \Bigl\{ \frac{1}{\zeta^2} \exp \Bigl(- \frac{1}{\zeta} \Bigr) 
\Bigr\} 
\\
& \simeq \SD_{\PP} \exp \Bigl(- \frac{1}{\zeta} \Bigr) 
\simeq \mathscr{E}_{\CC | \PP}^{-z}
\end{align}
(cf. \cite[Lemma 6.1.2 and Remark 6.1.3]{DK16}). 
In this case, we thus obtain an isomorphism 
\begin{equation}
DR_{\PP}^\rmt \big( \mathscr{E}_{\CC | \PP}^{-z} \big) 
\simeq 
Sol_{\PP}^\rmt \big( \mathscr{E}_{\CC | \PP}^{z} \big) [1]. 
\end{equation}
Consider the distinguished triangle 
\begin{equation}\label{dist-t} 
H^0 Sol_{\PP}^\rmt \big( \mathscr{E}_{\CC | \PP}^{z} \big) 
\longrightarrow 
Sol_{\PP}^\rmt \big( \mathscr{E}_{\CC | \PP}^{z} \big) 
\longrightarrow 
\tau^{\geq 1} 
Sol_{\PP}^\rmt \big( \mathscr{E}_{\CC | \PP}^{z} \big) 
\overset{+1}{\longrightarrow}
\end{equation}
in $\BDC(\I\CC_X)$. Then, according to 
\cite[Remark 7.4]{KS03} (see also \cite[Proposition 3.14]{IT20a}), 
for the closed 
embedding $i_{\infty}: \{ \infty \} 
\hookrightarrow \PP$ there exists an isomorphism 
\begin{equation}
\tau^{\geq 1} 
Sol_{\PP}^\rmt \big( \mathscr{E}_{\CC | \PP}^{z} \big) 
\simeq (i_{\infty})_* \CC_{\{ \infty \}}[-1]. 
\end{equation}
From this and $i_{\infty}^{-1} \CC_{\CC} \simeq 0$ 
we obtain the vanishing 
\begin{equation}
\rihom \Big( \CC_{\CC} , 
\tau^{\geq 1} 
Sol_{\PP}^\rmt \big( \mathscr{E}_{\CC | \PP}^{z} \big) 
\Big) \simeq 
\rihom \Big( i_{\infty}^{-1} \CC_{\CC} , 
\CC_{\{ \infty \}}[-1] \Big)
\simeq 0. 
\end{equation}
On the other hand, by \cite[Proposition 7.3]{KS03} 
(see also \cite[Proposition 3.14]{IT20a})  
we have an isomorphism 
\begin{equation}
H^0 Sol_{\PP}^\rmt \big( \mathscr{E}_{\CC | \PP}^{z} \big) 
\simeq
\underset{a\to+\infty}{\inj}\ 
\CC_{\{ z \in \CC  | \Re z <a \}}. 
\end{equation}
Then the assertion follows immediately from 
the isomorphisms 
\begin{equation}
DR_{\PP}^\rmt \big( \mathscr{E}_{\CC | \PP}^{-z} \big) 
\simeq 
DR_{\PP}^\rmt \big( \mathscr{E}_{\CC | \PP}^{-z}
( \ast \{ \infty \}) \big) 
\simeq 
\rihom \Big( \CC_{\CC} , 
DR_{\PP}^\rmt \big( \mathscr{E}_{\CC | \PP}^{-z} \big) 
\Big)
\end{equation}
(see \cite[Theorem 7.4.12]{KS01}) and the 
distinguished triangle \eqref{dist-t}. 
Next, we consider the general exponential D-module 
$\mathscr{E}_{U | X}^{-f}$ for $f\in\SO_X(\ast D)$. 
Let $\nu : Y \longrightarrow X$ be a proper 
morphism of complex manifolds such that 
$E:= \nu^{-1} D \subset Y$ is a normal crossing 
divisor in $Y$, the restriction 
$\nu |_{Y \setminus E}: Y \setminus E 
\longrightarrow X \setminus D$ of 
$\nu$ is an isomorphism and the meromorphic 
function $g:=f \circ \nu \in\SO_Y(\ast E)$ 
on $Y$ has no point of indeterminacy on the 
whole $Y$. Such a resolution of singularities of 
$D \subset X$ always exists. See for example 
the proof of \cite[Theorem 3.6]{MT13}. 
Let $E_0$ (resp. $E_1$) be the union of 
the irreducible components of $E$ along 
which $g$ has no pole (resp. has a pole) 
so that we have $E=E_0 \cup E_1$. Set 
$V:=Y \setminus E$ and $\tilde{V} 
:=Y \setminus E_1 \supset V$. Then $g \in 
\SO_Y(\ast E)$ extends to a holomorphic 
function on $\tilde{V}$ and we obtain a 
meromorphic function 
$\tilde{g} \in \SO_Y(\ast E_1)$. Moreover 
there exists an isomorphism
\begin{equation}
\mathscr{E}_{V | Y}^{-g} \simeq 
\mathscr{E}_{\tilde{V} | Y}^{- \tilde{g}} 
( \ast E_0). 
\end{equation}
Since we have an isomorphism 
\begin{equation}
\mathscr{E}_{U | X}^{-f} \simeq 
\Big( 
{\bfD} \nu_\ast 
\mathscr{E}_{V | Y}^{-g} 
\Big) ( \ast D),
\end{equation}
by \cite[Theorems 7.4.6 and 7.4.12]{KS01} 
we obtain an isomorphism 
\begin{equation}
DR_X^\rmt \big(\mathscr{E}_{U | X}^{-f} \big) 
\simeq 
\rmR \nu_* 
\rihom \Big( \CC_{V} , 
DR_{Y}^\rmt \big( 
\mathscr{E}_{\tilde{V} | Y}^{- \tilde{g}} 
 \big) \Big). 
\end{equation}
Since the meromorphic function 
$\tilde{g} \in \SO_Y(\ast E_1)$ has no point 
of indeterminacy on the whole $Y$, 
we obtain a holomorphic map from $Y$ to 
$\PP$. We denote it by the same letter $\tilde{g}$ 
for simplicity. Then by 
\cite[Theorem 7.4.1]{KS01} it follows from 
the isomorphism 
\begin{equation}
\mathscr{E}_{\tilde{V} | Y}^{- \tilde{g}} 
\simeq {\bfD} \tilde{g}^* 
\mathscr{E}_{\CC | \PP}^{-z}
\end{equation}
that we obtain an isomorphism
\begin{equation}
DR_{Y}^\rmt \big( 
\mathscr{E}_{\tilde{V} | Y}^{- \tilde{g}} 
 \big) \simeq 
\tilde{g}^! 
DR_{\PP}^\rmt \big( \mathscr{E}_{\CC | \PP}^{-z} \big) 
[1-d_X] \simeq 
\rihom \Big( \CC_{\tilde{V}} , 
\underset{a\to+\infty}{\inj}\ 
\tilde{g}^! \CC_{\{ \Re z <a \}}
\Big) [2-d_X]. 
\end{equation}
Since for an open neighborhood $W$ of the normal 
crossing divisor $E_1= \tilde{g}^{-1}( \infty ) \subset 
Y$ the restriction of the morphism 
$\tilde{g}: Y \longrightarrow \PP$ to 
$\tilde{V} \cap W$ is a topological submersion, 
we have an isomorphism 
$\tilde{g}^! \CC_{\{ \Re z <a \}} \simeq 
\tilde{g}^{-1} \CC_{\{ \Re z <a \}} [2 d_X-2]$ 
on $\tilde{V} \cap W$. Moreover, if $a>0$ is 
large enough, on $\tilde{V} \setminus W$ 
we have 
\begin{equation}
\tilde{g}^! \CC_{\{ \Re z <a \}} \simeq 
\tilde{g}^! \CC_{\PP} \simeq 
\tilde{g}^! \omega_{\PP} [-2] \simeq 
\omega_Y [-2] \simeq 
\CC_Y [2d_X-2] \simeq 
\tilde{g}^{-1} \CC_{\{ \Re z <a \}} [2 d_X-2]. 
\end{equation}
Combining these results together, we finally 
get the desired isomorphism as follows: 
\begin{align*}
& DR_X^\rmt \big(\mathscr{E}_{U | X}^{-f} \big) 
\simeq 
\rmR \nu_* 
\rihom \Big( \CC_{V} , 
\rihom \Big( \CC_{\tilde{V}} , 
\underset{a\to+\infty}{\inj}\ 
 \CC_{\{ \Re \tilde{g} <a \}}
\Big) [d_X] \Big)
\\
& \simeq 
\rmR \nu_* 
\rihom \Big( \CC_{V} , 
\underset{a\to+\infty}{\inj}\ 
 \CC_{\{ \Re g <a \}} \Big) [d_X] 
\\
& \simeq 
\rihom \Big( \CC_{U} , \rmR \nu_* \Big( 
\underset{a\to+\infty}{\inj}\ 
 \CC_{\{ \Re g <a \}} \Big) \Big) [d_X] 
\\
& \simeq 
\rihom \Big( \CC_{U} ,
\underset{a\to+\infty}{\inj}\ 
\CC_{\{ \Re f <a \}} \Big) [d_X]. 
\end{align*} 
This completes the proof. 
\end{proof}

\section{Several Micro-supports Related to 
Holonomic 
\\ D-modules}\label{sec:19}

In this section, we introduce several micro-supports 
related to holonomic D-modules and study 
their properties especially for exponentially twisted 
ones. Let $X$ be a complex manifold of dimension $N$ 
and $\BEC_+(\CC_{X})$ the full 
subcategory of $\BEC(\CC_{X})$ 
consisting of objects $F \in \BEC(\CC_{X})$ 
such that $F \simeq \CC_{\{ t \geq 0 \}} 
\Potimes F$. If for a holonomic $\SD_X$-module 
$\SM \in\Modhol(\SD_X)$ there exists an 
object $F \in \BEC_+(\CC_{X})$ of 
$\BEC_+(\CC_{X})$ such that 
\begin{equation}\label{eq-sit-net} 
Sol_{X}^\rmE( \SM ) \simeq 
\CC^\rmE_{X} \Potimes F, 
\end{equation}
the following micro-supports of 
$F$ introduced in Tamarkin \cite{Tama08} 
are useful in the study of $\SM$. Let $X_{\RR}$ be the 
underlying real analytic 
manifold of $X$ and by the standard coordinate 
$(t;t^*)$ of $T^* \RR$ identify $(T^*X_{\RR}) \times \RR$ 
with the subset $(T^*X_{\RR}) \times \{ t^*=1 \} \subset 
(T^*X_{\RR}) \times (T^* \RR )$ of 
$T^*(X_{\RR} \times \RR ) 
\simeq (T^*X_{\RR}) \times (T^* \RR )$. Let 
\begin{equation}
\iota_{\RR} : (T^*X_{\RR}) \times \RR \hookrightarrow 
T^*(X_{\RR} \times \RR ) 
\end{equation}
be the inclusion map. 
For an object $F \in \BEC_+(\CC_{X})$ taking 
$\widetilde{F} \in \BDC(\CC_{X \times \RR})$ such that 
$\Q ( \widetilde{F} )=F$ we set 
\begin{equation}
{\rm SS}^{{\rm E}}(F):= 
\iota_{\RR}^{-1} {\rm SS}( \widetilde{F} ) 
\ \subset \ (T^*X_{\RR}) \times \RR. 
\end{equation}
This definition does not depend on the 
choice of $\widetilde{F}$ (see \cite[Section 2.6]{DK17}). 
We call ${\rm SS}^{{\rm E}}(F)$ 
the enhanced micro-support of $F$. 
For a local coordinate $z=x+iy=(z_1,z_2, \ldots, z_N)$ of $X$ 
let $(x,y,t;x^*,y^*,t^*)$ be the 
corresponding coordinate of $T^*(X_{\RR} \times \RR ) 
\simeq (T^*X_{\RR}) \times (T^* \RR )$. 
Then on the open subset $\{ t^*>0 \} \subset 
T^*(X_{\RR} \times \RR )$ locally we define a 
real analytic map 
$\gamma : \{ t^*>0 \} \longrightarrow 
(T^*X_{\RR}) \times \RR$ by 
\begin{equation}
(x,y,t;x^*,y^*,t^*) \longmapsto 
\Bigl(
\bigl( x,y; \frac{x^*}{t^*}, \frac{y^*}{t^*} \bigr), t 
\Bigr).
\end{equation}
By the conicness of ${\rm SS}( \widetilde{F} )$ 
we can easily see that 
\begin{equation}
{\rm SS}^{{\rm E}}(F)= 
\gamma \bigl( {\rm SS}( \widetilde{F} )
\cap \{ t^*>0 \} \bigr).  
\end{equation}
Let ${\rm pr}_{T^*X_{\RR}}: (T^*X_{\RR}) \times \RR 
\longrightarrow T^*X_{\RR}$ be the projection and set 
\begin{equation}
{\rm SS}_{{\rm irr}}(F):= 
\overline{{\rm pr}_{T^*X_{\RR}} 
\bigl( {\rm SS}^{{\rm E}}(F) \bigr)} 
\ \subset \ T^*X_{\RR}
\end{equation}
(see \cite[Section 2.6]{DK17} for a 
different notation for it). 
We call it  the irregular micro-support of $F$. 

From now on, we assume that $X$ is a smooth algebraic 
variety over $\CC$ and denote by $X^{\an}$ its  
underlying complex manifold that we sometimes 
denote by $X$ for short. We set $X_{\RR}:= 
(X^{\an})_{\RR}$. We shall consider 
the micro-supports of the objects 
$F \in \BEC_+(\CC_{X^{\an}})$ related to 
the following special but basic 
holonomic D-modules. The problem being local, 
we may assume that $X$ is affine. Then for a rational function 
$f= \frac{P}{Q}: X \setminus Q^{-1}(0) 
\longrightarrow \CC$ ($P, Q \in 
\Gamma(X; \SO_X)$, 
$Q \not= 0$) on $X$ we set 
$U:= X \setminus Q^{-1}(0)$ and 
define an algebraic 
exponential $\SD_X$-module $\SE^f_{U|X} 
\in\Modhol (\SD_X)$ as in the analytic case. 
Let $i_U:U \hookrightarrow X$ be the inclusion map and 
$\SO_U(f) \in \Modhol (\SD_U)$ the algebraic 
integrable connection on $U$ associated to the 
regular function $f:U \longrightarrow \CC$. Then 
there exist isomorphisms  
\begin{equation}
\SE^f_{U|X} \simeq \bfD i_{U*} \SO_U(f) 
\simeq (i_U)_* \SO_U(f). 
\end{equation}
We define the analytification
$( \SE^f_{U|X})^{\an} \in\Modhol(\SD_{X^{\an}})$
of $\SE^f_{U|X}$ by
$( \SE^f_{U|X})^{\an} := \SO_{X^{\an}}
\otimes_{\SO_{X}} \SE^f_{U|X}$ 
and set 
\begin{equation}
Sol_{X}^\rmE( \SE^f_{U|X} ) := 
Sol_{X^{\an}}^\rmE( ( \SE^f_{U|X} )^{\an} ) 
\qquad  \in \ \BEC(\I\CC_{X^{\an}}).
\end{equation}
Then by Theorem \ref{thm-4} (vi) we have an isomorphism 
\begin{equation}
Sol_{X}^\rmE(  \SE^f_{U|X} ) \simeq 
\EE_{U^{\an}| X^{\an}}^{{\rm Re}f}. 
\end{equation}

\begin{definition}\label{defi-etd}
We say that a holonomic $\SD_X$-module 
$\SM \in\Modhol(\SD_X)$ is an exponentially twisted holonomic 
D-module if there exist a regular holonomic 
$\SD_X$-module $\SN \in\Modrh(\SD_X)$ and a rational function 
$f= \frac{P}{Q}: U=X \setminus Q^{-1}(0) 
\longrightarrow \CC$ ($P, Q \in 
\Gamma(X; \SO_X), 
Q \not= 0$) on $X$ such that we have an isomorphism 
\begin{equation}
\SM \simeq 
\SN \Dotimes \SE^f_{U|X}. 
\end{equation}
\end{definition}

Let $\SM \in\Modhol(\SD_X)$ be an 
exponentially twisted holonomic $\SD_X$-module such that 
$\SM \simeq 
\SN \Dotimes \SE^f_{U|X}$ 
for a regular holonomic 
$\SD_X$-module $\SN \in\Modrh(\SD_X)$ and a rational function 
$f= \frac{P}{Q}: U= X \setminus Q^{-1}(0) 
\longrightarrow \CC$ on $X$.  Let us set 
\begin{equation}
K:= Sol_X ( \SN ) \qquad \in \ \Dbc(X^{\an}). 
\end{equation}
We denote the support of $K \in \Dbc(X^{\an})$ 
by $Z \subset X$. Note that $Z= {\rm supp} ( \SN )$ 
and if $Z \subset Q^{-1}(0)$ then we have 
$\SM \simeq 0$. Hence in what follows, we assume 
that $Z$ is not contained in $Q^{-1}(0)$. 
Then in this algebraic case, 
considering $f$ as a rational function on 
a (possibly singular) compactification 
$\overline{Z}$ of $Z$ and using the enhanced 
solution complex of $\SM$, we see that the 
restriction of $f$ to $Z$ is uniquely 
determined by $\SM$ modulo constant functions 
on $Z$. Note also that the shift $K[N] \in \Dbc(X^{\an})$ of $K$ 
is a perverse sheaf on $X^{\an}$. Then by 
Theorem \ref{thm-4} (vi) 
and Corollary \ref{cor-41} for the enhanced sheaf 
$F:= \pi^{-1}  K \otimes 
{\rm E}_{U^{\an}| X^{\an}}^{{\rm Re}f}  
\in \BEC_+(\CC_{X^{\an}})$ 
on $X^{\an}$ we have 
an isomorphism 
\begin{equation}\label{eq-sit-new} 
Sol_{X}^\rmE( \SM ) \simeq 
\CC^\rmE_{X^{\an}} \Potimes F. 
\end{equation}
From now on, we shall study 
the enhanced micro-support ${\rm SS}^{{\rm E}}(F) \subset 
(T^*X_{\RR}) \times \RR$ and the irregular one 
${\rm SS}_{{\rm irr}}(F) 
\subset T^*X_{\RR}$ of the enhanced sheaf 
$F= \pi^{-1}  K \otimes 
{\rm E}_{U^{\an}| X^{\an}}^{{\rm Re}f}  
\in \BEC_+(\CC_{X^{\an}})$. 
It turns out that 
over the open subset 
$T^*U_{\RR} \subset T^*X_{\RR}$ the structures 
of these two micro-supports are very simple. 
Define a (not necessarily $\CC^*$-conic) 
complex Lagrangian submanifold $\Lambda^f$ of 
$T^*U$ by 
\begin{equation}
\Lambda^f := 
\{ (z, df(z)) \ | \ z \in U=X \setminus Q^{-1}(0) \} 
\qquad \subset T^*U. 
\end{equation}
Then we can easily show that 
via the natural identification $(T^*U)_{\RR} 
\simeq T^*U_{\RR}$ we have 
\begin{equation}\label{inclss} 
{\rm SS}_{{\rm irr}}(F) \cap T^*U_{\RR}
= \Bigl( {\rm SS}(K) \cap T^*U_{\RR} \Bigr) + \Lambda^f
\end{equation}
and 
\begin{align*}
& {\rm SS}^{{\rm E}}(F) \cap \Bigl( (T^*U_{\RR}) \times \RR \Bigr)  
\\ 
& = \Bigl\{ (x,y;x^*,y^*, -{\rm Re }f(x+iy) \ | \ 
(x,y;x^*,y^*) \in \Bigl( {\rm SS}(K) \cap T^*U_{\RR} \Bigr) 
+ \Lambda^f \Bigr\}. 
\end{align*}
Moreover we have the following results. 

\begin{lemma}\label{empty-pole-micro}
We denote the preimage of the locally closed subset 
$Q^{-1}(0) \setminus I(f)= 
Q^{-1}(0) \setminus P^{-1}(0) \subset X$ 
by the projection $(T^*X_{\RR}) \times \RR 
\longrightarrow X_{\RR}$ simply by 
$\{ P \not= 0, \ Q=0 \} \subset (T^*X_{\RR}) \times \RR$. Then 
we have 
\begin{equation}
{\rm SS}^{{\rm E}}(F) \cap \{ P \not= 0, \ Q=0 \}
= \emptyset. 
\end{equation}
In particular, if we assume 
that $I(f)=P^{-1}(0) \cap Q^{-1}(0) = \emptyset$, 
then we have 
\begin{equation}
{\rm SS}_{{\rm irr}}(F) 
= \Bigl( {\rm SS}(K) \cap T^*U_{\RR} \Bigr) + \Lambda^f. 
\end{equation}
\end{lemma}

\begin{proof}
The problem being local, after shrinking $X$ we may 
assume that $I(f)= \emptyset$. 
First, we consider the special case where $
X= \CC_z$, $K= \CC_X$ and $P(z)=1$, $Q(z)=z$, 
$f(z)= \frac{1}{z}$. In this case, we have 
\begin{equation}
F= \CC_{ \{ z \not= 0, \ t+ {\rm Re}( \frac{1}{z} ) \geq 0 \} }. 
\end{equation}
Then it suffices to show that for any point 
$(0,t_0) \in (Q^{-1}(0) \setminus I(f)) \times \RR = 
Q^{-1}(0)  \times \RR$ there exists no covector 
of the form $\xi dx + dt$ in $T^*_{(0,t_0)}(X_{\RR} \times \RR ) 
\cap {\rm SS}(F)$. For such a point $(0,t_0)$ 
we can easily see that there exist its 
neighborhoods $W, W^{\prime}$ in $X_{\RR} \times \RR 
\simeq \RR^3$ and a diffeomorphism 
$\Phi : W \simto W^{\prime}$ between them such that 
$\Phi ((0,t_0))=(0,t_0)$, 
\begin{equation}
\Phi \Bigl( \{ z \not= 0, \ t+ {\rm Re} \Bigl( \frac{1}{z} \Bigr) \geq 0 \} \cap W \Bigr)
= \{ z \not= 0, \ t_0+ {\rm Re} \Bigl( \frac{1}{z} \Bigr) \geq 0 \} \cap W^{\prime}
\end{equation}
and the tangent map 
\begin{equation}
T_{(0,t_0)} \Phi : T_{(0,t_0)}(X_{\RR} \times \RR )
\longrightarrow 
T_{(0,t_0)}(X_{\RR} \times \RR )
\end{equation}
of $\Phi$ at it is the identity. Then locally 
we can replace $F$ by the sheaf 
\begin{equation}
F^{\prime}= \CC_{ \{ z \not= 0, \ t_0+ {\rm Re} ( \frac{1}{z} ) \geq 0 \} } 
\end{equation}
to check the condition. The same argument can be applied 
even if we replace $f(x)$ by $\frac{1}{z^m}$ ($m \geq 1$). 
Next consider the case where 
$X= \CC^N_z$, $K= \CC_X$ and $P(z)=1$, $Q(z)=z_1^{m_1} z_2^{m_2} 
\cdots z_k^{m_k}$  ($m_i \geq 1$), 
$f(z)= \frac{1}{Q(z)}$ for some $1 \leq k \leq N= {\rm dim} X$. 
In this case, for the sheaf  
\begin{equation}
F= \CC_{ \{ Q(z) \not= 0, \ t+ {\rm Re}f(z) \geq 0 \} } 
\end{equation}
on $X_{\RR} \times \RR$ 
we can similarly check the condition. 
Let us now consider the more general case where 
$K= \CC_X$ and $P(z)=1$, $Q(z) \not= 0$ 
$f(z)= \frac{1}{Q(z)}$. 
Let $\nu : \widetilde{X} \longrightarrow X$ be a 
proper morphism of complex manifolds such that 
the restriction $\widetilde{X} \setminus 
\nu^{-1} Q^{-1}(0) 
 \longrightarrow X \setminus Q^{-1}(0)$ of $\nu$ 
is an isomorphism and $\nu^{-1} Q^{-1}(0) \subset \widetilde{X}$ is 
a normal crossing divisor in $\widetilde{X}$. Then 
for the sheaf $F$ on $X_{\RR} \times \RR$ and 
the morphism $\widetilde{\nu}:= \nu \times {\rm id}_{\RR}: 
\widetilde{X}_{\RR} \times \RR \longrightarrow X_{\RR} \times \RR$ 
we have the condition 
\begin{equation}
{\rm SS}^{{\rm E}}( \widetilde{\nu}^{-1} F) \cap 
\{ P \circ \nu \not= 0, \ Q  \circ \nu =0 \}
= \emptyset
\end{equation}
and there exists an isomorphism 
\begin{equation}
F \simto 
\rmR \widetilde{\nu}_* \widetilde{\nu}^{-1} F. 
\end{equation}
Hence we can check the condition on $F$ by \cite[Proposition 5.4.4]{KS90}. 
Finally, we consider the general case. 
Let $\CS$ be a stratification of 
$Z$ adapted to $K$ such that $Q^{-1}(0) \cap Z$ 
is a union of some strata in it. Then by decomposing 
the support of $K$ with respect to $\CS$, it suffices 
to show that for any stratum $S \in \CS$ in $\CS$ 
such that $S \subset Z \setminus Q^{-1}(0)$ and any 
local system $L$ on $S$ we have the condition 
\begin{equation}\label{neqas} 
{\rm SS}^{{\rm E}} \Bigl(  
\pi^{-1} (i_{!} L) \otimes {\rm E}_{U| X}^{{\rm Re}f} \Bigr) \cap 
\{ P  \not= 0, \ Q  =0 \} = \emptyset, 
\end{equation}
where $i :S \hookrightarrow X$ is the inclusion map. 
Let $\nu : T \longrightarrow X$ be a 
proper morphism of complex manifolds such that 
$\nu (T)= \overline{S}$ and 
the restriction $T^{\circ}:= T \setminus 
\nu^{-1} ( \overline{S} \setminus S) 
 \longrightarrow S$ of $\nu$ 
is an isomorphism and $\nu^{-1} Q^{-1}(0), 
\nu^{-1} ( \overline{S} \setminus S)  \subset T$ are 
normal crossing divisors in $T$. We thus 
can identify $T^{\circ}$ with $S$ and regard 
$L$ as a local system on $T^{\circ}$. 
Let us consider the meromorphic function 
$f \circ \nu := ( P \circ \nu ) / 
( Q \circ \nu )$ on $T$ and the inclusion map 
$i^{\prime} : T^{\circ} \hookrightarrow T$. 
Then for the sheaf 
$\pi^{-1} ( i^{\prime}_! L) \otimes {\rm E}_{T^{\circ}| T}^{{\rm Re}(f \circ \nu )}$ 
on $T_{\RR} \times \RR$ 
we can similarly show the condition 
\begin{equation}
{\rm SS}^{{\rm E}} \Bigl( 
\pi^{-1} ( i^{\prime}_! L) \otimes {\rm E}_{T^{\circ}| T}^{{\rm Re}(f \circ \nu )}
\Bigr) \cap 
\{ P \circ \nu \not= 0, \ Q  \circ \nu =0 \}
= \emptyset. 
\end{equation}
Moreover, for the morphism 
$\widetilde{\nu}:= \nu \times {\rm id}_{\RR}: 
T_{\RR} \times \RR \longrightarrow X_{\RR} \times \RR$ 
there exists an isomorphism 
\begin{equation}
\pi^{-1} (i_{!} L) \otimes {\rm E}_{U| X}^{{\rm Re}f}
\simto 
\rmR \widetilde{\nu}_*
\Bigl(
\pi^{-1} ( i^{\prime}_! L) \otimes {\rm E}_{T^{\circ}| T}^{{\rm Re}(f \circ \nu )}
\Bigr).
\end{equation}
Then we obtain \eqref{neqas} by \cite[Proposition 5.4.4]{KS90}. 
This completes the proof. 
\end{proof}

\begin{lemma}\label{empty-micro}
Assume that there exists a stratification $\CS$ of 
$Z$ adapted to $K$ such that any stratum $S \in \CS$ 
in it is not contained in $P^{-1}(0) \cap Z \subset Z$ nor 
$Q^{-1}(0)  \cap Z \subset Z$ and the meromorphic 
function $f|_{S \setminus Q^{-1}(0)} : 
S \setminus Q^{-1}(0) \longrightarrow \CC$ on $S$ 
satisfies the condition: For any point of 
$S \cap P^{-1}(0) \cap Q^{-1}(0)$ there exists 
a local coordinate $z_1,z_2, \ldots, z_l$ 
($l:= {\rm dim }S$) of $S$ around it such that 
$(P|_S)(z_1, \ldots, z_l)=z_1, \ (Q|_S)(z_1, \ldots, z_l)=z_2^m$ 
for some $m \geq 1$ and hence 
\begin{equation}
(f|_S)(z_1, \ldots, z_l)=
\frac{z_1}{z_2^m}. 
\end{equation}
We denote the preimage of the closed subset 
$I(f)= P^{-1}(0) \cap Q^{-1}(0) \subset X$ 
by the projection $(T^*X_{\RR}) \times \RR 
\longrightarrow X_{\RR}$ simply by 
$\{ P=Q=0 \} \subset (T^*X_{\RR}) \times \RR$. Then 
we have 
\begin{equation}
{\rm SS}^{{\rm E}}(F) \cap \{ P=Q=0 \}
= \emptyset. 
\end{equation}
\end{lemma}

\begin{proof}
The problem being local, we may assume that 
locally $X= \CC^N$. 
First, we consider the special case where $N=2$, $Z 
=X= \CC^2$, $K= \CC_X$ and $P(z_1,z_2)=z_1$, $Q(z_1,z_2)=z_2$, 
$f(z_1,z_2)= \frac{z_1}{z_2}$. We set $z=(z_1,z_2)=(x,y)$ 
so that we have $f(x,y)= \frac{x}{y}$. We shall show 
${\rm SS}^{{\rm E}}(F) \cap \{ x=y=0 \}
= \emptyset$. By an explicit calculation, we can easily 
show that the closure of the set 
\begin{equation}
{\rm SS}^{{\rm E}}(F) \setminus \{ x=y=0 \} 
\subset (T^*X_{\RR}) \times \RR \setminus \{ x=y=0 \} 
\end{equation}
in $(T^*X_{\RR}) \times \RR$ does not intersect 
$\{ x=y=0 \}$. Let $\nu : \widetilde{X} \longrightarrow X$ 
be the blow-up of $X= \CC^2$ along the origin 
$\{ (0,0) \} \subset X= \CC^2$ and 
$M= \CC^2_{u,y}$ the affine chart of $\widetilde{X}$ 
such that we have 
\begin{equation}
\nu |_M: M= \CC^2_{u,y} \longrightarrow X= \CC^2 \qquad 
((u,y) \longmapsto (uy,y)). 
\end{equation}
Then for the exceptional divisor $\nu^{-1}( \{ (0,0) \} ) 
\simeq \PP := \PP^1$ we have $\nu^{-1}( \{ (0,0) \} ) \cap 
M= \CC_u \times \{ 0 \} = \{ y=0 \} \subset M= \CC^2_{u,y}$ 
and $(f \circ \nu )(u,y)= \frac{uy}{y}=u$ on $M= \CC^2_{u,y}$. 
Let $M^{\prime} = \CC^2_{x,v}$ be the affine chart of $\widetilde{X}$ 
such that we have 
\begin{equation}
\nu |_{M^{\prime}}: M^{\prime} = \CC^2_{x,v} \longrightarrow X= \CC^2 \qquad 
((x,v) \longmapsto (x,vx)). 
\end{equation}
Then we have $M \cup M^{\prime} = \widetilde{X}$, 
$\nu^{-1}( \{ (0,0) \} ) \cap 
M^{\prime} =  \{ 0 \} \times \CC_v  = \{ x=0 \} \subset 
M^{\prime}= \CC^2_{x,v}$ 
and $(f \circ \nu )(x,v)= \frac{x}{vx}= 
\frac{1}{v}$ on $M^{\prime} = \CC^2_{x,v}$. 
Hence the meromorphic function $f \circ \nu$ 
on $\widetilde{X}$ has no point of indeterminacy 
on the whole $\widetilde{X}$ and its pole is 
contained in $\widetilde{X} \setminus M = 
\{ v=0 \} \subset M^{\prime} = \CC^2_{x,v}$. 
Now we define an enhanced sheaf $\widetilde{F}$ on 
$\widetilde{X}$ by 
\begin{equation}
\widetilde{F} :=  {\rm E}_{M \setminus \{ y=0 \}| 
\widetilde{X}}^{{\rm Re} (f \circ \nu )} 
\quad \in \BEC_+(\CC_{\widetilde{X}^{\an}})
\end{equation}
and set $\widetilde{\nu}:= \nu \times {\rm id}_{\RR}: 
\widetilde{X} \times \RR \longrightarrow X \times \RR$ 
so that there exists an isomorphism 
\begin{equation}
\rmR \widetilde{\nu}_* \widetilde{F} \simeq F. 
\end{equation}
Let 
\begin{equation}
T^*(X \times \RR )
\overset{\varpi}{\longleftarrow}
( \widetilde{X} \times \RR )
\times_{(X \times \RR )} T^*(X \times \RR ) 
\overset{\rho}{\longrightarrow}
T^*( \widetilde{X} \times \RR )
\end{equation}
be the natural morphisms associated to $\widetilde{\nu}$. 
Then by \cite[Proposition 5.4.4]{KS90} we have 
\begin{equation}\label{estim} 
{\rm SS} (F) 
\subset \varpi \rho^{-1} 
{\rm SS} ( \widetilde{F} ). 
\end{equation}
Let 
\begin{equation}
(0,0,t; \xi , \eta , 1)
\qquad \in \{ t^* >0 \} \subset T^*(X \times \RR )
\end{equation}
be a point which corresponds to the covector 
$\xi dx+ \eta dy+dt \in T^*_{(0,0,t)} (X \times \RR )$ 
at the point $(0,0,t) \in \{ P=Q=0 \} \subset 
X \times \RR$. 
Then its pull-back by the map 
\begin{equation}
\widetilde{\nu} |_{M \times \RR}: M \times \RR  
\longrightarrow X \times \RR \qquad 
((u,y,t) \longmapsto (uy,y,t))
\end{equation}
is calculated as 
\begin{equation}
\bigl( \widetilde{\nu} |_{M \times \RR} \bigr)^* 
( \xi dx+ \eta dy+dt ) 
=( \xi u+ \eta )dy+dt. 
\end{equation}
Since ${\rm Re }(f \circ \nu )(u,y)= {\rm Re }u$ on $M= \CC^2_{u,y}$ 
and hence together with Lemma \ref{empty-pole-micro}  
there is no covector of the form 
$( \xi u+ \eta )dy+dt$ in 
$\{ t^*=1 \} \cap {\rm SS} ( \widetilde{F} )$ 
over the exceptional divisor $\nu^{-1}( \{ (0,0) \} )$, 
by \eqref{estim} we obtain the assertion ${\rm SS}^{{\rm E}}(F) \cap \{ x=y=0 \}
= \emptyset$. We can treat also the case 
where $N=2$, $Z 
=X= \CC^2$, $K= \CC_X$ and $P(z_1,z_2)=z_1$, $Q(z_1,z_2)=z_2^m$, 
$f(z_1,z_2)= \frac{z_1}{z_2^m}$ for some 
$m \geq 2$ by repeating blow-ups (see e.g. 
the proof of \cite[Theorem 3.6]{MT13}). 
The general case can be treated similarly by decomposing 
the support of $K$ with respect to the stratification 
$\CS$. This completes the proof. 
\end{proof}

By the natural identification $(T^*X)_{\RR} \simeq 
T^*X_{\RR}$ we regard $T^*X_{\RR}$ as a complex symplectic 
manifold endowed with the canonical 
holomorphic symplectic 2-form 
$\sigma_X$. 

\begin{lemma}\label{iso-analy}
The irregular micro-support 
${\rm SS}_{{\rm irr}}(F)
\subset T^*X_{\RR}$ of 
$F= \pi^{-1}K \otimes {\rm E}_{U|Z}^{{\rm Re}f} \in 
\BEC_+(\CC_{X^{\an}})$ is contained in a complex 
isotropic analytic subset of 
$(T^*X)_{\RR} \simeq T^*X_{\RR}$. 
\end{lemma}

\begin{proof}
Let us consider the exact sequence 
\begin{equation}
0 \longrightarrow F_{\pi^{-1}(X \setminus P^{-1}(0))} 
\longrightarrow  F
 \longrightarrow  F_{\pi^{-1}(P^{-1}(0))} 
\longrightarrow 0
\end{equation}
of enhanced sheaves on $X$. Since there exists an isomorphism 
\begin{equation}
F_{\pi^{-1}(P^{-1}(0))} 
\simeq 
\pi^{-1} K \otimes 
\CC_{ \{ (z,t) \in U \times \RR \ | \ z \in U \cap P^{-1}(0), 
\ t \geq 0 \} }, 
\end{equation}
we see that ${\rm SS}_{{\rm irr}}
(F_{\pi^{-1}(P^{-1}(0))}) = {\rm SS} 
(K_{U \cap P^{-1}(0)})$ is a Lagrangian 
analytic subset of $(T^*X)_{\RR} \simeq T^*X_{\RR}$. 
Let $\nu : \widetilde{X} \longrightarrow X$ be a 
proper morphism of complex manifolds such that 
the restriction $\widetilde{X} \setminus 
\nu^{-1} ( P^{-1}(0) \cup Q^{-1}(0) ) 
 \longrightarrow X \setminus  
 ( P^{-1}(0) \cup Q^{-1}(0) )$ of $\nu$ 
is an isomorphism, $\nu^{-1} ( P^{-1}(0) 
\cup Q^{-1}(0) ) \subset \widetilde{X}$ is 
a normal crossing divisor, and the meromorphic function 
$f \circ \nu := ( P \circ \nu ) / 
( Q \circ \nu )$ on $\widetilde{X}$ has 
no point of indeterminacy on the whole $\widetilde{X}$ 
(see e.g. the proof of \cite[Theorem 3.6]{MT13}). 
We thus obtain a function $f \circ \nu : 
\nu^{-1}(U)= \widetilde{X} \setminus 
(Q \circ \nu )^{-1}(0) \longrightarrow \CC$. 
Moreover we define a complex Lagrangian submanifold 
$\Lambda^{f \circ \nu} \subset T^* \widetilde{X}$ by 
\begin{equation}
\Lambda^{f \circ \nu} := \{ (z,d(f \circ \nu)(z)) \ | \ z \in 
\nu^{-1}(U) \} 
\quad \subset T^* \widetilde{X}
\end{equation}
and set $\widetilde{\nu}:= \nu \times {\rm id}_{\RR}: 
\widetilde{X} \times \RR \longrightarrow X \times \RR$. 
Then for the enhanced sheaf $\widetilde{\nu}^{-1} 
F_{\pi^{-1}(X \setminus P^{-1}(0))}$ on 
$\widetilde{X}$ we can easily show that 
\begin{equation}
{\rm SS}_{{\rm irr}}
\bigl( \widetilde{\nu}^{-1} 
F_{\pi^{-1}(X \setminus P^{-1}(0))} 
\bigr)
= {\rm SS} \bigl( \nu^{-1} K_{U \setminus P^{-1}(0)} \bigr) 
+ \Lambda^{f \circ \nu} 
\end{equation}
in $T^* \widetilde{X}$. This implies that ${\rm SS}_{{\rm irr}}
\bigl( \widetilde{\nu}^{-1} 
F_{\pi^{-1}(X \setminus P^{-1}(0))} 
\bigr)$ is a Lagrangian 
analytic subset of $T^* \widetilde{X}$. Let 
\begin{equation}
T^*X
\overset{\varpi}{\longleftarrow}
\widetilde{X} \times_X T^*X 
\overset{\rho}{\longrightarrow}
T^* \widetilde{X}
\end{equation}
be the natural morphisms associated to 
$\nu : \widetilde{X} \longrightarrow X$. 
Since $\nu$ is proper and there exists 
an isomorphism 
\begin{equation}
F_{\pi^{-1}(X \setminus P^{-1}(0))} \simto 
\rmR \widetilde{\nu}_* \widetilde{\nu}^{-1} 
(F_{\pi^{-1}(X \setminus P^{-1}(0))}), 
\end{equation}
by \cite[Proposition 5.4.4]{KS90} we have 
\begin{equation}
{\rm SS}_{{\rm irr}}
(F_{\pi^{-1}(X \setminus P^{-1}(0))}) 
\subset \varpi \rho^{-1} 
{\rm SS}_{{\rm irr}}
\bigl( \widetilde{\nu}^{-1} 
F_{\pi^{-1}(X \setminus P^{-1}(0))} 
\bigr). 
\end{equation}
It is clear that the right hand side is an analytic 
subset of $T^*X$. Moreover, by 
\cite[page 43]{GM88} there exist some 
stratifications of it and $\rho^{-1} 
{\rm SS}_{{\rm irr}}
\bigl( \widetilde{\nu}^{-1} 
F_{\pi^{-1}(X \setminus P^{-1}(0))} 
\bigr)$ such that the morphism 
\begin{equation}
\rho^{-1} 
{\rm SS}_{{\rm irr}}
\bigl( \widetilde{\nu}^{-1} 
F_{\pi^{-1}(X \setminus P^{-1}(0))} 
\bigr)
\longrightarrow 
 \varpi \rho^{-1} 
{\rm SS}_{{\rm irr}}
\bigl( \widetilde{\nu}^{-1} 
F_{\pi^{-1}(X \setminus P^{-1}(0))} 
\bigr)
\end{equation}
induced by $ \varpi$ is a 
stratified fiber bundle over 
each stratum. Then we can slightly modify the proof 
of \cite[Proposition A.54]{K-book} to show that 
the symplectic 2-form 
$\sigma_X$ vanishes on 
$ \varpi \rho^{-1} 
{\rm SS}_{{\rm irr}}
\bigl( \widetilde{\nu}^{-1} 
F_{\pi^{-1}(X \setminus P^{-1}(0))} 
\bigr)$. We thus 
obtain the assertion. 
\end{proof}

In the special case where $X$ is 
$\CC^N_z$, $Y$ is its 
dual vector space $\CC^N_w$ and 
$T^*X \simeq X \times Y$, we shall 
introduce a new subset of $(T^*X)_{\RR}
\simeq T^*X_{\RR}$ as follows. 
For a point $w_0 \in Y= \CC^N$ we define a 
rational function $f^{w_0}$ on $X= \CC^N$ by 
\begin{equation}
f^{w_0} : U=X \setminus Q^{-1}(0) \longrightarrow \CC \qquad 
(z \longmapsto   \langle  z, w_0 \rangle -f(z)). 
\end{equation}

\begin{definition}\label{defi-13}
In the case where $X$ is 
$\CC^N_z$, $Y$ is its 
dual vector space $\CC^N_w$ and 
$T^*X \simeq X \times Y$, by using $K$ and $f$ in $F= \pi^{-1}K 
\otimes {\rm E}_{U^{\an}|X^{\an}}^{{\rm Re}f}$ 
we define a subset 
${\rm SS}_{{\rm eva}}(F)
\subset (T^*X)_{\RR}$ of $(T^*X)_{\RR}
\simeq T^*X_{\RR}$ by 
\begin{equation}
{\rm SS}_{{\rm eva}}(F):= \{ (z_0,w_0) \in (T^*X)_{\RR} 
\ | \ \phi^{\merc}_{f^{w_0}-c}(K)_{z_0} \not= 0 
\ \text{for some} \ c \in \CC \}. 
\end{equation}
\end{definition}
Unfortunately, we do not know if this set 
${\rm SS}_{{\rm eva}}(F)
\subset T^*X_{\RR}$ is contained in the 
irregular micro-support ${\rm SS}_{{\rm irr}}(F) 
$ of $F$ or not. With Lemmas \ref{empty-pole-micro}, 
\ref{empty-micro} and \ref{iso-analy} at hands, 
now we have obtained a rough picture of the irregular 
micro-support ${\rm SS}_{{\rm irr}}(F) \subset T^*X_{\RR}$. 
However, for the moment it is not clear for us if it 
is a complex Lagrangian analytic subset 
of $(T^*X)_{\RR} \simeq T^*X_{\RR}$ or not. 
For these reasons, we define a (not necessarily 
homogeneous) complex Lagrangian analytic 
subset of $(T^*X)_{\RR} \simeq T^*X_{\RR}$ 
in a different way 
as follows and use it instead of ${\rm SS}_{{\rm irr}}(F)$. 
We return to the case where $X$ is a smooth 
algebraic variety over $\CC$. 
First, as a complex analogue of the 
$\RR$-constructible sheaf 
$F= \pi^{-1}K 
\otimes {\rm E}_{U^{\an}|X^{\an}}^{{\rm Re}f}$ 
on $X_{\RR} \times \RR$, by applying 
the proof of Theorem \ref{the-1} (i) to 
the (not necessarily) closed embedding 
\begin{equation}
i_{-f}: U= X \setminus Q^{-1}(0) \hookrightarrow 
X \times \CC, \qquad (z \longmapsto (z,-f(z)))
\end{equation}
associated to the function 
$-f:U \longrightarrow \CC$, we obtain a complex 
constructible sheaf 
\begin{equation}
\SF := (i_{-f})_! (K|_U) \quad 
\in \Dbc (X \times \CC )
\end{equation}
on $X \times \CC$. 
Then for the morphism ${\rm id}_X \times {\rm Re}: X \times 
\CC \longrightarrow X \times \RR$ induced by 
the one ${\rm Re}: \CC \longrightarrow \RR$ 
($\tau \longmapsto {\rm Re} \tau$) 
we have an isomorphism 
\begin{equation}
F \simeq \CC_{\{ t \geq 0 \}} 
\Potimes 
\rmR ( {\rm id}_X \times {\rm Re} )_! \SF. 
\end{equation}
Note that $\SF [N]$ is a perverse 
sheaf and its micro-support 
${\rm SS} ( \SF [N])= {\rm SS} ( \SF )$ 
is a homogeneous complex Lagrangian 
analytic subset of $T^*(X \times \CC )$. 
Then as in the definitions of ${\rm SS}^{{\rm E}}(F)
\subset (T^*X_{\RR}) \times \RR$ 
and ${\rm SS}_{{\rm irr}}(F) \subset T^*X_{\RR}$, 
by forgetting the homogeneity of ${\rm SS} ( \SF )$ 
we define two subsets: 
\begin{equation}
{\rm SS}^{{\rm E}, \CC}( \SF )
\subset T^*X \times \CC, \qquad 
{\rm SS}_{{\rm irr}}^{\CC}( \SF ) \subset T^*X 
\end{equation}
in the following way. First, by the closed embedding
\begin{equation}
\iota : (T^*X) \times \CC \hookrightarrow 
T^*(X \times \CC ) \qquad 
(((z,w), \tau ) \longmapsto 
(z, \tau ; w, 1)) 
\end{equation}
we set 
\begin{equation}
{\rm SS}^{{\rm E}, \CC}( \SF ):= 
\iota^{-1} {\rm SS}( \SF ) \ \subset \ 
(T^*X) \times \CC. 
\end{equation}
It is clear that ${\rm SS}^{{\rm E}, \CC}( \SF )$ 
is a complex analytic subset of $(T^*X) \times \CC$. 
Let us call it the enhanced micro-support of 
$\SF \in \Dbc (X \times \CC )$. Then we 
define ${\rm SS}_{{\rm irr}}^{\CC}( \SF ) \subset T^*X$ 
to be the closure of the image of 
${\rm SS}^{{\rm E}, \CC}( \SF )$ by the 
projection $(T^*X) \times \CC \longrightarrow T^*X$. 
We call it the irregular micro-support of $\SF$. 
As the restriction of $f$ to $Z$ is uniquely 
determined by $\SM$ modulo constant functions 
on $Z$, we see that 
${\rm SS}_{{\rm irr}}^{\CC}( \SF )$ does 
not depend on the expression $\SN \Dotimes \SE^f_{U|X}$ 
of $\SM$. It is easy to see that the results  
analogous to Lemmas \ref{empty-pole-micro} and 
\ref{empty-micro} 
holds true also for ${\rm SS}_{{\rm irr}}^{\CC}( \SF ) \subset T^*X$. 

\begin{lemma}\label{iso-analy-new}
The irregular micro-support 
${\rm SS}_{{\rm irr}}^{\CC}( \SF ) \subset T^*X$ of 
$\SF \in \Dbc (X \times \CC )$ is a complex 
Lagrangian analytic subset of 
$(T^*X)_{\RR} \simeq T^*X_{\RR}$. If we assume moreover 
that $I(f)=P^{-1}(0) \cap Q^{-1}(0) = \emptyset$, 
then we have 
\begin{equation}
{\rm SS}_{{\rm irr}}^{\CC} ( \SF ) 
= \bigl( {\rm SS}(K) \cap T^*U_{\RR} \bigr) + \Lambda^f. 
\end{equation}
\end{lemma}

\begin{proof}
By the inclusion map $j: X \times \CC \hookrightarrow 
X \times \PP$ we obtain an algebraic constructible 
sheaf 
\begin{equation}
j_! \SF = j_! (i_{-f})_! (K|_U) \quad 
\in \Dbc (X \times \PP )
\end{equation}
on $X \times \PP$. Then by the theorem 
in \cite[page 43]{GM88} there exist Whitney 
stratifications $\CS$ and $\CS_0$ of 
$X \times \PP$ and $X$ respectively such that 
\begin{equation}
{\rm SS }( j_! \SF ) \subset \bigsqcup_{S \in \CS} 
T^*_S(X \times \PP ), 
\end{equation}
$X \times \CC$ is a union of some strata in $\CS$ 
and the projection $X \times \PP 
\longrightarrow X$ is a stratified fiber 
bundle as in its assertion. 
Note that for each stratum $S_0 \in \CS_0$ in $\CS_0$ 
there exist only finitely many strata 
$S \in \CS$ in $\CS$ projecting to it such 
that $S \subset X \times \CC$, 
${\rm dim }S= {\rm dim }S_0$ and 
the enhanced micro-support ${\rm SS}^{{\rm E}, \CC}( \SF )
\subset (T^*X) \times \CC$ of $\SF$ is 
determined only by the conormal bundles 
$T^*_S(X \times \PP )$ of such ``horizontal"  
strata $S \in \CS$ such that $S 
\subset X \times \CC$. Then we obtain the first assertion 
by simple calculations. We obtain also the second one 
as in the proof of Lemma \ref{empty-pole-micro}.  
\end{proof}

By the proof of Lemma \ref{iso-analy-new}, 
considering also the multiplicities of $\SF$ 
we obtain a Lagrangian cycle in 
$T^*X$ supported by 
the irregular micro-support 
${\rm SS}_{{\rm irr}}^{\CC}( \SF ) \subset T^*X$ 
of $\SF$. As it depends only on $\SM$, 
we call it the irregular characteristic 
cycle of $\SM$ and denote it by ${\rm CC}_{{\rm irr}}( \SM )$. 
Moreover it satisfies also the following 
functorial property. 
Let $\phi : X \longrightarrow Y$ be 
a proper morphism of smooth algebraic varieties and 
$\SM^{\prime} := 
{\bfD} \phi_\ast \SM \in \BDC_{\rm hol}(\SD_Y)$ 
the direct image of the exponentially 
twisted holonomic $\SD_X$-module $\SM$ by it. 
Then by Theorem \ref{thm-4} (iii) for the 
enhanced sheaf $F^{\prime} := \bfE \phi_* F[d_X-d_Y] 
\in \BEC_+(\CC_{Y^{\an}})$ on $Y^{\an}$ 
we have an isomorphism 
\begin{equation}
Sol_{Y}^\rmE( \SM^{\prime} ) \simeq 
\CC^\rmE_{Y^{\an}} \Potimes F^{\prime}. 
\end{equation}
Similarly to the case of $\SM$, 
also for $\SM^{\prime} = {\bfD} \phi_\ast \SM$ 
we can define a Lagrangian cycle in $T^*Y$ 
by using the complex constructible sheaf 
$\SF^{\prime} := 
\rmR ( \phi \times {\rm id}_{\CC})_* \SF [d_X-d_Y] \in 
\Dbc (Y \times \CC )$. Denote it by 
${\rm CC}_{{\rm irr}}( \SM^{\prime} )$ and let 
\begin{equation}
T^*Y
\overset{\varpi}{\longleftarrow}
X \times_Y T^*Y 
\overset{\rho}{\longrightarrow}
T^* X
\end{equation}
be the natural morphisms associated to 
$\phi : X \longrightarrow Y$. Then there 
exists an equality 
\begin{equation}
{\rm CC}_{{\rm irr}}( \SM^{\prime} )= 
\varpi_* \rho^* {\rm CC}_{{\rm irr}}( \SM )
\end{equation}
of Lagrangian cycles in $T^*Y$ (see \cite[Proposition 9.4.2]{KS90}). 
Note also that by the commutative diagram 
\begin{equation}
\begin{CD}
X \times \CC  
@>{{\rm id}_X \times {\rm Re}}>> 
X \times \RR 
\\
@V{\phi \times {\rm id}_{\CC}}VV   
@VV{\phi \times {\rm id}_{\RR}}V
\\
Y \times \CC  
 @>{{\rm id}_Y \times {\rm Re}}>>  
Y \times \RR  
\end{CD}
\end{equation}
we obtain an isomorphism 
\begin{equation}
F^{\prime}  \simeq \CC_{\{ t \geq 0 \}} 
\Potimes 
\rmR ( {\rm id}_Y \times {\rm Re} )_! 
\SF^{\prime}. 
\end{equation}
Also for a morphism $\psi : Y \longrightarrow X$ 
of smooth algebraic varieties and the inverse image 
$\SM^{\prime} := 
{\bfD} \psi^\ast \SM \in \BDC_{\rm hol}(\SD_Y)$ 
of $\SM$ by it, 
we can define a Lagrangian cycle 
${\rm CC}_{{\rm irr}}( \SM^{\prime} )$ in $T^*Y$ 
and obtain a formula which expresses it 
in terms of ${\rm CC}_{{\rm irr}}( \SM )$ 
under some non-characteristic condition. 
Note that in this case the inverse image 
$\SM^{\prime} = {\bfD} \psi^\ast \SM$ is 
also of exponentially twisted type. 

\begin{proposition}\label{lem-ess-phi-new} 
In the case where $X$ is 
$\CC^N_z$, $Y$ is its 
dual vector space $\CC^N_w$ and 
$T^*X \simeq X \times Y$, 
for a point $(z_0,w_0)$ of $T^*X \simeq X \times Y$ 
assume that there exists a complex number $c \in \CC$ such that 
\begin{equation}
\phi^{\merc}_{f^{w_0}-c}(K)_{z_0} \not= 0
\end{equation}
and set $\tau_0:= c-
\langle  z_0, w_0 \rangle \in \CC$. 
Then we have $(z_0,w_0, \tau_0) \in 
{\rm SS}^{{\rm E}, \CC}( \SF )$, equivalently 
$(z_0, \tau_0, w_0, 1) \in 
{\rm SS}( \SF )$. 
\end{proposition} 

\begin{proof}
Let us consider 
the (not necessarily) closed embedding 
\begin{equation}
i_{f^{w_0}}: U= X \setminus Q^{-1}(0) \hookrightarrow 
X \times \CC, \qquad (z \longmapsto (z,f^{w_0}(z)))
\end{equation}
associated to the rational function 
$f^{w_0} :U \longrightarrow \CC$. Then we obtain 
a non-vanishing 
\begin{equation}
\phi_{\tau - c} \bigl( 
(i_{f^{w_0}})_! (K|_U) \bigr)_{(z_0, c)} 
\simeq 
\phi^{\merc}_{f^{w_0}-c}(K)_{z_0} \not= 0. 
\end{equation}
By \cite[Proposition 8.6.3]{KS90}, 
this implies that the point 
$(z_0, c, 0, 1) \in T^*(X \times \CC )$ is 
contained in ${\rm SS } \bigl( 
(i_{f^{w_0}})_! (K|_U) \bigr) \subset T^*(X \times \CC )$. 
On the other hand, for the 
automorphism $T_{w_0}$ of $X \times \CC$ defined by 
\begin{equation}
T_{w_0}: X \times \CC \simto 
X \times \CC, \qquad ((z, \tau ) \longmapsto 
(z, \tau + \langle  z, w_0 \rangle  )), 
\end{equation}
we have an isomorphism 
\begin{equation}
(T_{w_0})_* ( \SF )= (T_{w_0})_* (i_{-f})_! (K|_U) 
\simeq (i_{f^{w_0}})_! (K|_U). 
\end{equation}
Then the assertion $(z_0, \tau_0, w_0, 1) \in 
{\rm SS}( \SF )$ immediately follows from it. 
\end{proof}

By Proposition \ref{lem-ess-phi-new} the subset 
${\rm SS}_{{\rm eva}}(F)
\subset (T^*X)_{\RR}$ is contained in the 
irregular micro-support 
${\rm SS}_{{\rm irr}}^{\CC}( \SF )$ 
of $\SF \in \Dbc (X \times \CC )$. 
It is also clear that we have 
\begin{equation}\label{incluss-new}
{\rm SS}_{{\rm irr}}^{\CC}( \SF ) \cap T^*U_{\RR}
= \bigl( {\rm SS}(K) \cap T^*U_{\RR} \bigr) + \Lambda^f. 
\end{equation}

\section{Fourier Transforms of Exponentially Twisted 
\\ Holonomic D-modules}\label{sec:9}

In this section, we study the Fourier transforms of 
exponentially twisted holonomic 
D-modules. We inherit the situation 
and the notations in 
Section \ref{sec:1}. Let
\begin{equation}
X\overset{p}{\longleftarrow}X\times 
Y\overset{q}{\longrightarrow}Y
\end{equation}
be the projections. 
Then by Katz-Laumon \cite{KL85}, for an 
algebraic holonomic $\SD_X$-module 
$\SM\in\Modhol(\SD_X)$ we have an isomorphism
\begin{equation}
\SM^\wedge\simeq
\bfD q_\ast(\bfD p^\ast\SM\Dotimes \SO_{X\times Y}
e^{- \langle z, w \rangle }),
\end{equation}
where $\bfD p^\ast, \bfD q_\ast, \Dotimes$ are
the operations for algebraic D-modules
and $\SO_{X\times Y}
e^{- \langle z, w \rangle }$ is the integral connection
of rank one on $X\times Y$ associated to the canonical
paring $\langle \cdot , \cdot \rangle : X\times Y \longrightarrow \CC$.
In particular the right hand side is concentrated in degree zero. 
Let $\var{X}\simeq\PP^N$ (resp. $\var{Y}\simeq\PP^N$)
be the projective compactification of $X$ (resp. $Y$). 
By the inclusion map $i_X : 
X=\CC^N\xhookrightarrow{\ \ \ }\var{X}=\PP^N$
we extend a holonomic $\SD_X$-module $\SM\in\Modhol(\SD_X)$
on $X$ to the one $\tl{\SM} := i_{X\ast}\SM\simeq\bfD i_{X\ast}\SM$ 
on $\var{X}$.
Denote by $\var{X}^{\an}$ the underlying complex manifold
of $\var{X}$ and define the analytification
$\tl{\SM}^{{\ }\an}\in\Modhol(\SD_{\var{X}^{\an}})$
of $\tl{\SM}$ by
$\tl{\SM}^{{\ }\an} := \SO_{\var{X}^{\an}}
\otimes_{\SO_{\var{X}}}\tl{\SM}$.
Then we set 
\begin{equation}
Sol_{\var X}^\rmE(\tl\SM) := 
Sol_{\var{X}^{\an}}^\rmE(\tl{\SM}^{{\ }\an})
\ \in \ \BEC(\I\CC_{\var{X}^{\an}}).
\end{equation}
Similarly for the Fourier transform $\SM^\wedge\in\Modhol(\SD_Y)$, 
by the inclusion map $i_Y : 
Y=\CC^N\xhookrightarrow{\ \ \ }\var{Y}=\PP^N$ 
we define $\tl{\SM^\wedge}$ and 
$Sol_{\var Y}^\rmE(\tl{\SM^\wedge}) 
 \in\BEC(\I\CC_{\var{Y}^{\an}})$. 
Let
\begin{equation}
\var{X}^{\an}\overset{\var{p}}{\longleftarrow}
\var{X}^{\an}\times\var{Y}^{\an}\overset{\var{q}}{\longrightarrow}
\var{Y}^{\an}
\end{equation}
be the projections. 
Then the following theorem is essentially due to
Kashiwara-Schapira \cite{KS16-2} and D'Agnolo-Kashiwara \cite{DK17}. 
For $F\in\BEC(\I\CC_{\var{X}^{\an}})$ we set
\begin{equation}
{}^\rmL F := \bfE\var{q}_{*}(\bfE\var{p}^{-1}F
\Potimes\EE_{X\times Y|\var{X}\times\var{Y}}^{
-{\Re} \langle  z, w \rangle }[N])
\ \in \ \BEC(\I\CC_{\var{Y}^{\an}} ) 
\end{equation}
(here we denote $X^{\an}\times Y^{\an}$ etc.
by $X\times Y$ etc. for short) and call 
it the Fourier-Sato (Fourier-Laplace) transform of $F$. 

\begin{theorem}\label{thm-1}
For $\SM\in\Mod_{\rm hol}(\SD_X)$ there exists an isomorphism
\begin{equation}
Sol_{\var Y}^\rmE(\tl{\SM^\wedge})
\simeq{}^\rmL Sol_{\var X}^\rmE(\tl{\SM}).
\end{equation}
\end{theorem}
By \cite[Lemma 2.5.1]{DHMS17} we can take an 
enhanced sheaf $F \in \BEC_+(\CC_{X^{\an}})$ on 
$X^{\an}$ such that $F \simeq \CC_{\{ t \geq 0 \}} 
\Potimes F$ and 
\begin{equation}\label{given}
Sol_{\var X}^\rmE(\tl{\SM}) \simeq 
\CC^\rmE_{\var{X}^{\an}} \Potimes \bfE i_{X!} F. 
\end{equation}
For an enhanced sheaf $G\in\BEC(\CC_{\var{X}^{\an}})$
on $\var{X}^{\an}$ we define its Fourier-Sato 
(Fourier-Laplace) transform
${}^\rmL G \in\BEC(\CC_{\var{Y}^{\an}})$ by
\begin{equation}
{}^\rmL G := \bfE\var{q}_{*}(\bfE\var{p}^{-1} G \Potimes
\rmE_{X\times Y|\var{X}\times\var{Y}}^{
-{\Re} \langle z, w \rangle}[N])
\ \in \ \BEC(\CC_{\var{Y}^{\an}}).
\end{equation}
Since by \cite[Proposition 4.7.17]{DK16} 
we have 
\begin{equation}\label{eq-expon} 
{}^\rmL(\CC^\rmE_{\var{X}^{\an}}\Potimes( \cdot ))
\simeq
\CC^\rmE_{\var{Y}^{\an}}\Potimes{}^\rmL( \cdot ), 
\end{equation}
by Theorem \ref{thm-1} and \eqref{given}  
we obtain an isomorphism 

\begin{equation}\label{eq-efl} 
Sol_{\var Y}^\rmE(\tl{\SM^\wedge})
\simeq  \CC^\rmE_{\var{Y}^{\an}} \Potimes
{}^\rmL ( \bfE i_{X!} F). 
\end{equation}
Hence it suffices to study the enhanced sheaf 
\begin{equation}\label{eq-expl} 
{}^\rmL ( \bfE i_{X!} F) 
\simeq 
\bfE i_{Y!} 
\bfE q_{!}(\bfE p^{-1} F \Potimes
\rmE_{X\times Y|X \times Y}^{
-{\Re} \langle z, w \rangle}[N]) \ \in \ 
\BEC(\CC_{\var{Y}^{\an}}) 
\end{equation}
on $\var{Y}^{\an}$. As was explained in 
\cite[Sections 2.6 and 7.3]{DK17}, 
we can know the geometric structure of 
$L(F):= \bfE q_{!}(\bfE p^{-1} F \Potimes
\rmE_{X\times Y|X \times Y}^{
-{\Re} \langle z, w \rangle}[N])  \in 
\BEC(\CC_{Y^{\an}})$
to some extent by Tamarkin's 
microlocal thery of sheaves in \cite{Tama08} as follows. 
Since we have 
$L(F) \in \BEC_+(\CC_{Y^{\an}})$ we can define 
${\rm SS}^{{\rm E}}(L(F)) 
\subset (T^*Y_{\RR}) \times \RR$ and 
${\rm SS}_{{\rm irr}}(L(F)) 
\subset T^*Y_{\RR}$ also for $L(F)$. 
Then we have the following theorem 
due to Tamarkin \cite{Tama08}. 
Let us consider the maps 
\begin{equation}
\widetilde{\chi} : (T^*X^{\an}) \times \RR \simto 
(T^*Y^{\an}) \times \RR, \quad 
((z,w),t)  \longmapsto ((w,-z), t+ 
{\rm Re} \langle z, w \rangle )
\end{equation}
and 
\begin{equation}
\chi : T^*X^{\an} \simto 
T^*Y^{\an}, \quad 
(z,w) \longmapsto (w,-z)). 
\end{equation}

\begin{theorem}[{\cite[Theorem 3.6]{Tama08}}]\label{tama-th}
For $F \in \BEC_+(\CC_{X^{\an}})$ and 
$L(F)= \bfE q_{!}(\bfE p^{-1} F \Potimes
\rmE_{X\times Y|X \times Y}^{
-{\Re} \langle z, w \rangle}[N]) \in \BEC_+(\CC_{Y^{\an}})$ 
we have 
\begin{equation}
{\rm SS}^{{\rm E}}(L(F))= \widetilde{\chi}
( {\rm SS}^{{\rm E}}(F) ), 
\quad 
{\rm SS}_{{\rm irr}}(L(F))= \chi 
( {\rm SS}_{{\rm irr}}(F) ). 
\end{equation}
\end{theorem}

For a rational function 
$f= \frac{P}{Q}: X \setminus Q^{-1}(0) 
\longrightarrow \CC$ ($P, Q \in 
\Gamma(X; \SO_X) \simeq \CC [z_1,z_2, \ldots, z_N]$, 
$Q \not= 0$) on $X= \CC^N_z$ we set 
$U:= X \setminus Q^{-1}(0)$ and 
define an 
exponential $\SD_X$-module $\SE^f_{U|X} 
\in\Modhol (\SD_X)$ as in the analytic case. 
Then by Theorem \ref{thm-4} (vi) 
for its extension 
$( \SE^f_{U|X})^{\sim} :=
 i_{X\ast} \SE^f_{U|X}
 \simeq\bfD i_{X\ast} \SE^f_{U|X}$
to $\var{X} = \PP^N$ we have an isomorphism 
\begin{equation}
Sol_{\var X}^\rmE( ( \SE^f_{U|X})^{\sim} ) \simeq 
\EE_{U^{\an}| \var{X}^{\an}}^{{\rm Re}f}. 
\end{equation}

From now on, we fix an 
exponentially twisted holonomic $\SD_X$-module 
$\SM \in\Modhol(\SD_X)$ such that 
$\SM \simeq 
\SN \Dotimes \SE^f_{U|X}$ 
for a regular holonomic 
$\SD_X$-module $\SN \in\Modrh(\SD_X)$ and a rational function 
$f= \frac{P}{Q}: U= X \setminus Q^{-1}(0) 
\longrightarrow \CC$ on $X$ and study 
the basic properties of its 
Fourier transform $\SM^\wedge$. Let us set 
\begin{equation}
K:= Sol_X ( \SN ) \in \Dbc(X^{\an}). 
\end{equation}
Note that the shift $K[N] \in \Dbc(X^{\an})$ of $K$ 
is a perverse sheaf on $X^{\an}$. 
We denote the support of $K \in \Dbc(X^{\an})$ 
by $Z \subset X$. Then by 
Theorem \ref{thm-4} (vi) 
and Corollary \ref{cor-41} there exists 
an isomorphism 
\begin{equation}\label{eq-sit} 
Sol_{\var X}^\rmE(\tl{\SM}) \simeq 
\bigl\{ \pi^{-1} (i_{X})_! K \bigr\} 
\otimes 
\EE_{U^{\an}| \var{X}^{\an}}^{{\rm Re}f}. 
\end{equation}
Moreover 
for the enhanced sheaf
\begin{equation}\label{eq-sit} 
\bigl\{ \pi^{-1} (i_{X})_! K \bigr\} 
\otimes 
{\rm E}_{U^{\an}| \var{X}^{\an}}^{{\rm Re}f}
\simeq 
 \bfE i_{X!} \Bigl\{  
\pi^{-1}K \otimes {\rm E}_{U^{\an}|X^{\an}
}^{{\rm Re}f}
\Bigr\} \qquad \in \BEC_+(\CC_{\var{X}^{\an}})
\end{equation}
on $\var{X}^{\an}$ we have an isomorphism 
\begin{equation}\label{eq-expo} 
Sol_{\var X}^\rmE(\tl{\SM}) \simeq 
\CC^\rmE_{\var{X}^{\an}} \Potimes
\Bigl\{ 
\bigl\{ \pi^{-1} (i_{X})_! K \bigr\} 
\otimes 
{\rm E}_{U^{\an}| \var{X}^{\an}}^{{\rm Re}f} \Bigr\}. 
\end{equation}  
It follows from Theorem \ref{thm-1} that there exists 
an isomorphism 

\begin{equation}\label{eq-expofl} 
Sol_{\var Y}^\rmE(\tl{\SM^\wedge})
\simeq  \CC^\rmE_{\var{Y}^{\an}} \Potimes
{}^\rmL \Bigl\{ 
\bigl\{ \pi^{-1} (i_{X})_! K \bigr\} 
\otimes 
{\rm E}_{U^{\an}| \var{X}^{\an}}^{{\rm Re}f}
 \Bigr\}. 
\end{equation}
Hence it suffices to study the enhanced sheaf 
\begin{equation}\label{eq-expoeis} 
{}^\rmL \Bigl\{ 
\bigl\{ \pi^{-1} (i_{X})_! K \bigr\} 
\otimes 
{\rm E}_{U^{\an}| \var{X}^{\an}}^{{\rm Re}f}
 \Bigr\}  \ \in \ \BEC(\CC_{\var{Y}^{\an}}). 
\end{equation}
From now on, for the enhanced sheaf 
\begin{equation}
F:= \pi^{-1}K \otimes {\rm E}_{U^{\an}|X^{\an}
}^{{\rm Re}f} \in \BEC_+(\CC_{X^{\an}})
\end{equation}
on $X_{\RR}$ we shall study 
$L(F)= \bfE q_{!}(\bfE p^{-1} F \Potimes
\rmE_{X\times Y|X \times Y}^{
-{\Re} \langle z, w \rangle}[N]) \in \BEC_+(\CC_{Y^{\an}})$. 
For this purpose, as a complex analogue of the 
$\RR$-constructible sheaf 
$F= \pi^{-1}K 
\otimes {\rm E}_{U^{\an}|X^{\an}}^{{\rm Re}f}$ 
on $X_{\RR} \times \RR$, by applying 
the proof of Theorem \ref{the-1} (i) to 
the (not necessarily) closed embedding 
\begin{equation}
i_{-f}: U= X \setminus Q^{-1}(0) \hookrightarrow 
X \times \CC, \qquad (z \longmapsto (z,-f(z)))
\end{equation}
associated to the rational function 
$-f:U \longrightarrow \CC$, we obtain a complex 
constructible sheaf 
\begin{equation}
\SF := (i_{-f})_! (K|_U) \quad 
\in \Dbc (X \times \CC )
\end{equation}
on $X \times \CC$. Note that $\SF [N]$ is a perverse 
sheaf and its micro-support 
${\rm SS} ( \SF [N])= {\rm SS} ( \SF )$ 
is a homogeneous complex Lagrangian 
analytic subset of $T^*(X \times \CC )$. 
Then by forgetting the homogeneity of ${\rm SS} ( \SF )$ 
we define two subsets: 
\begin{equation}
{\rm SS}^{{\rm E}, \CC}( \SF )
\subset T^*X \times \CC, \qquad 
{\rm SS}_{{\rm irr}}^{\CC}( \SF ) \subset T^*X
\end{equation}
(see Section \ref{sec:19} for the details). 
Recall that the irregular micro-support 
${\rm SS}_{{\rm irr}}^{\CC}( \SF )$ of $\SF \in \Dbc (X \times \CC )$ 
is a (not necessarily homogeneous) 
complex Lagrangian analytic subset of $T^*X$ and 
${\rm SS}_{{\rm eva}}(F)
\subset (T^*X)_{\RR}$ is contained in 
${\rm SS}_{{\rm irr}}^{\CC}( \SF )$. 
In what follows, we set $\Lambda := 
{\rm SS}_{{\rm irr}}^{\CC}( \SF ) \subset (T^*X)_{\RR}$. 
Recall that we have 
\begin{equation}\label{incluss}
\Lambda \cap  T^*U_{\RR} =
{\rm SS}_{{\rm irr}}^{\CC}( \SF ) \cap T^*U_{\RR}
= \bigl( {\rm SS}(K) \cap T^*U_{\RR} \bigr) + \Lambda^f. 
\end{equation}
Since $\Lambda$ is Lagrangian and 
${\rm dim } \Lambda = {\rm dim }X=N$, 
by e.g. the theorem in \cite[page 43]{GM88} we can 
easily prove the following lemma. 

\begin{lemma}\label{lem-fin-cover}
There exists a non-empty 
Zariski open subset $\Omega \subset Y=\CC_w^N$ 
of $Y$ such that the restriction $q^{-1}( \Omega ) \cap 
\Lambda \longrightarrow  \Omega$ of the projection 
$q : X\times Y \longrightarrow Y$ is 
an unramified finite covering and any connected component 
of the open subset $q^{-1}( \Omega ) \cap 
\Lambda \subset \Lambda$ is a fiber bundle over a 
complex submanifold of $X= \CC^N_z$ contained 
in $U=X \setminus Q^{-1}(0)$ or 
$I(f)=P^{-1}(0) \cap Q^{-1}(0)$. 
\end{lemma}

\begin{remark}
In the special case where $I(f)=P^{-1}(0) \cap Q^{-1}(0) 
= \emptyset$, we have $\Lambda = 
\bigl( {\rm SS}(K) \cap T^*U_{\RR} \bigr) + \Lambda^f$ 
by Lemma \ref{iso-analy-new} and hence there exists a 
(unique) maximal Zariski open subset 
$\Omega \subset Y= \CC^N$ satisfying the 
conditions in Lemma \ref{lem-fin-cover}. 
\end{remark}

Let $\Omega \subset Y=\CC_w^N$ be as in 
Lemma \ref{lem-fin-cover} and $V \subset \Omega$ 
a contractible open subset of it. 
Then for the decomposition 
\begin{equation}
q^{-1}(V)\cap \Lambda
= \Lambda_{V,1} 
\sqcup \Lambda_{V,2} \sqcup \cdots \cdots \sqcup \Lambda_{V,d}
\end{equation}
of $q^{-1}(V)\cap \Lambda$ into its connected components 
$\Lambda_{V,i} \subset \Lambda$ 
($1 \leq i \leq d$) 
the morphism $q|_{\Lambda}$ induces an isomorphism 
$\Lambda_{V,i} \simto V$ for any $1 \leq i \leq d$. 
Now let us consider the symplectic transformation 
of \cite{DK17}: 
\begin{equation}
\chi : T^*X \simto T^*Y \qquad ((z,w) \longmapsto (w,-z)), 
\end{equation}
where we used the natural identification $T^*Y \simeq Y \times X$. 
Then $\chi ( \Lambda_{V,i} ) \subset T^*Y$ is a 
Lagrangian submanifold of $T^*Y$ for which 
the natural 
projection $T^*Y \longrightarrow Y$ induces an isomorphism 
$\chi ( \Lambda_{V,i} ) \simto V$. It follows that 
there exists a holomorphic function $g_i$ on $V$ such that 
\begin{equation}
\chi ( \Lambda_{V,i})= 
\Lambda^{g_i}:= \{ (w,dg_i(w)) \ | \ w \in V \} 
\subset T^*Y. 
\end{equation}
Namely, if for a point $w \in V$ we denote by 
$\zeta^{(i)} (w)=(\zeta^{(i)}_1(w), \zeta^{(i)}_2(w), 
\ldots, \zeta^{(i)}_N(w)) \in X= \CC^N_z$ the 
unique point of $X$ such that $(\zeta^{(i)}(w), w) \in  \Lambda_{V,i}  
\subset \Lambda$, then we have 
\begin{equation}\label{eq-conrm}
\Bigl(
\frac{\partial g_i}{\partial w_1}(w), 
\frac{\partial g_i}{\partial w_2}(w), 
\ldots \ldots, 
\frac{\partial g_i}{\partial w_N}(w)
\Bigr) 
= - (\zeta^{(i)}_1(w), \zeta^{(i)}_2(w), \ldots, 
\zeta^{(i)}_N(w)). 
\end{equation}
More precisely, we have the following higher-dimensional 
analogue of \cite[Lemma-Definition 7.4.2]{DK17}. 
For $1 \leq i \leq d$ by our choice of $\Omega \subset Y= 
\CC^N$ in Lemma \ref{lem-fin-cover}, 
$Z_i:=p ( \Lambda_{V,i} ) \subset X= \CC^N_z$ is a 
complex submanifold of $X$. 
We renumber them so that 
for some $0 \leq r \leq d$ we have 
$Z_i \subset U=X \setminus Q^{-1}(0)$ (resp. 
$Z_i \subset I(f)=P^{-1}(0) \cap Q^{-1}(0)$) if $1 \leq i \leq r$ 
(resp. if $r+1 \leq i \leq d$). 

\begin{lemma}\label{lem-10} 
For any $1 \leq i \leq r$ there exists a unique 
holomorphic function $g_i$ on $V \subset \Omega$ 
such that $\chi ( \Lambda_{V,i} )= 
\Lambda^{g_i}$ and the equality 
\begin{equation}
g_i(w)=f(\zeta^{(i)}(w))- \langle \zeta^{(i)}(w), w \rangle 
\end{equation}
holds on $V$. 
\end{lemma}

\begin{proof}
By the condition 
$Z_i= p( \Lambda_{V,i} ) \subset U=X \setminus Q^{-1}(0)$ 
and \eqref{incluss},  
the subset $\Lambda_{V,i} - \Lambda^f \subset 
(T^*U)_{\RR}$ of $(T^*U)_{\RR} \simeq T^*U_{\RR}$ is 
contained in the $\CC^*$-conic Lagrangian analytic 
set ${\rm SS}(K) \cap (T^*U)_{\RR}$. 
By \cite[Lemma A.52]{K-book} this implies that 
we have $\Lambda_{V,i} \subset T^*_{Z_i}X + \Lambda^f$. 
Choose a holomorphic function $g_i$ on $V \subset \Omega$ 
such that $\chi ( \Lambda_{V,i} )= 
\Lambda^{g_i}$ and set 
\begin{equation}
k_i(w):= g_i(w)-f(\zeta^{(i)}(w))+ \langle 
\zeta^{(i)}(w), w \rangle 
\qquad (w \in V). 
\end{equation}
Then by using the conditions $\zeta^{(i)}(w) \in Z_i$ 
($w \in V$) and $(\zeta^{(i)}(w), w) \in 
\Lambda_{V,i} \subset T^*_{Z_i}X + \Lambda^f$ 
($w \in V$) 
we can easily show that the 
condition \eqref{eq-conrm} is equivalent to the one 
\begin{equation}
\frac{\partial k_i}{\partial w_j}(w)=0 
\qquad (1 \leq j \leq N, w \in V). 
\end{equation}
This implies that $k_i$ is constant on $V$. 
Then by substracting its value from $g_i$ 
we obtain the assertion. 
\end{proof}

\bigskip 
\par \indent

\begin{example}\label{ex:mero-1}
Consider the case where $N=2$ and set $z=(z_1,z_2)=(x,y)$ and 
$w=(w_1,w_2)=( \xi, \eta )$. Assume that 
$U= \{ (x,y) \in X= \CC^2 \ | \ x \not= 0 \} \subset X$, 
$K= \CC_X$ and 
the function $f:U \longrightarrow \CC$ is defined to be 
the ratio of $P(x,y)=y$ and $Q(x,y)=x$ as 
\begin{equation}
f(x,y)= \frac{P(x,y)}{Q(x,y)} = \frac{y}{x}. 
\end{equation}
Then by (an analogue of) Lemma \ref{empty-micro} we have 
\begin{equation}
\Lambda = {\rm SS }_{\rm irr}^{\CC}( \SF ) 
= \Lambda^f=  \Bigl\{ 
\bigl(x,y, - \frac{y}{x^2}, \frac{1}{x} \bigr) \ | \ 
(x,y) \in U \Bigr\} \subset T^*X \simeq X \times Y. 
\end{equation}
By solving the equation $\xi = - \frac{y}{x^2}$, 
$\eta =  \frac{1}{x}$, we find that $d=1$, 
$\Omega = \{ ( \xi , \eta ) \in Y= \CC^2 \ | \ 
\eta \not= 0 \} \subset Y$ and for 
$( \xi, \eta ) \in \Omega$ we have 
\begin{equation}
x= \zeta_1( \xi, \eta ):= \frac{1}{\eta}, \quad 
y= \zeta_2( \xi, \eta ):= - \frac{\xi}{\eta^2}. 
\end{equation}
We can also verify that the function 
\begin{equation}
g( \xi, \eta ):= f \bigl( \frac{1}{\eta}, - \frac{\xi}{\eta^2} \bigr) 
- \Bigl\langle 
\bigl( \frac{1}{\eta}, - \frac{\xi}{\eta^2} \bigr), ( \xi, \eta )
\Bigr\rangle = - \frac{\xi}{\eta}
\end{equation}
satisfies the condition $\chi ( \Lambda )= \Lambda^g$. 
\end{example} 

\bigskip 
\par \indent

\begin{example}\label{ex:mero-2}
Let $N=2$ and $z=(z_1,z_2)=(x,y)$, 
$w=(w_1,w_2)=( \xi, \eta )$ be as in Example \ref{ex:mero-1}. 
Assume that 
$U= \{ (x,y) \in X= \CC^2 \ | \ x \not= 0 \} \subset X$, 
$K= \CC_X$ and 
the function $f:U \longrightarrow \CC$ is defined to be 
the ratio of $P(x,y)=x-y^3$ and $Q(x,y)=x$ as 
\begin{equation}
f(x,y)= \frac{P(x,y)}{Q(x,y)} = \frac{x-y^3}{x}
=1- \frac{y^3}{x}. 
\end{equation}
Then we have 
\begin{equation}
\Lambda_0 := {\rm SS }_{\rm irr}^{\CC}( \SF ) 
\cap (T^*U)_{\RR} = \Bigl\{ 
\bigl(x,y, \frac{y^3}{x^2}, - \frac{3y^2}{x} \bigr) \ | \ 
(x,y) \in U \Bigr\} \subset T^*X \simeq X \times Y. 
\end{equation}
Let us set 
\begin{equation}
\Omega_0 := \{ ( \xi , \eta ) \in Y= \CC^2 \ | \ 
\xi \not= 0, \eta \not= 0 \} \quad \subset Y. 
\end{equation}
Then by solving the equation $\xi = \frac{y^3}{x^2}$, 
$\eta =  - \frac{3y^2}{x}$, we find that for 
$( \xi, \eta ) \in \Omega_0$ we have 
\begin{equation}
x= \zeta_1( \xi, \eta ):= - \frac{\eta^3}{27 \xi^2}, \quad 
y= \zeta_2( \xi, \eta ):= \frac{\eta^2}{9 \xi}. 
\end{equation}
We can also verify that the function 
\begin{equation}
g( \xi, \eta ):= f \bigl( - \frac{\eta^3}{27 \xi^2}, 
 \frac{\eta^2}{9 \xi} \bigr) 
- \Bigl\langle 
\bigl(  - \frac{\eta^3}{27 \xi^2}, 
 \frac{\eta^2}{9 \xi} \bigr), ( \xi, \eta )
\Bigr\rangle = 1- \frac{\eta^3}{27 \xi}
\end{equation}
on $\Omega_0$ satisfies the condition 
$\chi ( \Lambda_0 ) \cap q^{-1}(\Omega_0) = \Lambda^g$. 
\end{example}

\bigskip 
\par \indent
Recall that for the closed embedding 
\begin{equation}
\iota : (T^*X) \times \CC \hookrightarrow 
T^*(X \times \CC ) \qquad 
(((z,w), \tau ) \longmapsto 
(z, \tau , w, 1)) 
\end{equation}
we have 
\begin{equation}
{\rm SS}^{{\rm E}, \CC}( \SF )= 
\iota^{-1} {\rm SS}( \SF ). 
\end{equation}
Then for $r+1 \leq i \leq d$, by the proof of 
Lemma \ref{iso-analy-new} we obtain the following 
description of ${\rm SS}^{{\rm E}, \CC}( \SF ) 
\cap ( \Lambda_{V,i} \times \CC )$.
For $r+1 \leq i \leq d$ let 
$\rho_i: Z_i \times_X T^*X \longrightarrow T^*Z_i$ 
be the natural morphism associated to the 
inclusion map $Z_i \hookrightarrow X= \CC^N$. 

\begin{lemma}\label{iso-anaconti}
For any $r+1 \leq i \leq d$ there exist holomorphic functions 
$h_{ik}:Z_i \longrightarrow \CC$ 
($1 \leq k \leq n_i$) on $Z_i$ and a 
neighborhood $W$ of $\Lambda_{V,i}$ in $T^*X$ such that for 
the complex submanifolds 
\begin{equation}
\Gamma_k:= \{ (z, -h_{ik}(z)) \in X \times \CC \ | \ 
z \in Z_i \} \subset X \times \CC
\qquad (1 \leq k \leq n_i) 
\end{equation}
of $X \times \CC$ we have 
\begin{equation}\label{incl-ssec} 
{\rm SS}^{{\rm E}, \CC}( \SF ) =  
\bigcup_{k=1}^{n_i} 
\iota^{-1} \Bigl\{  
T^*_{\Gamma_k}(X \times \CC ) \Bigr\}  
\end{equation}
in the open subset $W \times \CC \subset (T^*X) \times \CC$ 
and 
\begin{equation}\label{incl-ssec-new}
\rho_i( \Lambda_{V,i}) = 
\Lambda^{h_{ik}} 
\qquad (1 \leq k \leq n_i), 
\end{equation}
where we set 
\begin{equation}
\Lambda^{h_{ik}} := 
\{ (z, dh_{ik}(z)) \ | \ z \in Z_i \} \subset T^*Z_i. 
\end{equation}
\end{lemma}

\begin{proof}
Note that as $\Lambda_{V,i} \simeq V$ 
the complex manifold $Z_i = p( \Lambda_{V,i} ) 
\subset I(f)$ 
is contractible. Let $\Lambda_i^{\circ} 
\subset q^{-1}( \Omega ) \cap \Lambda 
= q^{-1}( \Omega ) \cap {\rm SS}_{{\rm irr}}^{\CC}( \SF )$ 
be the (unique) connected component of 
$\Lambda^{\circ}:=q^{-1}( \Omega ) \cap \Lambda$ containing 
the connected set $\Lambda_{V,i} \simeq V$. 
Then by our choice of $\Omega \subset Y= 
\CC^N$ in Lemma \ref{lem-fin-cover}, the smooth variety 
$\Lambda_i^{\circ}$ is a fiber bundle over 
a smooth (quasi-affine) subvariety $S_i$ of $X= \CC^N$ 
contained in $I(f)$ such that $Z_i \subset 
S_i$. Note that even if $i \not= i^{\prime}$ 
we may have $\Lambda_i^{\circ} = 
\Lambda_{i^{\prime}}^{\circ}$. 
As the natural morphism 
$S_i \times_X T^*X \longrightarrow T^*S_i$ 
associated to the 
inclusion map $S_i \hookrightarrow X= \CC^N$ 
is an extension of $\rho_i$, we denote it 
also by $\rho_i$. Note that the closure 
$\overline{\Lambda_i^{\circ}}$ of $\Lambda_i^{\circ}$ 
in $T^*X$ is an irreducible component of $\Lambda$. 
Moreover by the proof of 
Lemma \ref{iso-analy-new}, 
there exists a ``horizontal" stratum $M_i 
\subset X \times \CC$ 
over a Zariski open subset of $S_i$ 
such that the closure of the smooth Lagrangian subvariety 
of $T^*X$ obtained by forgetting the homogeneity of 
the conormal bundle $T^*_{M_i}(X \times \CC ) 
\subset T^*(X \times \CC )$ is equal to 
$\overline{\Lambda_i^{\circ}}$. 
This implies that the subset 
$\overline{\Lambda_i^{\circ}} \cap (S_i \times_X T^*X)$ 
of $S_i \times_X T^*X$ 
is a union of some fibers of 
$\rho_i$. Namely there exists a Lagrangian 
subvariety $\Lambda_{S_i} \subset T^*S_i$ 
such that 
\begin{equation}\label{lag-eq} 
\overline{\Lambda_i^{\circ}} \cap (S_i \times_X T^*X) 
= \rho_i^{-1} \Lambda_{S_i}. 
\end{equation}
Recall that $\Lambda_i^{\circ}$ is smooth and a fiber bundle 
over $S_i$. Since $\Lambda_i^{\circ}$ is an open subset of 
$\rho_i^{-1} \Lambda_{S_i}$, for any point 
$(z_0,w_0) \in \Lambda_i^{\circ}$ there exists 
its neighborhood $W(z_0,w_0)$ in $S_i \times_X T^*X$ 
such that $\Lambda_i^{\circ}$ is a union of 
some fibers of $\rho_i$ in $W(z_0,w_0)$. This implies 
that the open subset $\Lambda_{S_i}^{\circ}:= 
\rho_i ( \Lambda_i^{\circ} )$ of $\Lambda_{S_i}$ 
is a complex Lagrangian submanifold of $T^*S_i$ and 
the morphism $\Lambda_i^{\circ} \longrightarrow 
\Lambda_{S_i}^{\circ}$ induced by $\rho_i$ is 
a submersion. As the structure morphism 
$\Lambda_i^{\circ} \longrightarrow S_i$ of 
the fiber bundle $\Lambda_i^{\circ}$ over $S_i$ 
is the composite of the morphisms 
\begin{equation}
\Lambda_i^{\circ} \longrightarrow \Lambda_{S_i}^{\circ}
\longrightarrow S_i, 
\end{equation}
the morphism $\Lambda_{S_i}^{\circ}
\longrightarrow S_i$ induced by $T^*S_i \longrightarrow S_i$ 
is also a submersion. Moreover by the dimensional reason it is a 
finite unramified covering. We denote its 
degree by $e_i \geq 1$. 
Restricting the covering $\Lambda_{S_i}^{\circ}
\longrightarrow S_i$ 
to the contractible open 
subset $Z_i \subset S_i$, we obtain a trivial 
covering $\Lambda_{Z_i} := 
Z_i \times_{S_i} \Lambda_{S_i}^{\circ} \longrightarrow 
Z_i$ of $Z_i$ of degree $e_i$. 
Since $\rho_i( \Lambda_{V,i}) \subset 
\Lambda_{Z_i}$ is connected, 
it is contained in only one connected component 
of $\Lambda_{Z_i}$. 
Together with the condition $p( \Lambda_{V,i})=Z_i$ 
we obtain $\rho_i( \Lambda_{V,i}) \simeq Z_i$. 
Then there exists a holomorphic function  
$h:Z_i \longrightarrow \CC$ on 
$Z_i$ such that 
\begin{equation}
\rho_i( \Lambda_{V,i})= \Lambda^h:=
 \{ (z,dh(z)) \ | \ z \in Z_i \}
\end{equation}
in $T^*Z_i$. Such $h$ is uniquely 
determined up to constant 
functions on $Z_i$. Since $\Lambda_{V,i}$ is 
contained in the smooth part of the Lagrangian 
subvariety $\Lambda = {\rm SS}_{{\rm irr}}^{\CC}( \SF )
\subset T^*X$, we then 
immediately obtain the assertion. 
This completes the proof. 
\end{proof}

Moreover by the condition \eqref{incl-ssec-new} 
we see also that for any 
$1 \leq k, k^{\prime} \leq n_i$ 
such that $k \not= k^{\prime}$ 
we have $\Lambda_{h_{ik}}= 
\Lambda_{h_{ik^{\prime}}}$ and hence 
$h_{ik}-h_{ik^{\prime}} :Z_i \longrightarrow \CC$ 
is a non-zero constant function on $Z_i$. 
Then as in the proof of Lemma \ref{lem-10} 
we obtain the following result. 

\begin{lemma}\label{lem-10-new} 
For any $r+1 \leq i \leq d$ and $1 \leq k \leq n_i$ 
there exists a unique 
holomorphic function $g_{ik}$ on $V \subset \Omega$ 
such that $\chi ( \Lambda_{V,i} )= 
\Lambda^{g_{ik}}$ and the equality 
\begin{equation}
g_{ik}(w)=
h_{ik}(\zeta^{(i)}(w))- \langle \zeta^{(i)}(w), w \rangle 
\end{equation}
holds on $V$. 
\end{lemma}

\bigskip 
\par \indent

\begin{example}\label{ex:mero-3}
We consider the situation in Example \ref{ex:mero-2} and 
inherit the notations there. 
Moreover we assume that $Z=X= \CC^2$ and $K= \CC_X$. 
For $w=( \xi, \eta ) \in \Omega_0$ the meromorphic 
function $f^w:X \setminus Q^{-1}(0)= 
\CC^2 \setminus \{ x=0 \} \longrightarrow \CC$ 
is written as  
\begin{equation}
f^w(x,y)= ( \xi x + \eta y)-f(x,y) 
= \frac{\xi x^2+ \eta xy+y^3-x}{x}. 
\end{equation}
For $t \in \CC$ we set also 
\begin{equation}
R_t^w(x,y):= ( \xi x^2+ \eta xy+y^3-x)-tx 
\end{equation}
so that we have 
\begin{equation}
f^w(x,y)-t=
\frac{R_t^w(x,y)}{Q(x,y)} \qquad 
(z=(x,y) \in U=X \setminus Q^{-1}(0))
\end{equation}
and 
\begin{equation}
\overline{(f^w)^{-1}(t)}= 
(R_t^w)^{-1}(0) \subset X= \CC^2. 
\end{equation}
Then for $t \not= -1$ the complex curve 
$(R_t^w)^{-1}(0)$ in $X= \CC^2$ is smooth 
on a neighborhood of the origin 
$\{ (0,0) \} =P^{-1}(0) \cap Q^{-1}(0)$. 
If $t=-1$, for any $w=( \xi, \eta ) \in \Omega_0$ 
the defining polynomial 
\begin{equation}
R_{-1}^w(x,y):= \xi x^2+ \eta xy+y^3
\end{equation}
of the curve $(R_t^w)^{-1}(0)$ 
is Newton non-degenerate at the origin 
$\{ (0,0) \}$ (see \cite[Section 3]{Take23} 
for the details). 
Then by applying the 
classical Kouchnirenko's theorem 
(see \cite{Kou76}) to it, 
we see that for any $w=( \xi, \eta ) \in \Omega$ 
its Milnor number at the origin is equal to $1$. 
Moreover in Example \ref{exe-2} we studied the 
Milnor fiber of the rational function 
\begin{equation}
f^w(x,y)-(-1)= \frac{R_{-1}^w(x,y)}{Q(x,y)} 
= \frac{\xi x^2+ \eta xy+y^3}{x} \qquad 
(z=(x,y) \in U=X \setminus Q^{-1}(0))
\end{equation}
precisely and obtained the isomorphisms 
\begin{equation}
H^j \phi^{\merc}_{f^w-(-1)}( K[2] )_{(0,0)}
\simeq 
\begin{cases}
\CC & (j=-1)\\
\\
0 & (\mbox{otherwise}). 
\end{cases}
\end{equation}
By Proposition \ref{lem-ess-phi-new} this implies that 
for the conormal bundle $T^*_{\{ (0,0) \}}X 
( \simeq \{ (0,0) \} \times Y ) \subset 
T^*X ( \simeq X \times Y)$ of 
the complex submanifold $\{ (0,0) \} \subset 
X$ the subset $(T^*_{\{ (0,0) \}}X ) \times 
\{ -1 \} \subset (T^*X) \times \CC$ is 
contained in ${\rm SS}^{{\rm E}, \CC}( \SF )$. 
We thus obtain 
\begin{equation}
{\rm SS}^{\CC}_{{\rm irr}}( \SF )= 
\Bigl( T^*_{\{ (0,0) \}}X \Bigr) \cup \Lambda_0
\end{equation}
(for the definition of $\Lambda_0$ see 
Example \ref{ex:mero-2}). In particular, 
we can take $\Omega$ to be $\Omega_0$ 
in this case. 
\end{example} 

\bigskip 
\par \indent
For a point $w \in V \subset \Omega$ we define 
two regular functions $\ell^w : X= \CC^N 
\longrightarrow \CC$ and 
$f^w: U=X \setminus Q^{-1}(0) 
\longrightarrow \CC$ by 
\begin{equation}
\ell^w: X \longrightarrow \CC \qquad 
(z \longmapsto   \langle  z, w \rangle ) 
\end{equation}
and 
\begin{equation}
f^w = \ell^w|_U 
-f:U=X \setminus Q^{-1}(0) \longrightarrow \CC \qquad 
(z \longmapsto   \langle  z, w \rangle -f(z)) 
\end{equation}
respectively. 

\begin{lemma}\label{lem-nodeg}
For any  $1 \leq i \leq r$ and 
$w \in V \subset \Omega$ the holomorphic function 
$f^w|_{Z_i} : Z_i \longrightarrow \CC$ on $Z_i$ 
has a non-degenerate (complex Morse)
critical point at $\zeta^{(i)}(w) \in Z_i$. 
\end{lemma} 

\begin{proof}
In the proof of Lemma \ref{lem-10},  
for $1 \leq i \leq r$ we saw that 
$\Lambda_{V,i}$ is an open subset of 
$T^*_{Z_i}X + \Lambda^f \subset T^*U$. 
In other words, for the diffeomorphism 
\begin{equation}
\Phi : T^*U \simto T^*U \qquad 
((z,w) \longmapsto (z,w+df(z))
\end{equation}
of $T^*U$ we have the inclusion 
\begin{equation}
\Lambda_{V,i} \subset \Phi (T^*_{Z_i}X ). 
\end{equation}
Moreover, for $w \in V$ the condition 
$( \zeta^{(i)}(w), w) \in \Lambda_{V,i}$ implies 
\begin{equation}
 ( \zeta^{(i)}(w), df^w ( \zeta^{(i)}(w))) 
= ( \zeta^{(i)}(w), w-df( \zeta^{(i)}(w)))
\in T^*_{Z_i}X
\end{equation}
and hence the holomorphic function 
$f^w|_{Z_i} : Z_i \longrightarrow \CC$ 
has a critical point at 
$\zeta^{(i)}(w) \in Z_i$. 
By the choice of $\Omega$, for any $w \in V 
\subset \Omega$ the fiber 
$q^{-1}(w) \simeq X= \CC^N$ of $q: X \times Y 
\longrightarrow Y$ intersects $\Lambda_{V,i} \subset 
\Phi (T^*_{Z_i}X )$ transversally. 
Moreover, we have $q^{-1}(w)= \Lambda^{\ell^w}$. 
This implies that $\Phi^{-1}( \Lambda^{\ell^w} ) 
=  \Lambda^{\ell^w-f} = \Lambda^{f^w}$ intersects 
$\Phi^{-1}( \Lambda_{V,i}) \subset T^*_{Z_i}X$ 
transversally. Then by Kashiwara-Schapira 
\cite[Lemma 7.2.2]{KS85}, the holomorphic function 
$f^w|_{Z_i}=( \ell^w-f)|_{Z_i} : Z_i 
\longrightarrow \CC$ on $Z_i$ 
has a non-degenerate (complex Morse)
critical point at 
$\zeta^{(i)}(w) \in Z_i$.  
\end{proof}

As in the proof of Lemma \ref{lem-nodeg}, 
by Lemma \ref{iso-anaconti}  
we obtain also the following result. 
For $r+1 \leq i \leq d$, $1 \leq k \leq n_i$ and 
$w \in V \subset \Omega$ we 
define a holomorphic function 
$h_{ik}^w: Z_i \longrightarrow \CC$ on $Z_i$ by 
\begin{equation}
h_{ik}^w:= \ell^w|_{Z_i}-h_{ik} 
: Z_i \longrightarrow \CC, \qquad 
(z \longmapsto \langle z,  w \rangle - h_{ik}(z))
\end{equation}
so that we have 
\begin{equation}
g_{ik}(w)= h_{ik}(\zeta^{(i)}(w))
- \langle \zeta^{(i)}(w), w \rangle 
= -h_{ik}^w ( \zeta^{(i)}(w) ) 
\qquad (w \in V). 
\end{equation}

\begin{lemma}\label{lem-nodegen-new}
For any $r+1 \leq i \leq d$, $1 \leq k \leq n_i$ and 
$w \in V \subset \Omega$ the function 
$h_{ik}^w: Z_i \longrightarrow \CC$ on $Z_i$ 
has a non-degenerate (complex Morse)
critical point at $\zeta^{(i)}(w) \in Z_i$.
\end{lemma} 

Note that 
for $1 \leq i \leq r$ and $w \in V \subset \Omega$ 
the point 
\begin{equation}
( \zeta^{(i)}(w), df^w( \zeta^{(i)}(w)))= 
( \zeta^{(i)}(w), w- df( \zeta^{(i)}(w))) \in 
(T^*X)_{\RR} 
\end{equation}
is contained in the intersection of 
$T^*_{Z_i}X$ and 
the smooth part of ${\rm SS}(K|_U)
\subset (T^*U)_{\RR}$. 
We denote by $m(i) \geq 1$ the multiplicity of 
the regular holonomic 
$\SD_X$-module $\SN \in\Modrh(\SD_X)$ (or 
of the perverse sheaf $K[N]$) there. 

\begin{proposition}\label{prop-cevan1}
For any $1 \leq i \leq r$ and $w \in V \subset \Omega$ 
we have isomorphisms 
\begin{equation}
H^j \phi^{\merc}_{f^w-c}( K[N] )_{\zeta^{(i)}(w)}
\simeq H^j \phi_{f^w-c}( K[N] )_{\zeta^{(i)}(w)}
\simeq 
\begin{cases}
\CC^{m(i)} & (j=-1)\\
\\
0 & (\mbox{otherwise}), 
\end{cases}
\end{equation}
where we set $c:= f^w(\zeta^{(i)}(w))= 
\langle \zeta^{(i)}(w), w \rangle 
-f(\zeta^{(i)}(w)) \in \CC$. 
\end{proposition} 

\begin{proof}
Since for $1 \leq i \leq r$ we have 
$\zeta^{(i)}(w) \in Z_i \subset U=
X \setminus Q^{-1}(0)$ and hence 
the function $f^w-c$ is holomorphic 
on a neighborhood of $\zeta^{(i)}(w)$, 
there exists an isomorphism 
\begin{equation}
\phi^{\merc}_{f^w-c}( K[N] )_{\zeta^{(i)}(w)}
\simeq \phi_{f^w-c}( K[N] )_{\zeta^{(i)}(w)}. 
\end{equation}
Moreover by our definition of 
$m(i) \geq 1$ the perverse sheaf $K[N] \in \Dbc(X^{\an})$ 
is isomorphic to $\CC_{Z_i}^{\oplus m(i)} 
[ {\rm dim } Z_i] \in \Dbc(X^{\an})$ in the 
localized category 
$\Db (X^{\an}, \{ ( \zeta^{(i)}(w), df^w( \zeta^{(i)}(w))) \} )$ 
at the point $( \zeta^{(i)}(w), df^w( \zeta^{(i)}(w))) \in T^*X^{\an}$ 
(see \cite[Definition 6.1.1]{KS90} 
for the definition). Then by 
\cite[Proposition 8.6.3]{KS90} 
we obtain isomorphisms 
\begin{align*}
\phi_{f^w-c}( K[N] )_{\zeta^{(i)}(w)} 
& \simeq 
\phi_{f^w-f^w(\zeta^{(i)}(w))}( \CC_{Z_i}^{\oplus m(i)} 
[ {\rm dim } Z_i] )_{\zeta^{(i)}(w)} 
\\ 
& \simeq 
\phi_{f^w|_{Z_i}-f^w(\zeta^{(i)}(w))}( \CC_{Z_i}^{\oplus m(i)} 
[ {\rm dim } Z_i] )_{\zeta^{(i)}(w)}. 
\end{align*} 
Recall now that by Lemma \ref{lem-nodeg} the holomorphic function 
$f^w|_{Z_i} : Z_i \longrightarrow \CC$ on $Z_i$ 
has a non-degenerate (complex Morse)
critical point at $\zeta^{(i)}(w) \in Z_i$. 
Then the Milnor fiber at it is 
equal to one and hence we obtain the desired 
isomorphisms 
\begin{align*}
H^j \phi_{f^w-c}( K[N] )_{\zeta^{(i)}(w)} 
& \simeq 
H^j 
\phi_{f^w|_{Z_i}-f^w(\zeta^{(i)}(w))}( \CC_{Z_i}^{\oplus m(i)} 
[ {\rm dim } Z_i] )_{\zeta^{(i)}(w)}
\\ 
& \simeq 
\begin{cases}
\CC^{m(i)} & (j=-1)\\
\\
0 & (\mbox{otherwise})
\end{cases}
\end{align*}
(see e.g. 
\cite[Theorem 2.6]{Take23}). 
\end{proof}

Note that by \eqref{incl-ssec} 
for $r+1 \leq i \leq d$, $1 \leq k \leq n_i$ 
and $w \in V \subset \Omega$ 
the point
\begin{equation}
(\zeta^{(i)}(w), -h_{ik}(\zeta^{(i)}(w)), w, 1) 
\in T^*(X \times \CC )
\end{equation}
is contained in the smooth part of 
${\rm SS}( \SF [N] ) = {\rm SS}( \SF )
\subset T^*(X \times \CC )$. 
We denote by $m(i,k) \geq 1$ the multiplicity of 
of the perverse sheaf 
$\SF [N]= (i_{-f})_!(K[N]|_U) \in 
\Dbc (X^{\an} \times \CC )$ there. 

\begin{proposition}\label{prop-cevan2}
For any $r+1 \leq i \leq d$, $1 \leq k \leq n_i$ 
and $w \in V \subset \Omega$ 
we have isomorphisms 
\begin{equation}
H^j \phi^{\merc}_{f^w-c_k}( K[N] )_{\zeta^{(i)}(w)}
\simeq 
\begin{cases}
\CC^{m(i,k)} & (j=-1)\\
\\
0 & (\mbox{otherwise}), 
\end{cases}
\end{equation}
where we set $c_k:= h_{ik}^w(\zeta^{(i)}(w))= 
\langle \zeta^{(i)}(w), w \rangle 
-h_{ik}(\zeta^{(i)}(w)) \in \CC$. Moreover, 
for any complex number 
$c \in \CC$ satisfying the condition $c \not= c_k$ 
($1 \leq k \leq n_i$) we have a vanishing 
\begin{equation}
\phi^{\merc}_{f^w-c}( K[N] )_{\zeta^{(i)}(w)}
\simeq 0. 
\end{equation}
\end{proposition} 

\begin{proof}
Note that for the (not necessarily closed) embedding 
$i_{f^w}:U \hookrightarrow X \times \CC$ 
($z \longmapsto (z,f^w(z))$) and any $c \in \CC$ 
we have an isomorphism 
\begin{equation}
\phi^{\merc}_{f^w-c}( K[N] )_{\zeta^{(i)}(w)}
\simeq 
\phi_{\tau -c} \bigl( 
(i_{f^w})_!(K[N]|_U) 
\bigr)_{(\zeta^{(i)}(w), \ c)}. 
\end{equation}
As in the proof of Proposition \ref{lem-ess-phi-new}, 
let us consider the 
automorphism $T_{w}$ of $X \times \CC$ defined by 
\begin{equation}
T_{w}: X \times \CC \simto 
X \times \CC, \qquad ((z, \tau ) \longmapsto 
(z, \tau + \langle  z, w \rangle  )). 
\end{equation}
Then we have $i_{f^{w}}=T_w \circ i_{-f}$ and 
hence there exists an isomorphism 
\begin{equation}
(T_{w})_* ( \SF [N])= (T_{w})_* (i_{-f})_! (K[N]|_U) 
\simeq (i_{f^{w}})_! (K[N]|_U). 
\end{equation}
On the other hand, by \eqref{incl-ssec} 
and our definition of $m(i,k) \geq 1$ the perverse 
sheaf $\SF [N] \in \Dbc(X^{\an} \times \CC )$ 
is isomorphic to $\CC_{\Gamma_k}^{\oplus m(i,k)} 
[ {\rm dim } \Gamma_k] \in \Dbc(X^{\an} \times \CC )$ in the 
localized category 
$\Db (X^{\an} \times \CC , 
\{ ( \zeta^{(i)}(w), -h_{ik}( \zeta^{(i)}(w)), w,1) \} )$ 
at the point $( \zeta^{(i)}(w), -h_{ik}( \zeta^{(i)}(w)), w,1) 
\in T^*(X^{\an} \times \CC )$ 
(see \cite[Definition 6.1.1]{KS90}). This implies 
that for the closed embedding 
\begin{equation}
i_k : Z_i \hookrightarrow 
X \times \CC, \qquad (z \longmapsto 
(z, h_{ik}^w(z) ))
\end{equation}
we have an isomorphism 
\begin{equation}
 (i_{f^{w}})_! (K[N]|_U) \simeq 
(i_k)_* ( \CC_{Z_i}^{\oplus m(i,k)} 
[ {\rm dim } Z_i])
\end{equation}
in the localized category 
$\Db (X^{\an} \times \CC , 
\{ ( \zeta^{(i)}(w), c_k, 0,1) \} )$ 
at the point $( \zeta^{(i)}(w), c_k, 0,1)
\in T^*(X^{\an} \times \CC )$. Then by 
\cite[Proposition 8.6.3]{KS90} 
we obtain isomorphisms 
\begin{align*}
& \phi_{\tau -c_k} \bigl( 
(i_{f^w})_!(K[N]|_U) \bigr)_{(\zeta^{(i)}(w), \ c_k)}
\\
& \simeq 
\phi_{\tau - h_{ik}^w(\zeta^{(i)}(w))} \bigl( 
(i_k)_* ( \CC_{Z_i}^{\oplus m(i,k)} 
[ {\rm dim } Z_i])
\bigr)_{(\zeta^{(i)}(w), \ h_{ik}^w(\zeta^{(i)}(w)))}
\\
& \simeq 
\phi_{h_{ik}^w - h_{ik}^w(\zeta^{(i)}(w))} \bigl( 
\CC_{Z_i}^{\oplus m(i,k)} [ {\rm dim } Z_i]
\bigr)_{\zeta^{(i)}(w)}. 
\end{align*} 
Recall that by Lemma \ref{lem-nodegen-new} the holomorphic function 
$h_{ik}^w|_{Z_i} : Z_i \longrightarrow \CC$ on $Z_i$ 
has a non-degenerate (complex Morse)
critical point at $\zeta^{(i)}(w) \in Z_i$.  
Then the first assertion immediately follows 
from the standard fact that the Milnor number at it 
is equal to one (see e.g. 
\cite[Theorem 2.6]{Take23}). Similarly, 
we can show the second assertion. 
\end{proof}

By the proof of Proposition \ref{prop-cevan2} and 
\eqref{incl-ssec}, for any 
$r+1 \leq i \leq d$, 
$w \in V \subset \Omega$ and 
$c \in \CC$ 
such that $c \not= h_{ik}^w(\zeta^{(i)}(w))$ 
($1 \leq k \leq n_i$) we have also a vanishing 
\begin{equation}
\phi^{\merc}_{f^w-c}( K[N] )_{\zeta^{(i)}(w)}
\simeq 0. 
\end{equation}
Together with Proposition \ref{prop-cevan2}  
this implies that for any $r+1 \leq i \leq d$ 
the positive integer 
\begin{equation}
m(i):= \dsum_{k=1}^{n_i} m(i,k) \geq 1
\end{equation}
satisfies the condition 
\begin{equation}
m(i) = 
\dsum_{c \in \CC} {\rm dim } 
H^{N-1} \phi^{\merc}_{f^w -c}( K)_{\zeta^{(i)}(w)} 
\qquad (w \in V). 
\end{equation}

\begin{corollary}\label{equiv-2ss}
For the open subset $q^{-1}( \Omega )= 
X \times \Omega \subset X \times Y$ of 
$X \times Y \simeq T^*X$ we have 
\begin{equation}
{\rm SS}_{{\rm irr}}^{\CC}( \SF ) \cap q^{-1}( \Omega )
= {\rm SS}_{{\rm eva}}( F ) \cap q^{-1}( \Omega ). 
\end{equation}
\end{corollary} 

\begin{proof}
By Proposition \ref{lem-ess-phi-new} it suffices 
to prove only the inclusion 
\begin{equation}
\Lambda \cap q^{-1}( \Omega )= 
{\rm SS}_{{\rm irr}}^{\CC}( \SF ) \cap q^{-1}( \Omega )
\subset {\rm SS}_{{\rm eva}}( F ) \cap q^{-1}( \Omega ). 
\end{equation} 
For a point $(z_0,w_0) \in \Lambda \cap q^{-1}( \Omega )$ 
we take a contractible open subset 
$V \subset \Omega$ such that $q((z_0,w_0))=w_0 \in V$ 
and use the several notations that we introduced 
for it above. Then by Propositions \ref{prop-cevan1} 
and \ref{prop-cevan2} 
we obtain the desired condition 
$(z_0,w_0) \in {\rm SS}_{{\rm eva}}( F ) \cap q^{-1}( \Omega )$. 
\end{proof}

\begin{lemma}\label{lem-31} 
Let $\CS$ be a Whitney stratification of $Z \cap U 
\subset U=X \setminus Q^{-1}(0)$ adapted to 
$K|_U \in \Dbc ( U)$ such that 
\begin{equation}
{\rm SS }(K|_U) \subset \bigcup_{S \in \CS}T^*_SX 
\end{equation}
and $\CS_{{\rm red}} \subset \CS$ its subset consisting of 
strata $S \in \CS$ satisfying the condition 
$T^*_SX \subset {\rm SS }(K|_U)$ so that we have 
\begin{equation}
{\rm SS }(K|_U) = \bigcup_{S \in \CS_{{\rm red}}}
\overline{T^*_SX}. 
\end{equation}
Then for any $w \in V \subset \Omega$ 
and $S \in \CS_{{\rm red}}$ the 
holomorphic function $f^w|_S: S \longrightarrow 
\CC$ on $S$ is tame at infinity. 
\end{lemma}

\begin{proof}
Since by definition the morphisms $\zeta^{(i)}: V  
 \longrightarrow 
U=X \setminus Q^{-1}(0)$ ($1 \leq i \leq r$) are holomorphic, 
there exist $0< \varepsilon \ll 1$ and $R \gg 0$ 
such that for the open ball $B_{\varepsilon}(w) 
\subset V \subset \Omega$ (resp. $B_{R}(0) 
\subset X= \CC^N$) with radius 
$\varepsilon >0$ (resp. $R>0$) 
centered at $w \in V$ (resp. the origin $0 \in X= \CC^N$) 
we have the inclusion 
\begin{equation}
\zeta^{(i)}( B_{\varepsilon}(w) ) \subset U \cap B_{R}(0) 
\qquad (1 \leq i \leq r). 
\end{equation}
In other words, the subset 
\begin{equation}
\Lambda \cap p^{-1}(U) \cap q^{-1}(B_{\varepsilon}(w)) = 
\Bigl\{ {\rm SS}(K|_U)+ \Lambda^f \Bigr\} \cap q^{-1}(B_{\varepsilon}(w))
\end{equation}
of $p^{-1}(U) \cap q^{-1}(B_{\varepsilon}(w))$ is contained in 
$p^{-1}(U \cap B_{R}(0)) \cap q^{-1}(B_{\varepsilon}(w))$. 
This implies that for the stratum $S \in \CS_{{\rm red}}$ and 
any $u \in B_{\varepsilon}(0) \subset 
Y= \CC^N$ the linear perturbation of $f^w|_S$: 
\begin{equation}
(f^{w+u}|_S)(z)=(f^w|_S)(z)+ \langle z, u \rangle \qquad 
(z \in S) 
\end{equation}
has no critical point in the set 
\begin{equation}
S \setminus  B_{R}(0)= \{ 
z \in S \ | \ ||z|| \geq R \}. 
\end{equation}
Suppose now that there exists a point $z \in S \setminus  B_{R}(0)$ 
such that $||d {\rm Re} (f^w|_S)(z)|| < \varepsilon$. 
Then, the image of $B_{\varepsilon}(0) \subset 
\CC^N$ by the surjective linear map 
$\Psi_z: T^*_z X_{\RR} \simeq \CC^N \longrightarrow 
T^*_z S_{\RR}$ being also an open ball with 
radius $\varepsilon >0$, we have the 
equality $d {\rm Re} (f^w|_S)(z) 
= \Psi_z(u)$ for some $u \in B_{\varepsilon}(0)$. 
Moreover, for the holomorphic function 
$\ell^{\var{u}} :S \longrightarrow \CC$ on $S$ 
defined by 
\begin{equation}
\ell^{\var{u}}(z):= \langle z, \var{u} \rangle 
=(z,u) \qquad (z \in S) 
\end{equation}
we have $d {\rm Re} \ell^{\var{u}}(z) 
= \Psi_z(u)$ by \eqref{C-R-eqn}. 
This implies that the holomorphic 
function $(f^{w- \var{u}}|_S)(z)$ on $S$ satisfies 
the condition $d {\rm Re} (f^{w- \var{u}}|_S)(z)=0$ 
at the point $z \in S \setminus  B_{R}(0)$. 
Then by the Cauchy-Riemann equation, we obtain 
also $d (f^{w- \var{u}}|_S)(z)=0$. 
Since we have $- \var{u} \in B_{\varepsilon}(0)$, 
we get a contradiction. 
\end{proof}

Similarly, by the proofs of Lemma \ref{iso-anaconti}  
and Proposition \ref{prop-cevan2}  
we obtain the following result. For 
a point $w \in \Omega$ by the (not necessarily 
closed) embedding $i_{f^w}: U \hookrightarrow 
X \times \CC$ ($z \longmapsto (z,f^w(z))$) we set 
\begin{equation}
L^w:=(i_{f^w})_! (K|_U) \quad \in \Dbc (X \times \CC ). 
\end{equation}

\begin{lemma}\label{lem-31-new} 
For a point $w \in V \subset \Omega$ let $\CS$ and $\CS_0$ 
be Whitney stratifications of 
$X \times \CC$ and $X$ respectively such that 
\begin{equation}
{\rm SS }(L^w) \subset \bigcup_{S \in \CS}
T^*_S(X \times \CC ) 
\end{equation}
and the projection $X \times \CC \longrightarrow 
X$ is a stratified fiber bundle as in 
the assertion of the theorem in 
\cite[page 43]{GM88}. Then for any stratum $S \in \CS$ 
in $\CS$ such that $T^*_S(X \times \CC ) 
\subset {\rm SS }(L^w)$ 
the restriction $h|_S: S \longrightarrow \CC$ 
of the function $h: X \times \CC \longrightarrow \CC$ 
($(z, \tau ) \longmapsto \tau$) to $S \subset X \times \CC$ 
is relatively tame at infinity for the projection 
$X \times \CC \longrightarrow 
X$ in the sense of Definition \ref{def-7new}. 
\end{lemma}

With Theorem \ref{np-vc}, Lemmas \ref{lem-10}, 
\ref{lem-nodeg} and \ref{lem-31} 
and Proposition \ref{prop-cevan1} 
at hands, 
in the special case where $I(f)=P^{-1}(0) \cap Q^{-1}(0) 
= \emptyset$ ($\Longrightarrow r=d$) we obtain the following 
result as in the proof of \cite[Theorem 4.4]{IT20a} 
(see the proof of Theorem \ref{thm-2} below for the details). 

\begin{theorem}\label{thm-2-spe}
In the situation as above, assume also that 
$I(f)=P^{-1}(0) \cap Q^{-1}(0) 
= \emptyset$ so that we have $r=d$. 
Then we have an isomorphism
\begin{equation}
\pi^{-1}\CC_V \otimes
\big(Sol_{\var Y}^\rmE( \tl{\SM^\wedge} )\big)
\simeq
\bigoplus_{i=1}^d 
\bigl( \EE_{V^{\an}| \var{Y}^{\an}}^{{\rm Re}g_i}
\bigr)^{\oplus m(i)}
\end{equation}
of enhanced ind-sheaves on $\var{Y}^{\an}$. 
\end{theorem}

In the general case i.e. if we do not 
assume the condition $I(f)= \emptyset$, 
we have the following result. 
Note that by Lemma \ref{lem-10-new}  
for any $r+1 \leq i \leq d$ and 
$1 \leq k, k^{\prime} \leq n_i$ 
such that $k \not= k^{\prime}$ the difference 
$g_{ik}-g_{ik^{\prime}} :Z_i \longrightarrow \CC$ 
is a non-zero constant function on $Z_i$. 
Then for any 
$r+1 \leq i \leq d$ and $1 \leq k, k^{\prime} \leq n_i$ 
we have an isomorphism 
\begin{equation}\label{enha-isom} 
\EE_{V^{\an}| \var{Y}^{\an}}^{{\rm Re}g_{ik}}
\simeq 
\EE_{V^{\an}| \var{Y}^{\an}}^{{\rm Re}g_{ik^{\prime}}}
\end{equation}
of enhanced ind-sheaves on $\var{Y}^{\an}$. 
For this reason, for each $r+1 \leq i \leq d$ by 
choosing an index $1 \leq k \leq n_i$ we 
set $g_i(w):=g_{ik}(w)$ ($w \in V$) 
and consider the enhanced ind-sheaf 
$\EE_{V^{\an}| \var{Y}^{\an}}^{{\rm Re}g_i}$ 
in what follows. 

\begin{theorem}\label{thm-2}
In the situation as above, we have an isomorphism
\begin{equation}
\pi^{-1}\CC_V \otimes
\big(Sol_{\var Y}^\rmE( \tl{\SM^\wedge} )\big)
\simeq
\bigoplus_{i=1}^d 
\bigl( \EE_{V^{\an}| \var{Y}^{\an}}^{{\rm Re}g_i}
\bigr)^{\oplus m(i)}
\end{equation}
of enhanced ind-sheaves on $\var{Y}^{\an}$. 
\end{theorem}

\begin{proof}
By \eqref{eq-expofl} and \eqref{enha-isom},  
it suffices to prove that there exists an isomorphism
\begin{align*}
& \pi^{-1}\CC_V \otimes {}^\rmL \Bigl\{ 
\bigl\{ \pi^{-1} (i_{X^{\an}})_! K \bigr\} 
\otimes 
{\rm E}_{U^{\an}| \var{X}^{\an}}^{{\rm Re}f}
 \Bigr\}
\\ 
& \simeq
\Bigl( \bigoplus_{i=1}^r 
\bigl( {\rm E}_{V^{\an}| \var{Y}^{\an}}^{{\rm Re}g_i } 
\bigr)^{\oplus m(i)}
\Bigr) 
\bigoplus 
\Bigl\{ \bigoplus_{i=r+1}^d 
\Bigl( \bigoplus_{k=1}^{n_i} 
\bigl( {\rm E}_{V^{\an}| \var{Y}^{\an}}^{{\rm Re}g_{ik}}
\bigr)^{\oplus m(i,k)}
\Bigr) 
\Bigr\}. 
\end{align*}
of enhanced sheaves on $V \subset \Omega \subset Y$.
Let 
\begin{equation}
X\times\RR_s\overset{p_1}{\longleftarrow}
(X\times\RR_s)\times (Y\times\RR_t)
\overset{p_2}{\longrightarrow}Y\times\RR_t
\end{equation}
be the projections.
Then by D'Agnolo-Kashiwara \cite[Lemma 7.2.1]{DK17}
on $Y^{\an} \subset \var{Y}^{\an}$ we have an isomorphism
\begin{align*}
& {}^\rmL \Bigl\{ 
\bigl\{ \pi^{-1} (i_{X^{\an}})_! K \bigr\} 
\otimes 
{\rm E}_{U^{\an}| \var{X}^{\an}}^{{\rm Re}f}
 \Bigr\}
\\
& \simeq \Q 
\Bigl(\rmR p_{2!} \Bigl(p_1^{-1}
\Bigl\{ 
\bigl( \pi^{-1} K \bigr)  
\otimes 
{\rm E}_{U^{\an}| X^{\an}}^{{\rm Re}f}
 \Bigr\} \otimes \CC_{\{t-s-{\Re}
\langle z, w \rangle \geq0\}}[ N ] \Bigl) \Bigr),
\end{align*}
where $\Q : \BDC(\CC_{Y^{\an}\times\RR})
\to\BEC(\CC_{Y^{\an}})$ is the quotient functor.
For a point $(w, t)\in Y^{\an}\times\RR$
we have also isomorphisms
\begin{align*}
& \Bigl(\rmR p_{2!} \Bigl(p_1^{-1}
\Bigl\{ 
\bigl( \pi^{-1} K \bigr) 
\otimes 
{\rm E}_{U^{\an}| X^{\an}}^{{\rm Re}f}
 \Bigr\} \otimes \CC_{\{t-s-{\Re}
\langle z, w \rangle \geq0\}}[ N ] \Bigl) \Bigr)_{(w, t)}
\\
&\simeq
\rmR\Gamma_c(\{ (z, s)\in U^{\an} \times \RR \ | \ 
t-s-{\Re} \langle z, w \rangle \geq0, 
s + {\Re} f(z) \geq0\}; \pi^{-1} K[N])
\\
&\simeq
\rmR\Gamma_c( U^{\an} ; \rmR \pi_!(
\CC_{\{ (z, s)\in U^{\an} \times \RR \ | \ 
t-s-{\Re} \langle z, w \rangle \geq 0, \ s 
 + {\Re} f(z) \geq 0 \}} 
\otimes 
\pi^{-1} K[N])) 
\\
&\simeq
\rmR\Gamma_c( U^{\an} ; (\rmR \pi_! 
\CC_{\{  (z, s)\in U^{\an} \times \RR \ | \ 
 t-s-{\Re} \langle z, w \rangle \geq 0, \ 
s  + {\Re} f(z) 
\geq 0 \}} ) 
\otimes K[N]) 
\\
&\simeq 
\rmR\Gamma_c(\{  z \in U^{\an}\ |\ {\Re}
f^w(z) \leq t\}; K[N]), 
\end{align*}
where we used 
\begin{equation}
\rmR \pi_! 
\CC_{\{   (z, s)\in U^{\an} \times \RR \ | \ 
 t-s-{\Re} \langle z, w \rangle \geq 0, 
\ s  + {\Re} f(z) \geq 0 \}} 
\simeq 
\CC_{ \{ z \in U^{\an} \ | \ {\Re} \langle  z, w \rangle - 
{\Re} f(z) \leq t \} } 
\end{equation}
in the last isomorphism. 
Fix a point $w\in V \subset \Omega \subset Y=\CC^N$. 
Then as in the proof 
of \cite[Theorem 4.4]{IT20a} 
we can prove the vanishing 
\begin{equation}
\rmR\Gamma_c(\{  z \in U^{\an}\ |\ {\Re}
f^w(z) \leq t\}; K[N]) \simeq 0 
\end{equation}
for $t\ll0$ as follows. Let us set 
\begin{align*}
& L( Sol_X(\SM) ) 
:= 
\\
& \Bigl(\rmR p_{2!} \Bigl(p_1^{-1}
\Bigl\{ 
\bigl( \pi^{-1} K \bigr)  
\otimes 
{\rm E}_{U^{\an}| X^{\an}}^{{\rm Re}f}
 \Bigr\} \otimes \CC_{\{t-s-{\Re}
\langle z, w \rangle \geq0\}}[ N ] \Bigl) \Bigr).
\end{align*}
Then for $t\in\RR$ its stalk 
at $(w,t) \in Y \times \RR$: 
\begin{equation}
\big(  L( Sol_X(\SM) ) \big)_{(w, t)} 
\simeq
\rmR\Gamma_c\big(
\{z\in U^{\an}\ |\ {\Re} f^w \leq t\}; K[N]
\big)
\end{equation}
is calculated as follows:
\begin{align*}
&\rmR\Gamma_c\big(\{z\in U^{\an}\ |\ {\Re} f^w 
\leq t\}; K[N] \big)\\
&\simeq
\rmR\Gamma_c\big(\{\tau\in \CC\ |\ {\Re}\tau\leq t\};
\rmR f^w_! (K|_U[N]) \big).
\end{align*}
Since the morphism $f^w: U=X \setminus 
Q^{-1}(0) \longrightarrow \CC$ is algebraic, 
the proper direct image $\rmR f^w_! (K|_U[N])$ 
of the perverse sheaf 
$(K|_U[N]) \in \Dbc ( U^{\an} )$ 
is constructible. Then in particular, 
for $t\ll0$ the restrictions of 
its cohomology sheaves to 
the closed half space $\{\tau\in\CC\ | \ 
{\Re}\tau\leq t\}\subset\CC$ of $\CC$ are 
locally constant. 
Thus for $t\ll0$ we obtain the vanishing 
\begin{equation}\label{van-ext} 
\rmR\Gamma_c \bigl( \{\tau\in \CC\ |\ {\Re}\tau\leq t\};
\rmR f^w_! (K|_U[N]) \bigr)
\simeq 0.
\end{equation}
For $1 \leq i \leq r$ we set 
\begin{equation}
c(i):= f^w(\zeta^{(i)}(w))= 
\langle \zeta^{(i)}(w), w \rangle 
-f(\zeta^{(i)}(w)) \in \CC. 
\end{equation}
Moreover for $r+1 \leq i \leq d$ and 
$1 \leq k \leq n_i$ we set 
\begin{equation}
c(i,k):= h_{ik}^w(\zeta^{(i)}(w))= 
\langle \zeta^{(i)}(w), w \rangle 
-h_{ik}(\zeta^{(i)}(w)) \in \CC. 
\end{equation}
Let $B \subset \CC$ be the union 
of these points in $\CC$ and write it as 
\begin{equation}
B= \{ c_1, c_2, \ldots, c_l \} =
\{ c_j \in \CC \ | \ 
1 \leq j \leq l \} \subset \CC  
\end{equation}
so that for any $j \not= j^{\prime}$ 
we have $c_j \not= c_{j^{\prime}}$. 
Then for any $1 \leq i \leq r$ there 
exists a unique index $1 \leq j \leq l$ 
such that $c_j=c(i)$ and we denote it by 
$j(i)$. Similarly, for any $r+1 \leq i \leq d$ and 
$1 \leq k \leq n_i$ there exists a unique index $1 \leq j \leq l$ 
such that $c_j=c(i,k)$ and we denote it by 
$j(i,k)$. Then by 
Corollary \ref{equiv-2ss} and  Lemma \ref{lem-31-new}, 
for any $1 \leq j \leq l$ we can apply 
Theorem \ref{prop-8-new}  
to obtain an isomorphism 
\begin{align*}
& \phi_{\tau -c_j}( \rmR f^w_!  (K|_U[N]) ) 
\\ 
& \simeq
\Bigl( \bigoplus_{i: \ j(i)=j} 
\phi_{f^w-c_j}( K[N] )_{\zeta^{(i)}(w)}
\Bigr) 
\bigoplus 
\Bigl( \bigoplus_{(i,k): \ j(i,k)=j} 
\phi^{\merc}_{f^w-c_j}( K[N] )_{\zeta^{(i)}(w)}
\Bigr), 
\end{align*}
where in the first (resp. the last) direct sum 
$\oplus$ the index $1 \leq i \leq r$ (resp. 
the pair $(i,k)$ of $r+1 \leq i \leq d$ and 
$1 \leq k \leq n_i$) ranges over the ones 
satisfying the condition $j(i)=j$ (resp. 
$j(i,k)=j$). Recall that for any $r+1 \leq i \leq d$ 
there exists at most one index $1 \leq k \leq n_i$ 
such that $j(i,k)=j$. 
By Proposition \ref{prop-cevan1} 
for any $1 \leq i \leq r$ 
such that $j(i)=j$ 
we have isomorphisms 
\begin{equation}
H^n \phi_{f^w-c_{j}}( K[N] )_{\zeta^{(i)}(w)}
\simeq 
\begin{cases}
\CC^{m(i)} & (n=-1)\\
\\
0 & (\mbox{otherwise}). 
\end{cases}
\end{equation}
Moreover by Proposition \ref{prop-cevan2} 
for any pair $(i,k)$ of $r+1 \leq i \leq d$ and 
$1 \leq k \leq n_i$ 
such that $j(i,k)=j$ 
we have isomorphisms 
\begin{equation}
H^n \phi^{\merc}_{f^w-c_{j}}( K[N] )_{\zeta^{(i)}(w)}
\simeq 
\begin{cases}
\CC^{m(i,k)} & (n=-1)\\
\\
0 & (\mbox{otherwise}). 
\end{cases}
\end{equation}
For $1 \leq j \leq l$ let us set 
\begin{equation}
d_j:= \Bigl( \dsum_{i: \ j(i)=j} m(i) \Bigr) + 
\Bigl( \dsum_{(i,k): \ j(i,k)=j} m(i,k) \Bigr). 
\end{equation}
Then for any $1 \leq j \leq l$ we thus obtain isomorphisms 
\begin{equation}
H^n \phi_{\tau -c_j}( \rmR f^w_!  (K|_U[N]) ) 
\simeq
\begin{cases}
\CC^{d_j} & (n=-1)\\
\\
0 & (\mbox{otherwise}). 
\end{cases}
\end{equation}
Again by Theorem \ref{prop-8-new}, for any $c \in \CC$ 
such that $c \notin B= \{ c_1, c_2, \ldots, c_l \}$ 
we can apply Proposition \ref{prop-cevan2} to  
obtain a vanishing 
\begin{equation}
\phi_{\tau -c}( \rmR f^w_! (K|_U[N]) ) \simeq 0. 
\end{equation}
By the proof of \cite[Lemma 2.1]{IT20b} 
this implies that 
the constructible sheaf $\rmR f^w_! (K|_U[N])
\in\BDC_{\CC-c}( \CC )$ on $\CC$ 
is smooth outside the finite subset 
$B= \{ c_1, c_2, \ldots, c_l \} \subset \CC$. 
For $1\leq j \leq l$ set $\gamma_j := {\Re} c_j
 \in \RR$. 
For the fixed point $w\in V \subset \Omega$, 
after reordering $\gamma_1, \gamma_2, \ldots, \gamma_l \in \RR$ 
we may assume that
\[  \gamma_1 \ \leq \gamma_2 \ \leq \cdots  \cdots \ \leq \gamma_l.\]
If for some $1\leq i<j \leq l$ such that 
$ \gamma_i <  \gamma_j$ 
the open interval $( \gamma_i,  \gamma_j)\subset\RR$ 
does not intersect the set $\{ \gamma_1, \gamma_2, \ldots, \gamma_l \}$, 
then for any $t_1, t_2\in\RR$ such that 
$\gamma_i < t_1 < t_2 < \gamma_j$ 
we can show an isomorphism 
\begin{align}
&\rmR\Gamma_c(\{z\in U^{\an}\ |\ {\Re}
f^w(z)   \leq t_2\}; K[N] ) \\
&\simto
\rmR\Gamma_c(\{z\in U^{\an}\ |\ {\Re}
f^w(z) \leq t_1\}; K[N] ).
\end{align}
Equivalently, we shall show an isomorphism 
\begin{align}
&\rmR\Gamma_c(\{ \tau \in \CC \ |\ {\Re}
\tau  \leq t_2\}; \rmR f^w_!  (K|_U[N]) ) \\
&\simto
\rmR\Gamma_c(\{ \tau \in \CC \ |\ {\Re}
\tau \leq t_1\}; \rmR f^w_! (K|_U[N]) ).
\end{align}
This follows from 
Kashiwara's non-characteristic 
deformation lemma (see \cite[Proposition 2.7.2]{KS90}) 
as follows. Let $\iota : \CC = \RR \times \sqrt{-1} \RR 
\hookrightarrow \RR \times \sqrt{-1} \ \overline{\RR}$ be the 
inclusion map. Then there exists 
a continuous map 
$\ell : \RR \times \sqrt{-1} \ \overline{\RR} \longrightarrow \RR$ 
which extends the one ${\rm Re} : \CC  \longrightarrow \RR$. 
Now, by applying Kashiwara's non-characteristic 
deformation lemma to the Morse function 
$\ell : \RR \times \sqrt{-1} \ \overline{\RR} \longrightarrow \RR$, 
we obtain an isomorphism 
\begin{align}
&\rmR\Gamma_c(\{ \tau \in \RR \times \sqrt{-1} \ \overline{\RR} \ |\ 
\ell ( \tau ) \leq t_2\}  ; 
\iota_! \rmR f^w_! (K|_U[N]) ) \\
&\simto
\rmR\Gamma_c(\{ \tau \in \RR \times \sqrt{-1} \ \overline{\RR} \ |\ 
\ell ( \tau ) \leq t_1\}  ; 
\iota_! \rmR f^w_! (K|_U[N]) ), 
\end{align}
which is equivalent to the desired one. 
For $1 \leq j \leq l$ we define a closed 
half space $G_j \subset \CC$ of $\CC$ by 
\begin{equation}
G_j := \{ \tau \in \CC \ |\ {\Re} \tau 
 \geq \gamma_j \} \subset \CC. 
\end{equation}
Then for any  $1 \leq j \leq l$  we have 
isomorphisms 
\begin{align}
H^n \rmR \Gamma_{G_j} ( \rmR f^w_!  (K|_U[N]) )_{c_j} 
&\simeq
H^n \phi_{\tau -c_j}( \rmR f^w_!  (K|_U[N]) ) [-1] 
\\
&\simeq
\begin{cases}
\CC^{d_j} & (n=0)\\
\\
0 & (\mbox{otherwise}).
\end{cases}
\end{align}
Starting from the situation \eqref{van-ext}, 
by Morse theory, we can show that 
for any $t= \gamma_j 
= {\Re} c_j \in \RR$ 
($1\leq j \leq l$)
there exists $0<\e\ll1$ such that
\begin{align*}
&\rmR\Gamma_c(\{ \tau \in \CC  \ |\ {\Re}  \tau  \leq 
t+\e\}; \rmR f^w_!  (K|_U[N]) )\\
&\simeq
\rmR\Gamma_c(\{ \tau \in \CC \ |\ {\Re}  \tau  \leq t\}; 
 \rmR f^w_! (K|_U[N]) )\\
&\simeq
\rmR\Gamma_c(\{ \tau \in \CC \ |\ {\Re}  \tau  \leq t-\e\}; 
 \rmR f^w_! (K|_U[N]) ) \oplus \CC^{d_j}.
\end{align*}
This implies that the restriction of $L( Sol_X(\SM) )$
to the fiber $\pi^{-1}(w) \simeq\RR$ of 
the point $w\in V \subset \Omega$
is isomorphic to that of the sheaf
\begin{align*}
& \Bigl( \bigoplus_{i=1}^r 
\bigl( \CC_{\{t + {\Re} g_i(w) \geq 0 \}} 
\bigr)^{\oplus m(i)}
\Bigr) 
\bigoplus 
\Bigl\{ \bigoplus_{i=r+1}^d 
\Bigl( \bigoplus_{k=1}^{n_i} 
\bigl( \CC_{\{t + {\Re} g_{ik}(w) \geq 0 \}} 
\bigr)^{\oplus m(i,k)}
\Bigr) 
\Bigr\}. 
\\ 
& \simeq
\Bigl( \bigoplus_{i=1}^r 
\bigl( {\rm E}_{V^{\an}| \var{Y}^{\an}}^{{\rm Re}g_i} 
\bigr)^{\oplus m(i)}
\Bigr) 
\bigoplus 
\Bigl\{ \bigoplus_{i=r+1}^d 
\Bigl( \bigoplus_{k=1}^{n_i} 
\bigl( {\rm E}_{V^{\an}| \var{Y}^{\an}}^{{\rm Re}g_{ik}} 
\bigr)^{\oplus m(i,k)}
\Bigr) 
\Bigr\}. 
\end{align*}
Since the subsets $(V\times\RR)\cap \{ 
t + {\Re} g_i(w) \geq 0 \} 
\simeq V \times\RR_{\geq0}$ and 
$(V\times\RR)\cap \{ 
t + {\Re} g_{ik}(w) \geq 0 \} 
\simeq V \times\RR_{\geq0}$
of $V\times\RR$
are connected and simply connected,
we can extend this isomorphism to the whole $V \times \RR 
\subset Y^{\an}\times\RR$.
This completes the proof.
\end{proof}

\begin{corollary}\label{cor-3}
In the situation of Theorem \ref{thm-2}, 
the restriction of the Fourier transform
$\SM^\wedge\in\Mod_{\rm hol}(\SD_Y)$ of $\SM$
to $\Omega \subset Y= \CC^N_w$ is an integrable connection.
Moreover its rank is equal to 
\begin{equation}
\dsum_{i=1}^d m(i) = 
\Bigl( \dsum_{i=1}^r m(i) \Bigr) + 
\dsum_{i=r+1}^d 
\Bigl( \dsum_{k=1}^{n_i} m(i,k) \Bigr). 
\end{equation}
\end{corollary}

Since the proof of this corollary is completely the 
same as that of \cite[Corollary 4.5]{IT20a}, 
we omit it here. Next fix a point 
$w \in Y= \CC^N$ such that $w\neq0$ and set
\[\LL := \CC w = \{\lambda w\ |\ \lambda\in\CC\}\subset Y=\CC^N.\]
Then $\LL$ is a complex line isomorphic to $\CC_\lambda$. 
Assume that $\LL$ is not contained in 
$D:= Y \setminus \Omega \subset Y= \CC^N$. 
Let $\PP := \LL \sqcup \{\infty\} \subset \var{Y} = \PP^N$
be the projective compactification of $\LL$ 
and $i_{\PP}: \PP \hookrightarrow \var{Y} = \PP^N$ 
the inclusion map. Then by Corollary \ref{cor-3} 
the holonomic D-module 
$\SL := H^0 \bfD (i_{\PP})^\ast 
\SM^\wedge \in \Modhol(\SD_\PP)$ on $\PP$ is a 
meromorphic connection on a 
neighborhood of the point $\infty \in \PP$. 
The following result is a generalization 
of \cite[Theorem 5.6]{ET15} 
and \cite[Theorem 4.6]{IT20a}.  
Let $\varpi_\PP : \tl{\PP}\to\PP$ 
be the real oriented blow-up of $\PP$ along the divisor 
$\{ \infty \} \subset \PP$. 

\begin{theorem}\label{thm-3}
In the situation of Theorem \ref{thm-2}, 
for any point $\theta\in 
\varpi_{\PP}^{-1}( \{ \infty \} )\simeq S^1$
there exists its open neighborhood $W$ in $\tl{\PP}$
such that we have an isomorphism
\[\SL^\SA|_W
\simeq
\Bigr(\bigoplus_{i=1}^d
\bigl(
(\SE_{\LL|\PP}^{g_i( \lambda w)})^\SA
\bigr)^{\oplus m(i)}
\Bigl)|_W
\]
of $\SD_\PP^\SA$-modules 
(see Section \ref{sec:11} for the definition) 
on $W$. In particular, 
the functions $g_i( \lambda w)$ 
of $\lambda$ are the exponential factors of 
the meromorphic connection $\SL$ at the point 
$\infty \in \PP$. Moreover the multiplicity of 
$g_i( \lambda w)$ is equal 
to $m(i)$. 
\end{theorem}

\begin{proof}
By \cite[Proposition 3.5]{IT20a}, 
for any point $\theta\in 
\varpi_{\PP}^{-1}( \{ \infty \} )\simeq S^1$
there exists its sectorial neighborhood
$V_\theta\subset\PP\bs \{ \infty \}$ such that
we have isomorphisms
\begin{equation}
\pi^{-1}\CC_{V_\theta}\otimes 
Sol_\PP^\rmE(\SE_{ \LL |\PP}^{g_i( \lambda w)})
\simeq
\EE_{ V_\theta |\PP}^{{\Re} g_i( \lambda w)} 
\simeq
\pi^{-1}\CC_{V_\theta}\otimes 
\Big( \underset{a\to+\infty}{\inj}\ 
 \CC_{\{t\geq - {\Re} g_i( \lambda w)+a\}} \Big).
\end{equation} 
On the other hand, by Theorem \ref{thm-2}
there exists an isomorphism
\begin{equation}
\pi^{-1}\CC_{V_\theta}\otimes Sol_\PP^\rmE(\SL)
\simeq
\bigoplus_{i=1}^d 
\pi^{-1}\CC_{V_\theta}\otimes 
\Big( \underset{a\to+\infty}{\inj}\ 
\CC_{\{t\geq - {\Re} g_i( \lambda w)+a\}}^{\oplus m(i)} \Big). 
\end{equation} 
We thus obtain an isomorphism
\[\pi^{-1}\CC_{V_\theta}\otimes Sol_\PP^\rmE(\SL)
\simeq
\pi^{-1}\CC_{V_\theta}\otimes Sol_\PP^\rmE\Bigl(
\bigoplus_{i=1}^d \big(\SE_{\LL|\PP}^{
g_i( \lambda w)}\big)^{\oplus m(i)}
\Bigr).
\]
Then the assertions follow from 
\cite[Corollary 3.11 and Theorem 3.18]{IT20a}. 
\end{proof}

As in Esterov-Takeuchi \cite[Remark 5.7]{ET15}, 
by Theorems  \ref{thm-2} and \ref{thm-3}
we easily obtain the Stokes lines of the meromorphic connection 
$\SL\in\Modhol(\SD_\PP)$ at $\infty \in \PP$. 
We leave the precise formulation to the readers.

\medskip 
\indent By Theorem \ref{thm-2}, at generic 
points $v \in D= Y \setminus \Omega \subset Y= \CC^N$ at which 
$D$ is a smooth hypersurface in $Y= \CC^N$ 
we obtain also the irregularity and the exponential 
factors of $\SM^\wedge$ along it as follows. 
Let $D_{\rm reg} \subset D$ be 
the smooth part of $D$ and $v \in D_{\rm reg}$ 
such a generic point. 
Take a subvariety $M\subset Y$ of $Y= \CC^N$ which
intersects $D_{\rm reg}$ at $v$ transversally. 
We call it a normal slice of $D$ at $v$. 
By definition $M$ is smooth and of dimension $1$ 
on a neighborhood of $v$. 
Let $i_M : M\xhookrightarrow{\ \ \ } Y=\CC^N$ be 
the inclusion map and set 
$\SK = {\bfD}i_M^\ast\SM^\wedge\in\Modhol(\SD_M)$. 
Then we can describe the irregularity 
${\rm irr} ( \SK(\ast\{v\}) )$ of the 
meromorphic connection 
$\SK(\ast\{v\})$ on $M$ along $\{ v \} \subset M$ 
as follows. Recall that the irregularity 
${\rm irr} ( \SK(\ast\{v\}) )$ is a non-negative integer and 
equal to $- \chi_v \big(Sol_M ( \SK(\ast\{v\}) ) \big)$, 
where 
\begin{equation}
\chi_v\big(Sol_M( \SK(\ast\{v\}) )\big) :=
\sum_{j\in\ZZ}(-1)^j\dim H^jSol_M ( \SK(\ast\{v\}) )_v
\end{equation} 
is the local Euler-Poincar\'e index of 
$Sol_M( \SK(\ast\{v\}) )$ at the point $v \in M$ 
(see e.g. Sabbah \cite{Sab93}). 
Shrinking the normal slice $M$ 
if necessary we may assume that 
$M= \{ u \in \CC \ | \ | u | < 
\varepsilon \}$ for some $\varepsilon >0$, 
$\{ v \} = \{ u =0 \}$ and 
$M \setminus \{ v \} \subset \Omega$. 
Let $i_0 : 
M \setminus \{v\} \hookrightarrow \Omega$ 
be the inclusion map and define 
(possibly multi-valued) holomorphic 
functions $\varphi_i: M \setminus \{ v \} 
\rightarrow \CC$ ($1 \leq i \leq d$) by 
\begin{equation}
\varphi_i ( u )=  g_i( i_0(u) ).
\end{equation}
Then it is easy to see that $\varphi_i ( u )$ 
are Laurent Puiseux series of $u$ 
(see Kirwan \cite[Section 7.2]{Kir92} etc.). 
For each Laurent Puiseux series 
\begin{equation}
\varphi_i ( u )= 
\sum_{a \in \QQ} c_{i,a} u^a \qquad (c_{i,a} \in \CC ) 
\end{equation}
set $r_i= \min \{ a \in \QQ \ | \ c_{i,a} \not= 0 \}$ and 
define its pole order 
${\rm ord}_{\{ v \}}( \varphi_i ) \geq 0$ by 
\begin{equation}
{\rm ord}_{\{ v \}}( \varphi_i )
=
\begin{cases}
- r_i
 & ( r_i <0)\\
\\
0 & (\mbox{\rm otherwise}).
\end{cases}
\end{equation}
Then we obtain the following theorem.

\begin{theorem}\label{thm-neww}
The exponential factors appearing in the 
Hukuhara-Levelt-Turrittin decomposition
of the meromorphic connection $\SK(\ast\{v\})$ at $v\in M$
are the pole parts of $\varphi_i\ (1\leq i \leq d)$.
Moreover for any $1\leq i\leq d$ the 
multiplicity of the pole part of $\varphi_i$
is equal to $m(i)$.
In particular, the irregularity of 
the meromorphic connection $\SK(\ast\{v\})$ along $v\in M$  
is given by 
\begin{equation}
{\rm irr} ( \SK(\ast\{v\}) ) = 
\sum_{i=1}^d m(i) \cdot 
{\rm ord}_{\{ v \}}( \varphi_i ). 
\end{equation}
\end{theorem}

\section{Toward the Study of Fourier Transforms 
of General Holonomic D-modules}\label{sec:11}

\subsection{Preliminary Results for Holonomic D-modules}\label{subsec:7.1}
In this subsection, we prove some new formulas 
which might be useful to extend our results in 
Section \ref{sec:9} to 
arbitrary holonomic D-modules. 
First of all, let us recall some notions and results in \cite[\S 7]{DK16}. 
Let $X$ be a complex manifold and $D \subset X$ a 
normal crossing divisor in it. Denote by 
$\varpi_X : \tl{X}\to X$ the real oriented blow-up of $X$ along 
$D$ (sometimes we denote it simply by $\varpi$). 
Then we set
\begin{align*}
\SO_{\tl{X}}^\rmt &:= \rhom_{\varpi^{-1}
\SD_{\overline{X}}}(\varpi^{-1}\SO_{\overline{X}}, 
\mathcal{D}b_{\tl{X_\RR}}^\rmt),\\
\SA_{\tl{X}} &:= \alpha_{\tl{X}}\SO_{\tl{X}}^\rmt,\\
\SD_{\tl{X}}^\SA &:= \SA_{\tl{X}}\otimes_{\varpi^{-1}\SO_X}
\varpi^{-1}\SD_X, 
\end{align*}
where $\mathcal{D}b_{\tl{X}}^\rmt$ stands for the ind-sheaf 
of tempered distributions on $\tl{X}$ 
(for the definition see \cite[Notation 7.2.4]{DK16}). 
Recall that a section of $\SA_{\tl X}$ is a holomorphic function having
moderate growth at $\varpi_X^{-1}(D)$.
Note that $\SA_{\tl{X}}$ and 
$\SD_{\tl{X}}^\SA$ are sheaves of rings on $\tl{X}$. 
For $\SM\in\BDC(\SD_X)$ we define 
an object $\SM^{\SA}\in\BDC(\SD_{\tl{X}}^\SA)$ by 
\begin{align*}
\SM^{\SA} &:= \SD_{\tl{X}}^\SA\Lotimes
{\varpi^{-1}\SD_X}\varpi^{-1}\SM
\simeq\SA_{\tl{X}}
\Lotimes{\varpi^{-1}\SO_X}\varpi^{-1}\SM.
\end{align*}
Note that if $\SM$ is a holonomic $\SD_X$-module such that 
$\SM\simto\SM(\ast D)$ and 
${\rm sing.supp} (\SM)\subset D$,
then one has $\SM^\SA\simeq
\SD_{\tl{X}}^\SA\otimes_{\varpi^{-1}\SD_X}\varpi^{-1}\SM$
(see \cite[Lemma 7.3.2]{DK16}).
Moreover we have an isomorphism $\M^\SA\simto \M(\ast D)^\SA$
for any holonomic $\D_X$-module $\M$ (see \cite[Lemma 7.2.2]{DK16}). 
Let us take local coordinates 
$(u,v)=(u_1, \ldots, u_l, v_1, \ldots, v_{n-l})$ 
of $X$ such that $D= \{ u_1 u_2 \cdots u_l=0 \}$. 
We define a partial order $\leq$ on the 
set $\ZZ^l$ by 
\[ a=(a_1, \ldots, a_l) \leq a^{\prime}=(a^{\prime}_1, \ldots, a^{\prime}_l) 
 \ \Longleftrightarrow 
\ a_i \leq a_i^{\prime} \ (1 \leq i \leq l).\] 
Then for a meromorphic function $\varphi\in\SO_X(\ast D)$
on $X$ having a pole along $D$ by using its Laurent expansion 
\[ \varphi = \sum_{a \in \ZZ^l} c_a( \varphi )(v) \cdot 
u^a \ \in \SO_X(\ast D) \]
with respect to $u_1, \ldots, u_{l}$
we define its order 
${\rm ord}( \varphi ) \in \ZZ^l$ to be 
the minimum 
\[ \min \Big( \{ a \in \ZZ^l \ | \
c_a( \varphi ) \not= 0 \} \cup \{ 0 \} \Big) \]
if it exists. In \cite[Chapter 5]{Mochi11} 
Mochizuki defined the notion of 
good sets of irregular values on $(X,D)$ to be 
finite subsets $S \subset \SO_X(\ast D)/ \SO_X$ such that 
\par 
\medskip
\noindent (i): ${\rm ord}( \varphi )$ exists for any 
$\varphi \in S$ and if $\varphi \not= 0$ then its  
leading term $c_{{\rm ord}( \varphi )} ( \varphi ) (v)$ 
does not vanish 
at any point $v \in Y:= \{ u_1= \cdots =u_l=0 \} 
\subset D$. 
\par \noindent (ii): ${\rm ord}( \varphi - \psi )$ exists for any 
$\varphi \not= \psi$ in $S$ and then ${\rm ord}( \varphi - \psi ) 
\in \ZZ^l_{\leq 0} \setminus \{ 0 \}$ and the  
leading term $c_{{\rm ord}( \varphi - \psi )} ( \varphi - \psi ) (v)$ 
does not vanish 
at any point $v \in Y= \{ u_1= \cdots =u_l=0 \} 
\subset D$. 
\par \noindent (iii): the subset $\{ {\rm ord}( \varphi - \psi ) \ | \ 
\varphi, \psi \in S, \varphi \not= \psi \} \subset \ZZ^l$  
is totally ordered with respect to the order $\leq$ on $\ZZ^l$. 
\par 
\medskip
\begin{definition}\label{def-A}
Let $X$ be a complex manifold and $D \subset X$ a 
normal crossing divisor in it. 
Then we say that a holonomic $\SD_X$-module 
$\SM$ has a normal form along $D$ if\\
(i) $\SM\simto\SM(\ast D)$\\
(ii) $\rm{sing.supp}(\SM)\subset D$\\
(iii) for any $\theta\in\varpi^{-1}(D)\subset\tl{X}$,
there exist an open neighborhood $U\subset X$
of $\varpi({\theta}) \in D$ in $X$, a good set 
$S= \{ [ \varphi_1], [ \varphi_2], \ldots, 
[ \varphi_k] \} \subset \SO_X(\ast D)/ \SO_X$ 
($\varphi_i \in \SO_X(\ast D)$) 
of irregular values on $(U, D \cap U)$, 
positive integers 
$m_i >0$ ($1 \leq i \leq k$) and 
an open neighborhood $W$ of $\theta$ 
with $W \subset\varpi^{-1}(U)$
such that
\begin{equation}
\SM^\SA|_W
\simeq
\bigoplus_{i=1}^k 
\Bigl( \bigl(\SE_{U\bs D|U}^{\varphi_i}\bigr)^\SA |_W
\Bigr)^{\oplus m_i} .
\end{equation}
\end{definition}
By \cite[Proposition 3.19]{IT20a} the good set 
$S \subset \SO_X(\ast D)/ \SO_X$ of 
irregular values for $\SM$ 
in this definition does not depend on 
the point $\theta \in \varpi^{-1}(D)$. Moreover 
by \cite[Proposition 3.5]{IT20a} for any 
$\theta \in \varpi^{-1}(D \cap U)$ there exists its  
sectorial open neighborhood $V \subset U \setminus D$ 
such that 
\begin{equation}
\pi^{-1} \CC_V \otimes 
Sol_X^\rmE ( \SM ) 
\simeq
\bigoplus_{i=1}^k 
\Bigl(  \EE_{V|X}^{\Re \varphi_i} 
\Bigr)^{\oplus m_i}.
\end{equation}

\begin{lemma}\label{DK-lem}
In the situation as above, 
there exists a  
sectorial open neighborhood $V \subset U \setminus D$ 
of $\theta \in \varpi^{-1}(D \cap U)$ 
such that for any $1 \leq i,j \leq k$ the natural morphism 
\begin{equation}
{\rm Hom}^\rmE ({\rm E}_{V|M}^{\Re \varphi_i}, 
{\rm E}_{V|M}^{\Re \varphi_j}) 
\longrightarrow 
{\rm Hom}^\rmE ( \EE_{V|X}^{\Re \varphi_i}, \EE_{V|X}^{\Re \varphi_j}) 
\end{equation}
is an isomorphism. 
\end{lemma}

\begin{proof}
The proof is similar to that of \cite[Lemma 5.2.1 (ii)]{DK17}. 
It suffices to consider only pairs $(i,j)$ such that 
$i \not= j$. Then by $[ \varphi_i] \not= [ \varphi_j]$ 
the function $\varphi_j- \varphi_i \in 
\SO_X(\ast D)$ has a pole along the normal 
crossing divisor $D \subset X$. 
For a local coordinate system 
$(u,v)=(u_1, \ldots, u_l, v_1, \ldots, v_{n-l})$ 
of $X$ such that $D= \{ u_1 u_2 \cdots u_l=0 \}$ 
and $\varpi ( \theta ) \in 
Y:= \{ u_1=u_2= \cdots =u_l=0 \} \subset D$ let 
\begin{equation}
( \varphi_j- \varphi_i )(u,v)= 
\sum_{a \in \ZZ^l} c_a( \varphi_j- \varphi_i )(v) \cdot 
u^a \ \in \SO_X(\ast D) 
\end{equation}
be the Laurent expansion of $\varphi_j- \varphi_i$ 
with respect to $u_1,u_2, \ldots, u_l$. 
Then by the goodness of the set $S$ the order 
 $\alpha := {\rm ord}( \varphi_j- \varphi_i ) \in \ZZ^l$ 
is defined and satisfies the condition 
\begin{equation}
\alpha = ( \alpha_1, \alpha_2, \ldots, \alpha_l) 
\quad \in ( \ZZ_{\leq 0} )^l \setminus \{ 0 \}. 
\end{equation}
If there exists a sectorial open neighborhood 
$V \subset X \setminus D$ 
of $\theta \in \varpi^{-1}(D)$ 
such that $\Re ( \varphi_j- \varphi_i) \leq 0$ on $V$ , 
then by the proof of \cite[Lemmas 3.1.1 and 3.2.2]{DK17} 
we have isomorphisms 
\begin{equation}
{\rm Hom}^\rmE ({\rm E}_{V|M}^{\Re \varphi_i}, 
{\rm E}_{V|M}^{\Re \varphi_j}) \simeq \CC, \quad 
{\rm Hom}^\rmE ( \EE_{V|X}^{\Re \varphi_i}, \EE_{V|X}^{\Re \varphi_j}) 
\simeq \CC. 
\end{equation}
Otherwise, the point $\theta \in \varpi^{-1}(D)$ is 
contained in the closure $\overline{R}$ of the open subset 
\begin{equation}
R:= \{ (u,v) \in X \setminus D \ | \ 
\Re ( \varphi_j- \varphi_i) (u,v) > 0 \}
\end{equation}
of $\tl{X}$. As the meromorphic function 
$\varphi_j- \varphi_i \in 
\SO_X(\ast D)$ has no point of indeterminacy by 
the goodness of the set $S$, we can easily 
show that $R$ is a subanalytic open 
subset of $\tl{X}$. Then by the curve selection lemma 
there exists a 
real analytic curve $\gamma (t) : [0, \varepsilon )
\longrightarrow \overline{R}$ ($\varepsilon >0$) 
(defined on a neighborhood of $0 \in \RR$) 
such that $\gamma (0) = \theta \in 
\varpi^{-1}(D) \cap \overline{R}$ and 
$\gamma (t) \in R$ for any $t \in (0, \varepsilon )$. 
By the leading term 
$c_{\alpha}( \varphi_j- \varphi_i )(v) \cdot u^{\alpha}$
of the Laurent expansion of $\varphi_j- \varphi_i$, 
we define a complex valued real analytic 
function $\psi : \varpi^{-1}(Y) \longrightarrow \CC$ 
on the real analytic manifold 
$\varpi^{-1}(Y) \simeq (S^1)^l \times Y$ 
(defined on a neighborhood of the point 
$\theta \in \varpi^{-1}(Y)$) by 
\begin{equation}
\psi (e^{i \delta_1}, e^{i \delta_2}, \ldots, e^{i \delta_l}, 
v) := c_{\alpha}( \varphi_j- \varphi_i )(v) \cdot 
e^{i \alpha_1 \delta_1} e^{i \alpha_2 \delta_2} \cdots 
e^{i \alpha_l \delta_l}. 
\end{equation}
Recall that the holomorphic function 
$c_{\alpha}( \varphi_j- \varphi_i )$ on $Y$ 
does not vanish at any point $v \in Y$. 
Suppose that $\Re \psi ( \theta )<0$. Then 
for the real analytic curve $\gamma$ we have 
\begin{equation}
\lim_{t \to +0} \Re ( \varphi_j- \varphi_i) 
 ( \gamma (t)) =- \infty. 
\end{equation}
Then contradicts to the condition 
$\gamma (t) \in R$ ($t \in (0, \varepsilon )$). 
We thus obtain $\Re \psi ( \theta ) \geq 0$. 
Then, by the above definition of $\psi$, 
for any sectorial open neighborhood $V \subset X \setminus D$ 
of $\theta \in \varpi^{-1}(Y)$ there exists 
a point $\theta^{\prime} =  
(e^{i \delta_1}, \ldots, e^{i \delta_l}, 
v) \in \varpi^{-1}(Y) \cap {\rm Int} ( \overline{V} )$ 
such that $\Re \psi ( \theta^{\prime} ) > 0$. 
Define a line 
$\gamma^{\prime} (t) : (0, \varepsilon^{\prime} )
\longrightarrow V$ ($\varepsilon^{\prime} >0$) in $V$ by 
\begin{equation}
\gamma^{\prime} (t) = 
(t e^{i \delta_1}, t e^{i \delta_2}, \ldots, t e^{i \delta_l}, 
v) \in V \qquad (0<t< \varepsilon^{\prime} )
\end{equation}
so that we have 
\begin{equation}
\lim_{t \to +0} \gamma^{\prime} (t) = \theta^{\prime} 
= (e^{i \delta_1}, \ldots, e^{i \delta_l}, v). 
\end{equation}
Then by $\Re \psi ( \theta^{\prime} ) > 0$ we obtain 
\begin{equation}
\lim_{t \to +0} \Re ( \varphi_j- \varphi_i) 
 ( \gamma^{\prime} (t)) =+ \infty, 
\end{equation}
which implies also that $\theta^{\prime} \in \overline{R}$. 
By the proof of \cite[Lemma 3.1.1 (i)]{DK17} 
there exist isomorphisms 
\begin{equation}
{\rm Hom}^\rmE ({\rm E}_{V|M}^{\Re \varphi_i}, 
{\rm E}_{V|M}^{\Re \varphi_j}) \simeq 
{\rm Hom}_{\CC_X} 
( \CC_{V \cap \{ \Re ( \varphi_j- \varphi_i) \leq 0 \}}, \CC_X ) 
\simeq 0. 
\end{equation}
Moreover by the proof of \cite[Lemma 3.2.2 (i)]{DK17} 
we have 
\begin{equation}
{\rm Hom}^\rmE ( \EE_{V|X}^{\Re \varphi_i}, \EE_{V|X}^{\Re \varphi_j}) 
\simeq \varinjlim_{c \to + \infty}
{\rm Hom}_{\CC_X} 
( \CC_{V \cap \{ \Re ( \varphi_j- \varphi_i) \leq c \}}, \CC_X ) 
\simeq 0. 
\end{equation}
We thus obtain the desired isomorphism 
\begin{equation}
{\rm Hom}^\rmE ({\rm E}_{V|M}^{\Re \varphi_i}, 
{\rm E}_{V|M}^{\Re \varphi_j}) 
\simto 
{\rm Hom}^\rmE ( \EE_{V|X}^{\Re \varphi_i}, \EE_{V|X}^{\Re \varphi_j}) 
\end{equation}
for the pair $(i.j)$. Clearly we can take a  
sectorial open neighborhood $V \subset X \setminus D$ 
of $\theta \in \varpi^{-1}(D)$ 
so that this isomorphism holds for any 
pair $(i.j)$. 
\end{proof} 

A ramification of $X$ along the normal crossing divisor 
$D \subset X$ on a neighborhood $U$ 
of $x \in D$ is a finite map $ \rho : X' \longrightarrow U$
of complex manifolds 
of the form $w \longmapsto 
z=(z_1,z_2, \ldots, z_n)= 
 \rho (w) = (w^{d_1}_1,\ldots, w^{d_l}_l, w_{l+1},\ldots, w_n)$ 
for some $(d_1, \ldots, d_l)\in (\ZZ_{>0})^l$, where 
$(w_1,\ldots, w_n)$ is a local coordinate system of $X'$ and 
$(z_1, \ldots, z_n)$ is that of 
$U$ such that $D \cap U=\{z_1\cdots z_l=0\}$. 

\begin{definition}\label{def-B}
Let $X$ be a complex manifold and $D \subset X$ a 
normal crossing divisor in it. 
Then we say that a holonomic $\SD_X$-module $\SM$ has a 
quasi-normal form along $D$ if it satisfies 
the conditions (i) and (ii) of Definition \ref{def-A},
and if for any point $x \in D$ 
there exists a ramification $\rho  : X'\to U$ 
on a neighborhood $U$ of it such that $\bfD \rho^\ast(\SM|_U)$
has a normal form along the normal crossing divisor 
$\rho^{-1}(D\cap U)$.
\end{definition}

Note that $\bfD \rho^\ast(\SM|_U)$ as well 
as $\bfD \rho_\ast\bfD \rho^\ast(\SM|_U)$
is concentrated in degree zero and $\SM|_U$ is a 
direct summand of $\bfD \rho_\ast\bfD \rho^\ast(\SM|_U)$. 
Now let $\SM$ be a holonomic $\SD_X$-module having a 
quasi-normal form along the normal crossing divisor 
$D \subset X$. Then for any point $x \in D$ 
there exists a ramification $\rho  : X'\to U$ 
on a neighborhood $U$ of it such that $\bfD \rho^\ast(\SM|_U)$
has a normal form along the normal crossing divisor 
$D^{\prime} :=\rho^{-1}(D\cap U) \subset X^{\prime}$. 
Note that $\rho^{-1}(x) \subset D^{\prime}$ is a point and 
denote it by $x^{\prime}$. 
Let $\varpi^{\prime}: 
\widetilde{X^{\prime}} \to X^{\prime}$ be the real 
oriented blow-up of $X^{\prime}$ along 
$D^{\prime}$ and $\tl{\rho}: \widetilde{X^{\prime}} \to 
\tl{X}$ the morphism induced by $\rho$. 
Then by \cite[Propositions 3.5 and 3.19]{IT20a} 
there exist a unique good set 
$S= \{ [ \varphi_1], [ \varphi_2], \ldots, 
[ \varphi_k] \} \subset \SO_{X^{\prime}}(\ast D^{\prime})
/ \SO_{X^{\prime}}$ 
($\varphi_i \in  \SO_{X^{\prime}}(\ast D^{\prime})$) 
of irregular values on a neighborhood of 
$x^{\prime} \in D^{\prime}$ in $X^{\prime}$ and 
positive integers 
$m_i >0$ ($1 \leq i \leq k$) such that 
for any $\theta^{\prime} \in ( \varpi^{\prime})^{-1}
(D^{\prime})$ and 
its sufficiently small sectorial open neighborhood 
$V^{\prime} \subset X^{\prime} \setminus D^{\prime}$ 
we have an isomorphism 
\begin{equation}
\pi^{-1} \CC_{V^{\prime}} \otimes 
Sol_{X^{\prime}}^\rmE ( \bfD \rho^\ast(\SM|_U) ) 
\simeq
\bigoplus_{i=1}^k 
\Bigl(  \EE_{V^{\prime}|X^{\prime}}^{\Re \varphi_i} 
\Bigr)^{\oplus m_i}.
\end{equation}
For a point $\theta\in\varpi^{-1}(D \cap U)$ and 
its sufficiently small sectorial open neighborhood 
$V \subset U \setminus D$ we take a point 
$\theta^{\prime} \in ( \varpi^{\prime})^{-1}
(D^{\prime})$ such that $\tl{\rho} ( \theta^{\prime}) = 
\theta$ and its sectorial open neighborhood 
$V^{\prime} \subset X^{\prime} \setminus D^{\prime}$ 
such that $\rho |_{V^{\prime}} : V^{\prime} \simto V$. 
Define holomorphic functions 
$f_i: V \to \CC$ ($1 \leq i \leq k$) by 
$f_i:= \varphi_i \circ ( \rho |_{V^{\prime}})^{-1}$. 
Then by \cite[Proposition 3.5]{IT20a} we obtain 
an isomorphism 
\begin{equation}\label{mon-eq} 
\pi^{-1} \CC_V \otimes 
Sol_X^\rmE ( \SM ) 
\simeq
\bigoplus_{i=1}^k 
\Bigl(  \EE_{V|X}^{\Re f_i} 
\Bigr)^{\oplus m_i}.
\end{equation}
As $\tl{\rho}: \widetilde{X^{\prime}} \to 
\tl{X}$ is locally an isomorphism, then it is 
also clear that on an open neighborhood $W$ of $\theta$ 
in $\tl{X}$ we have an isomorphism 
\begin{equation}\label{mono-eqn} 
\SM^\SA|_W
\simeq
\bigoplus_{i=1}^k 
\Bigl( \bigl(\SE_{U\bs D|U}^{f_i}\bigr)^\SA |_W
\Bigr)^{\oplus m_i} .
\end{equation}
Moreover, by the proof of Lemma \ref{DK-lem}  
we obtain the following result. 

\begin{lemma}\label{DK-lem-2}
In the situation as above, 
there exists a  
sectorial open neighborhood $V \subset U \setminus D$ 
of $\theta \in \varpi^{-1}(D \cap U)$ 
such that for any $1 \leq i,j \leq k$ the natural morphism 
\begin{equation}
{\rm Hom}^\rmE ({\rm E}_{V|M}^{\Re f_i}, 
{\rm E}_{V|M}^{\Re f_j}) 
\longrightarrow 
{\rm Hom}^\rmE ( \EE_{V|X}^{\Re f_i}, \EE_{V|X}^{\Re f_j}) 
\end{equation}
is an isomorphism. 
\end{lemma}

In order to improve \eqref{mon-eq} and obtain a 
higher-dimensional analogue of 
D'Agnolo-Kashiwara \cite[Proposition 5.4.5]{DK17}, 
let us prepare some notations (see \cite[Section 5]{DK17} 
for the details in the one dimensional case).  For 
the real oriented blow-up 
$\varpi : \tl{X}\to X$ of $X$ along the normal crossing divisor 
$D \subset X$ consider the following commutative diagram
\begin{equation}
\vcenter{
\xymatrix@M=5pt{
\varpi^{-1}(D) \ar[r]^-{\tl{\imath}} & \tl{X} \ar[d]^-{\varpi}& \\
X\bs D \ar[r]^-{j} \ar[ur]^-{\tl{\jmath}} & X, 
}}\end{equation}
where $\tl{\imath},\tl{\jmath},j$ 
are the natural embeddings. For an open subset $\Omega\subset\tl{X}$, 
$f\in \Gamma(\Omega;\tl{\jmath}_{\ast} 
j^{-1}\SO_X) \simeq \Gamma( \tl{\jmath}^{-1} \Omega;
\SO_{X \setminus D})$ and $\theta \in \Omega \cap \varpi^{-1}(D)$ 
we say that $f$ admits a Puiseux expansion 
along $D \subset X$ at $\theta$ if there exist  
a ramification $\rho  : X'\to U$ 
of a neighborhood $U$ of $\varpi ( \theta ) \in D$ 
along $D \cap U \subset U$, a sectorial 
neighborhood $V \subset U \setminus D$ of $\theta$ 
contained in $\tl{\jmath}^{-1} \Omega 
= \varpi ( \Omega \setminus \varpi^{-1}(D)) 
\subset X \setminus D$ and a 
meromorphic function $g \in \SO_{X^{\prime}}(* D^{\prime})$ 
along the normal crossing divisor $ D^{\prime}:= \rho^{-1}(D) 
\subset X^{\prime}$ defined on an open neighborhood 
$W$ of $\rho^{-1}( \overline{V} \cap D)= 
\overline{\rho^{-1}(V)} \cap D^{\prime}$ in $X^{\prime}$ such that 
the pull-back of $f|_V \in \SO_X(V)$ by $\rho$ conicides 
with $g$ on the open subset $W \cap \rho^{-1}(V) \subset W$. 
We denote by $\SP_{\tl{X}}$ the subsheaf of 
$\tl{\jmath}_{\ast} j^{-1}\SO_X$ whose sections are defined by 
\begin{multline*}
\Gamma(\Omega;\SP_{\tl{X}}) \coloneq 
\{f\in\Gamma(\Omega;\tl{\jmath}_{\ast} 
j^{-1}\SO_X)\mid \textit{For any $\theta\in\Omega\cap \varpi^{-1}(D)$,} \\ 
\textit{$f$ admits a Puiseux expansion along $D \subset X$ at $\theta$.}\}
\end{multline*}
for open subsets $\Omega\subset\tl{X}$.
Then we define the sheaf of 
Puiseux germs $\SP_{\varpi^{-1}(D)}$ on $\varpi^{-1}(D)$ by   
\begin{align}
\SP_{\varpi^{-1}(D)}\coloneq\tl{\imath}^{\,-1}\SP_{\tl{X}}.
\end{align}
For a point $\theta \in \varpi^{-1}(D)$ if we take a local coordinate 
$(u,v)=(u_1, \ldots, u_l, v_1, \ldots, v_{n-l})$ of $X$ 
on a neighborhood of $\varpi ( \theta ) \in D$ in $X$ such that 
$\varpi ( \theta )=(0,0) \in D= \{ u_1 u_2 \cdots u_l=0 \}$ 
then the stalk of $\SP_{\varpi^{-1}(D)}$ at $\theta$ 
is isomorphic to the ring 
\begin{equation}
\bigcup_{p\in\ZZ_{\geq1}} \ 
\CC \{ u_1^{\frac{1}{p}}, \ldots, u_l^{\frac{1}{p}}, 
v_1, \ldots, v_{n-l} \} 
\bigl[
u_1^{- \frac{1}{p}}, \ldots, u_l^{- \frac{1}{p}}
\bigr].
\end{equation}
of Puiseux series along $D \subset X$. We 
denote by $\SP_{\varpi^{-1}(D)}^{\leq 0}$ the subsheaf of 
$\SP_{\varpi^{-1}(D)}$ consisting of sections locally contained in 
the ring 
\begin{equation}
\bigcup_{p\in\ZZ_{\geq1}} \ 
\CC \{ u_1^{\frac{1}{p}}, \ldots, u_l^{\frac{1}{p}}, 
v_1, \ldots, v_{n-l} \} 
\end{equation}
for some (hence, any) local coordinate 
$(u,v)=(u_1, \ldots, u_l, v_1, \ldots, v_{n-l})$ of $X$ 
as above. By this definition, it is clear that 
for any point $x \in D$ there exist its neighborhood 
$U$ in $X$ and a subsheaf $\SP_{\varpi^{-1}(D \cap U)}^{\prime} 
\subset \SP_{\varpi^{-1}(D \cap U)}$ of $\CC_{\varpi^{-1}(D \cap U)}$-modules  
defined on 
the open subset $\varpi^{-1}(D \cap U) \subset \varpi^{-1}(D)$ 
such that the natural morphism 
\begin{equation}
\SP_{\varpi^{-1}(D \cap U)}^{\prime} \longrightarrow 
\SP_{\varpi^{-1}(D \cap U)}/ \SP_{\varpi^{-1}(D \cap U)}^{\leq 0}
\end{equation}
is an isomorphism. We call such $\SP_{\varpi^{-1}(D \cap U)}^{\prime}$ 
a representative subsheaf of $\SP_{\varpi^{-1}(D \cap U)}$. 
By slightly modifying the definition of the multiplicities in 
D'Agnolo-Kashiwara \cite[Section 5.3]{DK17}, we shall use 
the following one (cf. \cite[Definition 2.4]{KT23}). 

\begin{definition}\label{def-multi}
(cf. \cite[Section 5.3]{DK17} and \cite[Definition 2.4]{KT23}) 
In the situation as above, we say that a morphism 
$N: \SP_{\varpi^{-1}(D \cap U)}^{\prime} \longrightarrow 
( \ZZ_{\geq 0})_{\varpi^{-1}(D \cap U)}$ of sheaves of 
sets is a multiplicity along $D \cap U \subset U$ if 
there exists 
a ramification $\rho  : X'\to U$ 
of $U$ along $D \cap U \subset U$ such that 
for any $\theta \in \varpi^{-1}(D \cap U)$ the subset 
$N_{\theta}^{>0}:= N_{\theta}^{-1}( \ZZ_{>0}) \subset 
\SP_{\varpi^{-1}(D \cap U), \theta}^{\prime}$ is finite 
and the pull-backs of its elements 
$f \in N_{\theta}^{>0}$ by $\rho$ are meromorphic 
functions on $X^{\prime}$ along $D^{\prime}:= \rho^{-1}(D) 
\subset X^{\prime}$ and form a good set 
$\{ [f \circ \rho ] \ | \ f \in N_{\theta}^{>0} \} 
 \subset \SO_{X^{\prime}}(\ast D^{\prime})/ \SO_{X^{\prime}}$ 
of irregular values on $(X^{\prime}, D^{\prime})$ 
on a neighborhood of the point $\rho^{-1}( \varpi ( \theta ))
\in D^{\prime}$. 
\end{definition}

\begin{definition}\label{def-qnf}
(cf. \cite[Definition 5.3.1]{DK17} and \cite[Definition 2.5]{KT23}) 
In the situation as above, we say that an 
$\RR$-constructible enhanced sheaf 
$F \in \BEC( \CC_X)$ on $X$ 
has a quasi-normal form along the normal crossing 
divisor $D \cap U \subset U$ if there exists a 
multiplicity $N: \SP_{\varpi^{-1}(D \cap U)}^{\prime} \longrightarrow 
( \ZZ_{\geq 0})_{\varpi^{-1}(D \cap U)}$ 
such that any point $\theta\in\varpi^{-1}(D \cap U)$ has 
its sectorial open neighborhood 
$V_{\theta} \subset U \setminus D \subset \tl{X}$ for 
which we have an isomorphism 
\begin{equation}
\pi^{-1} \CC_{V_{\theta}} \otimes F 
\simeq
\bigoplus_{f \in N_{\theta}^{>0}} 
\Bigl(  {\rm E}_{V_{\theta} |X}^{\Re f} 
\Bigr)^{N(f)}.
\end{equation}
\end{definition}
Enhanced ind-sheaves having a quasi-normal form along the 
normal crossing divisor $D \cap U \subset U$ are defined similarly. 

\begin{lemma}\label{lem-hda-dk} 
Assume that a holonomic $\SD_X$-module $\SM$ has a 
quasi-normal form along the normal crossing 
divisor $D \subset X$. Then for any point $x \in D$ 
there exist a subanalytic open neighborhood 
$U$ of $x$ in $X$ such that the 
$\RR$-constructible enhanced ind-sheaf 
\begin{equation}
\pi^{-1} \CC_{U} \otimes Sol_X^\rmE ( \SM ) 
\simeq
\pi^{-1} \CC_{U \setminus D} \otimes 
Sol_X^\rmE ( \SM ) 
\end{equation}
has a quasi-normal form along the normal crossing 
divisor $D \cap U \subset U$. 
\end{lemma}

\begin{proof}
The proof is similar to that of \cite[Lemma 5.4.4]{DK17}. 
With the representative subsheaf $\SP_{\varpi^{-1}(D \cap U)}^{\prime}$ 
of $\SP_{\varpi^{-1}(D \cap U)}$ at hands, it suffices to use 
\eqref{mon-eq}, \eqref{mono-eqn} and 
\cite[Propositions 3.10 and 3.19]{IT20a}. 
\end{proof}
In the situation of Lemma \ref{lem-hda-dk}, let 
$N: \SP_{\varpi^{-1}(D \cap U)}^{\prime} \longrightarrow 
( \ZZ_{\geq 0})_{\varpi^{-1}(D \cap U)}$ be the multiplicity 
for which the enhanced ind-sheaf 
$\SF  \simeq \pi^{-1} \CC_{U} \otimes Sol_X^\rmE ( \SM ) 
\in \BEC( \I \CC_X)$ has a quasi-normal form along 
the normal crossing divisor $D \cap U \subset U$. Then 
by the proof of Lemma \ref{lem-hda-dk}, the sections of 
the subsheaf $N^{>0}= N^{-1}(( \ZZ_{>0})_{\varpi^{-1}(D \cap U)}) 
\subset \SP_{\varpi^{-1}(D \cap U)}^{\prime}$ are 
the exponential factors of $\SM$. Moreover, if 
the divisor $D \cap U \subset U$ is smooth and connected, 
then the non-negative rational number 
\begin{equation}
\dsum_{f \in N_{\theta}^{>0}} N_{\theta}(f) \cdot {\rm ord}_{D \cap U}(f) 
\quad \in \QQ_{\geq 0}  
\end{equation}
associated to a point $\theta \in \varpi^{-1}(D \cap U)$ 
is an integer and does not depend on the choice of 
$\theta \in \varpi^{-1}(D \cap U)$, where for 
the exponential factor $f \in N_{\theta}^{>0}$ of $\SM$ 
the rational number ${\rm ord}_{D \cap U}(f)  \geq 0$ stands for 
the pole order of $f$ along $D \cap U$. We call it the 
irregularity of $\SM$ along $D \cap U$ and 
denote it by ${\rm irr}_{D \cap U} ( \SM )$. 
If $D \subset X$ itself is smooth and connected, 
we define the irregularity ${\rm irr}_{D} ( \SM ) 
\in \ZZ_{\geq 0}$ of $\SM$ along $D \subset X$ similarly. 
By Lemmas \ref{DK-lem-2} and \ref{lem-hda-dk},  
we obtain the following higher-dimensional analogue 
of \cite[Proposition 5.4.5]{DK17}. 
For a precise explanation of the proof of 
\cite[Proposition 5.4.5]{DK17}, see \cite[Remark 2.10]{KT23}. 

\begin{proposition}\label{hda-dk} 
Assume that a holonomic $\SD_X$-module $\SM$ has a 
quasi-normal form along the normal crossing 
divisor $D \subset X$. Then for any point $x \in D$ 
there exist a subanalytic open neighborhood 
$U$ of $x$ in $X$ and an 
$\RR$-constructible enhanced sheaf 
$F \in \BEC( \CC_X)$ on $X$ 
having a quasi-normal form along the normal crossing 
divisor $D \cap U \subset U$ 
such that 
\begin{equation}
\pi^{-1} \CC_{U} \otimes Sol_X^\rmE ( \SM ) 
\simeq
\pi^{-1} \CC_{U \setminus D} \otimes 
Sol_X^\rmE ( \SM ) 
\simeq
\CC_X^\rmE \Potimes F. 
\end{equation}
\end{proposition}

The following fundamental result is due to
Kedlaya and Mochizuki.

\begin{theorem}[\cite{Ked10, Ked11, Mochi11}]
For a holonomic $\SD_X$-module $\SM$ and $x\in X$,
there exist an open neighborhood $U$ of $x$, 
a closed hypersurface $Y\subset U$,
a complex manifold $X'$ and
a projective morphism $\nu : X'\to U$ such that\\
{\rm (i)} $\rm{sing.supp} (\SM)\cap U\subset Y$,\\
{\rm (ii)} $D:= \nu^{-1}(Y)$ is a normal crossing divisor in $X'$,\\
{\rm (iii)} $\nu$ induces an isomorphism $X'\bs D\simto U\bs Y$,\\
{\rm (iv)} $(\bfD \nu^\ast\SM)(\ast D)$ has a quasi-normal form along $D$.
\end{theorem}
This is a generalization of the classical 
Hukuhara-Levelt-Turrittin theorem to higher dimensions. 
Now let $X$ be a compact complex manifold, 
$Y \subset X$ a closed hypersurface and 
$\SM$ a holonomic $\SD_X$-module such that 
$\rm{sing.supp} (\SM ) \subset Y$ and 
$\SM (*Y) \simeq \SM$. Assume that there exist a 
projective morphism $\nu : Z \to X$ 
of a compact complex manifold $Z$ such that
$D:= \nu^{-1}(Y) \subset Z$ is a normal 
crossing divisor and $\nu |_{Z \setminus D}: 
Z \setminus D \to X \setminus Y$ is an 
isomorphism. Assume also that there exist a 
ramification $\rho  : Z^{\prime} \to \SU$ 
on a neighborhood $\SU$ of $D$ in $Z$, 
meromorphic functions 
\begin{equation}
\varphi_1, \varphi_2, \ldots, \varphi_k \in 
\Gamma \bigl( Z^{\prime} \ ; \ 
\SO_{Z^{\prime}}(\ast \rho^{-1}(D)) \bigr)
\end{equation}
and positive integers $m_i>0$ ($1 \leq i \leq k$) 
such that $\bfD \rho^\ast( \bfD \nu^\ast \SM |_{\SU})$
has a normal form along the normal crossing divisor 
$D^{\prime}:= \rho^{-1}(D) \subset Z^{\prime}$ 
for the set 
\begin{equation}
S= \{ [ \varphi_1], [ \varphi_2], \ldots, 
[ \varphi_k] \} \subset \SO_{Z^{\prime}}(\ast D^{\prime})
/ \SO_{Z^{\prime}}
\end{equation}
which is good at each point of $D^{\prime}$ 
and $m_i>0$ ($1 \leq i \leq k$). Then by the 
proof of Proposition \ref{hda-dk} 
we can also show that there 
exist a semi-analytic open neighborhood 
$\SV \subset \SU$ of $D$ in $Z$ and an 
$\RR$-constructible enhanced sheaf 
$F \in \BEC( \CC_Z)$ on $Z$ 
having a quasi-normal form along the normal crossing 
divisor $D$ at each point of $D$ 
such that 
\begin{equation}
\pi^{-1} \CC_{\SV \setminus D} \otimes 
Sol_Z^\rmE ( \bfD \nu^\ast \SM ) 
\simeq
\CC_Z^\rmE \Potimes F. 
\end{equation} 
Since in this situation there exists an isomorphism 
\begin{equation}
\SM \simto  \bfD \nu_\ast ( \bfD \nu^\ast \SM ), 
\end{equation} 
for the $\RR$-constructible enhanced sheaf 
$G:= \bfE \nu_\ast F \in \BEC( \CC_X)$ on $X$ 
and the open neighborhood $W:= \nu ( \SV )$ of 
$Y$ in $X$ we obtain an isomorphism 
\begin{equation}
\pi^{-1} \CC_{W \setminus Y} \otimes 
Sol_X^\rmE ( \SM ) 
\simeq
\CC_X^\rmE \Potimes G. 
\end{equation} 
As in Kudomi-Takeuchi \cite{KT23},  
one can also slightly modify the 
$\RR$-constructible enhanced sheaf 
$G \in \BEC( \CC_X)$ so that we have isomorphisms  
\begin{equation}
Sol_X^\rmE ( \SM ) \simeq 
\pi^{-1} \CC_{X \setminus Y} \otimes 
Sol_X^\rmE ( \SM ) 
\simeq
\CC_X^\rmE \Potimes G. 
\end{equation} 
By this very explicit description of 
the enhanced solution complex 
$Sol_X^\rmE ( \SM )$, in the one dimensional 
case $d_X= {\rm dim} X=1$ we can apply the 
Morse theoretical argument in \cite{IT20a} 
to improve the main results in 
\cite{DK17}. See \cite{KT23} for the details.

\subsection{Fourier Transforms of Standard Holonomic D-modules}\label{subsec:7.2}

In this subsection, let $X$ be the affine space $\CC^N$ of 
dimension $N$ and regard it as an algebraic variety endowed 
with the Zariski topology. Let $S \subset X$ be a smooth 
and connected quasi-affine subvariety of $X= \CC^N$ 
of dimension $n$ and $\SN$ an algebraic integrable connection 
on it. For some technical reason, we assume here that there exists 
an algebraic hypersurface $H \subset X$ of $X$ such that 
$S \subset X \setminus H$ and $S$ is closed in $X \setminus H$. 
This in particular implies that the inclusion map 
$i_S: S \hookrightarrow X$ is affine. Then we set 
\begin{equation}
\SM := \bfD i_{S*} \SN \simeq i_{S*} \SN \quad 
\in \Modhol(\SD_X). 
\end{equation}
By the standard operations for algebraic D-modules, 
for the initial study of Fourier transforms of general 
holonomic D-modules on $X= \CC^N$ it suffices to 
study those for such holonomic $\SD_X$-modules. Recall that the Fourier 
transform is an exact functor. 
For this reason, let us call 
$\SM$ a standard holonomic D-module on $X$. 
Let $\var{X}=\PP^N$ be the projective compactification 
of $X= \CC^N$ and $\var{S} \subset \var{X}$ the closure 
of $S$ in it. Let $\nu : Z \longrightarrow \var{X}$ be a 
proper morphism of a smooth variety $Z$ such that 
$\nu (Z)= \var{S}$, the restriction $\nu^{-1}S 
\longrightarrow S$ of $\nu$ is an isomorphism and 
$D:= \nu^{-1}( \var{S} \setminus S)= 
Z \setminus \nu^{-1}S \subset Z$ is a normal crossing 
divisor in $Z$. Let $i_S^{\prime}: S \simeq \nu^{-1}S  
\hookrightarrow Z$ be the inclusion map and 
consider the algebraic meromorphic connection 
\begin{equation}
\SM^{\prime} := \bfD i_{S*}^{\prime} \SN 
\simeq i_{S*}^{\prime} \SN  \quad 
\in \Modhol(\SD_Z) 
\end{equation}
on $Z$ for which we have an isomorphism  
\begin{equation}
\tl{\SM} \simeq 
\bfD \nu_{\ast} \SM^{\prime}. 
\end{equation}
Assume that the analytification $( \SM^{\prime})^{\an}$ 
of $\SM^{\prime}$ on $Z^{\an}$ has a quasi-normal form 
along the normal crossing divisor $D^{\an} \subset Z^{\an}$ 
and satisfies the nice property that we assumed at the 
end of Subsection \ref{subsec:7.1}. Then 
by the proof of Proposition \ref{hda-dk} there 
exists a semi-analytic open neighborhood $\SV$ of $D^{\an}$ 
in $Z^{\an}$ and an $\RR$-constructible enhanced 
sheaf $F \in \BEC( \CC_{Z^{\an}})$ on $Z^{\an}$ 
such that we have an isomorphism 
\begin{equation}
\pi^{-1} \CC_{\SV \setminus D^{\an}} \otimes 
Sol_Z^\rmE ( \SM^{\prime} ) 
\simeq
\CC_{Z^{\an}}^\rmE \Potimes F. 
\end{equation} 
For the $\RR$-constructible enhanced sheaf 
$G:= \bfE \nu_\ast F \in \BEC( \CC_{\var{X}^{\an}})$ on $\var{X}^{\an}$ 
and the semi-analytic open subset $\SW:= \nu ( \SV ) \subset \var{S}^{\an}$ of 
$\var{S}^{\an}$ containing $\partial S:=( \var{S} \setminus S)^{\an}$ 
we thus obtain an isomorphism 
\begin{equation}
\pi^{-1} \CC_{\SW \setminus \partial S} \otimes 
Sol_{\var{X}}^\rmE ( \tl{\SM} ) 
\simeq
\CC_{\var{X}^{\an}}^\rmE \Potimes G. 
\end{equation} 
Let us explain the structure of $G$ more precisely. 
First, by the consruction of $F$ there exist 
(possibly multi-valued) holomorphic functions 
$f_i: \SV \setminus D^{\an} \longrightarrow \CC$ 
($1 \leq i \leq k$) 
on $\SV \setminus D^{\an}$ and positive integers  
$m_i>0$ ($1 \leq i \leq k$) for which the 
enhanced sheaf $F \in \BEC( \CC_{Z^{\an}})$ has a 
quasi-normal form along the normal crossing 
divisor $D^{\an} \subset Z^{\an}$ in the sense of 
Definition \ref{def-qnf}. Recall that we 
constructed $F$ by glueing some enhanced 
sheaves $F_j$ on open sectors $V_j 
\subset \SV \setminus D^{\an} \subset  
Z^{\an} \setminus D^{\an}$ along $D^{\an}$ 
and for any $1 \leq i \leq k$ the analytic 
continuation of $f_i$ along any curve $\gamma 
\subset  \SV \setminus D^{\an}$ concides with 
$f_j$ for some $1 \leq j \leq k$ (see the 
proof of \cite[Proposition 5.4.5]{DK17}). 
Then we define a complex hypersurface 
$\tl{\SV \setminus D^{\an}} \subset 
( \SV \setminus D^{\an}) \times \CC$ of 
$( \SV \setminus D^{\an}) \times \CC$ by 
\begin{equation}
\tl{\SV \setminus D^{\an}} := 
\bigcup_{i=1}^k \ \{ (x,-f_i(x)) \ | \ 
x \in \SV \setminus D^{\an} \} \quad \subset 
( \SV \setminus D^{\an}) \times \CC, 
\end{equation} 
where in the right hand side we take the union 
considering all the possible branches of 
the (possibly multi-valued) functions  $f_i$. 
Since the pull-backs of $f_i$ ($1 \leq i \leq k$) 
by a ramification $\rho : Z^{\prime} \longrightarrow 
\SV$ of $\SV$ along the normal crossing divisor 
$D^{\an} \subset \SV$ are single-valued and form 
a good set of irregular values in the sense of 
Mochizuki \cite[Chapter 5]{Mochi11}, 
after shrinking $\SV$ if necessary we may assume that 
$\tl{\SV \setminus D^{\an}}$ is a smooth hypersurface in 
$(\SV \setminus D^{\an}) \times \CC$. 
Then by using the positive integers  
$m_i>0$ ($1 \leq i \leq k$) and the transition 
matrices that we used to glue the enhanced sheave $F_j$ in 
the construction of $F$ we define a local system 
$L$ on the smooth hypersurface 
$\tl{\SV \setminus D^{\an}} \subset 
( \SV \setminus D^{\an}) \times \CC$. Let 
$i: \tl{\SV \setminus D^{\an}} \hookrightarrow 
\SV \times \CC$ be the (not necessarily closed) 
embedding of $\tl{\SV \setminus D^{\an}}$ into 
$\SV \times \CC$ and set 
\begin{equation}
\SF := i_! L \quad \in \Db( \SV \times \CC ). 
\end{equation} 

\begin{lemma}\label{conper-1} 
The object $\SF [n] \in \Db( \SV \times \CC )$ is 
a perverse sheaf on $\SV \times \CC$. 
\end{lemma} 

\begin{proof}
Let $\rho : Z^{\prime} \longrightarrow 
\SV$ be a ramification 
of $\SV$ along the normal crossing divisor 
$D^{\an} \subset \SV$ such that the function 
$f_i \circ \rho$ is single-valued for any 
$1 \leq i \leq k$ and $\rho \times {\rm id}_{\CC}: 
Z^{\prime} \times \CC \longrightarrow 
\SV \times \CC$ the morphism associated to it. 
Then by the proof of Theorem \ref{the-1} (i), 
$( \rho \times {\rm id}_{\CC})^{-1} \SF$ 
is a constructible sheaf on $Z^{\prime} \times \CC$. 
Moreover we can easily see that the canonical morphism 
\begin{equation}
\SF \longrightarrow ( \rho \times {\rm id}_{\CC})_*
( \rho \times {\rm id}_{\CC})^{-1} \SF
\end{equation} 
of sheaves is injective and hence $\SF$ is a 
subsheaf of the constructible sheaf 
$( \rho \times {\rm id}_{\CC})_*
( \rho \times {\rm id}_{\CC})^{-1} \SF$. 
As the support of $\SF$ conicides with that of 
$( \rho \times {\rm id}_{\CC})_*
( \rho \times {\rm id}_{\CC})^{-1} \SF$ in 
$( \SV \setminus D^{\an}) \times \CC \subset 
\SV \times \CC$ and $\SF |_{D^{\an} \times \CC} 
\simeq 0$, we then see that $\SF$ itself is 
constructible. It follows that the restriction of 
$\SF [n]$ to $( \SV \setminus D^{\an}) \times \CC$ 
is perverse. Now let $j: ( \SV \setminus D^{\an}) \times \CC 
\hookrightarrow \SV \times \CC$ be the inclusion map. 
Then there exists an isomorphism 
\begin{equation}
\SF [n] \simeq j_!j^{-1} \SF [n]. 
\end{equation} 
From this the perversity of $\SF [n]$ follows 
as in the proof of 
Theorem \ref{the-1} (ii). 
\end{proof} 

Let $\tl{W} \subset \var{X}^{\an}$ be an open 
subset of $\var{X}^{\an}$ such that 
$\SW = \var{S}^{\an} \cap \tl{W}$ and set 
$W := \tl{W} \cap X^{\an} \subset X^{\an}$. 
Then for the restriction $\nu |_{\SV}: \SV \longrightarrow 
\tl{W}$ of $\nu$ and the morphism 
$\nu |_{\SV} \times {\rm id}_{\CC} : \SV \times \CC \longrightarrow 
\tl{W} \times \CC$ associated to it, 
we set 
\begin{equation}
\tl{\SG} :=  \rmR ( \nu |_{\SV} \times {\rm id}_{\CC} )_* \SF 
\simeq ( \nu |_{\SV} \times {\rm id}_{\CC} )_* \SF 
\quad \in \Db( \tl{W} \times \CC )
\end{equation} 
and let 
\begin{equation}
\SG := \tl{\SG} |_{W \times \CC} 
\quad \in \Db( W \times \CC )
\end{equation} 
be the restriction of $\tl{\SG}$ to the open 
subset $W \times \CC$ of $X^{\an} \times \CC$. 

\begin{lemma}\label{conper-2} 
The object $\SG [n]  \in \Db( W \times \CC )$ is 
a perverse sheaf on $W \times \CC$. 
\end{lemma} 

\begin{proof}
As $\nu |_{\SV} \times {\rm id}_{\CC}$ is a finite map, 
the constructibility of $\SG [n]$ is clear. 
By the construction of $\SG$, 
it is also clear that the 
restriction of $\SG [n]$ to the open 
subset $(W \setminus H^{\an}) \times \CC$ of $W \times \CC$ 
is perverse.  
Let $j_0: (W \setminus H^{\an}) \times \CC 
\hookrightarrow W \times \CC$ be the inclusion map. 
Then there exists an isomorphism 
\begin{equation}
\SG [n] \simeq j_{0!}j_0^{-1} \SG [n]. 
\end{equation} 
From this the perversity of $\SG [n]$ follows 
as in the proof of 
Theorem \ref{the-1} (ii). 
\end{proof} 

By Lemma \ref{conper-2} the micro-support ${\rm SS} ( \SG [n] )$ of 
the perverse sheaf $\SG [n]$ is a homogeneous complex Lagrangian 
analytic subset of $T^* (W \times \CC )$.  
Then as in Section \ref{sec:19}, 
by forgetting the homogeneity of ${\rm SS} ( \SG [n] )$ 
we define the following subsets: 
\begin{equation}
{\rm SS}^{{\rm E}, \CC}( \SG [n] )
\subset (T^*W) \times \CC, \qquad 
{\rm SS}_{{\rm irr}}^{\CC}( \SG [n] ) \subset T^*W. 
\end{equation}
Similarly we can show that 
$\Lambda := {\rm SS}_{{\rm irr}}^{\CC}( \SG [n] )$ is a 
(not necessarily 
homogeneous) complex Lagrangian analytic 
subset of $T^*W$. We call it 
the irregular micro-support of $\SG [n]$. Moreover, 
we obtain a Lagrangian cycle ${\rm CC}_{{\rm irr}}( \SM )$ 
supported on $\Lambda \subset T^*W$ and call it the 
irregular characteristic 
cycle of $\SM$ (for the case $N=1$ see 
Kudomi-Takeuchi \cite{KT23}). By the isomorphism 
$\SV \setminus D^{\an} \simeq \SW \setminus \partial S$ 
we regard $f_i$ ($1 \leq i \leq k$) as (possibly 
multi-valued) holomorphic functions on the open subset $\SW_0:= 
\SW \setminus \partial S \subset S^{\an}$ of $S^{\an}$ 
and set 
\begin{equation}
\Lambda_0 := 
\bigcup_{i=1}^k \  \{ (x,df_i(x)) \ | \ 
x \in \SW_0 \} \quad \subset 
T^* \SW_0 \subset T^*S^{\an}, 
\end{equation} 
where in the right hand side we take the union 
considering all the possible branches of 
the (possibly multi-valued) functions  $f_i$. 
Set $W_0 := W \setminus H^{\an}$ and let 
\begin{equation}
T^*W_0
\overset{\varpi_0}{\longleftarrow}
\SW_0 \times_{W_0} T^*W_0
\overset{\rho_0}{\longrightarrow}
T^* \SW_0
\end{equation}
be the natural morphisms associated to 
the closed embedding 
$\SW_0 \hookrightarrow W_0$. Then it is easy to see 
that for the open subset $T^*W_0 \subset T^*W$ of 
$T^*W$ we have 
\begin{equation}
\Lambda \cap (T^*W_0)= \varpi_0 \rho_0^{-1} \Lambda_0. 
\end{equation}

\begin{remark}
Although the pull-backs of $f_i$ ($1 \leq i \leq k$) 
by a ramification $\rho : Z^{\prime} \longrightarrow 
\SV$ of $\SV$ along the normal crossing divisor 
$D^{\an} \subset \SV$ form 
a good set of irregular values in the sense of 
\cite[Chapter 5]{Mochi11}, even after shrinking $\SV$ 
the complex Lagrangian analytic subset 
$\Lambda_0 \subset T^* \SW_0$ may be singular. 
\end{remark} 

From now on, assume also that there exists a non-empty 
Zariski open subset $\Omega \subset Y$ of $Y= \CC^N$ and 
$R > 0$ such that for the open subset $A(R):= \{ 
w \in Y^{\an}= \CC^N \ | \ ||w||>R \}$ of $Y^{\an}$ 
the restriction $q^{-1}( \Omega \cap A(R) ) \cap 
\Lambda \longrightarrow  \Omega \cap A(R)$ of the projection 
$q : X\times Y \longrightarrow Y$ is 
an unramified finite covering and any connected component 
of the open subset $q^{-1}( \Omega \cap A(R) ) \cap 
\Lambda \subset \Lambda$ is a fiber bundle over a 
complex manifold in $\SW \subset \var{S}^{\an}$ contained 
in $\SW \setminus \partial S \subset S^{\an}$ or 
$\partial S \cap X^{\an}$ (see Lemma \ref{lem-fin-cover}). 
Fix a sufficiently large such $R>0$ and let 
$V \subset \Omega \cap A(R)$ be 
a contractible open subset of $\Omega \cap A(R)$. 
Then for the decomposition 
\begin{equation}
q^{-1}(V)\cap \Lambda
= \Lambda_{V,1} 
\sqcup \Lambda_{V,2} \sqcup \cdots \cdots \sqcup \Lambda_{V,d}
\end{equation}
of $q^{-1}(V)\cap \Lambda$ into its connected components 
$\Lambda_{V,i} \subset \Lambda$ 
($1 \leq i \leq d$) 
the morphism $q|_{\Lambda}$ induces an isomorphism 
$\Lambda_{V,i} \simto V$ for any $1 \leq i \leq d$. 
In this situation, as in Section \ref{sec:9}, 
by $\Lambda$ and $\SG [n]  \in \Dbc ( W \times \CC )$ 
we define holomorphic functions $g_i: \longrightarrow \CC$ 
($1 \leq i \leq d$) and positive integers $m(i) \geq 1$ 
($1 \leq i \leq d$) to obtain the following result. 
For the proof, we use also a method similar to the one used in  
Kudomi-Takeuchi \cite{KT23}. 
 
\begin{theorem}\label{thm-ext}
In the situation as above, for the standard 
holonomic D-module 
$\SM = \bfD i_{S*} \SN \simeq i_{S*} \SN \quad 
\in \Modhol(\SD_X)$ on $X$ there exists an isomorphism
\begin{equation}
\pi^{-1}\CC_V \otimes
\big(Sol_{\var Y}^\rmE( \tl{\SM^\wedge} )\big)
\simeq
\bigoplus_{i=1}^d 
\bigl( \EE_{V^{\an}| \var{Y}^{\an}}^{{\rm Re}g_i}
\bigr)^{\oplus m(i)}
\end{equation}
of enhanced ind-sheaves on $\var{Y}^{\an}$. 
\end{theorem}

\subsection{Some Auxiliary Results}\label{subsec:7.3}

Recall that by 
Theorem \ref{thm-4} (vi) for an exponential D-module $\mathscr{E}_{U | X}^f$ 
on $X$ associated to a meromorphic function 
$f\in\SO_X(\ast D)$ along $D \subset X$ and $U=X \setminus D$ 
we have an isomorphism in $\BEC(\I\CC_X)$
\begin{equation}
Sol_X^\rmE\big( \mathscr{E}_{U | X}^f \big) 
\simeq \EE_{U | X}^{\Re f}.
\end{equation}
This formula for the enhanced solution complex 
$Sol_X^\rmE\big( \mathscr{E}_{U | X}^f \big)$ played 
a central role in the proof of the main results 
in \cite{DK16}. Aiming at the study of the 
Fourier transforms of general holonomic D-modules, 
we shall extend it to more general meromorphic 
connections. Our argument below is a 
higher-dimensional analogue of those 
of Kashiwara-Schapira \cite[Section 7]{KS03} 
and Morando \cite[Section 2.1]{Morando}. Assuming 
that $M$ is a real analytic manifold we 
recall some basic definitions and results. 

\begin{definition}\label{def-temp} 
Let $U \subset M$ be an open subset and $f:U \longrightarrow 
\CC$ a $\CC$-valued $C^{\infty}$-function on it. 
Then for a point $p \in M$ we say that $f$ has 
polynomial growth at $p$ if for a 
local coordinate system $x=(x_1,x_2. \ldots, x_n)$ 
of $M$ around $p$ there exists a compact 
neighborhood $K$ of $p$ in $M$ and $N>0$ 
such that 
\begin{equation}
\sup_{x \in K \setminus U} 
{\rm dist} (x, K \setminus U)^N \cdot 
 | f(x) | < + \infty. 
\end{equation}
Obviously $f$ has polynomial growth at any 
point of $U$. We say that $f$ is tempered 
at $p \in M$ if all its derivatives have 
polynomial growth at $p$. Moreover we 
say that $f$ is tempered on $U$ if it 
is tempered at any point $p \in M$. 
\end{definition}
For an open subset $U \subset M$ we denote 
by $C^{\infty, t}_M(U)$ the $\CC$-vector 
space consisting of $\CC$-valued $C^{\infty}$-functions 
on $U$ that are tempered on $U$. Then for 
any pair $(U_1,U_2)$ of relatively compact 
subanalytic open subsets $U_1, U_2 \subset M$ 
we have an exact sequence 
\begin{equation}
0  \longrightarrow C^{\infty, t}_M(U_1 \cup U_2)
\longrightarrow 
C^{\infty, t}_M(U_1) \oplus C^{\infty, t}_M(U_2)
\longrightarrow C^{\infty, t}_M(U_1 \cap U_2)
\longrightarrow 0
\end{equation}
and hence we get a sheaf $C^{\infty, t}_M$ on the 
subanalytic site of $M$ in the sense of 
\cite{KS01}. Note that this important result 
is an immediate consequence of the following 
Lojasiewicz's inequality (see Bierstone-Milman \cite{B-M} and 
Lojasiewicz \cite{Lojasiewicz} for the details).  

\begin{proposition}[Lojasiewicz's inequality]\label{Lojas}
Let $U \subset \RR^n_x$ be an open subset of 
$\RR^n_x$ and $Z_1, Z_2 \subset U$ its closed subanalytic 
subsets. Then for any point 
$p \in Z_1 \cap Z_2$ there exist an open neighborhood 
$V \subset U$ of $p$ in $U$ and positive real numbers 
$C, r>0$ such that we have 
\begin{equation}
{\rm dist} (x, Z_1) + {\rm dist} (x, Z_2) \geq 
 C \cdot {\rm dist} (x, Z_1 \cap Z_2)^r 
\end{equation}
for any $x \in V$. 
\end{proposition}
Following \cite[Definition 1.1.4]{Morando}, for 
a real analytic curve $\gamma (t) : [0, \varepsilon )
\longrightarrow M$ ($\varepsilon >0$) on $M$ 
(defined on a neighborhood of $0 \in \RR$) 
we call the subset $\Gamma := \gamma ((0, \varepsilon )) 
\subset M$ of its image a semi-analytic arc 
with an endpoint $\gamma (0) \in \overline{\Gamma}$. 
In what follows, let $X$ be a complex manifold and 
consider $\CC$-valued $C^{\infty}$-functions on 
its underlying real analytic manifold 
$X_{\RR}$. Then we have the following 
higher-dimensional analogue of 
Kashiwara-Schapira \cite[Lemma 7.2]{KS03} 
and Morando \cite[Proposition 2.1.1]{Morando}. 

\begin{proposition}\label{bdd.vs.temp}
Let $X$ be a compact complex manifold, 
$D \subset X$ a normal crossing divisor on it, 
and $f\in\SO_X(\ast D)$ a meromorphic function  
along $D$. Set $U:=X \setminus D$. 
Assume that $f$ does not have any point of 
indeterminacy on the whole $X$ and has a pole 
along each irreducible component of $D$. 
Then for any relatively compact subanalytic 
open subset $W \subset U= X \setminus D$ 
the following conditions are equivalent. 
\begin{enumerate}
\item[\rm{(i)}] The function 
$\Re f|_W: W \longrightarrow \RR$ is 
bounded from above. 
\item[\rm{(ii)}] The $\CC$-valued 
$C^{\infty}$-function 
$\exp (f)|_W: W \longrightarrow \CC$ is 
is tempered on $W \subset X_{\RR}$. 
\end{enumerate}
\end{proposition}

\begin{proof}
Since $f$ has no point of 
indeterminacy on the whole $X$, it defines 
a holomorphic map from $X$ to $\PP = \PP^1$. 
We denote it also by $f$. 
This non-constant map $f: 
X \longrightarrow \PP$ being 
proper and open, we obtain a 
relatively compact subanalytic open 
subset $f(W) \subset \PP$ of $\PP$ 
such that $f(W) \subset \PP \setminus 
\{ \infty \} = \CC$. Since the proof of 
(i) $\Longrightarrow$ (ii) is trivial, 
we shall prove only (ii) $\Longrightarrow$ (i). 
We prove it by showing a contradiction. 
Suppose that the function 
$\Re f|_W: W \longrightarrow \RR$ is not 
bounded from above. Then there exists a 
sequence $p_k \in W$ ($k=1,2,3, \ldots$) 
such that 
\begin{equation}
\lim_{k \to + \infty} {\rm Re} f (p_k) =+ \infty. 
\end{equation}
By taking a subsequence of it, we may assume also that 
it converges to a point $p \in X$. Then by the 
condition $f\in\SO_X(\ast D)$, 
we have $p \in D \cap \overline{W}$. By the 
curve selection lemma, there exists a 
real analytic curve $\gamma (t) : [0, \varepsilon )
\longrightarrow \overline{W}$ ($\varepsilon >0$) 
(defined on a neighborhood of $0 \in \RR$) 
such that $\gamma (0) =p \in D \cap \overline{W}$,  
$\gamma (t) \in W$ for any $t \in (0, \varepsilon )$, and 
\begin{equation}
\lim_{t \to +0} \Re f ( \gamma (t)) =+ \infty. 
\end{equation}
Consider the semi-analytic arc  
$\Gamma := \gamma ((0, \varepsilon )) 
\subset W \subset X_{\RR}$ with an endpoint 
$p= \gamma (0) \in D \cap \overline{W}$. 
Let $z=(z_1,z_2, \ldots, z_n)$ be a holomorphic 
coordinate system of $X$ on a neighborhood 
$V$ of the point $p \in D$ such that we have 
$\{ p \} = \{ z=0 \}$ and 
\begin{equation}
D= \{ z_1z_2 \cdots z_l=0 \}, \quad 
f(z)=\frac{1}{z_1^{m_1}z_2^{m_2} \cdots z_l^{m_l}}
\end{equation}
for some $1 \leq l \leq n$ and $m_i \in \ZZ_{>0}$ 
($1 \leq i \leq l$). Shrinking the semi-analytic 
arc $\Gamma \subset W$ if necessary, we may assume 
that $\overline{\Gamma} \subset V \subset \CC^n$. 
Then we can apply Proposition \ref{Lojas} 
the pair $(Z_1,Z_2)$ of the closed subanalytic 
subsets $Z_1= \overline{\Gamma}$ and 
$Z_2= \partial W$ of $V \subset \CC^n$ 
(satisfying the condition $Z_1 \cap Z_2= \{ p \} 
= \{ z=0 \}$) to show that there exist an open 
neighborhood $V_1 \subset V$ of $p$ in $V$ and 
positive real numbers $C_1,r_1>0$ such that we have 
\begin{equation}
 || z || = {\rm dist} (z,  \{ p \} ) \leq 
 C_1 \cdot {\rm dist} (z, \partial W )^{r_1}
\end{equation}
for any $z \in \overline{\Gamma} \cap V_1 
\subset V \subset \CC^n$. On the other hand, 
by the condition $\{ p \} = \{ z=0 \} \subset D$ 
we have also 
\begin{equation}
{\rm dist} (z, D ) \leq {\rm dist} (z,  \{ p \} )  = || z || 
\end{equation}
for any $z \in \overline{\Gamma} \cap V_1$. Set 
\begin{equation}
g(z):=\frac{1}{f(z)}= z_1^{m_1}z_2^{m_2} \cdots z_l^{m_l}. 
\end{equation}
Then we can easily show that there exist an open 
neighborhood $V_2 \subset V$ of $p$ in $V$ and 
positive real numbers $C_2,r_2>0$ such that we have 
\begin{equation}
 |g(z)| = \frac{1}{|f(z)|} \leq 
 C_2 \cdot {\rm dist} (z, D )^{r_2}
\end{equation}
for any $z \in V_2 \setminus D=V_2 \cap U$. Now let 
us set $V_0:=V_1 \cap V_2 \subset V$. Then 
there exist positive real numbers $C,r>0$ such that we have 
\begin{equation}\label{EQ-A} 
 |g(z)| \leq 
 C \cdot {\rm dist} (z, \partial W )^r
\end{equation}
for any $z \in (V_0 \cap \overline{\Gamma}) \setminus 
D= (V_0 \cap \overline{\Gamma}) \cap U$. 
Shrinking the semi-analytic 
arc $\Gamma$ once again, we may assume 
that $\Gamma \subset V_0 \cap W$. Note that the 
image $f( \Gamma ) \subset \PP \setminus 
\{ \infty \} = \CC^n$ of $\Gamma$ 
by $f: X \longrightarrow \PP$ is 
a semi-analytic arc in $\PP$ 
with an endpoint $\infty \in 
\PP$. Indeed, for the real analytic curve 
$(f \circ \gamma )(t) : [0, \varepsilon )
\longrightarrow \PP$ ($\varepsilon >0$) 
we have $(f \circ \gamma )(0) = \infty \in \PP$. 
It satisfies also the condition 
\begin{equation}
\lim_{t \to +0} \Re (f \circ \gamma ) (t) =+ \infty. 
\end{equation}
Then by the proofs of \cite[Lemma 7.2]{KS03} 
and \cite[Proposition 2.1.2]{Morando}, we see 
that for any $M>0$ and $N>0$ there exists 
a point $\tau \in f( \Gamma ) \subset \CC$ 
such that 
\begin{equation}
 | \exp ( \tau ) | > M | \tau |^N. 
\end{equation}
In particular, if we take increasing sequences $M_k>0$, 
$N_k>0$ ($k=1,2,3, \ldots$) 
satisfying the condition  
\begin{equation}\label{EQ-C} 
\lim_{k \to + \infty} N_k= + \infty, \quad 
\lim_{k \to + \infty} \frac{M_k}{C^{N_k}} = + \infty, 
\end{equation}
there exist points  $\tau_k \in f( \Gamma ) \subset \CC$ 
 ($k=1,2,3, \ldots$) such that 
\begin{equation}
 | \exp ( \tau_k ) | > M_k | \tau_k |^{N_k}. 
\end{equation}
We may assume also that $\lim_{k \to + \infty}  \tau_k 
= \infty \in \PP$. For each $k \geq 1$ we choose 
a point $z_k \in \Gamma$ such that $f(z_k)= \tau_k$. 
Then we obtain 
\begin{equation}\label{EQ-B} 
 | \exp ( f(z_k) ) | > M_k | f(z_k) |^{N_k}
= M_k | g(z_k) |^{-N_k}
\end{equation}
for any $k \geq 1$. By the condition 
\begin{equation}
\lim_{k \to + \infty} f(z_k)= 
\lim_{k \to + \infty} \tau_k = \infty, 
\end{equation}
the sequence $z_k \in \Gamma$ ($k=1,2,3, \ldots$) 
satisfies the condition 
\begin{equation}
\lim_{k \to + \infty} | g(z_k) |=0. 
\end{equation}
Then it follows also from Proposition \ref{Lojas} that 
we have 
\begin{equation}
\lim_{k \to + \infty} {\rm dist} (z_k, D ) = 
\lim_{k \to + \infty} || z_k ||= 0. 
\end{equation}
Hence we get 
\begin{equation}\label{EQ-D} 
\lim_{k \to + \infty} z_k=p \in D. 
\end{equation}
On the other hand, by \eqref{EQ-A} and 
\eqref{EQ-B} we have 
\begin{equation}
 | \exp ( f(z_k) ) | > \frac{M_k}{C^{N_k}} \cdot 
 {\rm dist} (z_k, \partial W )^{-rN_k} 
\end{equation}
for any $k \geq 1$. By \eqref{EQ-C} and \eqref{EQ-D} 
this implies that the $\CC$-valued 
$C^{\infty}$-function 
$\exp (f)|_W: W \longrightarrow \CC$ 
does not have polynomial growth 
at $p \in D \cap \overline{W}$. 
This completes the proof. 
\end{proof}

As in the the proofs of \cite[Proposition 7.3]{KS03} 
and \cite[Lemma 3.1.1]{Morando}, 
by Proposition \ref{bdd.vs.temp} we obtain 
the following result. 

\begin{theorem}\label{SOLT} 
In the situation of Proposition \ref{bdd.vs.temp}, 
for $R>0$ we set 
\begin{equation}
W_R:= \{ z \in U \ | \ {\rm Re } f(z) <R \} \subset X. 
\end{equation}
Then we have an isomorphism 
\begin{equation}
H^0 Sol_X^t \big( \mathscr{E}_{U | X}^f \big) 
\simeq 
\underset{R \to+\infty}{\inj}\ 
( \CC_X)_{W_R}. 
\end{equation}
\end{theorem}

\begin{proposition}\label{tem-sol}
Let $D \subset X$ be a normal crossing divisor in $X$ and 
assume that a holonomic $\SD_X$-module $\SM$ has a 
quasi-normal form along $D$. Let $i_D:D 
\hookrightarrow X$ be the inclusion map. Then for 
any $j \geq 1$ there exists a complex constructible 
``sheaf" $F_j$ on $D$ such that we have 
\begin{equation}
H^j Sol_X^t \big( \SM \big) \simeq (i_D)_* F_j. 
\end{equation}
\end{proposition}

\begin{proof}
For a point $x \in D$, by our assumption 
there exists a ramification map $\rho : X'\to U$ 
on a neighborhood $U$ of $x$ such that $\bfD \rho^\ast(\SM|_U)$
has a normal form along the normal 
crossing divisor $D' := 
\rho^{-1}(D\cap U) \simeq D \cap U$ in $X'$. Let 
$i_{D'}:D' \hookrightarrow X'$ be the inclusion map. 
Then by the proof of \cite[Proposition 3.14]{IT20a} 
for any $j \geq 1$ there exists a complex constructible 
``sheaf" $G_j$ on $D'$ such that we have 
\begin{equation}
H^j Sol_{X'}^t \big( \bfD \rho^\ast(\SM|_U) \big) 
\simeq (i_{D'})_* G_j. 
\end{equation}
Moreover, by \cite[Theorem 7.4.6]{KS01} and 
\cite[Propositon 4.39]{K-book} we have an isomorphism 
\begin{equation}
Sol_{U}^t \big( \bfD \rho_\ast \bfD \rho^\ast(\SM|_U) \big) 
\simeq \rmR \rho_* 
Sol_{X'}^t \big( \bfD \rho^\ast(\SM|_U) \big). 
\end{equation}
The ramification map $\rho$ being finite, for any $j \geq 1$ 
we obtain an isomorphism 
\begin{equation}
H^j Sol_{U}^t \big( \bfD \rho_\ast \bfD \rho^\ast(\SM|_U) \big) 
\simeq 
( \rho \circ i_{D'})_* G_j. 
\end{equation}
Since $\SM|_U$ is a 
direct summand of $\bfD \rho_\ast\bfD \rho^\ast(\SM|_U)$, 
the same is true also for 
$Sol_{U}^t \big( \SM|_U \big)$ and 
$Sol_{U}^t \big( \bfD \rho_\ast \bfD \rho^\ast(\SM|_U) \big)$. 
Hence we verified the assertion locally on 
$U \subset X$. We can easily check that it holds 
also globally on $X$. 
\end{proof}

From now on, let $\SM$ be a holonomic D-module 
on $X$ such that for a closed hypersurface $D 
\subset X$ of $X$ we have \\
(i) $\SM\simto\SM(\ast D)$\\
(ii) $\rm{sing.supp}(\SM)\subset D$. \\
Set $U:= X \setminus D$ and let $j:U \hookrightarrow 
X$ be the inclusion map. Then 
$L:= \mathcal{H} om_{\SD_U}( \SM |_U, \SO_U)$ 
is a local system on $U$. We denote its 
rank by $r \geq 0$. Let $\R \in \Modhol(\SD_X)$ 
be the regular meromorphic 
connection on $X$ along $D$ such that we have an 
isomorphism 
\begin{equation}
Sol_{X} ( \R) \simeq j_!L 
\end{equation}
(see e.g. \cite[Theorem 5.3.8]{HTT08}). We thus obtain an isomorphism 
\begin{equation}
\Phi: L= \mathcal{H} om_{\SD_U}( \SM |_U, \SO_U) 
\simeq \mathcal{H} om_{\SD_U}( \R |_U, \SO_U)
\end{equation}
of local systems on $U$. Fix a point $p \in U$ and 
let $\phi_1, \ldots, \phi_r \in \mathcal{H} om_{\SD_U}( \R |_U, \SO_U)$ 
(resp. $\psi_1, \ldots, \psi_r \in 
\mathcal{H} om_{\SD_U}( \SM |_U, \SO_U)$) be the 
$\CC$-basis of the local system $\mathcal{H} om_{\SD_U}( \R |_U, \SO_U)$ 
(resp. $L= \mathcal{H} om_{\SD_U}( \SM |_U, \SO_U)$) on 
a neighborhood of $p$ such that $\Phi ( \psi_i)= \phi_i$ 
for any $1 \leq i \leq r$. Assume that there exist 
holomorphic functions $f_1. \ldots, f_r$ defined on a 
neighborhood of $p$ such that 
\begin{equation}
\psi_i= \exp (f_i) \cdot \phi_i \quad (1 \leq i \leq r). 
\end{equation}
As $\phi_i$ (resp. $\psi_i$) ($1 \leq i \leq r$) 
extend to multi-valued global sections of 
the local system $\mathcal{H} om_{\SD_U}( \R |_U, \SO_U)$ 
(resp. $\mathcal{H} om_{\SD_U}( \SM |_U, \SO_U)$), 
the same is true also for the holomorphic functions 
$f_i$ ($1 \leq i \leq r$). Since we 
have $\Phi ( \psi_i)= \phi_i$ even after the extensions, 
their real parts ${\rm Re} f_i: U \longrightarrow \RR$ 
are single-valued. Moreover, the isomorphism $\Phi$ 
of sheaves being compatible with restrictions, 
if the section $\phi_i$ is continued to 
$\sum_{i=1}^r c_i \phi_i$ ($c_i \in \CC$) along 
a closed continuous curve $\gamma (t):[0,1] \longrightarrow 
U$ in $U$ such that $\gamma (0)= \gamma (1)=p$, 
the section $\psi_i$ is continued to $\sum_{i=1}^r c_i \psi_i$ 
along it. By this observation, we see that the 
enhanced sheaf 
\begin{equation}
\bigoplus_{i=1}^r \CC_{ \{ (z,t) \in X \times \RR \ | \ z \in U,\  
t + {\rm Re} f_i(z) \geq 0 \} }
\end{equation}
defined on a neighborhood of $p$ can be naturally 
extended to the one $F$ on the whole $X$ such that 
there exists a surjective morphism 
$\pi^{-1}j_! L \longrightarrow F$ and we have 
$\pi^{-1} \CC_U \otimes F \simeq F$. 

\begin{theorem}\label{SOLE} 
In the situation as above, 
we have an isomorphism 
\begin{equation}
Sol_{X}^\rmE ( \SM ) \simeq 
\CC_X^\rmE \Potimes F. 
\end{equation}
\end{theorem}

\begin{proof}
Let $\nu : Y \longrightarrow X$ be a projective 
morphism of complex manifolds such that 
$E:= \nu^{-1} D \subset Y$ is a normal crossing 
divisor in $Y$, the restriction 
$\nu |_{Y \setminus E}: Y \setminus E 
\longrightarrow U= X \setminus D$ of 
$\nu$ is an isomorphism and the holonomic 
D-module $\SN := {\bfD} \nu^\ast \SM$ 
on $Y$ has a quasi-normal form along $E$. 
Set $V:=Y \setminus E = \nu^{-1} U$. 
By Theorem \ref{thm-4} (ii) and (v), 
we have 
\begin{equation}
\pi^{-1} \CC_U \otimes 
 Sol_{X}^\rmE ( \SM ) 
\simeq  Sol_{X}^\rmE ( \SM )
\end{equation}
and 
\begin{equation}
\pi^{-1} \CC_V \otimes 
 Sol_{Y}^\rmE ( \SN ) 
\simeq  Sol_{Y}^\rmE ( \SN ) 
\simeq \bfE \nu^{-1} 
Sol_{X}^\rmE ( \SM ). 
\end{equation}
Since $\nu |_{Y \setminus E}: V  
\longrightarrow U$ is an isomorphism, 
there exist also isomorphisms 
\begin{align*}
\bfE \nu_* Sol_{Y}^\rmE ( \SN ) 
\simeq 
\bfE \nu_* \Bigl( \pi^{-1} \CC_V \otimes 
\bfE \nu^{-1} Sol_{X}^\rmE ( \SM ) \Bigr) \\
\simeq  \pi^{-1} \CC_U \otimes 
 Sol_{X}^\rmE ( \SM ) 
\simeq  Sol_{X}^\rmE ( \SM ). 
\end{align*} 
Then it suffices to calculate $Sol_{Y}^\rmE ( \SN )$ 
for $\SN$ having a quasi-normal 
form along $E \subset Y$. Let 
\begin{equation}
i: Y \times \RR_{\infty} \longrightarrow 
Y \times \PP
\end{equation}
be the morphism of bordered spaces 
obtained by composing the natural ones 
$Y \times \RR_{\infty} \longrightarrow 
Y \times P$ and $Y \times P \hookrightarrow 
Y \times \PP$ for the real projective line 
$P = \RR \sqcup \{ \infty \}$. Then 
we have an isomorphism 
\begin{equation}
Sol_{Y}^\rmE ( \SN ) \simeq 
i^! Sol^t_{Y \times \PP} \Bigl( \SN \boxtimes 
\mathscr{E}_{\CC | \PP}^{\tau} \Bigr) [2]  
\end{equation}
for the holonomic D-module $\SN \boxtimes 
\mathscr{E}_{\CC | \PP}^{\tau}$ on $Y \times \PP$ 
having a quasi-normal form along 
the normal crossing divisor 
\begin{equation}
E^{\prime}:= (E \times \PP ) \cup 
(Y \times \{ \infty \} ) \subset Y \times \PP  
\end{equation}
in $Y \times \PP$. 
Note that by Theorem \ref{thm-4} (v) we have also 
\begin{equation}
Sol_{Y}^\rmE ( \SN ) \simeq 
\pi^{-1} \CC_V \otimes 
i^! Sol^t_{Y \times \PP} \Bigl( \SN \boxtimes 
\mathscr{E}_{\CC | \PP}^{\tau} \Bigr) [2]  
\end{equation}
Now let us consider the distinguished triangle 
\begin{equation}\label{Dist-t} 
H^0 Sol^t_{Y \times \PP} \Bigl( \SN \boxtimes 
\mathscr{E}_{\CC | \PP}^{\tau} \Bigr) 
\longrightarrow 
Sol^t_{Y \times \PP} \Bigl( \SN \boxtimes 
\mathscr{E}_{\CC | \PP}^{\tau} \Bigr) 
\longrightarrow 
\tau^{\geq 1} 
Sol^t_{Y \times \PP} \Bigl( \SN \boxtimes 
\mathscr{E}_{\CC | \PP}^{\tau} \Bigr) 
\overset{+1}{\longrightarrow}. 
\end{equation}
Then by Proposition \ref{tem-sol} we can easily show 
the vanishing 
\begin{equation}
\pi^{-1} \CC_V \otimes 
i^! \Bigl\{ \tau^{\geq 1} 
Sol^t_{Y \times \PP} \Bigl( \SN \boxtimes 
\mathscr{E}_{\CC | \PP}^{\tau} \Bigr) \Bigr\} 
\simeq 0. 
\end{equation}
Moreover, by the proofs of Proposition \ref{bdd.vs.temp} 
and Theorem \ref{SOLT} we obtain an isomorphism 
\begin{equation}
H^0 Sol^t_{Y \times \PP} \Bigl( \SN \boxtimes 
\mathscr{E}_{\CC | \PP}^{\tau} \Bigr) 
\simeq \underset{R \to+\infty}{\inj}\ 
\iota_! \bigl( {\rm id}_Y \times {\rm Re} \bigr)^{-1} G_R, 
\end{equation} 
where $\iota: Y \times \CC \hookrightarrow 
Y \times \PP$ is the inclusion map, 
the morphism ${\rm id}_Y \times {\rm Re}: Y \times 
\CC \longrightarrow Y \times \RR$ is induced by 
the one ${\rm Re}: \CC \longrightarrow \RR$ 
($\tau \longmapsto {\rm Re} \tau$) and the enhanced 
sheaf $G_R$ ($R \geq 0$) on $Y$ is a natural extension of 
the one 
\begin{equation}
\bigoplus_{i=1}^r \CC_{ \{ (w,t) \in Y \times \RR \ | \ w \in V, \ 
t + {\rm Re} (f_i \circ \nu )(w) <R \} }
\end{equation}
defined on a neighborhood of the point $q= \nu^{-1}(p) 
\in V$ such that there exists an injective morphism 
$G_R \hookrightarrow  \pi^{-1} \nu^{-1}(j_!L)$ and we have 
$\pi^{-1} \CC_V \otimes G_R \simeq G_R$. 
Indeed, the sheaf $\iota_! \bigl( {\rm id}_Y 
\times {\rm Re} \bigr)^{-1} G_R$ on $Y \times \PP$ 
is an extension of the one 
\begin{equation}
\bigoplus_{i=1}^r \CC_{ \{ (w, \tau ) \in Y \times \PP \ | \ w \in V, \  
\tau \in \CC, \ {\rm Re} ( \tau + (f_i \circ \nu ) (w) ) <R \} }
\end{equation}
to the whole $Y \times \PP$. Then, as in the proof of 
\cite[Lemma 9.3.1]{DK16} we obtain isomorphisms 
\begin{align*}
 Sol_{Y}^\rmE ( \SN ) 
& \simeq 
\pi^{-1} \CC_V \otimes 
i^! \Bigl\{ H^0 Sol^t_{Y \times \PP} \Bigl( \SN \boxtimes 
\mathscr{E}_{\CC | \PP}^{\tau} \Bigr) \Bigr\} [2]  
\\
& \simeq 
\pi^{-1} \CC_V \otimes 
\underset{R \to+\infty}{\inj}\ 
i^! \Bigl\{ 
\iota_! \bigl( {\rm id}_Y \times {\rm Re} \bigr)^{-1} G_R 
\Bigr\} [2] 
\\
& \simeq 
\pi^{-1} \CC_V \otimes 
\underset{R \to+\infty}{\inj}\ G_R [1] 
\\
& \simeq 
\pi^{-1} \CC_V \otimes 
\Bigl( \CC_Y^\rmE \Potimes \tilde{\nu}^{-1} F \Bigr) 
\\
& \simeq 
\pi^{-1} \CC_V \otimes 
\bfE \nu^{-1} 
\Bigl( \CC_X^\rmE \Potimes F \Bigr), 
\end{align*} 
where we set $\tilde{\nu}:= \nu \times {\rm id}_{\RR}: 
Y \times \RR \longrightarrow X \times \RR$ and in 
the third isomorphism we used the one 
$G_0[1] \simeq \tilde{\nu}^{-1} F$ in the 
category $\BEC(\CC_Y)$. It follows that we get 
the desired isomorphism 
\begin{equation}
Sol_{X}^\rmE ( \SM ) \simeq 
\bfE \nu_* Sol_{Y}^\rmE ( \SN ) \simeq \CC_X^\rmE \Potimes F. 
\end{equation}
This completes the proof. 
\end{proof}

\end{document}